%% file: thesis.tex
\documentclass[11pt,twoside]{report}

\input{preamble}

\begin{document}

\begin{titlepage}
\null\setcounter{page}{1}\vfill
  \vbox to 0pt{\vss\vbox to 8.75in{\parskip 0pt \parindent 0pt\centering
      {\large \uppercase\expandafter{Texas A\&M University} \par} \vskip 0pt plus 3fil
      {\Large \bf On Maximizing Weighted Algebraic Connectivity for Synthesizing Robust Networks \par} \vskip 0pt plus 1.3fil
      {\large by \par} \vskip 0pt plus 1fil
      {\large \bf Harsha Nagarajan \par} \vskip 0pt plus 3fil
      {\large \bf Chair: Dr. Swaroop Darbha, Co-chair: Dr. Sivakumar Rathinam \par} 
      {\large \bf Members: Dr. K.R. Rajagopal, Dr. Sergiy Butenko \par} \vskip 0pt plus 3fil
      
   }}
\end{titlepage}

\include{Chapters/abstract}

\pagenumbering{roman}
\pagestyle{plain}
   \tableofcontents
   \newpage
   \listoffigures
   \newpage
   \listoftables

\newpage

\pagenumbering{arabic}
\pagestyle{fancy}

\sloppy


\include{Chapters/section1}

\include{Chapters/section2}

\include{Chapters/section3}

\include{Chapters/section6}

\include{Chapters/appendices}

%
%
%
%
%
%
%
%
%
%
%
\appendix
%
%
\addcontentsline{toc}{chapter}{Bibliography}
\bibliography{references}
\bibliographystyle{ieeetr}

\end{document}

%% file: preamble.tex
\usepackage{epsfig}
\usepackage{graphicx}
\usepackage{amsmath,amsthm,paralist,bm,yhmath}
\usepackage{amssymb}
\usepackage{epstopdf}
\usepackage{empheq}
\usepackage{comment}
\usepackage[algo2e,vlined,ruled]{algorithm2e}
\usepackage{amsbsy,citesort,url}
\usepackage{epstopdf}
\usepackage{algorithm,algorithmic}

 \usepackage[onehalfspacing]{setspace}
 \usepackage[subfigure]{tocloft}
 \usepackage[rm, tiny,center, compact]{titlesec}
 \usepackage{indentfirst}
 \usepackage{etoolbox}
\usepackage{tocvsec2}
 \usepackage[titletoc]{appendix}
\usepackage{rotating}

\usepackage{psfrag} 
\usepackage{graphicx} 
\usepackage{subfigure} 
\usepackage{booktabs} 
\usepackage{moreverb,fancyvrb,enumerate}

\newtheorem{lemma}{Lemma}
\newtheorem{remark}{Remark}

\makeatletter
\let\NAT@parse\undefined
\makeatother
\usepackage{subfigure}
\usepackage{psfrag}
\usepackage{url, upgreek}
\usepackage[headings]{fullpage}
\usepackage[usenames]{color}

\usepackage{dropping}
\usepackage[SonnyChin]{fncychap2}
\usepackage{fancyhdr}  
\newcommand{\br}{\mbox{\boldmath{$r$}}}

\newlength\longest

\makeatletter
\def\blfootnote{\xdef\@thefnmark{}\@footnotetext}
\makeatother

%% file: Chapters/abstract.tex
%
%
%

\chapter*{ABSTRACT}
\addcontentsline{toc}{chapter}{ABSTRACT} 

\pagestyle{plain} 
\pagenumbering{roman} 
\setcounter{page}{2}

\indent This dissertation deals with the following simpler version of an open problem in 
system realization theory which has several important engineering applications: 
Given a collection of masses and a set of linear springs with a specified cost 
and stiffness, a resource constraint in terms of a budget on the total cost, 
the problem is to determine an optimal connection of masses and springs so 
that the resulting structure is as stiff as possible, i.e., the structure 
is connected and its smallest non-zero natural frequency is as large as possible. 

One often encounters variants of this problem in deploying Unmanned Aerial 
Vehicles (UAVs) for civilian and military applications. In such problems, 
one must determine the pairs of UAVs that must maintain a communication 
link so that constraints on resources and performance, such as a limit on 
the maximum number of communication links deployed, power consumed and 
maximum latency in routing information from one UAV to the other, are met 
and a performance objective is maximized. In this dissertation, 
algebraic connectivity, or its mechanical analog - the smallest non-zero 
natural frequency of a connected structure was chosen as a performance 
objective.  Algebraic connectivity determines the convergence rate of 
consensus protocols and error attenuation in UAV formations and is chosen to be a 
performance objective as it can be viewed as a measure of robustness in 
UAV communication networks to random node failures. 

Underlying the mechanical and UAV network synthesis problems is a Mixed 
Integer Semi-Definite Program (MISDP), which was  recently shown to be a 
NP-hard problem. There has not been any systematic procedure in the 
literature to solve this problem. This dissertation is aimed at addressing 
this void in the literature. The {\it novel} contributions of this 
dissertation to the literature are as follows:
a) An iterative primal-dual algorithm and an algorithm based on the 
outer approximation of the semi-definite constraint utilizing a cutting 
plane technique were developed for computing optimal algebraic connectivity. 
These algorithms are based on a polyhedral approximation of the feasible set of MISDP, 
b) A bisection algorithm was developed to reduce the MISDP to a Binary 
Semi-Definite Program (BSDP) to achieve better computational efficiency, 
c) The feasible set of the MISDP was efficiently relaxed by replacing the 
positive semi-definite constraint with linear inequalities associated with a family of Fiedler vectors 
to compute {\it  a tighter upper bound for algebraic connectivity}, 
d) Efficient neighborhood search heuristics based on greedy methods 
such as the $k$-opt and improved $k$-opt heuristics were developed,
e) Variants of the problem occurring in UAV backbone networks and 
Air Transportation Management were considered in the dissertation 
along with numerical simulations corroborating the methodologies 
developed in this dissertation. 

\pagebreak{}

%% file: Chapters/section1.tex

%
%
%


\chapter{{Introduction}}
\label{sec:intro}

This dissertation\footnote{Online source TAMU archives: \cite{nagarajan2014thesis}} deals with the development of {\it novel} tools for addressing an {\it open}
problem in system realization theory which has relevance to several important problems in
biomedicine, altering the dynamic response of discrete and continuous systems, connectivity
of Very Large Scale Integrated (VLSI) circuits, as well as the co-ordination of Unmanned
Aerial/Ground vehicles. The simplest case of this open problem, referred to as the {\it Basic Problem}
(or simply, {\bf BP}) is the following:  {\it Given a {\it finite} set of masses, a set of linear springs
and dampers,  a given subset of springs or dampers that may only be connected between a 
specified pair of masses,  a transfer function to be realized with a {\it subset} of these
components by connecting them appropriately, the decision problem is to determine if there
is an interconnection which can accomplish this objective}. The resolution of {\bf BP} is
open and far from simple.

If we restrict ourselves to mechanical systems with springs and masses, 
and require further that the interconnections should be made so that the resulting structure
is one-dimensional,  the resulting problem has a nice connection to Graph Laplacians in
graph theory. Graph Laplacians (or simply Laplacians) play an important role in assessing 
robustness of connectivity and are similar
to stiffness matrices in discrete structural mechanical systems. The analogy may be made as
follows: a mass serves the role of a node and a spring serves the role of an edge that 
connects two nodes in a graph. If one assigns the cost of the edge to be the stiffness 
of the corresponding spring, the resulting Graph Laplacian is the same as the stiffness 
matrix that one obtains for the corresponding structural mechanical system. A brief
overview of Laplacian matrices is given in section \ref{Sec:laplacian}.

A typical transfer function in structural systems relates the input displacement or force
acting on a mass to the displacement of another mass in the structural network. In the 
absence of any damping in the structural systems, only transfer functions that have purely
imaginary zeros and poles can be realized. The poles of the system correspond to the natural
frequencies of the interconnected system, the interconnections being the sought quantities.
The zeros of the system can be thought of as the natural frequencies of an associated
constrained system obtained by setting the output displacement to be identically zero,
that is by constraining the appropriate mass to be stationary.  Essentially, the fundamental
problem of system realizability with a collection of springs and masses reduces to the 
following variant: Given a set $V$ of masses and $E$ of springs, and a set of bounds on
natural frequencies $w_{1l}, w_{1u}, w_{2l}, w_{2u}, \ldots, w_{pl}, w_{pu}$, (with $p \leq |V|$),
is there a connection of masses in which at most $q$  springs are used which results in the
interconnected structure having natural frequencies that lie between
$[w_{1l}, w_{1u}], [w_{2l}, w_{2u}], \ldots, [w_{pl}, w_{pu}]$? This is a reasonable relaxation
of the original problem which requires $w_{1l} = w_{1u}, \; w_{2l} = w_{2u}, \; \ldots, w_{pl} = w_{pu}$
in the following sense. The feasible set of the relaxed problem is bigger and the possibility 
of finding a solution should be better.

When $w_{il} = t, \; w_{iu} = \infty$ for all $i= 1, \ldots, p$, one obtains a decision
problem for a related problem involving the maximization of {\it augmented} algebraic 
connectivity in graphs. The difference between the {\bf BP} and the {\it augmented} 
algebraic connectivity maximization problem is as follows: In the {\bf BP}, none of 
the masses are connected to any springs initially, whereas in the {\it augmented} 
algebraic connectivity problem, the masses may initially be connected partially and 
one is seeking {\it additional} edges to maximize the algebraic connectivity. This 
problem was only recently shown to be NP-hard~\cite{mosk2008maximum} and may be 
stated as follows: Given an interconnected system of springs and masses, a prescribed 
number $q$ and a positive number, $t$, the decision problem is to determine if one 
can find at most $q$ {\it additional} springs that have not yet been used so as to 
make the second smallest natural frequency (which is also known as the algebraic 
connectivity for the associated Laplacian)
to be greater than $t$.  We will recall that the second smallest natural frequency is a
measure of the ``stiffness'' of the structure and the smallest natural frequency 
is always zero corresponding to the rigid body mode admitted by the structure 
(a detailed discussion on the measure for stiffness of mechanical systems can 
be found in section \ref{sec:intro_mech_systems}). 
Since every instance of the maximum augmented algebraic connectivity problem
is an instance of the system realization problem, NP-hardness of the former 
problem implies the NP-hardness of the latter problem and hence, non-trivial.

The problem of maximizing the augmented algebraic connectivity has applications to
stiffening existing or damaged structures. The need for repair and strengthening of 
damaged or deteriorated structures subject
to tight budget constraints has been an important challenge all over the world. As
discussed in \cite{bob2001rehabilitation}, there are many seismic resistant structures
built before the 1970s which are still in service beyond their design life. These existing
structures were designed with inadequate lateral load resistance because earlier building
codes  specified lower levels of seismic loads. Currently, it has been a topic of great
interest to address the problem of deficiency in the structural system by adequately 
strengthening the structural system in order to attain the desired level of seismic resistance.
Structural strengthening or rehabilitation, as defined in UNIDO (United Nations Industrial
Development Organization) manual, may consist of modification of the existing structural
members or addition of new structural members so that their structural strength, stiffness
and/or ductility are improved. An improvement in the overall stiffness of a structure can
be achieved through the addition of new structural members of known stiffness values to
increase the respective characteristics of the structure like bracing in a frame or skeleton
structure or new shear walls in a shear wall structure.  One must also take into consideration
the constraints on total budget/cost, while improving the overall stiffness of the structure.
One may abstract the problem of strengthening structures as follows: Given a budget,
a list of additional structural members to choose from along with their costs, the problem
is to augment the structure so as to make it as strong (stiff) as possible by retrofitting
additional structural members to the existing structure within the specified budget.

Algebraic connectivity, as the name indicates, is a measure of connectivity. In
the structural context, a single dimensional structure is connected if a force applied at almost any point on the
structure will influence the displacement of structure or the stress at almost all other locations
and hence, cannot admit more than a single rigid body mode. Moreover, if the structure may
be thought of as a linear discrete structural mechanical system, the structure is  tightly
connected if the stiffness of every spring is sufficiently high or equivalently, all its 
non-zero natural frequencies are sufficiently high. We may carry this analogy to 
applications involving a formation or collection of UAVs. 


\begin{figure}[htp]
	\centering
	\subfigure[Initial configuration]{
	\includegraphics[scale=0.52]{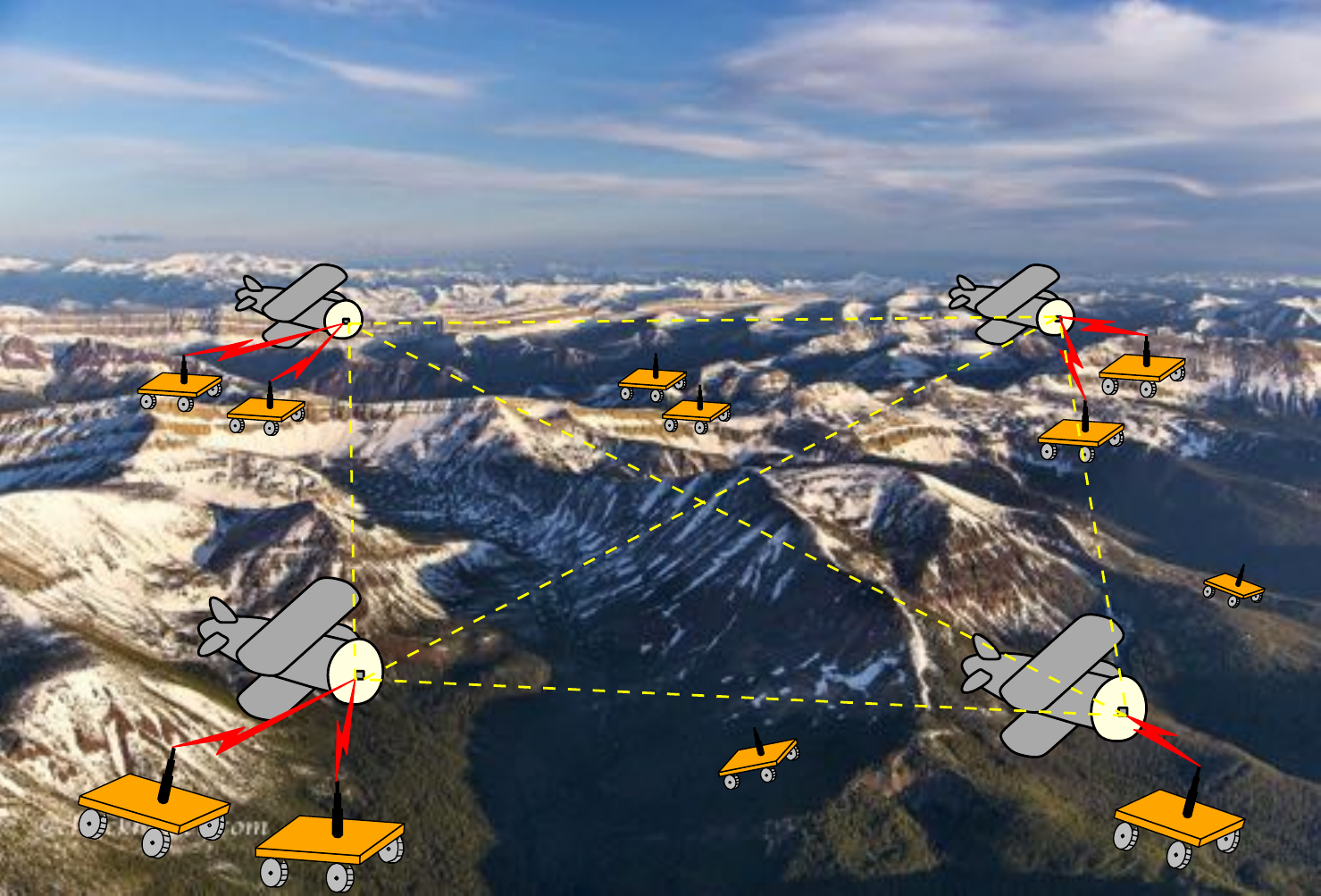}}
	\subfigure[Configuration after rigid body rotation]{
			  \includegraphics[scale=0.52]{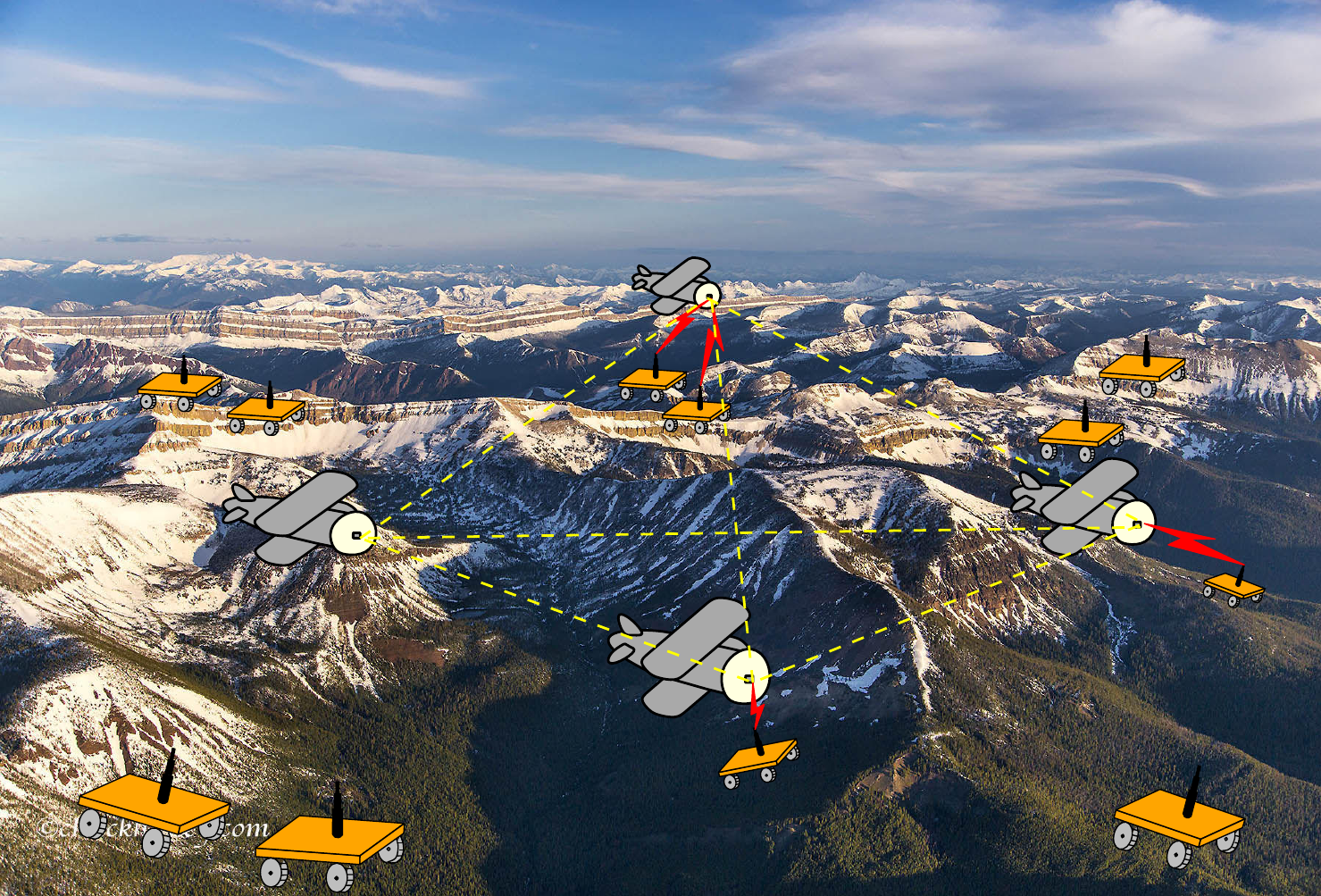}}
	\caption{In this figure, part (a) represents an initial configuration of backbone UAVs communicating with
	ground robots in disparate regions. As shown with the coloring of robots, not all robots are able to maintain
	a ground-to-air communication link with the UAVs. But in part (b), after a rigid body rotation of the backbone
	network about the centroid, the remaining ground robots are able to maintain a ground-to-air communication link with the UAVs. Source: \cite{nagarajan2015synthesizing}.}
   \label{fig:concept}
\end{figure}

This dissertation is also motivated by a scenario as shown in the figure \ref{fig:concept}.
In this scenario, there are clusters of ground robots moving in disparate regions 
that need to communicate their data and information amongst themselves. It is 
known a priori that the clusters will move slowly, and are known to be within a 
radius $R_{max}$ of their centroid. The ground-to-ground communication between 
these clusters may be hampered by obstacles such as mountains or tall buildings 
that prevent line-of-sight communications.  Since the power of a signal attenuates 
as the fourth power of distance in ground-to-ground communication, while it only 
decreases as the second power of the distance in ground-to-air and air-to-air 
communication \cite{Lehpamer,jakesmicrowave}, an ad-hoc network of UAVs is envisioned. 
The UAVs serve as backbone nodes and serve to establish communication between the 
clusters from ground. The robots use ground-to-air communication with the UAVs and the 
UAVs utilize air-to-air communication amongst themselves to reduce the overall 
power consumption for maintaining communication and transferring data.

The operational concept is as follows: The collection of UAVs maintain a 
fixed distance between them. The collection rotates about the centroid as a 
rigid body through an unit angle, stop at that configuration to facilitate 
communication with robots and then step through another unit angle and this 
procedure continues.  Associated with each UAV, one may associate a circular 
footprint on the ground; robots in the footprint can utilize ground-to-air 
communication with the UAV. As the collection of UAVs rotates about the 
centroid, the footprints of UAVs sweep/cover the area of the footprint on the ground. 
Rotating the UAVs about the centroid helps in providing a time window for 
ground-to-air communication between the robots and the UAVs without the UAVs 
having to track the robots.

Maintaining a rigid formation of UAVs provides a convenient way of 
maintaining the backbone UAV network. 
An important problem of maintaining a rigid formation is 
the problem of determining the underlying information flow graph, i.e., the 
determination of the pairs of UAVs that are maintaining communication. 
It is well-known that with a given decentralized controller as in 
\cite{yadlapalli2006information}, the convergence rate of the error in 
maintaining a desired constant spacing with respect to other UAVs in the 
formation is influenced by the algebraic connectivity of the (unweighted) 
information flow graph; if the algebraic connectivity is higher, the 
convergence is faster. In essence, this problem is identical to {\bf BP}. 

Another variant of {\bf BP} arises in the same scenario when we deal with the 
construction of an adhoc infrastructure network with UAVs,  $i.e.$, the 
determination of the relative location of UAVs as well as the pairs of UAVs 
that must maintain air-to-air communication. Since UAVs have limited battery 
power on-board, power consumption is an important issue. We use the following 
model of power consumption: to maintain a connection (or a communication link) 
between the $i^{th}$ and $j^{th}$ UAVs, the power consumed is given by 
$\alpha_{ij} d_{ij}^2$, where $d_{ij}$ is the distance between the UAVs and 
$\alpha_{ij}$ is the coefficient of proportionality and is dependent on the 
product of antenna gains of the transmitting and receiving UAVs. The 
coefficient $\alpha_{ij}$ may also be viewed as a strength of the communication 
link. If $\alpha_{ij}$ is higher, then the data rate that can be transmitted 
across the link is correspondingly higher. From the point of reducing 
interference in communication, there is an upper bound on the transmitted 
power by every UAV. This constraint limits communication between UAVs that 
are sufficiently far apart. The antennas and associated signal processing 
circuitry is typically powered by batteries on-board a UAV and this further 
limits the power that can be consumed in transmitting signals by every UAV. 
Instead of dealing with this constraint at the individual UAV level, we 
consider a surrogate constraint on the power consumption of the system as a whole. 
The total power consumed by the UAVs for maintaining air-to-air connectivity 
(or simply connectivity) is the sum of the power consumption associated with 
all the employed communication links. The total power consumption affects 
the cost of operation of the network and hence, can be treated as a resource.

One can naturally associate a graph with the network of backbone UAVs, 
with the UAVs serving as nodes, communication links being edges and a 
weight, $\alpha_{ij}$ associated with the communication link between the 
$i^{th}$ and $j^{th}$ UAVs. The desirable attributes of a communication 
network are: lower diameter so as to minimize latency in communicating 
data/information across the network, high isoperimetric number so that 
the bottlenecking in a network can only occur at higher data rates and 
robustness to node and link failures. It is known that a higher value 
of algebraic connectivity of a network is associated with a network 
with the previously mentioned desirable attributes\cite{HuijuanWang}. 
In relation to graph theory, algebraic connectivity  provides a measure 
of how weakly any subset of vertices is connected to the remaining graph. 
In this measure, a subset of vertices is considered to
be weakly connected if a normalized cut (sum of the number of edges 
leaving the subset) of the subset has a low value. Essentially, a tightly 
connected network with a larger normalized cut corresponds to a network 
with a higher algebraic connectivity. Algebraic connectivity as a measure 
of network connectivity is also superior to other measures such as the 
node or the link connectivity of a network ; for example, any (unweighted) 
spanning tree has a node or a link connectivity of one. On the other hand, 
it is known that a star network has a higher algebraic connectivity 
compared to that of any (Hamiltonian) path in the network. A star network, 
for instance, is considered to be more robust against a random removal 
of a node in the network as opposed to a path which gets disconnected 
upon the removal of any intermediate node \cite{HuijuanWang}. For this 
reason, we pose the network synthesis problem as that of determining 
the network with the maximum algebraic connectivity over all possible 
networks satisfying the given resource and operational constraints.

Simply put, a variant of the {\bf BP} that arises in this application 
is as follows: Given a collection of UAVs which can serve as backbone nodes, 
how should they be arranged and connected so that
\begin{itemize}
\item[(i)] the convex hull of the projections of their locations
on the ground spans a prespecified area of coverage,
\item[(ii)] the resources such as the total UAV power consumption for 
maintaining  connectivity  and the total number of communication
links employed are within their respective prescribed bounds, and
\item[(iii)] algebraic connectivity of the network is maximum among all 
possible networks satisfying the constraints (i) and (ii).
\end{itemize}

Variants of {\bf BP} have recently received attention in the UAV literature, 
for example, a few of the relevant references are ~\cite{kim2006maximizing}, \cite{martinez2007motion},
\cite{zavlanos2007potential}, \cite{yadlapalli2006information}, \cite{MRSR2009max_alg_conn}.
However, prior to this dissertation, a systematic and computationally efficient method for solving the
problem exactly was lacking.

Apart from mechanical systems, similar problems appear in disparate research areas including
biomedicine and VLSI circuit design. In biomedicine, of particular relevance is the field
of systems biology which aims to study the interplay between proteins, nucleic acids and 
other cellular components at the global level. In this research area, one is interested 
in engineering and achieving a desired output by either allowing certain new interactions 
or disallowing some interactions from taking place. In the simplest form, these interactions
may be modeled by systems of coupled ordinary differential equations and in more complicated
situations such as cascades of biochemical reactions that need to be controlled, the
interactions can be modeled by a system of coupled partial differential equations.

A similar problem is also encountered in VLSI circuit design\cite{CongVLSI}. Due to steady
miniaturization of VLSI devices and a quest for faster communication rates, there are
critical performance objectives placed on the design of interconnects \cite{chenVLSI},\cite{AcharVLSI}
between the components of a VLSI circuit including minimization of interconnect delays and
signal distortion, minimization of signal delays between time-critical components, minimization
of total wire length etc. A fundamental problem in VLSI circuit design \cite{CongVLSI} 
deals with designing a suitable network topology ($i.e.$, the  interconnects between the 
components) such that the specified performance objective is realized. The same problem also 
appears in disparate disciplines such as coding theory\cite{varshney2010distributed}, image
webs \cite{heath2010image}, air traffic management \cite{wei2010airport_transport,nagarajan2012icrat} 
and free space optical and communication networks \cite{son2010design},\cite{liu2009utility}.

Solving the {\bf BP}, i.e., finding the optimal network corresponding 
to the maximum value of augmented algebraic connectivity is  non-trivial. 
It is further compounded by the rapid increase in the size of the problem 
with an increase in the number of nodes (for example, masses). Even for 
instances of moderate size involving 8 identical masses, if one were 
asked to pick only 7 springs to form a connected structure, 
there are $8^6\approx 262144$ combinatorial possibilities (for a graph with $n$ masses, 
there are $n^{n-2}$ connected structures with $n-1$ springs). The difficulty is further
accentuated by the non-smooth and nonlinear nature of the objective function.
\emph{The focus of the dissertation is to develop numerical algorithms for 
computing optimal networks  as well as for computing sub-optimal 
networks along with a bound on their suboptimality.}


\section{A note on Laplacian matrix}
\label{Sec:laplacian}
A graph $G$ is specified by a set of vertices $V$, a set of edges 
$E \subset V \times V$ and a cost function $c: E \rightarrow \Re_{+}$. 
A graph $G$ is compactly represented as $G(V, E, c)$.
Let $n$ denote the cardinality of $V$ and let $I_n$ be the identity 
matrix of dimension $n$. Without any loss of generality, we can 
arbitrarily number the vertices and associate the numbers with the 
vertices. Let $i, j \in V$ and let $e_i, e_j$ correspond to the $i^{th}$ 
columns of $I_n$. If $a, b \in \Re^n$, let $a \otimes b$ denote the 
tensor product of $a$ and $b$. Let $c_{ij}$ denote the cost of the edge $\{i,j\}$.

The graph Laplacian of $G(V,E,c)$ is defined as:
$$ L := \sum_{e= \{i,j\} \in E} c_{ij} (e_i-e_j) \otimes (e_i-e_j). $$ 
The component of $L$ in the $i^{th}$ row and $j^{th}$ column is 
given by $L_{ij}$ and is as follows:


$$L_{ij} = \begin{cases}
    -c_{ij},& \text{if } i \ne j, \;  \{i,j\} \in E,\\
	\sum_{j: \{i,j\} \in E} c_{ij}, & \text{if} \  i = j, \\
    0,  & \text{otherwise}
\end{cases}
$$
As an example, Laplacian matrix for the graph shown in Figure \ref{Fig:laplacian1}(a) is as follows: 
\begin{align}
	\label{eq:laplacian_mat}
		L = \begin{pmatrix}
	 	1 & -1 & 0 & 0 & 0 & 0 \\
		-1 & 2 & -1 & 0 & 0  & 0\\
		0 & -1 & 2 & -1 & 0 & 0\\
		0 & 0 & -1 & 2 & -1 & 0\\
		0 & 0 & 0 & -1 & 2 & -1\\
		0 & 0 & 0 & 0 & -1 & 1 \\
		\end{pmatrix}
\end{align}

There are other variants of Laplacians that are used; this 
dissertation primarily focuses on the graph Laplacian.

\subsection{Relationship between Laplace's equation and graph Laplacian}
Consider the following one-dimensional Laplace's equation: 
\begin{align}
\label{eq:poisson}
-\frac{d^2 u}{dx^2} = f(x),
\end{align}
where $f(x)$ is the source, $u(x)$ is the response and $\frac{d^2}{dx^2}(\cdot)$ is the Laplacian operator. 
As shown in Figure \ref{Fig:laplacian1}(b), consider a discretized space such that the given domain 
$\Omega = [0,X]$ is discretized with equally spaced points $x_i, i=1\ldots 6$ 
and the grid size $h=1$. Hence, a one-dimensional stencil using a second order central differencing which approximates the Laplacian operator at point $x_i$ is given as follows: 

$$ \frac{d^2 u}{dx^2} \Big \rvert_{x=x_i} \approx \frac{u_{i+1} - 2u_i + u_{i-1} }{h^2} = u_{i+1} - 2u_i + u_{i-1}$$
where, $u_i \approx u(x_i)$. 
\begin{figure}[htp] \centering
	\subfigure[A graph with unit weights]{
			  \includegraphics[scale=1.09]{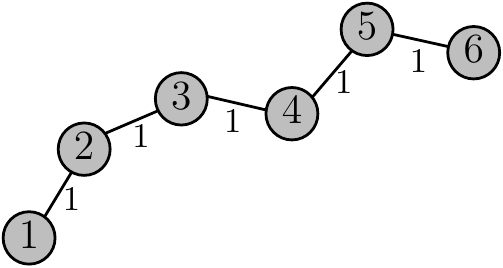}}       
	\subfigure[Finite difference grid]{
			  \includegraphics[scale=1.4]{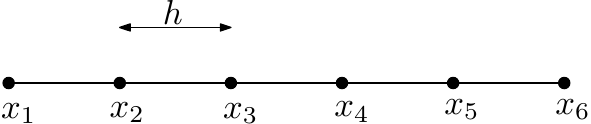} }
	\caption{A graph with unit weights and its equivalent finite difference grid.} 
	\label{Fig:laplacian1}
\end{figure}

In order to represent the discretized Laplace's equation in the matrix form,
we consider the following cases: \\
\noindent
\textbf{Dirichlet boundary condition:} 
For the Laplace's equation in \eqref{eq:poisson}, let the Dirichlet boundary 
conditions be as follows: 
$$u_1 = u(x_1) = \alpha_1, u_6 = u(x_6) = \alpha_6.$$
Under these boundary conditions, discretizing the Laplace's equation in \eqref{eq:poisson} over the grid shown in 
Figure \ref{Fig:laplacian1}(b), we obtain the following matrix form: 

\begin{align}
		\begin{pmatrix}
	 	2 & -1 & 0 & 0 \\
		-1 & 2 & -1 & 0 \\
		0 & -1 & 2 & -1 \\
		0 & 0 & -1 & 2 \\
		\end{pmatrix} 
		\begin{pmatrix}
				 {u}_2 \\ {u}_3 \\ {u}_4 \\ {u}_5
		\end{pmatrix}
		&= 
		\begin{pmatrix}
		f_2 - \alpha_1 \\ f_3 \\ f_4 \\ f_5- \alpha_6
		\end{pmatrix}
		\label{eq:finite_diff_1} 
\end{align}
It is evident that the square sub-matrix obtained by dropping the first and the last rows and columns of the 
Laplacian matrix in equation \eqref{eq:laplacian_mat} is the same as the coefficient matrix in equation \eqref{eq:finite_diff_1}.

\noindent
\textbf{A combination of Neumann and Dirichlet boundary conditions:} 
For the Laplace's equation in \eqref{eq:poisson}, let the Neumann boundary 
condition be: 
$$\frac{du}{dx} \Big \rvert_{x=x_1} = \beta_1$$
and the Dirichlet boundary condition be 
$$u_6 = u(x_6) = \alpha_6.$$
A natural first order approximation to the derivative at $x_1$ is a one sided difference
$$\frac{du}{dx} \Big \rvert_{x=x_1} \approx \frac{u_1 - u_2}{h} = \beta_1.$$
Under these boundary conditions, discretizing the Laplace's equation in \eqref{eq:poisson} over the grid shown in 
Figure \ref{Fig:laplacian1}(b), we obtain the following matrix form: 

\begin{align}
		\begin{pmatrix}
	 	1 & -1 & 0 & 0 \\
		-1 & 2 & -1 & 0 \\
		0 & -1 & 2 & -1 \\
		0 & 0 & -1 & 2 \\
		\end{pmatrix} 
		\begin{pmatrix}
				 {u}_2 \\ {u}_3 \\ {u}_4 \\ {u}_5
		\end{pmatrix}
		&= 
		\begin{pmatrix}
		f_2 - \beta_1 \\ f_3 \\ f_4 \\ f_5- \alpha_6
		\end{pmatrix}
		\label{eq:finite_diff_2} 
\end{align}
It is evident that the square sub-matrix obtained by dropping the the last two rows 
and columns of the Laplacian matrix in equation \eqref{eq:laplacian_mat} is the same as the 
coefficient matrix in equation \eqref{eq:finite_diff_2}.


The graph interpretation of the discretized problem is shown in Figure \ref{Fig:laplacian1}(a). 
In this interpretation, every graph vertex in Figure \ref{Fig:laplacian1}(a) can be treated as a grid point;
the edges of the graph shown in Figure \ref{Fig:laplacian1}(a) have a cost of $1$ unit.  
The finite difference stencil at the grid point can be treated as the local Laplacian 
matrix and the unit edge cost corresponds to the homogeneous material with a unit 
thermal conductivity in the case of heat conduction equation.

\subsection{Graph Laplacian and electrical systems}
Consider a simple electrical network with four resistors and five junctions
as shown in Figure \ref{Fig:electrical}. In graph theoretic terms, 
the junctions represent the vertices of the graph, the resistors 
represent the edges in the graph and the corresponding conductance values represent the 
edge weights. For the convenience in the notation, we describe 
each resistor by it's conductance values, which is  the inverse of its 
resistive values. As an example, if the resistance between vertices two 
and five is $\frac{1}{c_{25}} \Omega$, then it's conductance value is 
equal to $c_{25} \mho$.

\begin{figure}[!h] 
	\centering
	\includegraphics[scale=1.0]{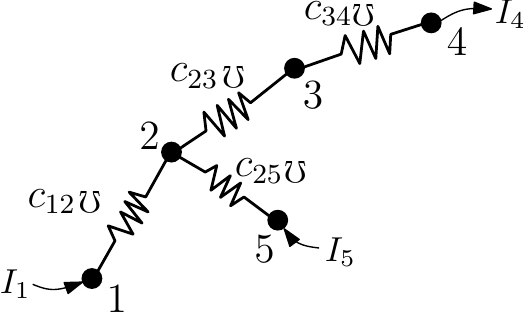}       
	\caption{An electric network with resistors labeled by their conductance values ($\mho$)} 
	\label{Fig:electrical}
\end{figure}

The problem of interest is to find the voltages $V_1,V_2,V_3,V_4$ and $V_5$ at all the vertices 
in the electrical resistive network given that $I_1$ and $I_5$ units of current enters
at vertices one and five respectively and $I_4$ units of current leaves 
from the fourth vertex. 

From Ohm's law, we know that the current flowing across the edges
$1 \rightarrow 2$, $2 \rightarrow 3$, $5 \rightarrow 2$ and
$3 \rightarrow 4$ are $c_{12}(V_1-V_2)$, $c_{23}(V_2-V_3)$, $c_{25}(V_5-V_2)$ and
$c_{34}(V_3-V_4)$ respectively. By applying Kirchoff's current balance 
law at all the vertices, we have the following set of linear equations: 
\begin{align*}
c_{12} (V_1 - V_2) &= I_1, \\
-c_{12} (V_1 - V_2) + c_{25} (V_2 - V_5) + c_{23} (V_2 - V_3)&= 0, \\
-c_{23} (V_2 - V_3) + c_{34} (V_3 - V_4)&= 0, \\
-c_{34} (V_3 - V_4) &= -I_4, \\
-c_{25} (V_2 - V_5) &= I_5.
\end{align*}
The above system can be expressed in matrix form as follows: 
\begin{align}
	\begin{split}
		\underbrace{\begin{pmatrix}
	 	c_{12}& -c_{12} & 0 & 0 & 0 \\
		-c_{12} & c_{12} + c_{23} +c_{25} & -c_{23} & 0 & -c_{25} \\
		0 & -c_{23} & c_{23} + c_{34} & -c_{34} & 0\\
		0 & 0 & -c_{34} & c_{34} & 0 \\
		0 & -c_{25} & 0 & 0 & c_{25} \\
		\end{pmatrix}}_{\mathrm{Admittance \  matrix}} 
		\begin{pmatrix}
		V_1 \\ V_2 \\ V_3 \\ V_4 \\ V_5
		\end{pmatrix} 
		&=
		\begin{pmatrix}
		I_1 \\ 0 \\ 0 \\ -I_4 \\ I_5
		\end{pmatrix}
		\label{eq:stiffness} 
	\end{split}
\end{align}
It can be noted from the above admittance matrix that
the $(i,j)^{th}$ entry is the negation of the conductance 
between the vertices $i$ and $j$ and the $i^{th}$ diagonal entry
is the sum of the conductances of all the resistors incident at 
$i^{th}$ vertex. Hence, this matrix is same as the Laplacian 
matrix of the weighted graph shown in \ref{Fig:laplacian}(a) and 
also the stiffness matrix shown in \eqref{eq:stiffness}.

\subsection{Graph Laplacian and discrete mechanical systems}


\begin{figure}[htp] \centering
	\subfigure[Weighted graph]{
			  \includegraphics[scale=1]{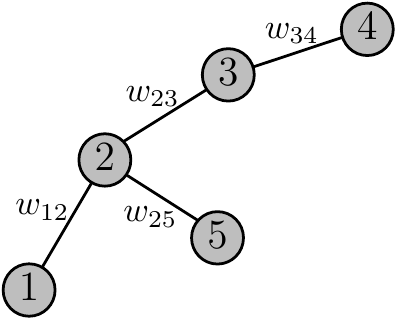}}       
	\subfigure[Spring mass system]{
			  \includegraphics[scale=0.7]{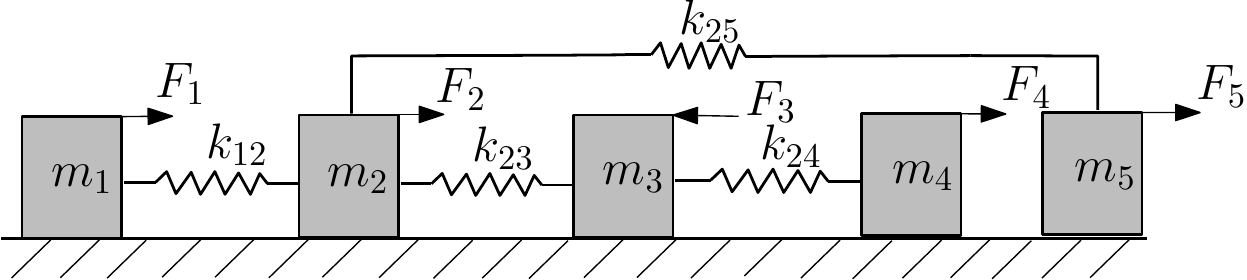} }
	\caption{A weighted graph and its equivalent form as a spring-mass system.} 
	\label{Fig:laplacian}
\end{figure}
Consider a simple five degree of freedom vibratory system shown 
in Figure \ref{Fig:laplacian}(b). At equilibrium, the springs are 
unstretched. The springs are assumed to be linear with a stiffness 
coefficient $k_{ij}$ if the spring connects masses $m_i$ and $m_j$. 
An application of Newton's laws yields the following governing 
equations of motion:

\begin{align}
	\begin{split}
		\begin{pmatrix}
	 	m_1 & 0 & 0 & 0 & 0 \\
		0 & m_2 & 0 & 0 & 0 \\
		0 & 0 & m_3 & 0 & 0\\
		0 & 0 & 0 & m_4 & 0 \\
		0 & 0 & 0 & 0 & m_5 \\
		\end{pmatrix}
		\begin{pmatrix}
				  \ddot{x}_1 \\ \ddot{x}_2 \\ \ddot{x}_3 \\ \ddot{x}_4 \\ \ddot{x}_5
		\end{pmatrix}
		+ \\
		\underbrace{\begin{pmatrix}
	 	k_{12}& -k_{12} & 0 & 0 & 0 \\
		-k_{12} & k_{12} + k_{23} +k_{25} & -k_{23} & 0 & -k_{25} \\
		0 & -k_{23} & k_{23} + k_{34} & -k_{34} & 0\\
		0 & 0 & -k_{34} & k_{34} & 0 \\
		0 & -k_{25} & 0 & 0 & k_{25} \\
		\end{pmatrix}}_{\mathrm{Stiffness \  matrix}} 
		\begin{pmatrix}
		x_1 \\ x_2 \\ x_3 \\ x_4 \\ x_5
		\end{pmatrix} 
		&=
		\begin{pmatrix}
		F_1 \\ F_2 \\ -F_3 \\ F_4 \\ F_5
		\end{pmatrix}
		\label{eq:stiffness} 
	\end{split}
\end{align}

Hence, for the example shown in Figure \ref{Fig:laplacian}(a), 
the corresponding Laplacian matrix would be: 
\begin{align}
		L &= \begin{pmatrix}
	 	k_{12}& -k_{12} & 0 & 0 & 0 \\
		-k_{12} & k_{12} + k_{23} +k_{25} & -k_{23} & 0 & -k_{25} \\
		0 & -k_{23} & k_{23} + k_{34} & -k_{34} & 0\\
		0 & 0 & -k_{34} & k_{34} & 0 \\
		0 & -k_{25} & 0 & 0 & k_{25} \\
		\end{pmatrix}
		\label{eq:lap}
\end{align}
Clearly, the stiffness matrix in equation \eqref{eq:stiffness} 
is the same as the Laplacian matrix in equation \eqref{eq:lap}.


\section{Algebraic connectivity as an objective of maximization}
Since algebraic connectivity is chosen as an objective of 
maximization in this dissertation, a motivation for the choice 
of the objective is in order. In this section, the motivation is 
provided through three different applications where maximizing 
algebraic connectivity is meaningful.

\subsection{Linear mechanical systems}
\label{sec:intro_mech_systems}
Let the mechanical system consist of $n$ identical masses and $|E|$ springs. If masses 
were to be treated as nodes, the linear springs as edges and stiffness coefficients of 
the  springs as the ``weight'' associated with each edge (spring), then the algebraic 
connectivity of the graph corresponds to the smallest non-zero natural frequency of the 
discrete mechanical system. Let $M, L$ respectively represent the mass and stiffness 
matrices respectively. The components of $L$ depend on the topology, $x$, of connections 
of masses with the aid of  springs.  Let $e_0$ denote a vector with every component being 
unity. If $\delta, f$ represent respectively the vectors of displacements 
and forces acting on the masses, then the governing equations corresponding to a 
given topology $x$ may be compactly expressed as 
$$M \ddot \delta + L(x) \delta = f. $$
Let $\|\delta \|_2, \|f \|_2$ represent the 2-norm of $\delta$ and $f$ respectively.
Let ${\cal F} = \{ f:  \| f \|_2 \leq 1, \; f \cdot e_0 = 0. \}$ The condition 
$f \cdot e_0 = 0$ implies that the net force acting on the system of masses is zero 
and hence, the centroid of the masses remains stationary if it is stationary initially.

 For a given topology $x$,  let $v_1, v_2, \ldots v_n$ be the unit eigenvectors of 
$L(x)$ corresponding to eigenvalues $\lambda_1 \le \lambda_2 \leq \cdots \lambda_n$. 
$$L(x) = \sum_{i=1}^n \lambda_i v_i \otimes v_i. $$
Since $L(x) e_0  = 0$, it implies that $\lambda_1 = 0, v_1 = \frac{e_0}{\sqrt{e_0 \cdot e_0}}$. Clearly,
when the displacements of all the masses are the same, the deflections in the springs 
are zero and this eigenvector corresponds to  a ``rigid-body'' mode. If the set of 
masses is connected, then it will admit only one rigid body mode; in this case 
$\lambda_2 >0$. Otherwise, $\lambda_2 = 0$ suggesting that there is another rigid 
body mode; in this case, one can find two disjoint sets of masses $S_1$ and $S_2$ 
that are not connected by any spring and correspondingly, the governing equations 
of motion in this case can be recast as: 
$$\underbrace{\left[ \begin{array}{cc} M_1 &0 \\0&M_2 \end{array} \right]}_{M} \underbrace{\left[ \begin{array}{c} \ddot \delta_1 \\ \ddot \delta_2 \end{array} \right]}_{ \ddot \delta} + \underbrace{\left[ \begin{array}{cc} L_1(x) &0 \\0&L_2(x) \end{array} \right]}_{L(x)} \underbrace{\left[ \begin{array}{c} \ddot \delta_1\\ \ddot \delta_2 \end{array} \right]}_{ \delta}  = \underbrace{\left[ \begin{array}{c} f_1\\ f_2 \end{array} \right]}_{f}. $$
We can construct a force $f \in {\cal F}$ as follows: the forces on the 
masses in $S_1$  and $S_2$ are  $\frac{1}{\sqrt{n}} \sqrt{\frac{|S_2|}{|S_1|}}$ and 
$-\frac{1}{\sqrt{n}} \sqrt{\frac{|S_1|}{|S_2|}}$ units respectively. Clearly $\|f\|_2 =1$ 
and the sum of the forces acting on the masses is zero. However, the masses in $S_1$ 
and $S_2$ move as if they are independent masses with a constant acceleration. Hence, 
in this case, the difference in the steady state displacements among the masses is 
unbounded. Since such a topology of connections is not desirable, let $\cal{X}$ 
denote the topology of connection of masses with springs which admits only a 
single rigid body mode. 

\begin{lemma}
Let $\delta_s$ be the vector of displacements of masses of the mechanical 
system due to the forcing function $f$. If $x \in {\cal X}$, and the 
initial value of average displacement and velocity of all masses is zero,  then 
$$\max_{f \in \cal{F}} \; \;  \|\delta_{s}\|_2  = \frac{1}{\lambda_2(L(x))}. $$
\end{lemma}
\begin{proof}
Since $f$ is a constant force, $\delta_s$ is a vector of constants and hence satisfies
$$L(x) \delta_s = f. $$
Let $f$ be decomposed along the eigenvectors $v_2, \ldots, v_n$ as 
$$f = \sum_{i=2}^n\alpha_i v_i,$$ so that 
$$\alpha_j = v_j \cdot f = v_j \cdot L(x) \delta_s = L(x)v_j \cdot \delta_s = \lambda_j v_j \cdot \delta_s. $$
From the assumption that the initial average displacement and velocity 
of all masses is zero,  it follows that the average displacement and 
velocity of masses is zero throughout as:
$$e_0 \cdot [M \ddot \delta + L(x) \delta] = e_0 \cdot f= 0, \Rightarrow e_0 \cdot \ddot \delta = 0. $$
Hence, $\delta_s$ cannot have a component along $v_1$ or equivalently 
along $e_0$. Since $x \in \cal{X}$,
$$\delta_s = \sum_{j=2}^n\frac{v_j \cdot f}{\lambda_j} v_j \Rightarrow \|\delta_s\|_2^2 = \sum_{j=2}^n \frac{\alpha_j^2}{\lambda_j^2}  \leq \frac{\sum_{j=2}^n \alpha_j^2}{\lambda_2^2} = \frac{1}{\lambda_2^2}.$$
Since the maximum is achieved when $f = v_2$, it follows that 
$$\max_{f \in \cal{F}} \|\delta_s\|_2 = \frac{1}{\lambda_2}. $$
\end{proof}

Clearly, the maximum value of the 2-norm of forced response of the mechanical 
system can be minimized when $\lambda_2(L(x))$ is a maximum. It is for this 
reason that algebraic connectivity (or the second smallest eigenvalue of $L(x)$ is maximized. 

\subsection{Application to rigid formations}
Consider a formation of $n$ identical UAVs in a single dimension trying to 
maintain a fixed distance from each other throughout their motion. Suppose 
the motion of the $i^{th}$ UAV is given by:
\begin{eqnarray}
\label{eq:Formation1}
X_i(s) = \frac{1}{s^2}[P(s) U_i(s)- D_i(s)], 
\end{eqnarray}
where $P(s)$ is a proper, rational transfer function, $X_i(s), U_i(s)$  
and $D_i(s)$ are respectively the Laplace transformation of the position of, 
control input to and disturbing force acting on the $i^{th}$ UAV. 
The term $P(s)$ represents the actuator transfer function and relates 
the control input to a UAV with the actuation force generated by the UAV. 

Suppose the UAVs desire to maintain a constant relative separation 
with respect to each other with the help of an identical on-board 
controller represented by the transfer function $C(s)$. Aiding the 
UAVs in accomplishing this task is a set of communication and sensing 
devices. An underlying information flow graph indicates the information 
available to each UAV. For example, if the $i^{th}$ UAV has the position 
and velocity information of the $j^{th}$ UAV in the collection, 
the $i^{th}$ and $j^{th}$ UAVs are considered adjacent or neighbors in 
the information flow graph. For the sake of simplicity of exposition, 
if $i^{th}$ UAV has the information of $j^{th}$ UAV,  it will be assumed 
that the converse also holds. Let ${\cal S}_i$ be the set of neighbors 
of the $i^{th}$ UAV. 

With this set up, one may associate a Laplacian with the information 
flow graph. For $i \neq j$, the component of Laplacian in the $i^{th}$ 
row and $j^{th}$ column is $-1$ if the $i^{th}$ and $j^{th}$ UAVs are 
neighbors and is zero otherwise. For each $i$, the $i^{th}$ diagonal 
element is the number of UAVs, $|S_i|$, that are neighbors of the 
$i^{th}$ UAV. Clearly,  the sum of the components of the 
corresponding row (and column) is zero.

Let $\bar X(s)$ denote the Laplace transformation of the position 
of the centroid of the formation and let the error in spacing 
$E_i(s) := X_i(s)  - \bar X(s) - \frac{L_i}{s}$, where $L_i$ is the 
desired position of the $i^{th}$ UAV from the centroid. The input to 
the controller is the aggregate error in maintaining a desired spacing 
relative to its neighbors and may be described as:
\begin{eqnarray}
\label{eq:Formation2}
U_i(s) = - C(s) \sum_{j \in S_i} [X_i(s) - X_j(s) - \frac{L_i-L_j}{s}] = -C(s) \sum_{j \in S_i} [E_i(s) - E_j(s)].
\end{eqnarray}

In this case, the error evolution equation can be described by:
\begin{eqnarray}
\label{eq:Formation3}
[s^2I_n + P(s)C(s)L] E(s) = -D(s) - s^2 \bar X(s) -s \ell,
\end{eqnarray}
where $E(s), D(s)$ are respectively the Laplace transformation of 
the vector of errors in spacing and disturbing forces acting on the UAVs. 
The term $\ell$ represents the vector of desired distances of the 
UAVs from the centroid and $L$ represents the Laplacian 
associated with the information flow graph.

 The associated characteristic equation is given by 
$$\Pi_{i=2}^n (s^2+ \lambda_i(L) P(s)C(s)) = 0. $$
The stability of motion of the formation of UAVs is governed by 
the eigenvalues of the Laplacian. Clearly, $P(0) C(0) \ne 0$; otherwise, 
the errors do not decay to zero as $0$ is one of the roots of the 
characteristic equation. The term $P(0)C(0) >0$ indicating that the 
steady state gain from the aggregate error to the force supplied by the 
actuation system is positive. Since $\lambda_i$ are the non-zero 
eigenvalues of the Laplacian, $\lambda_i >0$. The high frequency gain 
is positive as it is the coefficient of $s^2$. If $P(0)C(0) <0$, 
then for $\lambda >0$, the highest and lowest degree terms of the 
characteristic polynomial will be of opposite signs indicating instability 
of motion.  Physically, if $P(0)C(0)<0$ then a UAV  lagging behind will 
lag further behind and the motion of the UAVs will be unstable. 
Since $P(0)C(0)>0$, the sensitivity to low frequency disturbances is 
higher if $\lambda$ is lower.  This can be seen as follows: Let 
$E_v(s) := v \cdot E(s), D_v(s) := v \cdot D(s)$, where $v$ is an eigenvector 
of Laplacian corresponding to one its non-zero eigenvalues, $\lambda$. Then:
$$v \cdot [ s^2 I_n + P(s) C(s) E(s)] = - v \cdot D(s) - s (v \cdot \ell), $$
which simplifies to 
$$(s^2 + \lambda  P(s) C(s)) E_v(s) = - D_v(s) - s (v \cdot \ell). $$
Correspondingly, the disturbance attenuation transfer function is given by 
$\frac{-1}{s^2 + \lambda P(s) C(s)}$. At a low frequency, $w$, the 
low frequency attenuation is governed by  $|\frac{1}{-w^2 + \lambda P(0)C(0)}|$ 
since $P(0)C(0)>0$. Clearly, higher the value of $\lambda$, the better is the 
attenuation. Since disturbance attenuation at low frequencies is important, 
it is reasonable to maximize $\lambda_2$, the lowest non-zero eigenvalue 
of the Laplacian.

\subsection{Application to UAV network synthesis}
Earlier in this section, a variant of {\bf BP} involving the construction 
of an adhoc infrastructure network with UAVs has been alluded to. It 
required the maintenance of a rigid formation and a control law for 
maintaining such a formation can be constructed along the lines 
mentioned in the earlier subsection. The networking aspect of this 
problem also has relevance to the maximization of algebraic connectivity 
of the underlying communication graph. Since UAVs form a backbone network, 
each UAV must be able to transmit data at a constant rate, say 
$r \; bits/sec$ to every other UAV in the network; however, each UAV 
may be receiving data at a rate of $R \; bits/sec$ to be transmitted 
to other UAVs. In order to maintain a desirable quality of service, 
one may be interested in finding out the maximum value of $R$ that 
is allowable so that the network is not congested or ``bottlenecked''; 
from the viewpoint of designing a network, one would be interested in 
designing a UAV adhoc network so that $R$ is maximized subject to 
other constraints on resources. 

One may associate a vertex, $v$, with each UAV, an edge, $e$ with 
a communication link and a ``weight'' $\alpha_{ij}$ associated with 
the edge (communication link) connecting the $i^{th}$ and $j^{th}$ UAVs. 
Let $V, E$ represent the set of vertices and edges and let the 
underlying communication graph be represented as $G(V, E, \alpha)$. 
The term $\alpha_{ij}$ is proportional to the product of the antenna 
gains of the $i^{th}$ and $j^{th}$ UAVs and is reflective of the data 
rate that can be communicated across the link.  An important concept 
in addressing this issue is the value of a cut. A cut may be 
identified by a set $S \subset V, \; S \ne V$. The cut $\delta(S)$ is 
the set of edges with exactly one end in $S$. The value of the cut is 
the sum of the weights of the edges in $\delta(S)$ and is represented 
by $w(\delta(S))$. Let $\bar{S}$ be the complement of $S$ in $V$; 
i.e., $v \notin S \iff v \in \bar{S}$. Clearly, 
$\delta(S) = \delta(\bar{S})$ and the value of the cut is the same.

Suppose the network must be designed so that there is a guaranteed 
data rate of $R$ bits/second that every UAV can transmit without 
the network getting congested. Given a set $S \subset V$ and its 
complement $\bar{S} \subset V$, let $r(S)$ be the pairwise data 
exchange between every pair of nodes in $S$ and $\bar S$. Then, 
the total data transmission across the cut of $S$ 
is $r(S) |S| |\bar{S}|$ and must be no more than the value of 
the cut, $w(\delta(S))$ and hence:
$$r(S) \leq \frac{w(\delta(S))}{|S| \cdot |\bar{S}|}. $$
If $S$ is of lower cardinality than $\bar{S}$, then nodes in $S$ can 
transmit data at a rate of $r(S)$ to each node in $\bar{S}$ and consequently,
$$r(S) \max\{|S|, |\bar{S}|\} \leq \frac{w(\delta(S))}{|S| \cdot|\bar{S}|} \max\{|S|, |\bar{S}|\} = \frac{w(\delta(S))}{\min\{|S|, |\bar{S}|\}}. $$
Since 
$$ R = \min_{S \subset V} \; \; r(S) \max\{|S|, |\bar{S}|\}, $$ it follows that 
$$R \le \min_{S \subset V} \frac{w(\delta(S))}{\min\{|S|, |\bar{S}|\}}.$$
In fact, $R$ can be set to the minimum value of the right-hand side, 
which is referred to as the Cheeger constant or Cheeger number, $h$, 
or the isoperimetric number of a graph. If a link (or an edge) in 
$\delta(S)$ were to fail, the capacity of a cut decreases and the 
above inequality relates how the guaranteed data rate of transmission 
is reduced for every UAV in the network. In the problem of UAV network 
synthesis, it is desirable to maximize $R$ over all allowable ways 
of connecting the UAVs. 

The Cheeger number is difficult to compute for a sufficiently large 
size graph; for this reason, algebraic connectivity of a graph is 
used as its surrogate. The justification for using algebraic 
connectivity, $\lambda_2$, of a graph as a surrogate stems from the 
following inequality connecting Cheeger constant and algebraic connectivity:
$$h \le \lambda_2 \le \frac{h^2}{4}. $$
Clearly, if $h$ is small, $\lambda_2$ is small because of the upper 
bounding inequality and if $h$ is large, $\lambda_2$ is also large 
because of the lower bounding inequality. 

\subsection{Application to air transportation networks}
The fundamental objective of an air transportation network is to 
transport passengers from one airport to another in as efficient 
a manner as possible while meeting quality of service requirements 
of the passengers. In this dissertation,  a very simple model of an 
air transportation network will be considered, where every airport 
serves the role of a node, a {\it direct} route between two airports 
serves the role of an edge between the two nodes. The main issue of 
interest in this dissertation is  the sensitivity of ``connectivity'' 
of air transportation network to some edges being not operational 
due to a variety of factors such as weather etc. Here, the term 
``connectivity'' is used loosely. Connectivity is meant to mean the 
ability to transport the passengers from their respective origin to 
their intended destination. Connectivity can be affected by weather 
resulting in an edge connecting two nodes being deleted (i.e., the 
route being out of operation temporarily or a flight being cancelled). 
For example, a reduction in connectivity can result in passengers 
being stranded at an airport and the undesirable consequence of 
reduction in the quality of service to the passengers. 

Given that the motivation is to address connectivity, each edge 
will be ``weighed'' according to the passengers that {\it can be} 
transported across that edge or even more simply as the number of 
passenger flights flying that {\it can be flown} on the route in a 
day. This is a gross simplification; however, this is a first step 
towards a more complicated model of an air transportation network. 
The underlying assumption is that each flight carries the same number 
of passengers. Since the quality of service is associated with the 
passengers transported from a node, let $R(i)$ be the total number of 
passengers transported from node $i$. It is assumed that the travel 
demand from node $i$ to all other nodes is the same, i.e., the number 
of passengers desiring to travel from node $i$ to any other node is 
exactly the same. The {\it node capacity} of the network may be 
defined to the maximum value of $\min_i R(i)$ for which the network 
is congested, i.e., for some cut, $S \subset V$, the value of the 
cut equals the travel demand across $S$, that is, the number of 
passengers starting from $S$ and intending to reach some node in ${\bar S}$. 

A preliminary design of an air transportation network can be posed 
as follows: Suppose a graph, $G(V, E,w)$ of the nodes/vertices (airports), 
the set of edges $E$ (routes connecting a pair of airports) and the 
associated weights. Suppose further that the cost of operating a route 
is known a priori; one may even associate a priority/importance of the 
route as a cost. The problem is to find a network so that the minimum serving capacity 
(i.e., the number of passengers that can be transported from any node 
in the network) is maximized without the network getting congested 
and the sum of cost of operation of the routes is within a specified budget. 

This problem is analogous to the UAV adhoc infrastructure network 
design problem, where the objective is to maximize the Cheeger 
number of the network subject to resource constraint. Since Cheeger 
number is difficult to deal with, one can pose the closely related 
problem of maximizing algebraic connectivity subjected to 
resource constraints.

\section{Literature review}
\label{Sec:ch1_literature}
\subsection{Relation to current state of knowledge in system theory}
The problem of system realization considered in this dissertation was the topic of a plenary talk by Kalman in
an IFAC meeting \cite{kalman2005} in 2005. While the relevant reference to this work
is \cite{Seshu_Reed}, there has not been a formally written problem statement to this
effect to the best of the knowledge of the author. Neither has there been a resolution
of the problem. 

It is known that the algebraic connectivity of a structure is non-zero if and only if 
the structure or graph is connected \cite{fiedler1973algebraic}, \cite{merris1995survey}, 
\cite{mohar1991laplacian}. The problem of maximum algebraic connectivity has been
considered in \cite{maas1987transportation} for reducing the heights of the water 
columns at the junction in a network of pipes connecting them. The relevance of the
maximum algebraic connectivity to mixing of Markov Chains is shown in \cite{ghosh2006upper}.
The work in \cite{ghosh2006upper} is concerned only with unweighted graphs and provides
bounds on the maximum algebraic connectivity by exploiting the symmetry of the Laplacian
under the action of permutations. This problem is also relevant to information flow 
and motion planning of UAVs as considered in the works \cite{kim2006maximizing},
\cite{zavlanos2007potential}. However, none of them solve the mixed integer semi-definite program.
Recently, for the special case of the maximum augmented algebraic connectivity problem
where only one edge must be added, a bisection algorithm has been presented in \cite{kim2010}. 

The problem of maximizing augmented algebraic
connectivity was considered by Maas~\cite{maas1987transportation}; however, a systematic procedure
to solving for the maximum augmented algebraic connectivity is still lacking. It may be posed compactly
as a mixed-integer, semi-definite program; initial efforts to compute the upper bounds of the
maximum algebraic connectivity (which is the second smallest natural frequency in structural
systems) may be found in \cite{ghosh2006upper} and also in the recent work of the authors
\cite{MRSR2009max_alg_conn}.

\subsection{Relation to current state of knowledge in discrete optimization}
The problem of determining whether one can construct a constant factor approximation algorithm for
this problem is still open. From the viewpoint of constructing cuts for the semi-definite
integer programs, Atamturk and Narayanan recently developed non-linear cuts for conic 
programs\cite{atamturk2010lifting}, \cite{atamturk2010conic}. Since conic programs are
special instances of semi-definite programs, the general problem of constructing 
efficient cuts for semi-definite programs is still open. The recent work in \cite{kartikRPI}
develops efficient interior-point algorithms for infinite linear programs. Their work was
motivated by the need to solve mixed-integer, semi-definite programs through polyhedral
approximations. The book on convex optimization \cite{Boyd_convex_optimization} provides
an excellent overview of the algorithms required to solve linear semi-definite programs.

\section{Summary of contributions}
\label{Sec:summary_contri}
In the context of the problem of maximizing the algebraic connectivity
of networks under resource constraints, our contributions, as presented in
sections 2 and 3, are as follows:

\noindent a) Understanding the relevance of  {\bf BP} in the context 
of disparate fields of research and providing algorithms for solving {\bf BP}
to optimality, methods to obtain upper bounds and quick heuristics to obtain
sub-optimal solutions.

\noindent b) Providing algorithms for solving a variant of {\bf BP} that 
arises in the synthesis of robust UAV communication networks under 
resource constraints such as the total number of communication
links and the diameter of the network.

\noindent c) Providing algorithms for solving  a second variant of {\bf BP} that arises in synthesizing robust UAV communication networks under resource constraints such as the total number of communication
links and the power consumption constraint. 

\subsection{Organization of the dissertation}
This dissertation is organized as follows: 

In section \ref{ch2}, we pose the problem of
maximizing the algebraic connectivity as three equivalent 
formulations; Mixed Integer Semi-Definite Program (MISDP),
MISDP with connectivity constraints and the Fiedler vector formulation as 
an MILP and discuss the relative strengths and useful features of the proposed formulations.
Further, we study the importance of the choice of 
an appropriate family of finite number of vectors used to relax the semi-definite constraint
and discuss the quality of the associated upper bounds due to the 
relaxation. 

In section \ref{ch2}, we propose three cutting plane based algorithms to 
solve the proposed MISDP to optimality, namely: an algorithm based on the 
polyhedral approximation of the 
semi-definite constraint, an iterative primal-dual
algorithm that considers the Lagrangian relaxation of the semi-definite 
constraint and an algorithm based on the Binary Semi-Definite Program (BSDP) approach 
in conjunction with cutting plane and bisection techniques. Further, by an improved relaxation 
of the semi-definite constraint, we discuss the computational efficacy of the cutting 
plane algorithm in comparison with the state-of-the-art MISDP 
solvers in Matlab. Also, by adopting the BSDP approach and implementing the 
algorithm in CPLEX, we discuss the gain in the computation time. 
Section \ref{ch2} concludes with heuristics to synthesize feasible solutions for the \textbf{BP}. 
The proposed heuristics are based on neighborhood search, namely 
$k$-opt and an improved $k$-opt heuristic with a reduced search space. 
We corroborate the quality of the heuristic solutions with respect to 
optimal solutions for small instances and present the numerical results for 
large instances (up to sixty nodes). 

In section \ref{ch3}, we mathematically formulate various resource constraints such as the 
diameter constraint and the power consumption constraint. For the problem of maximizing 
algebraic connectivity under these resource constraints, we propose modified versions of the 
cutting plane algorithms and discuss their computational performance for relatively small instances. 
In the context of the problem with power consumption constraint, we extend the 
BSDP approach to obtain feasible solutions and discuss the quality of the associated lower bounds
for relatively large instances (up to ten nodes). Finally, we discuss the performance 
of the $k$-opt heuristic in the context of solving \textbf{BP} under resource constraints.

Lastly, in section \ref{ch4}, we summarize the results of the work and discuss 
possible directions to further develop this field of research.

%% file: Chapters/section2.tex
%
%
%


\chapter{\uppercase {Maximization of algebraic connectivity}}
\label{ch2}

In this chapter, the {\bf BP} of maximizing the algebraic connectivity of graphs is considered. 
The rationale for considering algebraic connectivity of graphs as an objective for optimization 
is illustrated via various applications, including a discrete mechanical system, 
air transportation network and UAV ad-hoc networks. Since the computation of solutions for combinatorial 
problems can be sensitive to mathematical formulation of the problem, different mathematical 
formulations of {\bf BP} are presented along with their features.  The rest of the section is 
focused on (1) developing upper bounds for the optimal value of algebraic connectivity using 
relaxation and cutting plane techniques, (2) developing techniques for computing the optimal 
value of algebraic connectivity and (3) to provide heuristic techniques for synthesizing 
sub-optimal graphs along with the percentage  deviation of their algebraic connectivity 
from the optimal value.

\section{Problem formulation for the basic problem of maximizing algebraic connectivity}
\label{Sec:ch1_notation}
Let $(V, E, w)$ represent a graph. Without any loss of generality, we will simplify
the problem by allowing at most one edge to be connected between any pair of nodes
in the graph. Let $w_{ij}$ represent the edge weight with the edge $e = \{i,j\} \in E$ and let
$x_{ij} \in \{0, 1\}$ represent the choice variable for every $\{i,j\} \in E$. 
Let $x$ be the vector of choice variables, $x_{ij}$. If $x_{ij} = 1$, it implies that the edge is chosen
in the construction of the network; otherwise, it is not. In the context of UAVs, 
vertices of the graph correspond to UAVs and the edges correspond to the communication links
between them. The edge weight corresponds to the strength of the communication link between 
a given pair of UAVs. 

If $v_1, v_2$ are two vectors in the same vector
space, we denote their tensor product by $v_1 \otimes v_2$
and their scalar or dot product by $v_1 \cdot v_2$. If $e = \{i,j\}$ connects UAVs 
$i$ and $j$, then the effective communication between 
that pair of UAVs may be expressed as $w_{ij} x_{ij}$. Let $e_i$ denote the $i^{th}$
column of the identity matrix $I_n$ of size $|V| = n$.
We may define $L_{ij} = w_{ij}(e_i-e_j)\otimes(e_i-e_j)$, and
correspondingly, the weighted Laplacian matrix (in the remainder of this dissertation, 
the usage of Laplacian matrix implies weighted Laplacian matrix unless specified) 
may be expressed as:
$$ L(x) = \sum_{i<j, \{i,j\} \in E}  x_{ij} L_{ij}.$$
Note that, for a given connected network, $L(x)$ is a symmetric,  positive 
semi-definite matrix, that is, 
$$ v \cdot L(x) v \geq 0 \ \forall v.$$
Let $(\lambda_1(L(x)) = 0) < \lambda_2(L(x)) \leq \lambda_3(L(x)) \ldots \leq \lambda_n(L(x))$  
be the eigenvalues of $L(x)$ and let $v_1, v_2, \ldots v_n$ be the corresponding eigenvectors 
of $L(x)$.

The {\bf BP} can be expressed as:
\begin{equation}
		\begin{array}{ll}
	  		\gamma^* = &\max  \lambda_2(L(x)), \\
			\text{s.t.} & \sum_{i<j, \; \{i,j\} \in E} x_{ij} \leq q, \\
			& x_{ij} \in \{0, 1\}^{|E|}
		\end{array}
		\label{eq:BP_form}
\end{equation}
where $q$ is some positive integer which is an upper bound on the number of edges to be chosen.
Since this is a non-linear binary program, it is paramount to develop efficient ways of formulate this 
problem. In the remainder of this section, we shall focus on developing various equivalent 
formulations for \textbf{BP}.

\subsection{Mixed integer semi-definite program}
\label{Subsec:F1}
The \textbf{BP} formulation in \eqref{eq:BP_form} may be equivalently expressed as a Mixed Integer
Semi-Definite Program (MISDP) as follows: let $e_0 = \frac{1}{\sqrt{n}}\sum_{i=1}^n e_i$
so that $e_0 \cdot e_0 = 1$. Then, formulation \eqref{eq:BP_form} may be expressed as:

\begin{equation}
		\label{eq:F1}
		\begin{array}{ll}
	  		\gamma^* = &\max  \gamma, \\
		\text{s.t.} &  \sum_{ i<j, \; \{i,j \}\in E} x_{ij} L_{ij} \succeq \gamma (I_n - e_0 \otimes e_0),\\
			& \sum_{i<j, \; \{i,j\} \in E} x_{ij} \leq q, \\
			& x_{ij} \in \{0, 1\}^{|E|}.
		\end{array}
\end{equation}

We will refer to it as the formulation ${\cal F}_1$. We first show that
this formulation correctly solves the algebraic connectivity problem.

\begin{lemma}
Let an optimal solution corresponding to the formulation ${\cal F}_1$ 
be $\gamma^*$ and $x^*$. Then, $x^*$ is a network that solves {\bf BP} to optimality with $\gamma^*$
being the second eigenvalue of $L(x^*)$.
\end{lemma}

\begin{proof}
Let the eigenvalues of the positive semi-definite matrix, $L(x)$
be given by
$(0=\lambda_1(L(x))) < \lambda_2(L(x)) \leq \ldots \leq \lambda_n(L(x))$.
We first show that for any connected network $x$, $L(x)$ and
$\gamma=\lambda_2(L(x))$ satisfy the constraints in the formulation
${\cal F}_1$. Let $e_0$ be an eigenvector corresponding to $\lambda_1(L(x)) =0$.
Then, $L(x)$ admits a spectral decomposition of the form

\begin{align}
\label{Eqn:proof1}
L(x) = \sum_{i=1}^{n} \lambda_{i}(L(x)) ~(v_i \otimes v_i),
\end{align}
where $v_{i}$ is the eigenvector corresponding to eigenvalue, $\lambda_{i}(L(x))$.
Since $\lambda_1(L(x)) =0$, and $v_1 = e_0$, the equation \eqref{Eqn:proof1} reduces to

\begin{align}
\label{Eqn:proof2a}
L(x) &= \sum_{i=2}^{n} \lambda_{i}(L(x))~ (v_i\otimes v_i).
\end{align}
Adding $\lambda_{2}(L(x))~ (e_0 \otimes e_0)$ to both sides of the equation
\eqref{Eqn:proof2a},

\begin{align}
\label{Eqn:proof2b}
L(x) +  \lambda_{2}(L(x))~ (e_0 \otimes e_0) &= \lambda_{2}(L(x))~ (e_0 \otimes e_0) + \sum_{i=2}^{n} \lambda_{i}(L(x))~ (v_i \otimes v_i).
\end{align}
Since $\lambda_i(L(x)) \geq \lambda_2(L(x)), \ \forall \ i \geq 2$, equation \eqref{Eqn:proof2b}
reduces to the following inequality:

\begin{align}
\label{Eqn:proof3}
L(x) +  \lambda_{2}(L(x))~ (e_0 \otimes e_0) &\succeq \lambda_{2}(L(x))  \underbrace{\left(\sum_{i=1}^{n} (v_i \otimes v_i)\right)}_{I_n},\\
L(x) &\succeq \lambda_{2}(L(x))~(I_n- e_0 \otimes e_0 ).
\end{align}
Therefore, for any connected network $x$, $L(x)$ and
$\gamma=\lambda_2(L(x))$ satisfy the constraints in the formulation
${\cal F}_1$. Now, to show that $\gamma^*$ =$\lambda_2(L(x^*))$,
it is enough to prove that $\gamma^* \geq \lambda_2(L(x^*))$ and
$\gamma^* \leq \lambda_2(L(x^*))$. \\

{\it Proof for $\gamma^* \geq \lambda_2(L(x^*))$}: We know that
$x^*$ is a feasible solution to formulation ${\cal F}_1$ with 
second eigenvalue, $\lambda_2(L(x^*))$. Since this is a maximization
problem, $\gamma^*$ must be an upper bound on $\lambda_2(L(x))$
for all possible feasible solutions. Hence,  $\gamma^* \geq \lambda_2(L(x^*))$. \\

{\it Proof for $\gamma^* \leq \lambda_2(L(x^*))$}: Since
($x^*,\gamma^*$) is a feasible solution, we have
\begin{align}
\label{Eqn:proof4}
L(x^*) &\succeq \gamma^*(I_n- e_0 \otimes e_0 ).
\end{align}
Let $\hat{v}$ be any unit vector perpendicular to $e_0$. Then
\begin{align}
\label{Eqn:proof5}
\hat{v} \cdot  L(x^*) \hat{v} &\geq \gamma^*.
\end{align}
Hence, from the Rayleigh quotient characterization of the
second eigenvalue, it follows that $\lambda_2(L(x^*)) \geq \gamma^* $.
\end{proof}

\noindent 
\textbf{Approximations of the feasible set of $\mathcal{F}_1$:}
One can approximate the feasible set of $\mathcal{F}_1$ in at least 
two different ways: 

\begin{itemize}
\item[(a)] \emph{Binary relaxation}: In this type of relaxation, 
the feasible set of the formulation $\mathcal{F}_1$ is 
expanded by replacing the integer constraint, $x_{ij} \in \{0, 1\}^{|E|}$ 
with $0\leq x_{ij} \leq 1, \ \ \forall i < j, \{i,j\} \in E$. 
\item[(b)] \emph{Relaxation of the semi-definite constraint}: 
The semi-definite constraint can be equivalently expressed as a family of linear 
inequalities parameterized as follows: 
$$v \cdot \left( \sum_{ i<j, \; \{i,j \}\in E} x_{ij} L_{ij}  - \gamma (I_n - e_0 \otimes e_0) \right) v \geq 0 \ \ \forall v.$$
where $v$ is any unit vector. One can relax the semi-definite constraint by picking a finite number of unit vectors,
say $v_1, v_2, \ldots , v_N$, and replacing the semi-definite constraint with the following 
linear inequalities: 
$$v_k \cdot \left( \sum_{ i<j, \; \{i,j \}\in E} x_{ij} L_{ij}  - \gamma (I_n - e_0 \otimes e_0) \right) v_k \geq 0 \ \ \forall k=1,\ldots, N .$$
\end{itemize}
Naturally, by solving formulation $\mathcal{F}_1$ with either of these relaxations, 
we are guaranteed to obtain an upper bound on the maximum algebraic connectivity. 
Hence, in the remainder of this section, we discuss the 
quality of the upper bounds obtained by considering the binary relaxation
of formulation $\mathcal{F}_1$ and its variants.
Later, in section \ref{sec:UB}, we discuss the quality of the 
upper bounds based on the relaxation of the semi-definite constraint 
with a finite number of unit vectors without relaxing the binary constraints.

\vspace{0.5cm}
\noindent 
\textbf{Performance of formulation $\mathcal{F}_1$:}
For the purposes of implementation, we restrict our 
feasible solutions to a set of undirected spanning trees since they
serve as minimally connected structures. Hence, we solve the following
version of formulation ${\cal F}_1$. 

\begin{equation}
		\label{eq:F1_tree}
		\begin{array}{ll}
	  		\gamma^* = &\max  \gamma, \\
		\text{s.t.} &  \sum_{ i<j, \; \{i,j \}\in E} x_{ij} L_{ij} \succeq \gamma (I_n - e_0 \otimes e_0),\\
			& \sum_{i<j, \; \{i,j\} \in E} x_{ij} = n-1, \\
			& x_{ij} \in \{0, 1\}^{|E|}.
		\end{array}
\end{equation}

From here on, we will prefix a formulation with ${\cal R}$ 
to indicate the relaxation of binary constraints (i.e., replacing the constraint 
$x_{ij} \in \{0, 1\}^{|E|}$ with $0\leq x_{ij} \leq 1, \ \ \forall i < j, \{i,j\} \in E$)
associated with the formulation.
Note that the feasible solutions
of  ${\cal F}_1$ are also feasible for  ${\cal RF}_1$; the optimal value of  ${\cal RF}_1$
is an upper bound represented by $\gamma_{RF_1}^*$.
A summary of  ${\cal RF}_1$'s solutions for various problem sizes 
is shown in Table \ref{Tab:summary_formulations}. It is clear from 
the table that the percentage deviation of the upper bound ($\gamma_{RF_1}^*$) from the 
best known feasible solution is unsatisfactory even for problems with small sizes (103.2\% gap for
five nodes problem). Also, one can observe that the percent 
deviation of the upper bound increases with the size of the problem (maximum gap up to 
181.9\% for twelve nodes problem), which is an undesirable feature.
However, having formulations with better upper bounds due to binary
relaxations are useful which can in turn reduce the computational
time of the Branch and Bound (B\&B) solver for solving the problem to optimality.
For example, any B\&B solver requires upper bounds on the optimal $\gamma^*$ and
one of the ways it generates this bound is by relaxing the binary
constraint on $x_e$. 

Also, since the binary relaxation of ${\cal F}_1$ 
allows for fractional values of $x_e$, it can
violate the following fundamental property of connectivity: If $S$ is a
strict subset of $V$, then there must be at least one edge between
the set of nodes in $S$ and $V-S$. The relaxation sometimes
allows the sum of the fractional values of the edges between
$S$ and $V-S$ to be less than unity. As an example, for a 
random cost matrix (Appendix \ref{ch:appendix}),  
a support graph constructed based on the binary relaxation 
solution of the ${\cal RF}_1$ is shown in Figure \ref{Fig:support_graph}.
It is clear from the figure that the min-cut ($x_{38} + x_{78}$) value
is equal to 0.805 and hence violates the connectivity property. 

\begin{figure}[!h]
	\centering
	\includegraphics[scale=0.7]{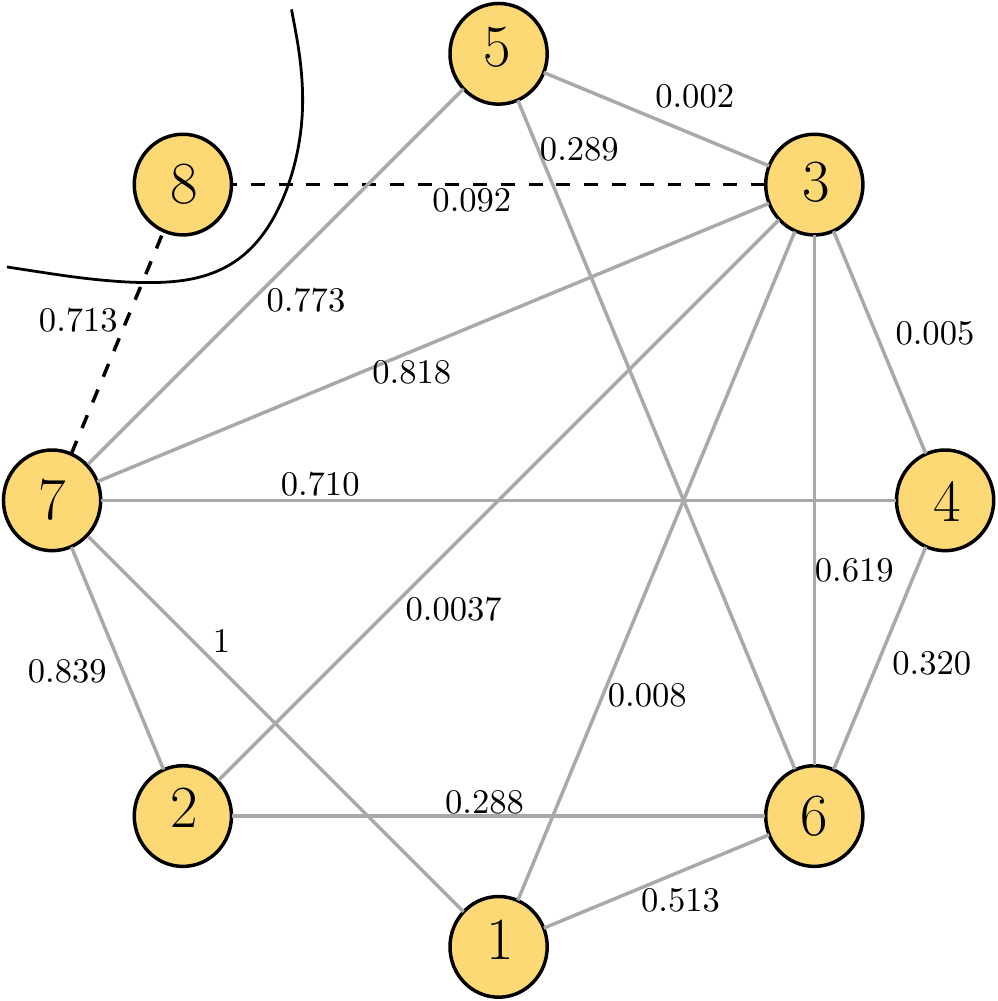}
	\caption{Support graph for a random instance where the connectivity constraints are violated. 
	The edges in the violated cutset are shown in dashed lines.} 
	\label{Fig:support_graph}
\end{figure}

A natural way to incorporate the connectivity constraints 
into the formulation is through the augmentation 
of flow cuts for the violated cutsets of the following form: 
$$\sum_{i<j \{i,j\} \in \delta(S)} x_{ij} \geq 1,$$
where $\delta(S)$ represents the edges in the cutset for the 
which the connectivity requirement is not satisfied. From
the implementation point of view, this may not be an efficient 
way of enforcing connectivity since the number of the violated 
cutsets may be exponential for large problems. 
Alternatively, one can also use the flow formulation of
Magnanti and Wong \cite{magnanti1995optimal} to obtain an
equivalently strong lifted formulation with a polynomial
number of constraints.

In the next subsection, we discuss a compact representation of 
connectivity using the flow formulation and study the performance 
of the MISDP with connectivity constraints.

\subsection{Mixed integer semi-definite problem with connectivity constraints}
\label{Subsec:F2}
As we discussed in the earlier subsection, though the formulation
${\cal F}_1$ with binary requirements on $x_{ij}$ enforces connectivity,
the relaxed problem ${\cal RF}_1$ does not ensure connectivity due to the 
fractional solutions.

MISDP formulation with cutset constraints which enforces the 
requirement of spanning trees as feasible solutions is as follows: 

\begin{equation}
		\label{eq:F2_tree}
		\begin{array}{ll}
	  		\gamma^* = &\max  \gamma, \\
		\text{s.t.} &  \sum_{ i<j, \; \{i,j \}\in E} x_{ij} L_{ij} \succeq \gamma (I_n - e_0 \otimes e_0),\\
			& \sum_{i<j, \; \{i,j\} \in E} x_{ij} = n-1, \\
			& \sum_{i<j, \; \{i,j\} \in \delta(S)} x_{ij} \geq 1 \ \ \forall S \subset V, \\
			& x_{ij} \in \{0, 1\}^{|E|}.
		\end{array}
\end{equation}

Clearly, if we ignore the integrality restrictions on $x_{ij}$ variables,
the fractional solutions still satisfy the connectivity requirements.
However, the main drawback of formulation \eqref{eq:F2_tree} 
is that the number of cutset constraints 
are exponential in the number of the nodes in the network. Hence,
we discuss an alternative formulation based on multicommodity
flow model proposed by Magnanti and Wolsey in \cite{magnanti1995optimal}.

In summary, in order to impose the connectivity constraints, 
the idea of the multicommodity flow model is as follows: Fix any
vertex in the graph as a source vertex $s$. Then construct the network
$x_{ij}$ such that a distinct unit commodity is shipped from $s$ to 
each of the vertices in $V$ while satisfying the flow and capacity
constraints. Flow constraints ensure that every distinct commodity 
indeed reaches its terminal vertex by satisfying the mass balance 
at every intermediate vertex. Capacity constraints ensure that sa
the flow of commodity across an edge occurs only if the capacity
of the edge is greater than or equal to the amount of the commodity shipped.
In the multicommodity flow formulation, let $f_{ij}^k$ 
be the the $k^{th}$ commodity flowing from $i$ to $j$. In this formulation,
although the edge variables are undirected, the flow variables will be directed.
Then, the MISDP formulation with multicommodity flow constraints, which we shall refer 
as ${\cal F}_2$ is as follows:

\begin{equation}
		\label{eq:F2_flow}
		\begin{array}{ll}
	  		\gamma^* = &\max  \gamma, \\
		\text{s.t.} &  \sum_{ i<j, \; \{i,j \}\in E} x_{ij} L_{ij} \succeq \gamma (I_n - e_0 \otimes e_0),\\
		&\sum_{j \in V \setminus \{s\}} (f^k_{ij}-f^k_{ji}) = 1, ~~ \forall k \in V ~ \text{and} ~ i = s, \\
		&\sum_{j \in V} (f^k_{ij}-f^k_{ji}) = 0, ~~ \forall i,k \in V ~ \text{and} ~ i \neq  k, \\
		&\sum_{j \in V} (f^k_{ij}-f^k_{ji}) = -1, ~~ \forall i,k \in V ~ \text{and} ~ i =  k, \\
		 &f^k_{ij} + f^k_{ji} \leq x_{ij}, ~~ \forall  ~ \{i,j\} \in E, \forall k \in V, \\
		& 0 \leq ~  f^k_{ij} ~ \leq 1, ~~ \forall i,j \in V,\forall k \in V, \\
		& \sum_{i<j, \; \{i,j\} \in E} x_{ij} = n-1, \\
		& x_{ij} \in \{0, 1\}^{|E|}.
		\end{array}
\end{equation}

\noindent 
\textbf{Performance of formulation $\mathcal{F}_2$:}
Although the MISDP formulation with multicommodity flow constraints ($\mathcal{F}_2$) 
circumvents the enumeration of exponential number of cutset constraints, the
computational performance of this formulation is very poor. With the 
integrality constraints, state-of-the-art MISDP solvers like Sedumi in 
Matlab \cite{sturm1999using} on a reasonably powerful workstation could not handle 
problems with five vertices and ten edge variables. One of the main 
reasons for the poor performance is the addition of $O(|V|^3)$ flow 
variables in addition to the $O(|V|^2)$ edge variables.

By relaxing the integrality constraints on $\mathcal{F}_2$ and solving 
$\mathcal{RF}_2$, the computational results are summarized in 
Table \ref{Tab:summary_formulations}. Again, the performance 
of $\mathcal{RF}_2$ was very poor and the solvers in Matlab
crashed for instances with more than six vertices and fifteen edge variables. 
For the case of five vertices, there is a slight improvement in the 
upper bound in comparison with the solution for $\mathcal{RF}_1$.

\subsection{Fiedler vector formulation}
\label{Subsec:F3}
There has been a great deal of interest in developing high performance solvers
for solving LPs and MILPs to optimality. Recently, there has also been 
a progress in the development of efficient programs for solving semi-definite problems 
and its variants with additional constraints such as polynomial constraints, 
second order conic constraints, etc. Here is a link to a comprehensive list of 
state-of-the-art semi-definite solvers:
\verb|http://www-user.tu-chemnitz.de/~helmberg/sdp_software.html|. \\
However, there has not been much focus on developing generic solvers for
solving MISDP problems. In order to utilize the available high performance MILP solvers, 
we present an equivalent formulation for $\mathcal{F}_1$ in the form of an MILP and 
study its performance in this subsection. 

We define the following notation before discussing the formulation based on the Fiedler vectors 
of feasible solutions:
Let $\Gamma$ represent the set of all feasible solutions to formulation $\mathcal{F}_1$
and
$$V_f := \{v \in \mathbb{R}^n : v \ \textrm{is a Fiedler vector for a feasible solution,} \ x \in \Gamma \}.$$
For the case of spanning trees as feasible solutions, $V_f$ contains the Fiedler vectors 
corresponding to the  $n^{n-2}$ spanning trees. 

The Fiedler vector formulation, which we shall refer as $\mathcal{F}_3$ is as follows:
\begin{equation}
		\label{eq:F3}
		\begin{array}{ll}
	  	\gamma^* =  &\max  \gamma, \\
		\text{s.t.} &  v \cdot (\sum_{ i<j, \; \{i,j \}\in E} x_{ij} L_{ij} ) v \succeq \gamma,  \ \ \forall v \in V_f, \\
			& \sum_{i<j, \; \{i,j\} \in E} x_{ij} \leq q, \\
			& x_{ij} \in \{0, 1\}^{|E|}.
		\end{array}
\end{equation}
We prove the following lemma to show the equivalence of formulations $\mathcal{F}_1$ and $\mathcal{F}_3$.
\begin{lemma}
\label{lemma:F3}
Let $(x_{F3}^*, \gamma_{F3}^*)$ be an optimal solution to $\mathcal{F}_3$ and
let $(x_{F1}^*, \gamma_{F1}^*)$ be an optimal solution to $\mathcal{F}_1$. Then, $\gamma_{F1}^* = \gamma_{F3}^*$.
\end{lemma}
\begin{proof}
  Clearly, the feasible set for $\mathcal{F}_1$ is a subset of the feasible set 
  for $\mathcal{F}_3$ since we replace the original semi-definite constraint with a finite number of constraints
  in $\mathcal{F}_3$. Hence, we have 
  $$\gamma_{F3}^* \geq  \gamma_{F1}^*  = \lambda_2(L(x_{F1}^*)).$$
	
  Let $v_{F3}$ represent the Fiedler vector of $x_{F3}^*$. From the definition of $V_f$, we
  know that $v_{F3}$ belongs to the set $V_f$. However, $x_{F3}^*$ is feasible for $\mathcal{F}_3$. Hence, 
  we have 
  $$v_{F3} \cdot L(x_{F3}^*) v_{F3} \geq \gamma_{F3}^*.$$
  that is,
  $$\lambda_2(L(x_{F3}^*)) \geq \gamma_{F3}^*.$$
Combining all the inequalities, we have 
  $$\gamma_{F1}^* = \lambda_2(L(x_{F1}^*)) \geq \lambda_2(L(x_{F3}^*)) \geq \gamma_{F3}^* \geq \lambda_2(L(x_{F1}^*)) = \gamma_{F1}^*.$$
 It follows that  $\gamma_{F1}^* = \gamma_{F3}^*$.
\end{proof}

\noindent
\textbf{Performance of Fiedler vector formulation:}
For the implementation purposes, we restrict the feasible solutions of formulation
${\cal F}_3$ to undirected spanning trees. Solving ${\cal F}_3$ with 
binary constraints is computationally very inefficient since
the number of the constraints are $8^6+1$ (262,145) even for the case
of eight nodes. Hence, we solve ${\cal RF}_3$ by relaxing the binary 
constraints on $x_{ij}$. From table \ref{Tab:summary_formulations}, it 
is clear that the upper bounds obtained are orders of magnitude 
higher than the optimal solutions. 

\vspace{1cm}

\noindent
\textbf{Relative strengths of the proposed formulations:}
Understanding the relative strengths of the formulations is 
easier by fixing the continuous variable $\gamma$ to a constant 
non-negative value, which shall be $\bar{\gamma}$ in all the three formulations. 
Since $\bar{\gamma}$ is chosen arbitrarily, the results 
hold true for any $\gamma$. 

For a given complete graph $G = (V,E)$, let $S$ denote the set 
of incidence vectors of spanning trees $s^j, \ j=1,\ldots, (N=|V|^{|V|-2})$. 
Let $conv(S)$ denote the convex hull of $S$, that is,
$$conv(S) := \{\sum_{j=1}^N \mu_j s^j:  \sum_{j=1}^N \mu_j = 1, \mu_j \geq 0 \ \forall j=1,\ldots, N \}$$

Since we are interested in undirected edges, we use the 
following notation for simplicity where $x_e$ represents an edge variable 
corresponding to edge $e:= \{i,j\}$. 

Based on our earlier discussion, $S$ can be defined for each of the formulations as 
follows: 
$$S_{F1} := \{x_e \in \{0,1\}^{|E|}:  \sum_{e \in E} x_{e} L_{e} \succeq \bar{\gamma} (I_n - e_0 \otimes e_0), \sum_{e \in E} x_e = n-1\},$$
$$S_{F2} := \{x_e \in \{0,1\}^{|E|}:  \sum_{e \in E} x_{e} L_{e} \succeq \bar{\gamma} (I_n - e_0 \otimes e_0), \sum_{e \in E} x_e = n-1, \; \sum_{e \in \delta(S)} x_e \geq 1 \ \forall S \subset V \}, $$
$$S_{F3} := \{x_e \in \{0,1\}^{|E|}:  v \cdot (\sum_{e \in E} x_{e} L_{e}) v  \succeq \bar{\gamma}, \sum_{e \in E} x_e = n-1\} .$$
We have shown that
$$conv(S_{F1}) = conv(S_{F2})= conv(S_{F3}).$$
Let $P_{F1},P_{F2}$ and $P_{F3}$ denote the polyhedrons 
obtained by relaxing the binary constraints on formulations ${\cal F}_1$, ${\cal F}_2$
and ${\cal F}_3$ respectively.

\begin{lemma} $P_{F2} \subseteq P_{F1} \subseteq P_{F3}$
\end{lemma}
\begin{proof}
A simple argument to prove this lemma is as follows.

From the definition of $P_{F1}$ and $P_{F2}$, we know that $P_{F2}$ 
has all the constraints of $P_{F1}$ in addition to the 
cutset constraints. Hence $P_{F2} \subseteq P_{F1}$.

Based on the definition of positive semi-definite matrices, 
the semi-definite constraint defining $P_{F1}$ can be replaced with
infinite linear constraints, that is,
$$v \cdot \left( \sum_{e \in E} x_{e} L_{e}  - \bar{\gamma} (I_n - e_0 \otimes e_0) \right) v \geq 0 \ \ \forall v.$$
Since this representation is true for any $v \in \mathbb{R}^n$, the set $P_{F1}$ can be 
relaxed by picking only a finite number of vectors, particularly  $v \in V_f$.
However, based on the definition, this relaxed set is also $P_{F3}$. 
Hence $P_{F1} \subseteq P_{F3}$. 

Combining the above two results, we have $P_{F2} \subseteq P_{F1} \subseteq P_{F3}$.
\end{proof}

The computational results summarizing the strengths of the formulations 
shown in Table \ref{Tab:summary_formulations} also matches well with the above lemma.

\begin{table}[htbp]
\caption{Summary of the binary relaxation solutions of proposed formulations.  The entries in the table represent the upper bounds due to binary relaxations and their corresponding percent  gaps from the best 
		  known feasible solution (Best FS). N/A implies that the Matlab's MISDP/MILP
		  solver could not handle those instances. For every $n$, the values 
	shown are averaged over ten random instances.}
\footnotesize
\begin{center}
\begin{tabular}{cccccccc}
		  \toprule
& & \multicolumn{2}{c}{}  \textbf{MISDP ($\mathcal{F}_1$)} & \multicolumn{2}{c}{} \textbf{MISDP with flow ($\mathcal{F}_2$)} & \multicolumn{2}{c}{}\textbf{Fiedler vector formulation ($\mathcal{F}_3$)}  \\ 
\cmidrule(l{0.25em}r{0.25em}){3-4}  \cmidrule(l{0.25em}r{0.25em}){5-6} \cmidrule(l{0.25em}r{0.25em}){7-8}
\textbf{$n$} & {Best FS} & \textbf{$\gamma_{RF_1}^*$} & {{\% gap}} & \textbf{$\gamma_{RF_2}^*$} & {{\% gap}} & \textbf{$\gamma_{RF_3}^*$} & {{\% gap}} \\ 
\cmidrule(r){1-8}
5 & 10.9751 & 22.1131 & 103.2 & 22.0626 & 102.8 & 23.6332 & 116.5 \\ 
6 & 15.7173 & 33.3133 & 113.6 & N/A & N/A & 33.8767 & 118.1 \\ 
7 & 17.6994 & 42.7876 & 144.5 & N/A & N/A & 43.2787 & 147.4 \\ 
8 & 25.5552 & 56.9862 & 123.1 & N/A & N/A & 58.2514 & 128.1 \\ 
9 & 28.0676 & 77.5151 & 177.8 & N/A & N/A & N/A & N/A \\ 
10 & 38.1984 & 105.7051 & 178.5 & N/A & N/A & N/A & N/A \\
12 & 52.0502 & 146.2322 & 181.9 & N/A & N/A & N/A & N/A \\ 
\bottomrule
\end{tabular}
\end{center}
\label{Tab:summary_formulations}
\end{table}

\vspace{0.8cm}

\noindent
\textbf{Useful features of the proposed formulations:}

\begin{itemize}
\item Formulation ${\cal F}_1$ provides a compact representation
of maximizing the algebraic connectivity via it's semi-definite constraint.
Since maximizing algebraic connectivity automatically ensures connectedness
in graphs, additional connectivity constraints can be omitted. 

\item As it is, solving  ${\cal F}_1$ to optimality 
is computationally inefficient since the available MISDP solvers have
limited features. However, by relaxing the semi-definite
constraint in ${\cal F}_2$ using a finite number of vectors, one can readily have
a tighter upper bound by solving the corresponding MILP. Of course, the quality 
of the upper bound depends on the number and the type of the vectors chosen.

\item The multicommodity flow constraints in ${\cal F}_2$ come in 
handy to enforce connectedness in feasible solutions while solving MILP with
relaxed semi-definite constraints. 

\item The Fiedler vectors of feasible solutions used in formulation ${\cal F}_3$ can 
be readily used to relax the semi-definite constraint by choosing 
a few of the many vectors from the set $V_f$.

\item The solution to ${\cal F}_2$ with relaxed semi-definite constraints need not be feasible
for ${\cal F}_1$, that is, the solution ($x_{e}^*,\gamma^*_{RF_1}$) need not satisfy the 
semi-definite constraint. This implies that the matrix 
$\sum_{e \in E} x_{e}^* L_{e}  - {\gamma^*_{RF_1}} (I_n - e_0 \otimes e_0)$ will have a negative 
eigenvalue. Hence, based on the eigenvector corresponding to the negative eigenvalue, 
one can develop a cutting plane\footnote{A brief discussion on the concept of 
cutting planes can be found in section \ref{sec:exact_algo}.} 
to eliminate the current solution and possibly many
other non-optimal solutions. Similarly, a sequence of 
cutting planes  can be generated until an optimal solution is obtained. 
This summarizes the basic idea of the cutting plane algorithms which are discussed 
in detail in section \ref{sec:exact_algo}.
\end{itemize}

In summary, this section has basically dealt with development of three equivalent formulations for the
\textbf{BP} and summarized the quality of upper bounds and the strengths of formulations 
obtained by considering their respective binary relaxations. In the next section,
we utilize the various features of these formulations as discussed 
to develop tighter upper bounds and ultimately obtain optimal solutions
asymptotically.


\section{Upper bounds on algebraic connectivity}
\label{sec:UB}
We discussed in the earlier section that the quality of the binary 
relaxations for all the three formulations was poor and became 
worse with an increase in the problem size.
In this section, we mainly focus on developing techniques 
to obtain tight upper bounds for the \textbf{BP}. For
any spanning tree as a feasible solution, we first develop 
a method to relax the semi-definite constraint 
using Fiedler vectors; this relaxation seems to provide better 
bounds than binary relaxations. 


\subsection{Relaxation of the semi-definite constraint using Fiedler vectors}
\label{subsec:UB_1}
From our earlier discussion in section \ref{Subsec:F3}, 
we know that the MILP formulation in $\mathcal{F}_3$ is equivalent to solving the MISDP 
formulation in $\mathcal{F}_1$. However, even for problems of moderate sizes ($n \geq 8$), 
it would be impractical to enumerate
all the Fiedler vectors of feasible solutions in $\mathcal{F}_3$. 
However, by considering only a few vectors from the many Fiedler vectors 
of the set $V_f$ and maintaining the integrality constraints, 
one can readily obtain upper bounds on the algebraic connectivity due to 
the relaxation of the feasible set. 
Earlier, in section \ref{Subsec:F1}, we had briefly alluded to the concept of approximating
the feasible set of $\mathcal{F}_1$ by relaxing the semi-definite 
constraint with finite set of linear inequalities with each inequality
identified with an appropriate unit vector. However, in this section, we 
restrict the relaxation of the feasible set using the Fiedler vectors and discuss the 
quality of the bounds obtained from such relaxations.


The quality of the upper bound from relaxing the semi-definite constraint
using Fiedler vectors depends on the following two factors: 1) Type
of feasible solutions whose Fiedler vectors are considered, 
2) The number of Fiedler vectors considered. 
Hence, the main focus of this section 
will be to construct appropriate Fiedler vectors which provide tight upper bounds  
and study their quality.

\vspace{0.5cm}
\noindent
\textbf{Choosing the type of Fiedler vectors to relax the semi-definite constraint:}
We observed that the relaxation of the semi-definite  
constraint with the Fiedler vectors of spanning trees with higher value of
algebraic connectivity gives very good upper bounds. A simple, but a 
rough geometric interpretation of this hypothesis is as follows: 
From Figure \ref{Fig:outer_approx_concept}, it is clear that the relaxation of 
the feasible semi-definite set with the Fiedler vectors corresponding to the
spanning trees with higher algebraic connectivity ($\gamma$) gives better upper bound.
However, it can also be observed that, without the Fiedler vector 
corresponding to the optimal solution, $\gamma^*_{UB1}$ will be
strictly greater than $\gamma^*$ irrespective of the number of Fiedler vectors 
used for the relaxation. This can also be easily deduced from
Lemma \ref{lemma:F3}.

\begin{figure}[!h]
	\centering
	\subfigure[Tighter relaxation]{
	\includegraphics[scale=0.69]{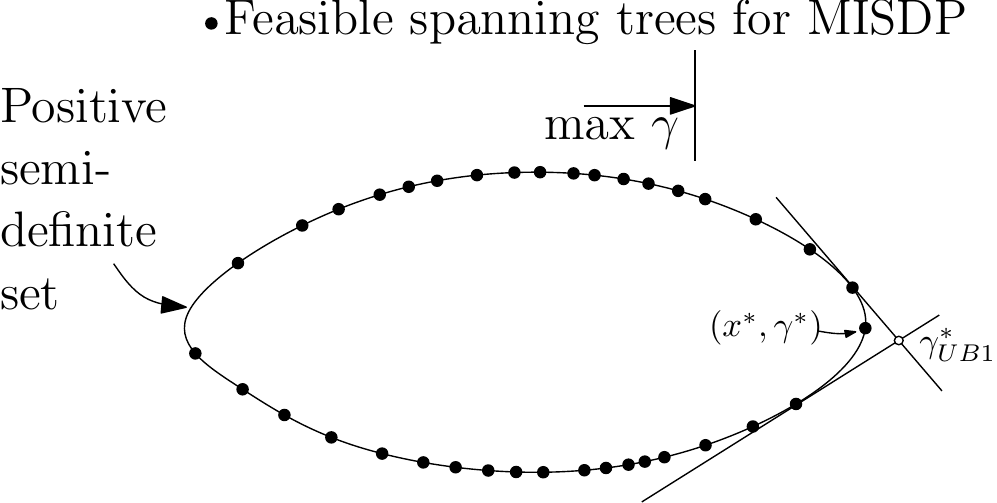}}
	\subfigure[Weaker relaxation]{
	\includegraphics[scale=0.69]{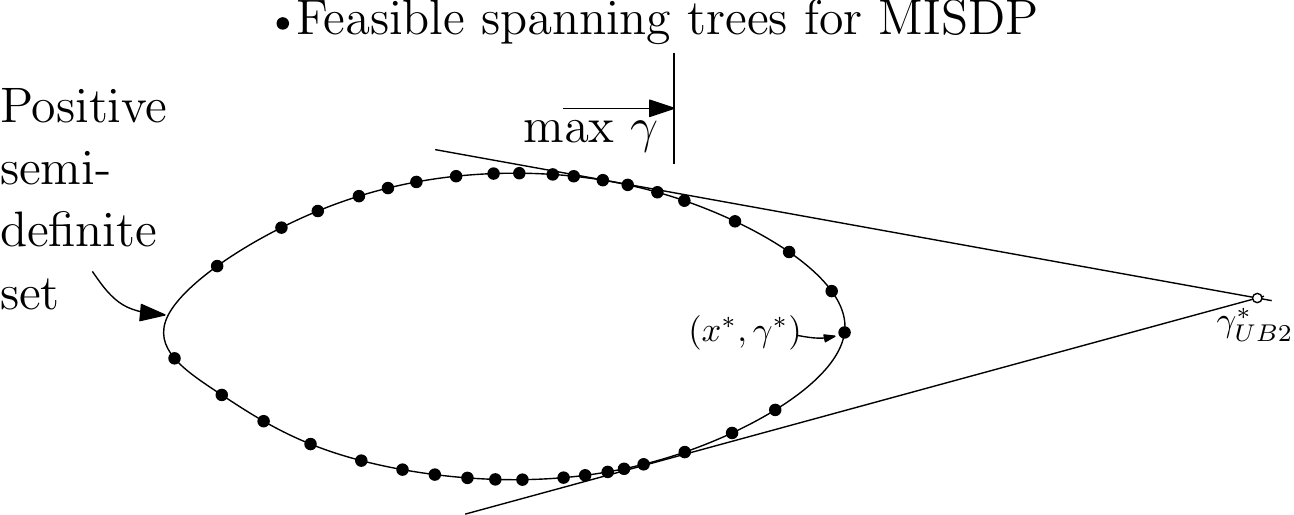}}
	\caption{A geometric interpretation of the relaxation of the 
	 semi-definite constraint using Fiedler vectors of feasible solutions. In (a), the upper bound obtained ($\gamma^*_{UB1}$)
        from Fiedler vectors of spanning trees with higher algebraic connectivity is tighter than the upper
	bound obtained ($\gamma^*_{UB2}$) from Fiedler vectors of spanning trees with lower algebraic connectivity.} 
	\label{Fig:outer_approx_concept}
\end{figure}

\vspace{0.5cm}
\noindent
\textbf{Constructing good feasible solutions:}
A priori, for a given complete weighted graph, we neither know the optimal spanning tree 
with maximum algebraic connectivity nor the sub-optimal spanning trees with 
higher algebraic connectivities. However, by enumerating all the spanning trees for small
instances, one can observe that the spanning trees with higher
algebraic connectivity tend to have larger values of the sum of the weights of the edges 
in the tree. This trend can be clearly observed in Figure \ref{Fig:enumeration} 
for instances with six and seven nodes. It can also be noted from the 
figure that the maximum spanning tree is not necessarily the spanning tree with maximum 
algebraic connectivity. However, from Table \ref{Tab:tree_ranks}, 
we can see that the tree with maximum algebraic connectivity occurs in
the first few thousands (up to eight nodes) while enumerating 
all the spanning trees in the decreasing order of the sum of the weights
of the edges in the tree. The enumeration of all the spanning trees in 
Figures \ref{Fig:enumeration}(a) and \ref{Fig:enumeration}(b) 
corresponds to the third and the first instance in Table \ref{Tab:tree_ranks} respectively. 
These are the worst case instances where the optimal spanning tree is farthest from the 
maximum spanning tree.

Based on these ideas, we now present a systematic 
procedure to construct the Fiedler vectors used for relaxation of the semi-definite constraint 
and discuss the quality of the upper bounds obtained. 

\begin{figure}[htp]
	\centering
	\subfigure[6 nodes]{
	\includegraphics[scale=0.44]{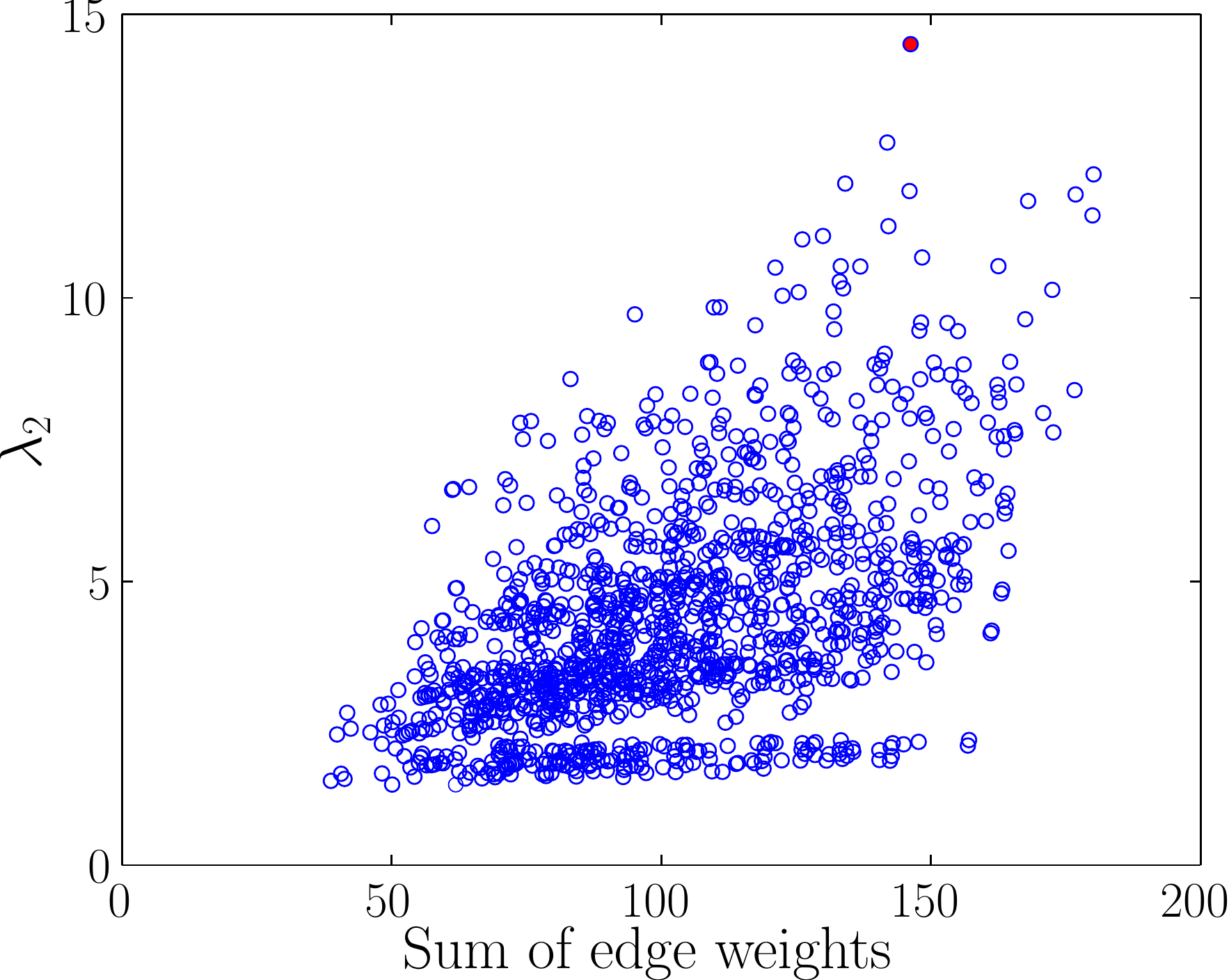}}
	\subfigure[7 nodes]{
	\includegraphics[scale=0.44]{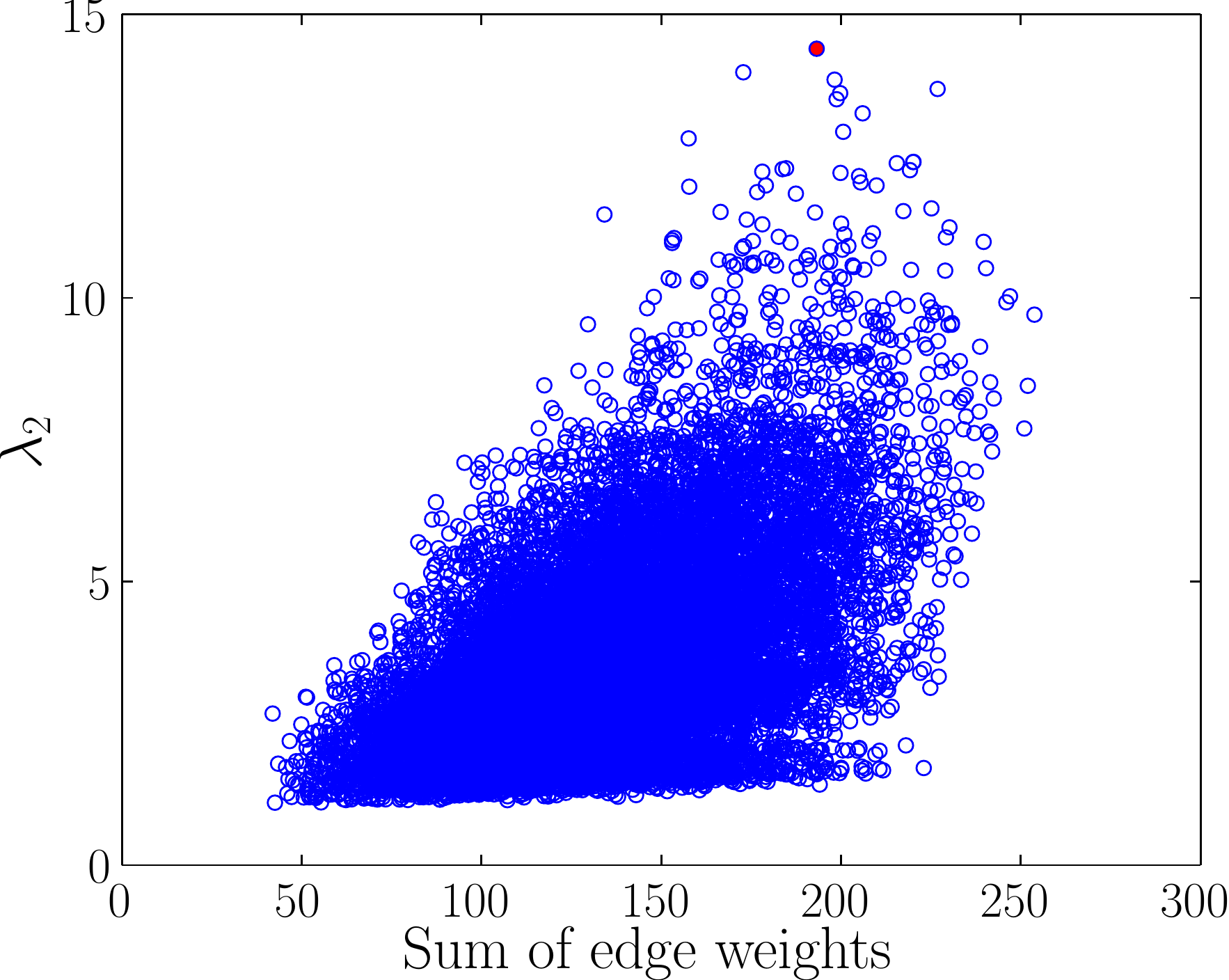}}
	\caption{Graphical representation of the distribution of algebraic
	connectivity ($\lambda_2$) values for all the spanning trees over a 
	random complete graph. The tree with 
	maximum $\lambda_2$ is indicated by the circle filled with red color. 
	It can be observed that the trees with larger values of $\lambda_2$ 
	tend to have larger sum of the edge weights.}
	\label{Fig:enumeration}
\end{figure}

\begin{table}[htbp]
\caption{The entries of this table represent the position of the spanning tree
with maximum algebraic connectivity (optimal solution) in the enumerated list of 
spanning trees, where the enumeration is in the decreasing order of the sum
of the edge weights.}
\footnotesize
\begin{center}
\begin{tabular}{cccc}
\toprule
\textbf{Instances}  & \textbf{6 nodes}& \textbf{7 nodes} & \textbf{8 nodes} \\ 
\cmidrule(r){1-4}
1 &39 &819 &5126 \\  
2 &21 &47 &530 \\  
3 &97 &48 &81 \\  
4 &2 &466 &1058 \\  
5 &3 &109 &704 \\  
6 &30 &10 &9312 \\  
7 &11 &243 &398 \\  
8 &2 &3 &12805 \\  
9 &92 &189 &11991 \\  
10 &2 &312 &1225 \\  
\cmidrule(r){1-4}
Average & 30 & 225 & 4323 \\
\bottomrule
\end{tabular}
\end{center}
\label{Tab:tree_ranks}
\end{table}

\begin{enumerate}[(a)]
\item Enumerate a fixed number of spanning trees in the
decreasing order of the sum of the weights of the edges in the tree, where the 
first tree in the enumerated list will be a maximum spanning tree.  
\item Rank the enumerated spanning trees in the decreasing order of their algebraic connectivity
values, that is, the tree with rank one will have the maximum value of 
algebraic connectivity among the enumerated spanning trees. 
\item Pick a fixed number of the first few ranked spanning trees in the
decreasing order of the algebraic connectivity values. 
Their Fiedler vectors can be used to relax the semi-definite constraint. 
\end{enumerate}

\noindent
\textbf{Quality of the upper bounds:}
In order to enumerate a fixed number of spanning trees from
the maximum spanning tree, a standard enumeration
algorithm for weighted graphs as given in \cite{gabow1986efficient} was implemented in Matlab.
The optimal spanning tree for up to eight nodes was within 4,323 trees 
while averaged over ten instances and the worst case being 12,805 as shown in 
Table \ref{Tab:tree_ranks}. Hence, from the maximum spanning tree, 
we enumerated 15,000 spanning trees for every random instance up to twelve nodes.
The computation time for enumerating up to 15,000 spanning trees for graphs of 
sizes up to twelve nodes was less than ten minutes.

The performance of the relaxation of formulation $\mathcal{F}_3$ with various number
the Fiedler vectors used for relaxation is shown in Figure \ref{Fig:UB_performance_1}.
The percent gap shown in Figure \ref{Fig:UB_performance_1} is defined as follows: 
$$\mathrm{percent \ gap} = \frac{\gamma^*_{UB} - \gamma_{bfs}}{\gamma_{bfs}} * 100, $$
where $\gamma^*_{UB}$ is the upper bound obtained by solving the relaxed formulation 
$\mathcal{F}_3$ and $\gamma_{bfs}$ is the algebraic connectivity 
of the best feasible solution known. The best feasible solution in this case 
will be the spanning tree with maximum algebraic connectivity among the 15,000
enumerated trees.

In Figure \ref{Fig:UB_performance_1}, it can be noted that the average 
percent gap obtained by relaxing 
with thousand Fiedler vectors for instances of eight, nine, ten and twelve nodes are 
4.13\%, 42.9\%, 65.7\% and 116.1\% respectively. Though the gaps grow with the 
increase in problem size, they are orders of magnitude
better than the binary relaxation gaps shown in Table \ref{Tab:summary_formulations}.
Also, for the best instances shown in Figure \ref{Fig:UB_performance_1}, 
the optimal solution was obtained with just eight hundred Fiedler vectors for the case of 
eight nodes.

\begin{figure}[!h]
	\centering
	\subfigure[8 nodes]{
	\includegraphics[scale=0.42]{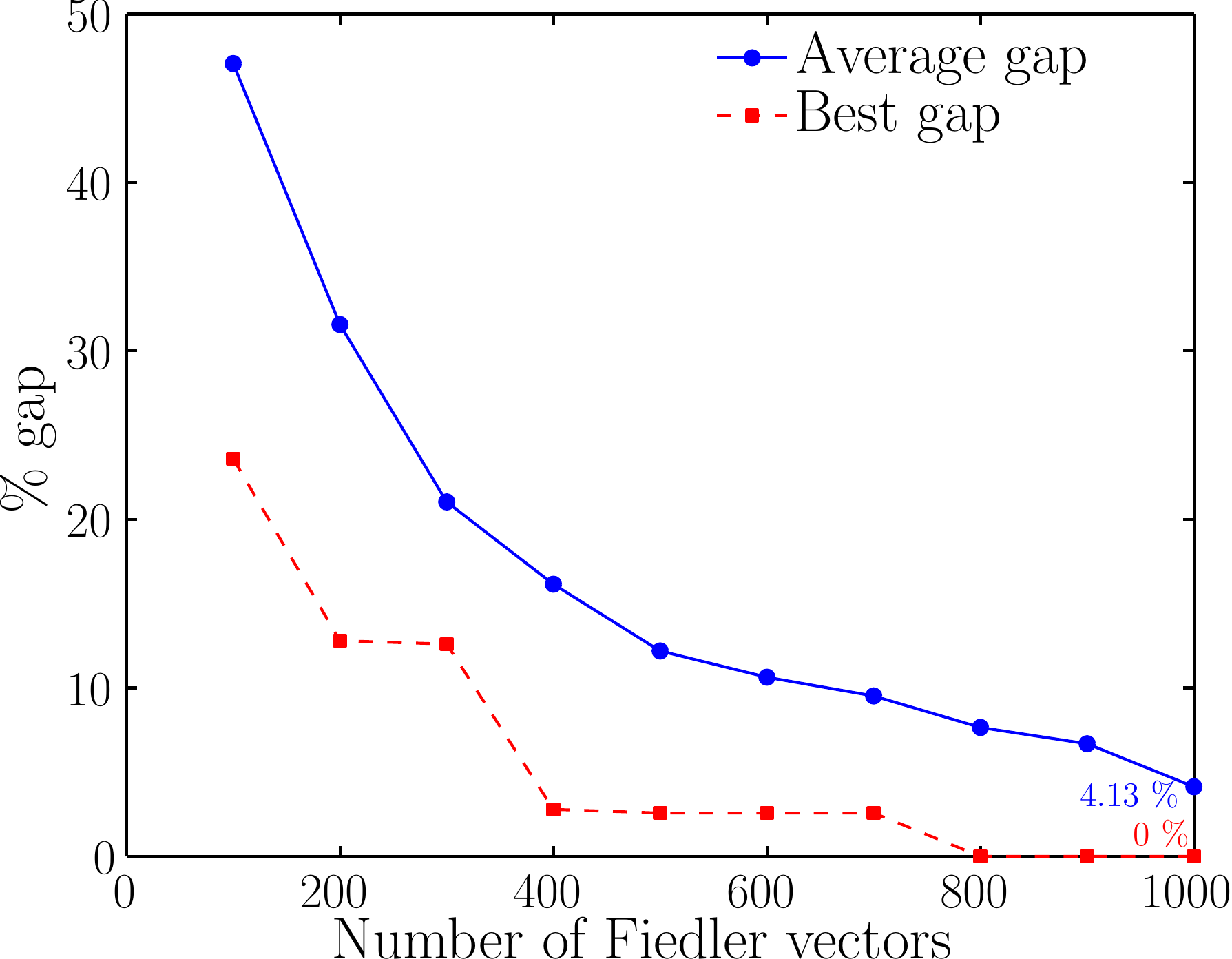}}
	\subfigure[9 nodes]{
	\includegraphics[scale=0.42]{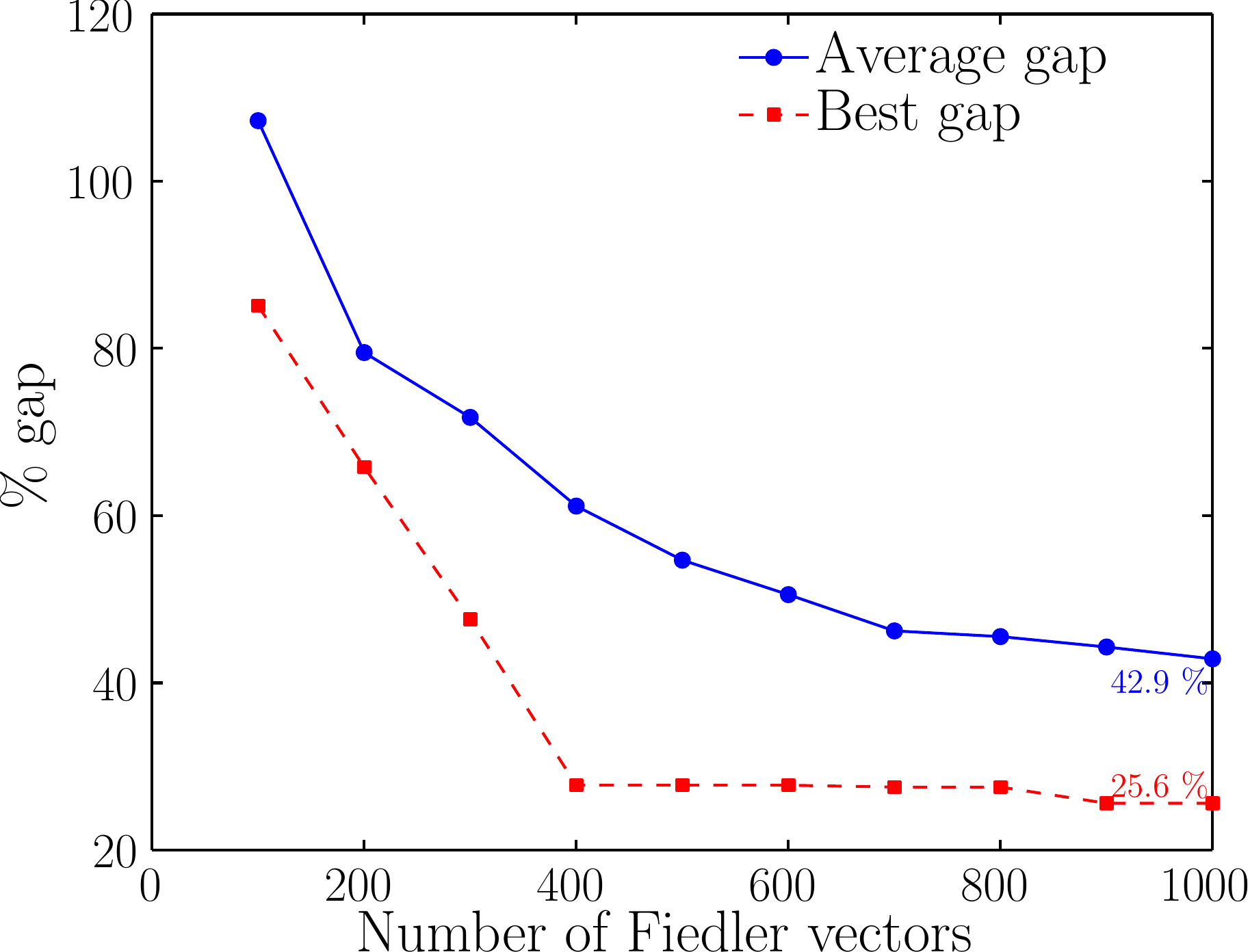}}
	\subfigure[10 nodes]{
	\includegraphics[scale=0.42]{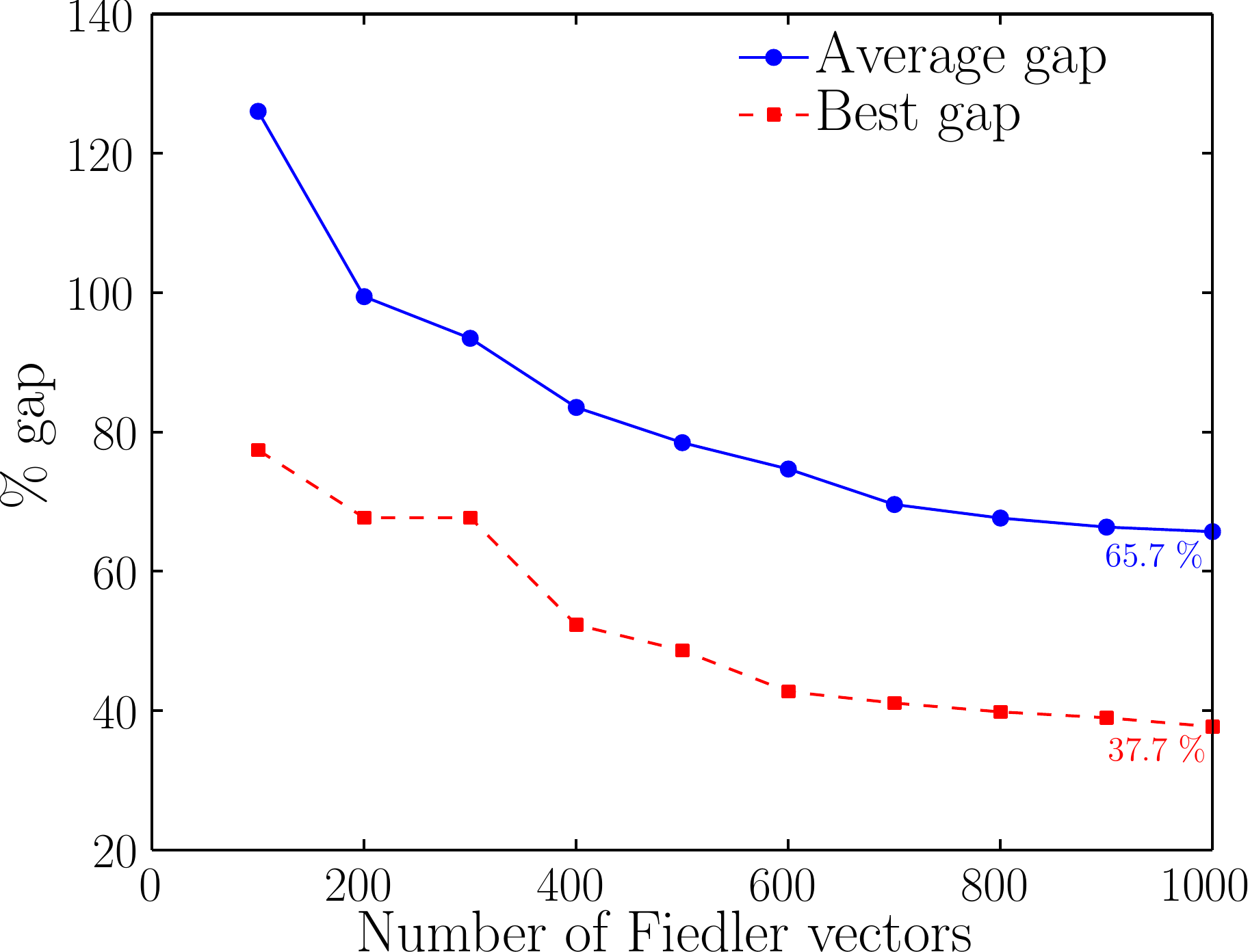}}
	\subfigure[12 nodes]{
	\includegraphics[scale=0.42]{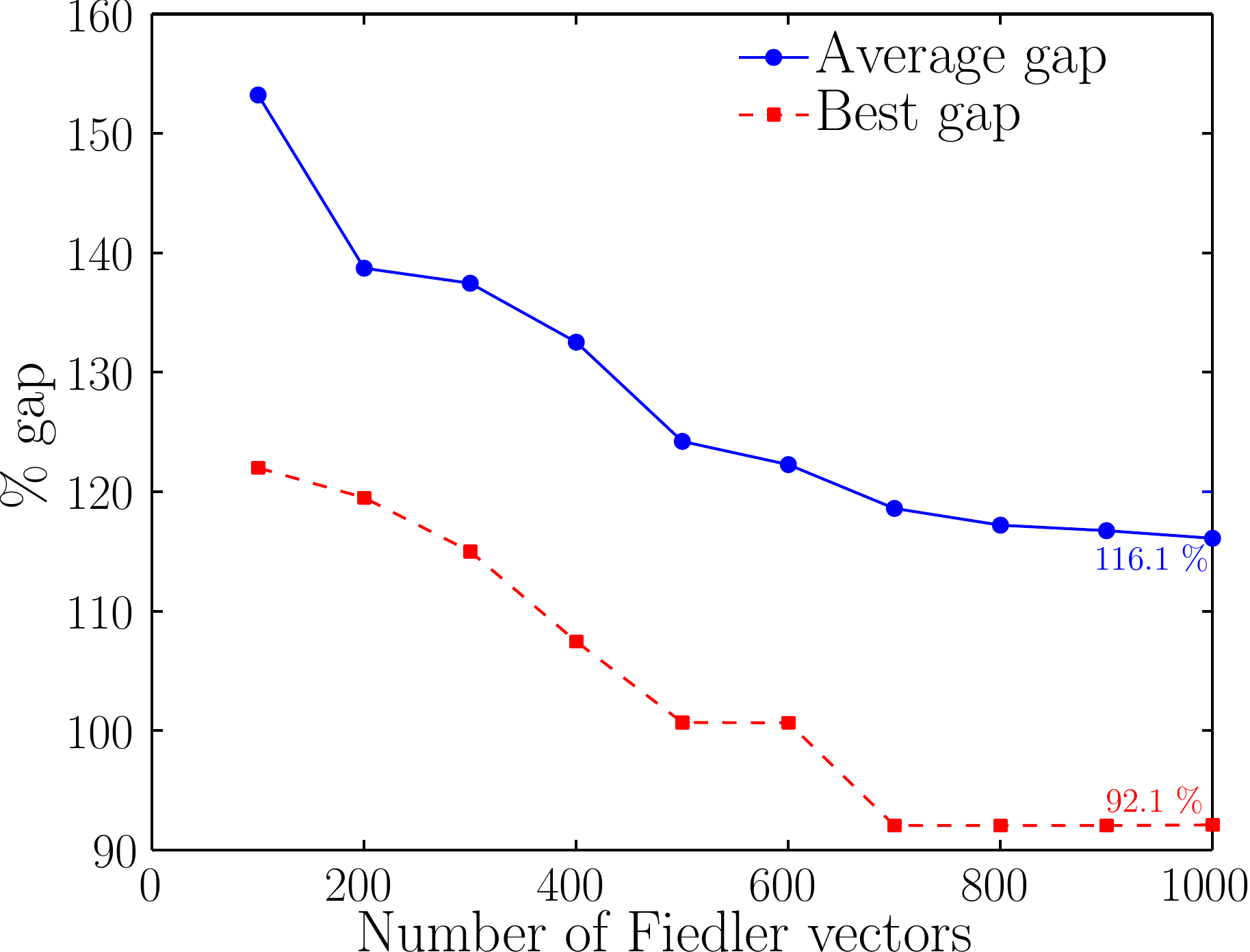}}
	\caption{Plot of the percent deviation of the upper bounds 
	obtained by relaxing the semi-definite constraint 
	using the Fiedler vectors of good solutions from the best known feasible solution. Average 
	gap corresponds to the average value evaluated over ten random instances and 
   the best gap corresponds to the instance for which the percent gap was minimum. Source: \cite{nagarajan2015maximizing}.}
	\label{Fig:UB_performance_1}
\end{figure}

%

\section{Algorithms for determining maximum algebraic connectivity}
\label{sec:exact_algo}
In this section, we focus on developing algorithms based on
cutting plane techniques to obtain 
optimal solutions for the problem of maximizing algebraic 
connectivity (\textbf{BP}).  

In principle, the available MISDP solvers in Matlab can be employed to solve the problem ${\cal F}_1$ in \eqref{eq:F1}.
A well known state-of-the-art solver employed for
solving MISDPs is the SEDUMI \cite{sturm1999using} toolbox which 
can be accessed with the YALMIP user interface \cite{lofberg2004yalmip}.
However, even on a powerful workstation, the time to compute an optimal solution using these solvers
was in the order of hours for instances involving at most eight nodes and  
couldn't handle instances involving nine nodes or more. (for 8 and 9 nodes, the number of
feasible solutions are 262,144 and 4,782,969 respectively). 

There is a need for developing algorithms that provide tight upper bounds, in case the 
optimal solution cannot be computed efficiently. The cutting plane technique can be used 
to provide a monotonically decreasing sequence of upper bounds that converge to the 
optimal value of algebraic connectivity. In the previous section,
we discussed in detail an efficient way to approximate the feasible set 
using Fiedler vectors of feasible solutions to obtain tight upper bounds. 
Based on the cutting plane techniques, one can always further tighten the upper bound 
and eventually obtain optimal solutions. Therefore, after a brief introduction to the 
concepts of cutting plane techniques, we propose three cutting plane algorithms 
to solve the problem of maximizing algebraic connectivity to optimality.

\subsection{Cutting plane techniques}
In combinatorial optimization problems, cutting plane 
method generally refers to an iterative refinement 
of the feasible set by means of valid linear inequalities
or ``cuts" or ``cutting planes". The procedure of adding 
cutting planes to obtain optimal solutions are popularly 
used for solving MILPs. In the early sixties, Gomory in his papers 
\cite{gomory1958outline,gomory1960algorithm} 
proposed to solve integer
programs by using cutting planes, thus reducing an integer
programming problem to the solution of a sequence of linear programs. 
Later, in the early 1990s, Ceria et. al., in their paper \cite{ceria1995combining}
introduced a branch-and-cut approach to solve MILPs
which effectively combined the usage of Gomory cuts 
with the branch-and-bound procedure. 

Cutting plane method for MILPs with a maximization 
objective works as follows: Solve
the MILP by relaxing the integrality constraints to obtain an 
easily solvable linear program. Since this
is a relaxation, the optimum value obtained will be an
upper bound to the original MILP. 
If the optimal solution obtained for the relaxed MILP is not an 
integer solution, then there is guaranteed to exist a linear 
inequality or a ``valid inequality'' or a ``cutting plane" or simply a cut that separates this optimal
solution from the convex hull of the feasible set of MILP.
Finding such a cutting plane is the ``separation problem". 
An improved relaxation to the MILP can be constructed by adding the cut
to the existing relaxation. The linear inequality is satisfied by the optimal 
solution of the MILP; however, it is not satisfied by the non-integral optimal 
solution of the relaxed linear program. The optimal value of the 
relaxed linear program provides a tighter upper bound to the MILP.
Solving a sequence of such linear programs with monotonically decreasing upper bound 
until an integer solution is found is the essence of the
``cutting plane method". Having a polynomial time
solvable separation problem for any MILP is not trivial. However, in 
the literature, there are many separation heuristics for specific problems,
where these heuristics are not guaranteed to generate cutting planes 
for every solution of the relaxed MILP. 

In this dissertation, we extend the idea of the standard 
cutting plane method for MILPs to solve the proposed
MISDP problem. Instead of relaxing the binary constraints
in formulation ${\cal F}_1$, we relax the semi-definite 
constraint using a finite number of Fiedler vectors and solve the corresponding MILP 
as discussed in the previous section. 
To enforce connectivity in the feasible solutions for the MILP, 
we invoke the multicommodity flow formulation as discussed in formulation
${\cal F}_2$. Clearly, the solution obtained by solving the
MILP need not be feasible for the MISDP since the semi-definite constraint
can be violated. Hence, we add a valid inequality which 
eliminates the current integral solution to obtain an  
augmented MILP. Solving a sequence of such augmented MILPs
terminates when the current solution is also feasible
for the MISDP. 

In the remainder of this section, we provide a detailed
discussion on developing three algorithms based on the cutting
plane techniques as discussed above. Firstly, we discuss 
a cutting plane algorithm, where a sequence of MILPs are 
solved by relaxing the semi-definite constraint using a finite number of 
Fiedler vectors. Secondly,
we provide a bisection algorithm to reduce the MISDP to 
a sequence of BSDPs and discuss the gains in the computational efficiency. 
Thirdly, we discuss an iterative primal-dual algorithm
based on the Lagrangian relaxation of the semi-definite constraint.

\subsection{$EA_1$: Algorithm to compute maximum algebraic connectivity }
\label{Sec:CP_algo}
$EA_1$ ($EA$ stands for an algorithm that computes an optimal solution exactly)
involves the construction of successively tighter polyhedral approximations 
of the positive semi-definite set corresponding the maximum algebraic connectivity problem 
given in formulation ${\cal F}_1$.

If one were to store the Fiedler vectors of some feasible solutions (spanning trees), 
one can relax the semi-definite constraint in ${\cal F}_1$ as follows: 
$$ \sum_{e \in E} x_e Q_i \cdot L_e - \gamma Q_i \cdot (I_n -e_0 \otimes e_0) \geq 0, \; \; i = 1, \ldots, N, $$
where $N$ is the pre-specified number of constraints used in the 
termination criteria and $Q_i, \; i=1, 2, \ldots, N$ are the dyads 
associated with the Fiedler vectors corresponding to the feasible solutions. 
If one were to directly approach the solution, one may pick a bunch of random feasible 
solutions and construct the associated $Q_i$'s from their Fiedler vectors. In order to have a tighter
initial relaxation of the feasible set, one can also construct the $Q_i$'s from the 
feasible solutions as discussed in section \ref{subsec:UB_1}. At the end of this section, we
shall discuss the computational efficiency of $EA_1$ 
by choosing such special Fiedler vectors for the relaxation of the semi-definite constraint. 
One may then perform the following iteration to obtain optimal solution:
\begin{enumerate}
\item Solve the following MILP using as follows:

{\begin{equation*}
		\begin{array}{ll}
			 & \max \gamma \\
 				\text{s.t.} & \sum x_e Q_i \cdot L_e - \gamma Q_i \cdot (I_n -e_0\otimes e_0) \geq 0, \ \textrm{ for }i = 1,\ldots, N, \\
			& \sum_{e \in E} x_e \leq q, \\
		  &\sum_{e \in \delta(S)} x_e \geq 1, \ \ \forall \ S \subset V, \\
			& x_e \in \{0, 1\}^{|E|}.
		\end{array}
\end{equation*}}
One may observe that the exponential number of cutset constraints can be replaced with the multicommodity
flow formulation as discussed in formulation ${\cal F}_2$.

\item Check if the optimal solution $x^*$ to the above MILP satisfies the semi-definite constraint:
$$ \sum_{e \in E} x_e^* L_e - \gamma (I_n - e_0 \otimes e_0) \succeq 0. $$
If not, one can construct a cut associated with the negative eigenvalue 
of $\sum_{e \in E} x_e^* L_e - \gamma (I_n - e_0 \otimes e_0)$ by first 
determining the corresponding eigenvector $v_{N+1}$ and constructing a 
semi-definite $Q_{N+1} = v_{N+1} \otimes v_{N+1}$. We can augment the 
MILP with the following scalar linear constraint which cuts off this undesirable solution:
$$ \sum_{e \in E} x_e Q_{N+1} \cdot L_e - \gamma Q_{N+1} \cdot (I_n -e_0 \otimes e_0) \geq 0, $$
which is clearly not satisfied when $x_e = x_e^*$, but is satisfied by the optimal solution.

\item One may then solve the augmented MILP using dual simplex algorithm.

\item This procedure is iterated until $x_e^*$ satisfies the semi-definite 
constraint. Hence, $x_e^*$ is an optimal solution.
\end{enumerate}

\subsection{$EA_2$: Algorithm to compute maximum algebraic connectivity }
\label{Sec:PD_algo}

\begin{algorithm}[h!]
\caption{\textbf{: Iterative primal-dual algorithm ($EA_2$)}}
\label{Algo:Iterative_primal_dual}
\begin{algorithmic}[1]
\STATE Input: A primal feasible solution
\STATE Let $P$ := Given primal feasible solution. Let the Fiedler vector of $P$ be denoted as $v_P$ and its corresponding eigenvalue represented as $\gamma_P$
\STATE primalCost $\gets$ $\gamma_P$
\STATE dualCost $\gets$ $\infty$
\STATE DualGap $\gets$ dualCost-primalCost
\IF{DualGap $>$ $0$}
\STATE Use $v_P$ to obtain another primal solution, $P^*$, by solving the following dual problem:

{\begin{equation*}
		\begin{array}{ll}
				  {\mathrm{dualCost}}_{\mathit{P^*}} = & \max \   v_P \cdot (\sum_{e \in E} x_e L_e)v_P \\
			\text{subject to} &  \sum_{e \in \delta(S)} x_e \geq 1, \ \ \forall \ S \subset V, \\
			& \sum_{e \in E} x_e \leq q, \\
			& x_e \in \{0, 1\}^{|E|}.
		\end{array}
\end{equation*}}

\STATE $P_{t}\gets P^*$
\STATE $Cuts \gets \emptyset$
\WHILE{$\gamma_{P} > \gamma_{P_t}$}
				\STATE Augment $Cuts$ with the following constraint:
                \begin{equation*}
                v_{P_{t}} \cdot (\sum_{e \in E} x_e L_e)v_{P_t} \geq \gamma_{P}
                \end{equation*}
                \STATE Find $P^*$ again by solving the above dual problem with all the additional constraints in $Cuts$.
                \STATE $P_t\gets P^*$
				\ENDWHILE
\STATE \STATE $P$ $\gets$ $P_t$
\STATE primalCost $\gets$ $\gamma_{P_t}$
\STATE dualCost $\gets$ $\min$ (dualCost,dualCost$_{P_t}$)
\STATE DualGap $\gets$ dualCost-primalCost
\STATE \textbf{Termination criterion:} \textit{if} DualGap $>0$ {\it return} to line 7, \textit{else }exit with $P$ as the optimal primal solution.
\ENDIF 
\end{algorithmic}
\end{algorithm}

$EA_2$ is a cutting plane algorithm 
based on the iterative primal-dual method
as outlined in Algorithm \ref{Algo:Iterative_primal_dual}. 
In this approach, we start with a feasible solution to the 
primal problem and iteratively
update this feasible solution with a new solution by solving a
related dual problem. The current feasible solution to the
primal problem is only updated with a new solution if the
algebraic connectivity of the new solution is greater than
the algebraic connectivity of the current feasible solution.
On the other hand, if it is certain that the new solution
found using the dual problem is not optimum, a cutting plane is augmented 
to the dual and the dual problem 
is solved again ({\it refer to lines 11-12 of the algorithm}). 
The dual problem is resolved with additional cutting planes
until it produces a new solution that is at least
as good as the current primal feasible solution
({\it refer to lines 10-14 of the algorithm}). The algorithm eventually
terminates when the dual cost equals the algebraic connectivity
of the best known primal solution 
({\it refer to line 20 of the algorithm}). A feature of this algorithm is that
the solutions from the dual problem can be continually
used to improve the primal feasible solution while
continuously decreasing the optimal dual cost and hence 
the upper bound.

In the following discussion, we discuss the formulation
of the dual problem related to the primal and some efficient
ways to solve the same. We also outline how to generate
cutting planes if the solution to the dual problem does
not produce an optimal solution.

We form the dual problem by relaxing the semi-definite constraint,
$$\sum_{e \in E} x_e L_e \succeq \gamma (I_n - e_0 \otimes e_0),$$
and penalizing the objective with a dual variable $Q \in \mathbb{R}^{|V| \times |V|}$ if the constraint
is violated. Let ${\cal T}$ be the set of networks on $(V,E,w_e)$
which are connected and have at most $q$ edges from $E$.
Then, one may express the dual function $\Pi(Q)$,
with its domain being $Q \succeq 0$
and $Q \cdot (I_n - e_0 \otimes e_0) = 1$.
One may compute $\Pi(Q)$ for every $Q$ in its domain as:
$$\Pi(Q) = \max_{x \in {\cal T}}  [\sum_{e \in E} x_e (Q \cdot L_e)]. $$
The computation of $\Pi(Q)$ may be carried out using the
greedy algorithms for spanning trees (which are the simplest of the connected
networks) given in \cite{cormen2001introduction}, \cite{lawler2001combinatorial}
mimicking Prim's or Kruskal algorithm. The property of connectivity
is taken into account by the algorithm and hence, is simple and yet
efficient. Since $\Pi(Q)$ is a dual function, it is automatically
an upper bound for the maximum algebraic connectivity for every $Q$ 
in its domain. In our approach, at any iteration of the algorithm,
the $Q$ we pick to solve the dual problem corresponds to the best
known feasible solution, $P$, available to the primal problem,
$i.e.$, $Q$ is chosen to be equal to $v_P\otimes v_P$ where $v_P$
is the Fiedler vector corresponding to $P$. Note that such a choice
of $Q$ always satisfies the constraints $Q \succeq 0$, $Q \cdot (I_n - e_0 \otimes e_0) = 1$
and is therefore always feasible to the dual. If a solution (say,
an optimal tree denoted by $P^*$) that solves the dual problem has an
algebraic connectivity greater than the algebraic connectivity of the
primal solution $P$, then the primal solution is replaced with the
optimal tree ($i.e.$, $P:=P^*$) and a new iteration is started again.
If algebraic connectivity of $P^*$ is less than that of $P$, then the
dual problem is augmented with the following cutting plane and
solved again. This procedure is repeated until either the dual problem
finds a tree with a greater algebraic connectivity or the dual cost
equals the primal cost in which case the algorithm terminates. The
cutting plane that is added is:
$$v_{P^*} \cdot (\sum_{e \in E} x_e L_e) v_{P^*} \geq \lambda_2(L(P)),$$
where $v_{P^*}$ denotes the Fiedler vector corresponding to
the tree $P^*$. Observe that the above inequality is violated
if $x$ is chosen to be $P^*$ since 
		  $v_{P^*} \cdot L({P^*})v_{P^*} < \lambda_2(L(P))$.
However, from Rayleigh's inequality, the optimal solution
to the primal problem always satisfies the above inequality.

\begin{remark}
The outer iteration of the algorithm ($lines~~ 6-21$)
terminates when the dual gap becomes zero. If at the end of an outer
iteration, the dual gap is not zero, several dual problems are solved
until a tree with better algebraic connectivity is found. In the worst
case, the number of dual problems that need to be solved in an outer
iteration will be at most equal to the number of feasible structures
available. Since during every outer iteration, the increment in the
algebraic connectivity is positive and the number of dual problems
that need to be solved is bounded, the algorithm will eventually
terminate with an optimal solution in finite steps.
\end{remark}

\subsection{Performance of algorithms $EA_1$ and $EA_2$}
\label{Sec:results1}
The MISDP formulation ${\cal F}_1$
was implemented using Matlab's toolboxes, SeDuMi 
and YALMIP which are state-of-the-art semi-definite solvers widely used
among the researchers in the area of semi-definite programing.
The proposed exact algorithms were implemented
in \texttt{C++} programing language and
the resulting MILP's were solved using CPLEX 12.2 with 
the default solver options.
All computational results in this paper were implemented on
a Dell Precision T5500 workstation (Intel Xeon E5630 processor @ 2.53GHz,
12GB RAM). \\

\noindent
\textbf{Construction of random instances:}
Random weighted adjacency matrix, $A$, for each instance was generated using $A = (M \circ R) + (M \circ R)^T$
where $\circ$ denotes the Hadamard product of magic square ($M$) and a
randomly generated square matrix ($R$) with zero diagonal entries. The entries of $R$ 
are the pseudorandom values drawn from the standard uniform distribution on the 
open interval (0,1) \cite{MATLAB2010}. 
The term $A_{ij}$ corresponds to the edge weight which may be chosen to
connect nodes $i$ and $j$. Every random cost matrix was chosen such that 
the maximum spanning tree's algebraic connectivity was greater than the 
algebraic connectivity of all the star graphs (spanning tree with $|V|$ nodes such that the 
internal node has a degree equal to $|V|-1$).
This ensured that the optimal solutions were non-trivial connected graphs. 
Adjacency matrices corresponding to the ten
weighted complete graphs of eight nodes are shown in Appendix \ref{ch:appendix}.

Corresponding to the weighted adjacency matrices in Appendix \ref{ch:appendix},
the optimal spanning trees with maximum algebraic connectivity
are shown in Figure \ref{Fig:Optimal_graphs_n8}.

In Table \ref{Table:Yalmip_Cplex_n_8_ch2}, for eight node networks, we compare the performance 
of the proposed algorithms implemented in CPLEX with the performance of directly 
solving the MISDP formulation ${\cal F}_1$ in MATLAB's SDP solver. 
On an average, the two proposed algorithms performed better than the SDP solver 
in Matlab. Moreover, the $EA_2$ based on iterative primal 
dual method performed 1.2 times faster than the $EA_1$ based on the
polyhedral relaxation of the semi-definite constraint. Also, we observed that
the MISDP solver in MATLAB ceased to reduce the 
gap between the upper and lower bounds it maintained during its 
branch-and-bound routine for networks with nine nodes and hence
was practically impossible to solve. The proposed algorithms solved the nine
node problem to optimality, but the computation time was in the order of many hours 
(8 to 9 hours). The optimal solutions for the problem with nine nodes
are shown in Table \ref{Table:improved_EA1_n9}.

\begin{table}[h!]
  \centering
  \caption{Comparison of CPU time to solve MISDP formulation
  using Matlab's SDP solver ($T_1$) with $EA_1$ ($T_2$) and 
  $EA_2$ ($T_3$) solved using CPLEX solver for networks with 8 nodes.}
    \label{Table:Yalmip_Cplex_n_8_ch2}
    {
  \begin{tabular}{ccccc} \toprule
   	Instance No. & Optimal & \textbf{$T_1$} & \textbf{$T_2$} & \textbf{$T_3$} \\
    & solution  & (seconds) & (seconds)& (seconds)  \\
	\cmidrule(r){1-5}
    1  & 22.8042 & 1187.07  & 428.45 & 610.31  \\
    2  & 24.3207 &  2771.24 & 1323.58  & 1003.56 \\
    3  & 26.4111 & 1173.02 & 630.39  & 655.32  \\
    4  & 28.6912  & 559.15 & 631.08 & 495.89 \\
    5  & 22.5051 & 715.61 & 515.51 & 608.78 \\
    6  & 25.2167 & 947.16 & 1515.15 & 801.10 \\
    7  & 22.8752 & 1139.56 & 1371.69 & 860.07 \\
    8  & 28.4397 & 753.48 & 564.80 & 274.93 \\
    9  & 26.7965 & 1127.46 & 824.64 & 1287.26 \\
   10  & 27.4913 & 862.81 & 383.88 & 213.48 \\
	\cmidrule(r){1-5}
	Avg. & & 1123.35 & 818.40 & 680.62 \\
   \bottomrule
  \end{tabular}
  }
\end{table}

\begin{figure}[h!]
  \centering
  \includegraphics[scale=0.50]{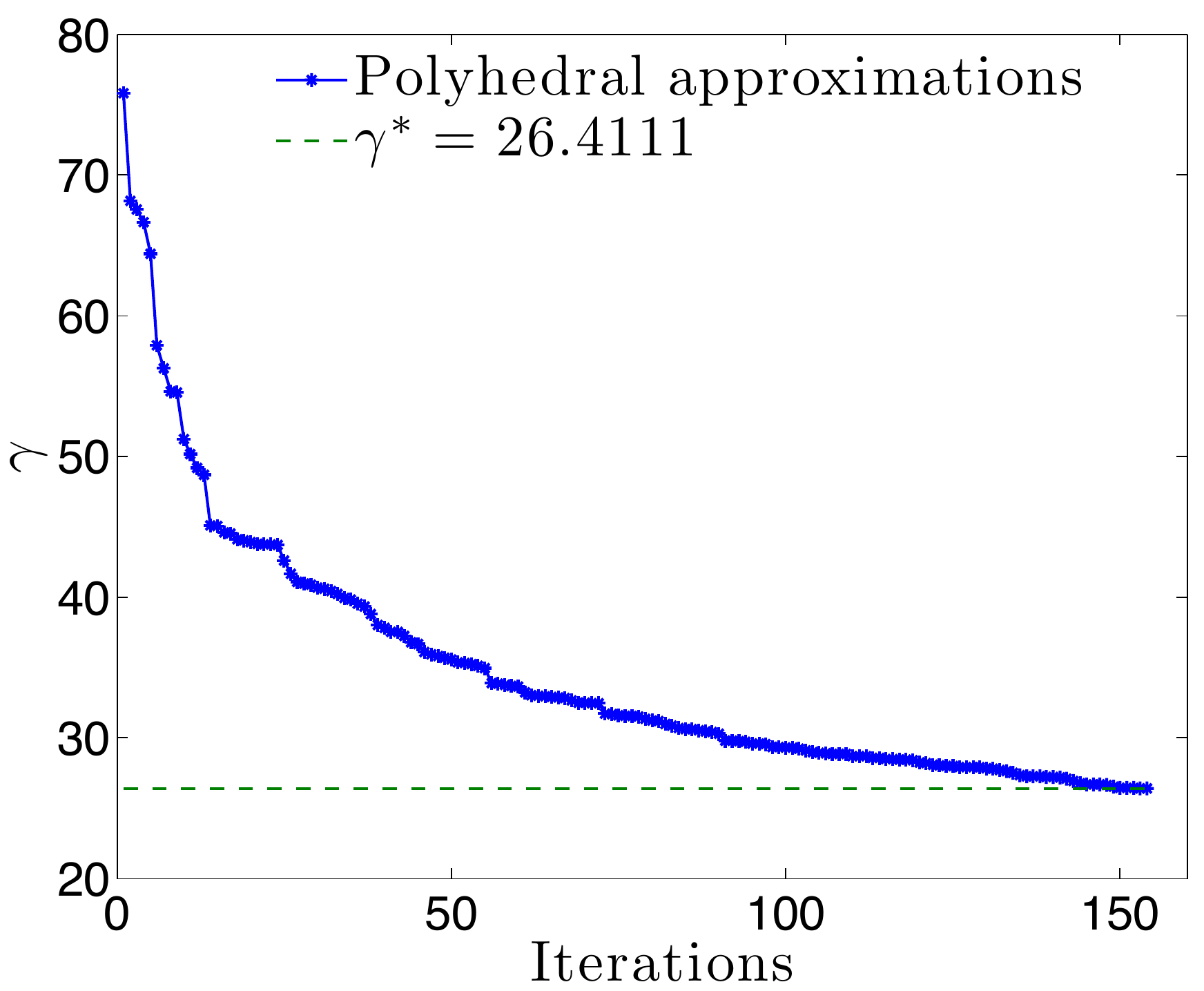}
  \caption{$EA_1$ based on polyhedral approximation of the feasible set: 
	Plot of the upper bound on the 
	algebraic connectivity versus iterations for
  instance 3 given in Table \ref{Table:Yalmip_Cplex_n_8_ch2}. Note that
  the construction of successively tighter polyhedral approximations
  of the feasible semi-definite set reduces the upper bound
  and finally terminates at the optimal solution with maximum
  algebraic connectivity ($\gamma^*$).}
   \label{Fig: Cutting_plane_n8_instance3}
\end{figure}

\begin{figure}[h!]
  \centering
  \includegraphics[scale=0.50]{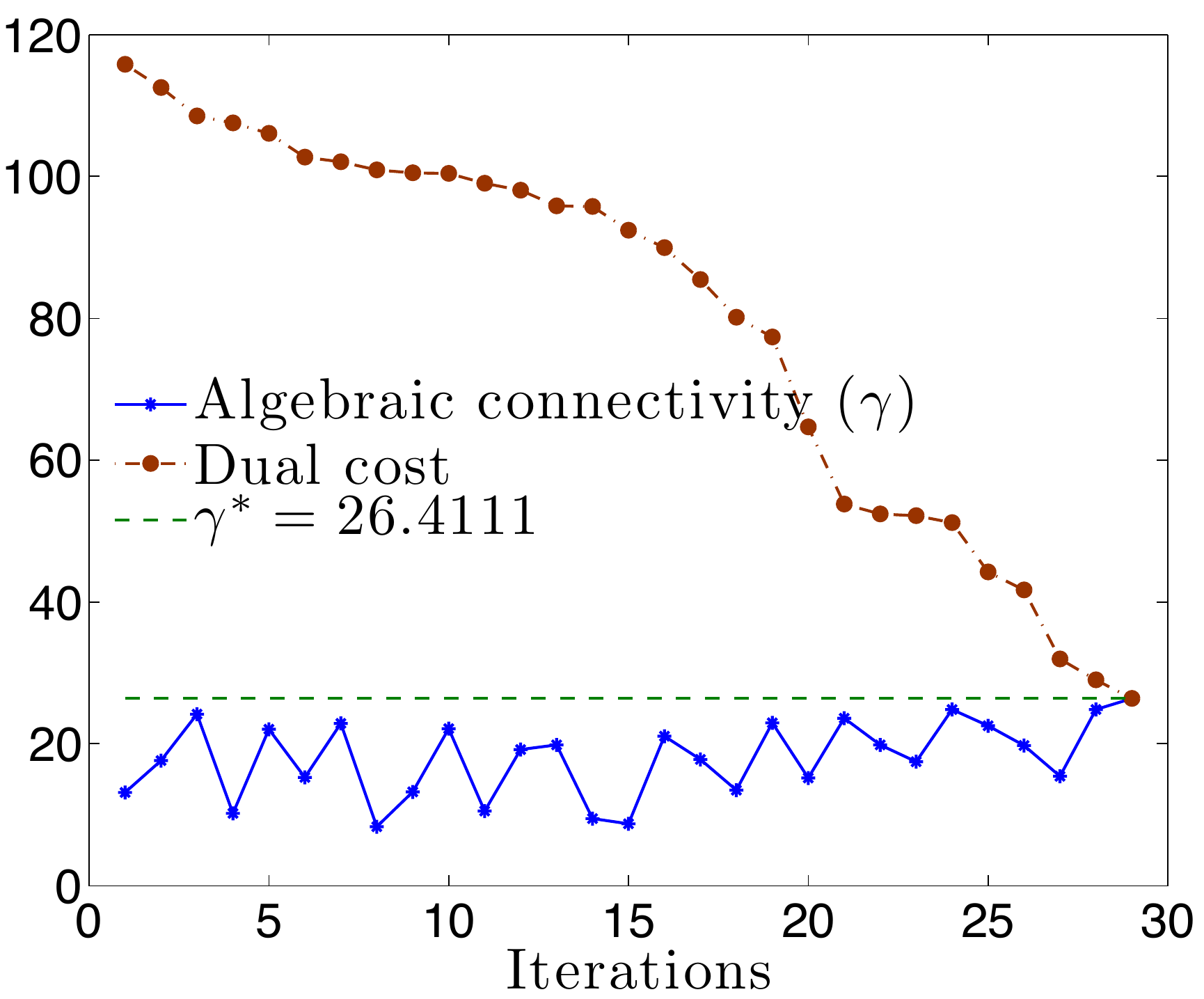}
  \caption{$EA_2$ based on iterative primal-dual method: 
	Plot of algebraic connectivity of primal feasible solutions and dual cost versus iterations
  for instance 3 given in Table \ref{Table:Yalmip_Cplex_n_8_ch2}.
  Note that this algorithm terminates when the dual cost equals
  the maximum algebraic connectivity ($\gamma^*$). Source: \cite{nagarajan2012algorithms}}
   \label{Fig:Primal_dual_gap_n_8_instance3}
\end{figure}

\begin{figure}[htp]
	\centering
	\subfigure[$\gamma^*$ = 22.8042]{
	\includegraphics[scale=0.32]{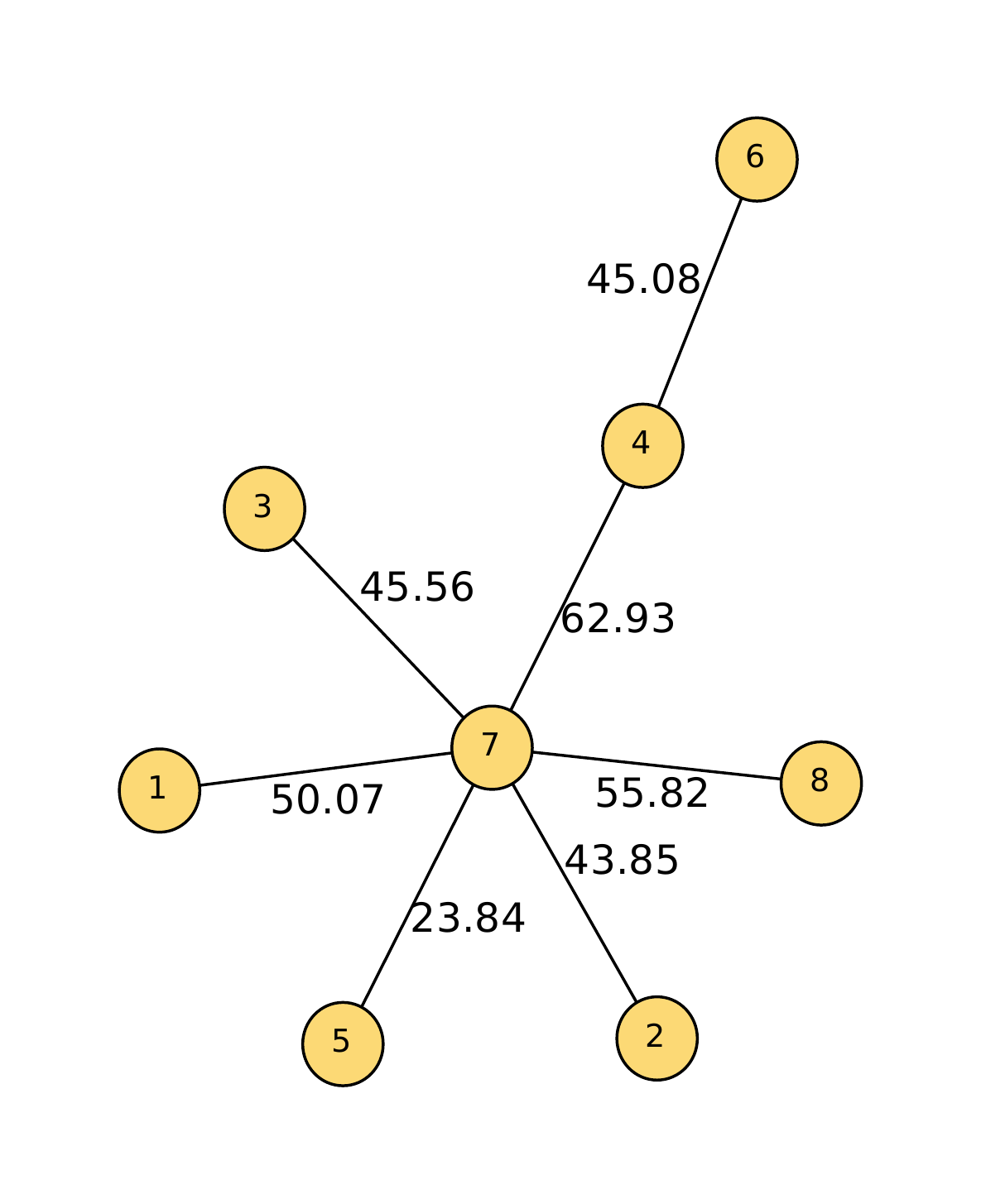}}
	\subfigure[$\gamma^*$ = 24.3207]{
	\includegraphics[scale=0.32]{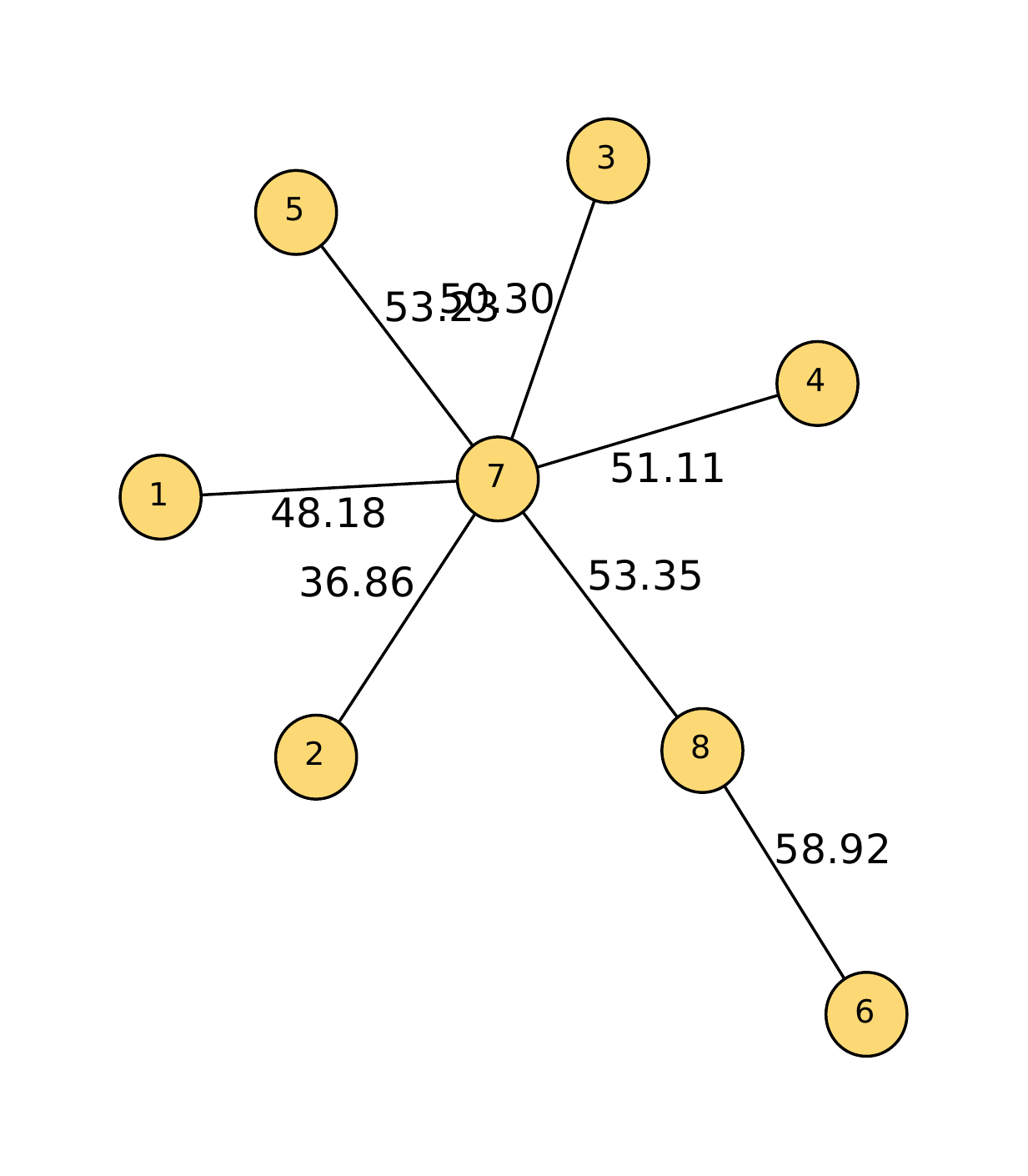}}
	\subfigure[$\gamma^*$ = 26.4111]{
	\includegraphics[scale=0.32]{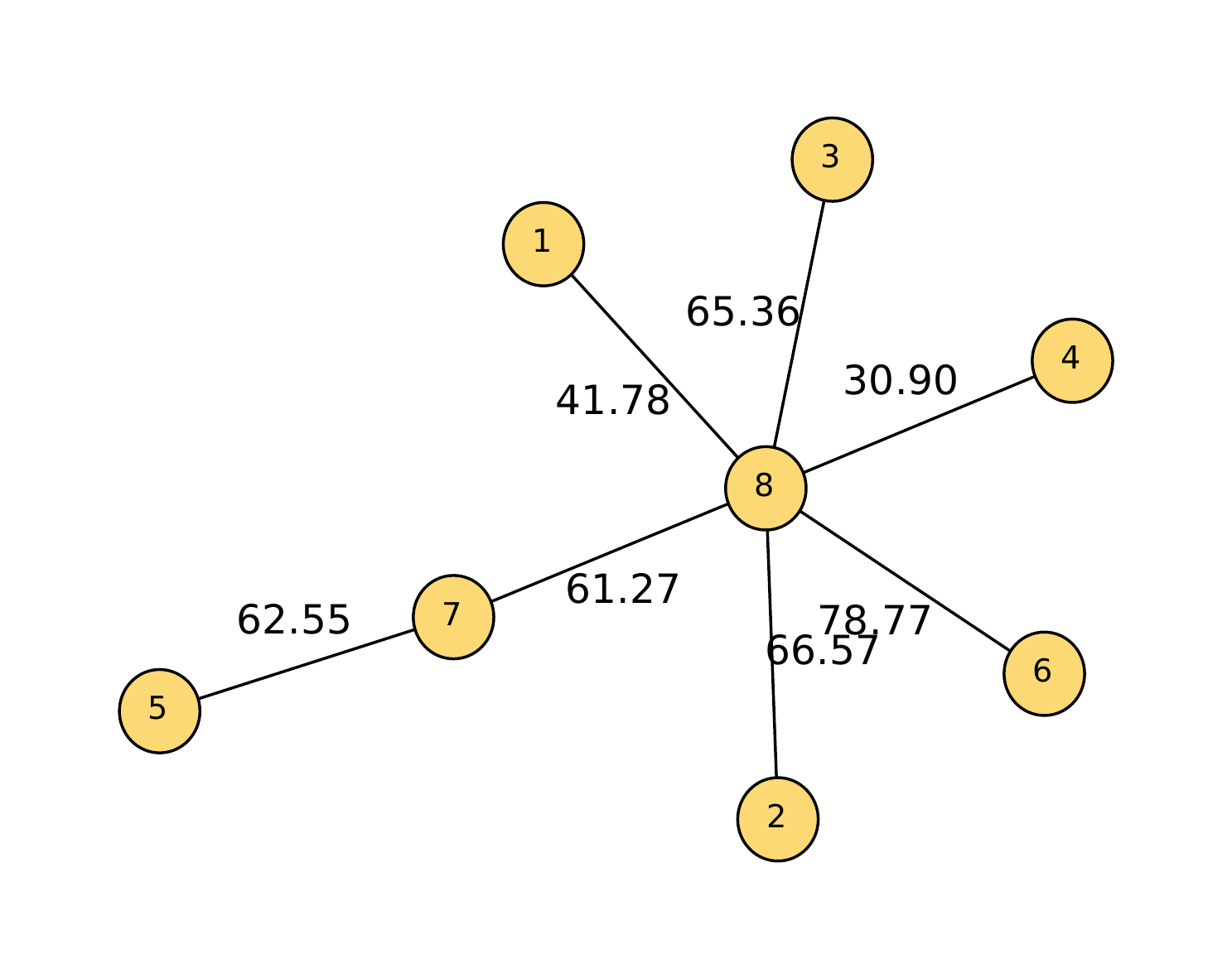}}
	\subfigure[$\gamma^*$ = 28.6912]{
	\includegraphics[scale=0.32]{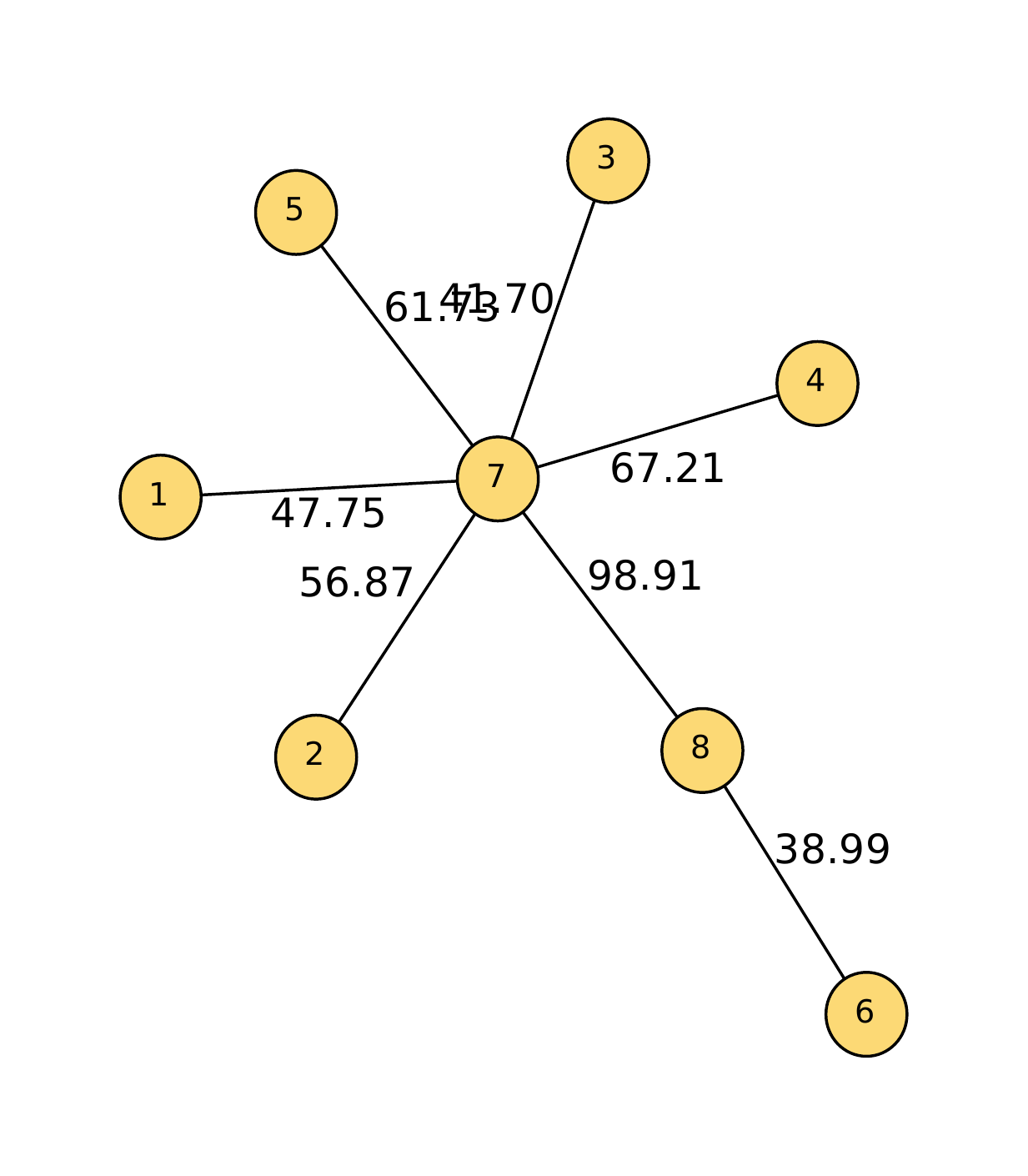}}
	\subfigure[$\gamma^*$ = 22.5051]{
	\includegraphics[scale=0.32]{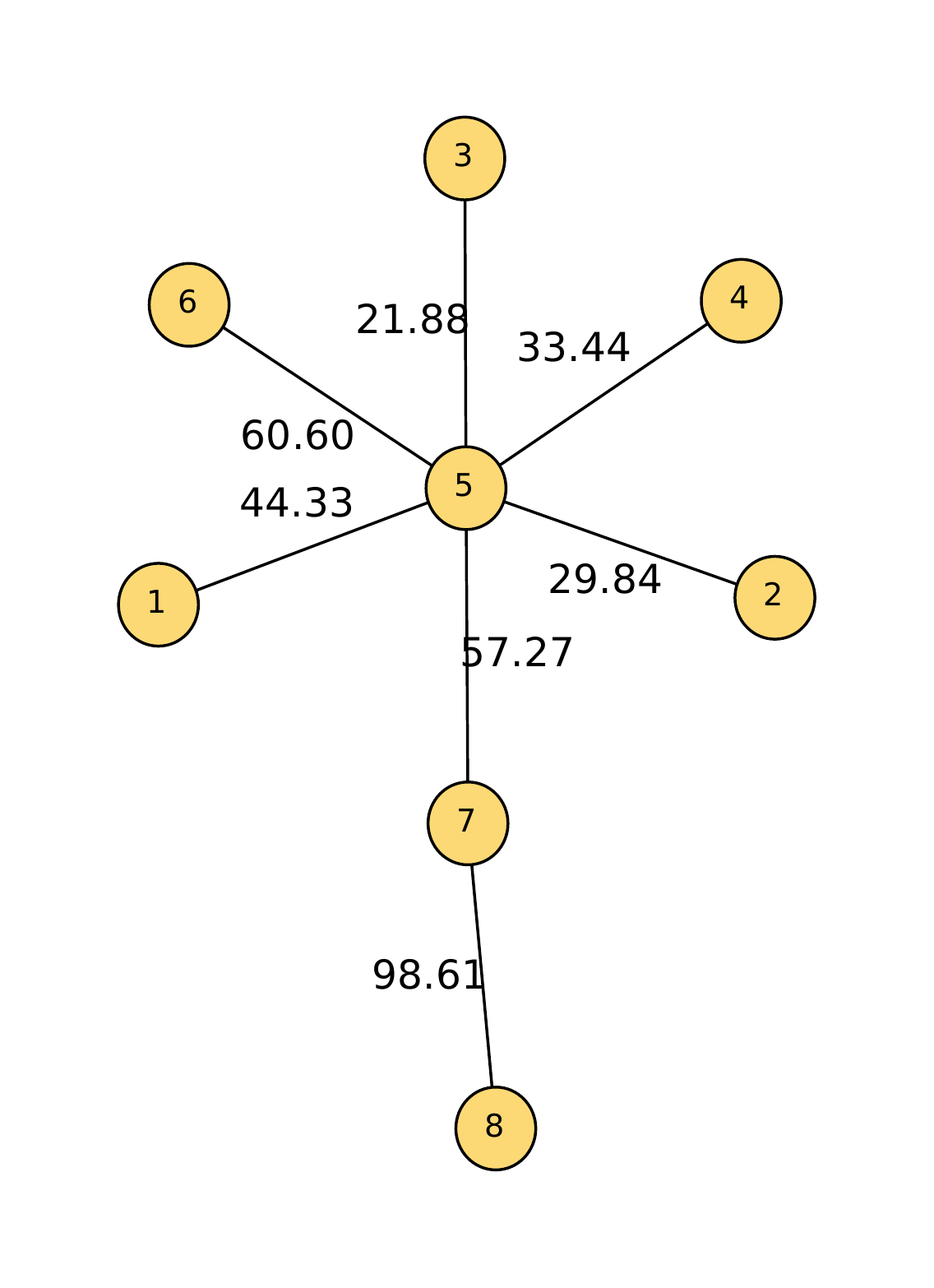}}
	\subfigure[$\gamma^*$ = 25.2167]{
	\includegraphics[scale=0.32]{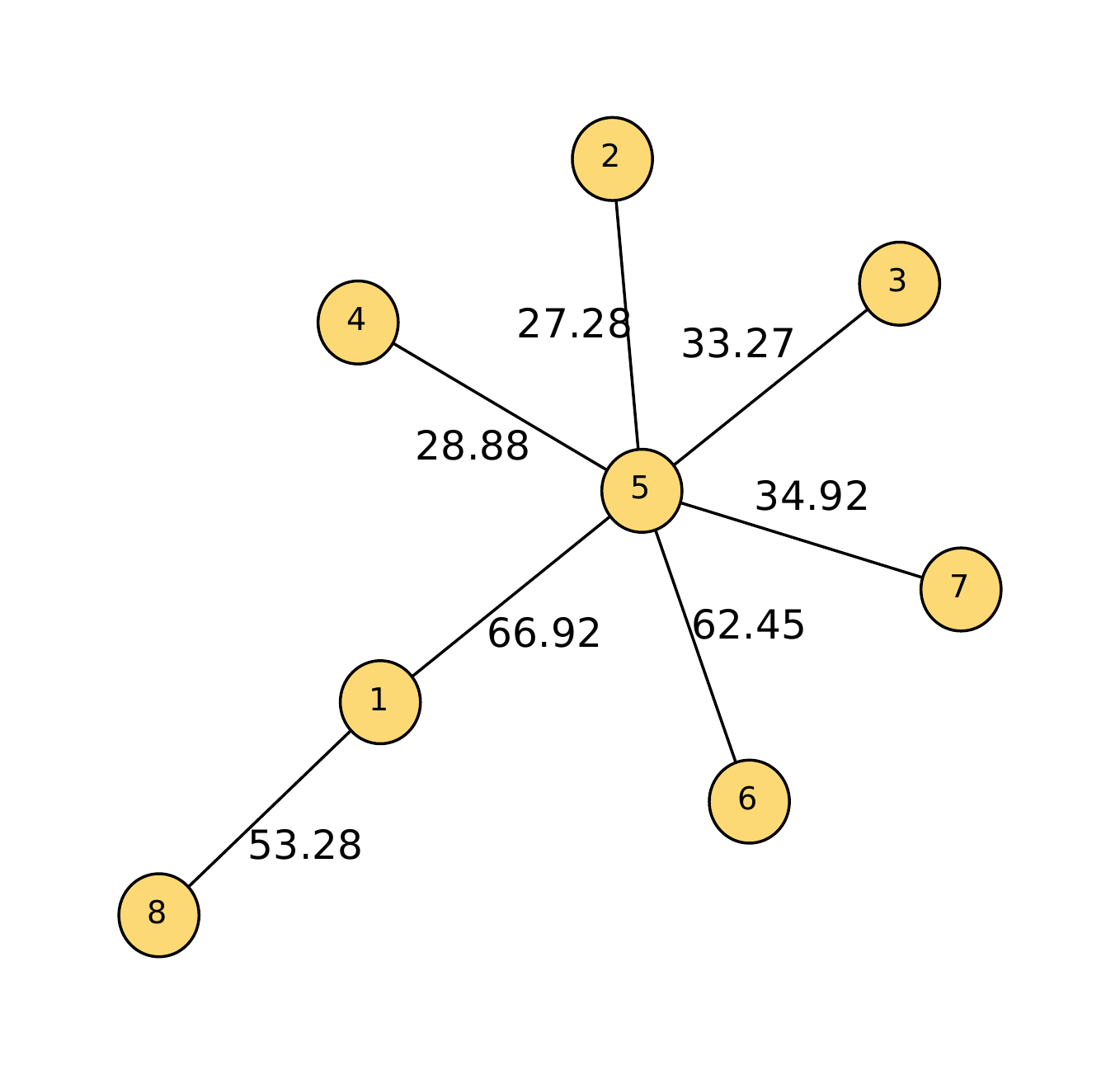}}
	\subfigure[$\gamma^*$ = 22.8752]{
	\includegraphics[scale=0.32]{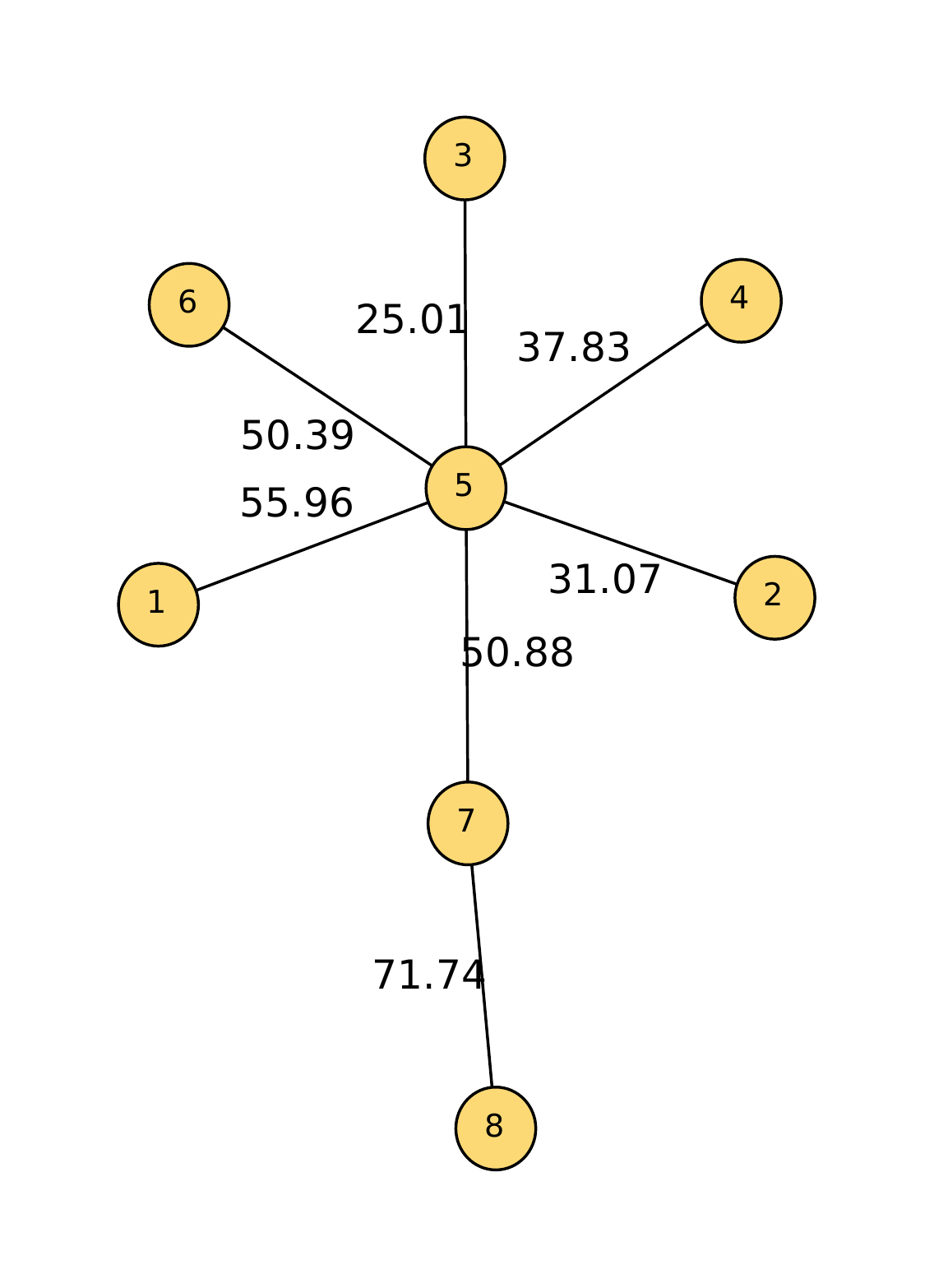}}
	\subfigure[$\gamma^*$ = 28.4397]{
	\includegraphics[scale=0.32]{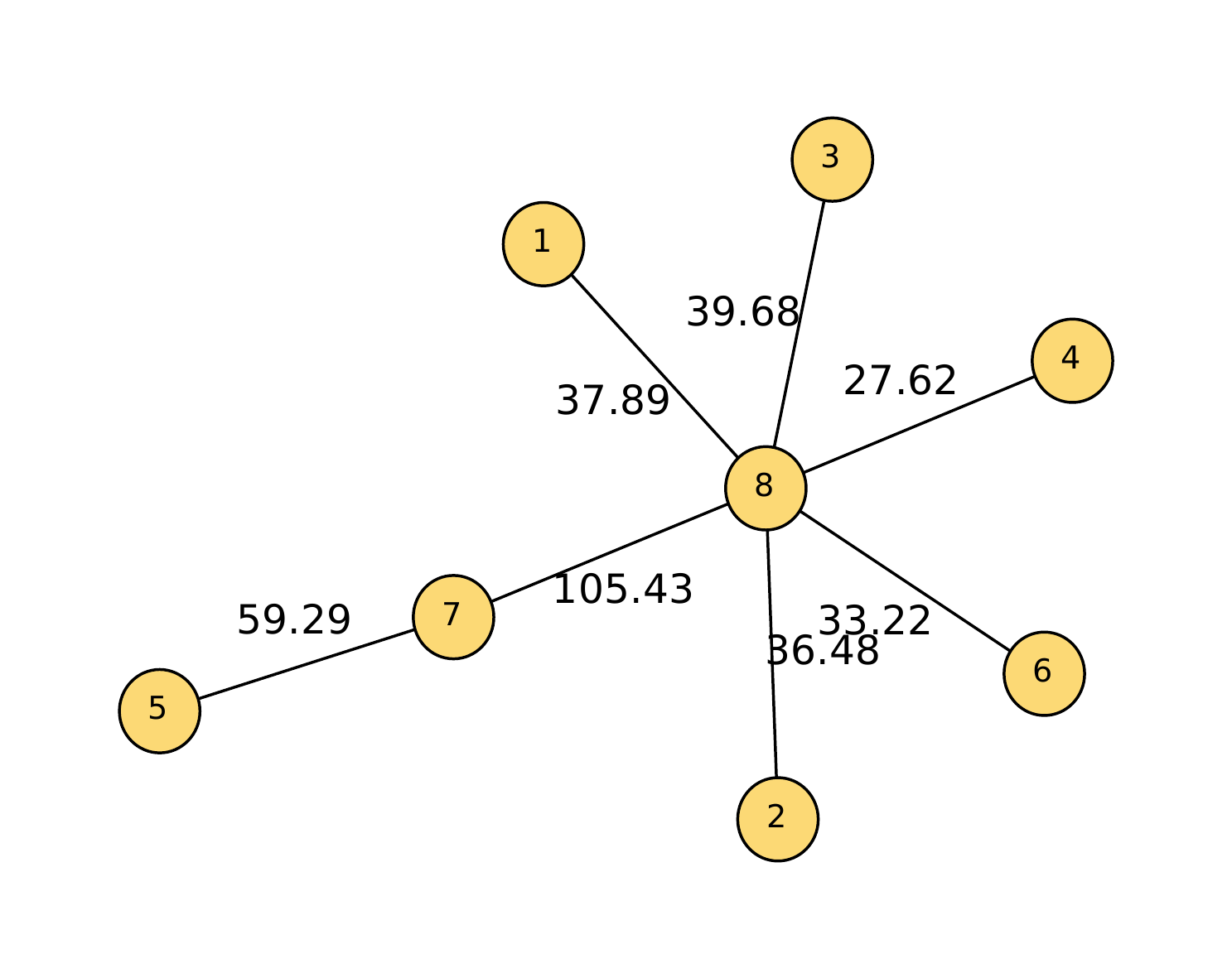}}
	\subfigure[$\gamma^*$ = 26.7965]{
	\includegraphics[scale=0.32]{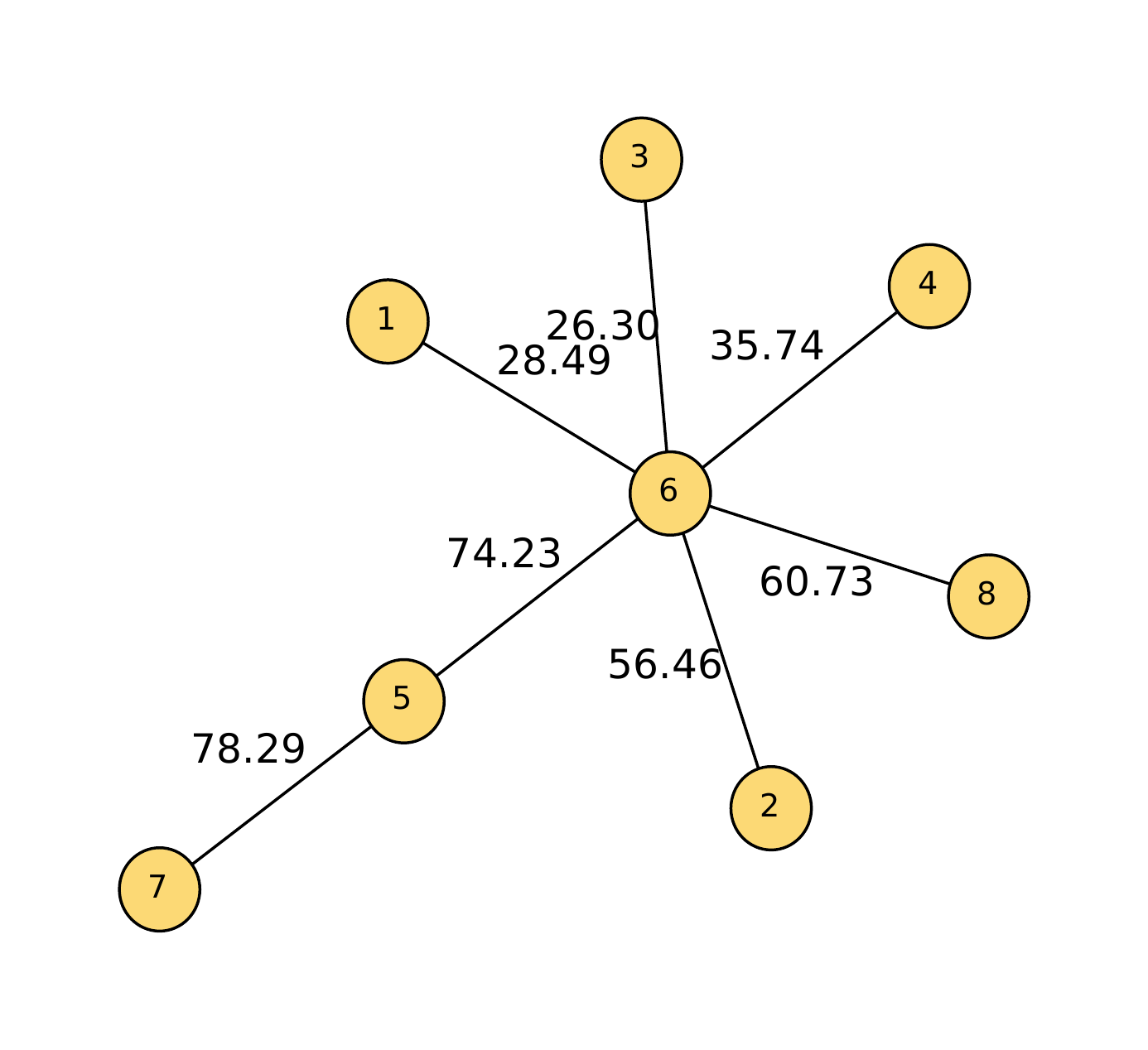}}
	\subfigure[$\gamma^*$ = 27.4913]{
	\includegraphics[scale=0.32]{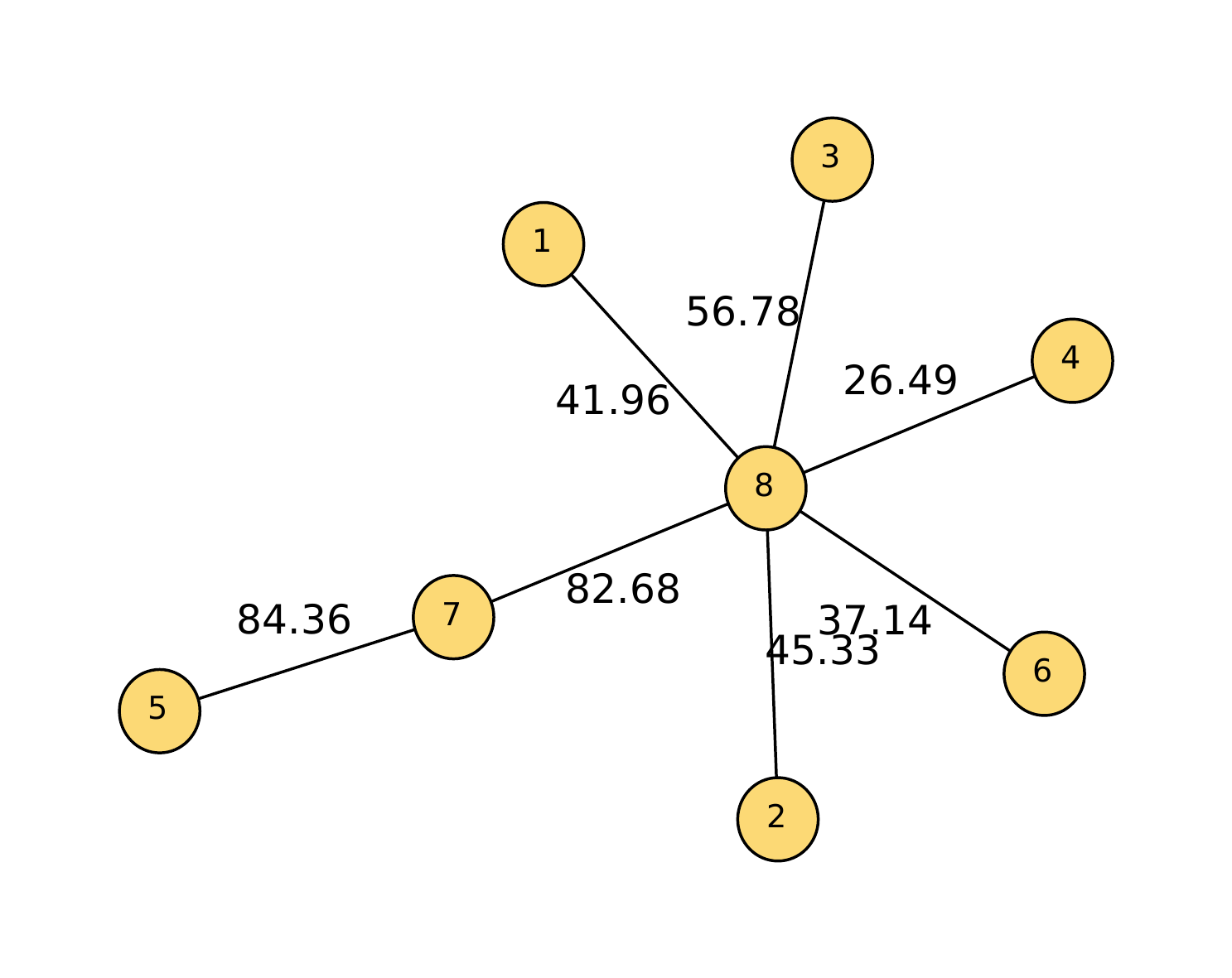}}
	\caption{Optimal networks of eight nodes with maximum algebraic connectivity 
	for the random instances shown in Table \ref{Table:Yalmip_Cplex_n_8_ch2} and the corresponding 
	adjacency matrices in Appendix \ref{ch:appendix}.}
	\label{Fig:Optimal_graphs_n8}
\end{figure}

In $EA_1$, twenty Fiedler vectors of random spanning trees
were used to relax the semi-definite constraint. The sequence of upper
bounds obtained by this algorithm for instance
3 can be seen in Figure \ref{Fig: Cutting_plane_n8_instance3}.
For this instance, the algorithm terminates with
an optimal solution after choosing approximately 150 feasible
spanning trees out of the possible 262144 feasible
solutions which we think is quite reasonable.

The exit criterion used for $EA_2$ can be clearly 
observed in Figure \ref{Fig:Primal_dual_gap_n_8_instance3}. The dual
cost which is also an upper bound on the optimal
solution continuously gets better with the augmentation of
cutting planes and finally exits when the dual
gap goes to zero. In this algorithm, the 
augmented dual problems were solved using the dual-simplex method.

\subsection{$EA_3$: Algorithm to compute maximum algebraic connectivity }
\label{Sec:BSDP}
The earlier sections dealt with two algorithms which 
synthesized optimal networks with eight nodes in 
a reasonable amount of time and was a huge improvement over the 
existing methods to handle MISDPs. However, the computation time 
for solving nine node problems to optimality was large. 
Therefore, we propose a different approach for finding an optimal solution in this section
by casting the algebraic connectivity problem as the following decision problem:
Is there a connected network $x$ with at most $q$ edges from $E$ such that
the algebraic connectivity of the network is at least equal to a
pre-specified value ($\hat{\gamma}$)?

One of the advantages of posing this question is that the resulting problem
turns out to be a Binary Semi-Definite Problem (BSDP) and correspondingly,
the tools associated with construction of cutting planes
are more abundant when compared to MILPs. Also, with further
relaxation of the semi-definite constraint, 
it can be solved using CPLEX, a high performance solver for ILPs.

The above decision problem can be mathematically posed as a BSDP by marking any vertex (say $r\in V$) in
this graph as a root vertex and then choosing to find a feasible tree
that minimizes the degree of this root vertex 
\footnote{There are several ways to
formulate the decision problem as a BSDP. 
We chose to minimize the degree of a node as it seems to produce reasonably
good feasible solutions in every iteration.}.
In this formulation, the only decision variables would be the binary variables
denoted by $x_e$. Therefore, the resulting BSDP is the following:

{\begin{equation}
	\label{eq:F2_degree}
		\begin{array}{ll}
			 \min & \sum_{e \in \delta(r)} x_e, \\
			 \text{s.t.} & \sum_{e \in E} x_e L_e \succeq \hat{\gamma} (I_n - e_0 \otimes e_0),\\
			& \sum_{e \in E} x_e \leq q, \\
		  &\sum_{e \in \delta(S)} x_e \geq 1, \ \ \forall \ S \subset V, \\
			& x_e \in \{0, 1\}^{|E|}.
		\end{array}
\end{equation}}
where, $\delta(r)$ denotes a cutset defined as
$\delta(r) = \{e=(r,j): \ j \in V \setminus {r} \}$.

If we can solve this BSDP efficiently, then we can use a bisection algorithm
to find an optimal solution that will maximize the algebraic connectivity.
In order to solve the BSDP using CPLEX, we do the following: we first consider the
relaxation of the semi-definite constraint 
by taking a finite subset of the infinite number of
linear constraints from the semi-infinite program, but however add
cutset constraints to ensure that the desired network is always connected.
These cutset constraints defined by the inequalities $$\sum_{e\in \delta(S)} x_e \geq 1, \quad 
 \forall S\subset V,$$ require that there is at least one edge chosen from the cutset of 
any subset $S$ (set of edges from $S$ to its complement $\bar{S}$ in $V$). 
If the solution to the relaxed BSDP does not satisfy the semi-definite constraint, 
we add an eigenvalue cut that ensures that this solution will not be chosen again and then solve 
the augmented but a relaxed BSDP again. This cutting plane procedure is continued until a feasible solution is found.
The idea of this procedure is to construct successively tighter
polyhedral approximations of the feasible set corresponding to the desired
level of algebraic connectivity which is very similar to the procedure discussed in $EA_1$. 
Clearly, the algebraic connectivity of the
feasible solution we have is a lower bound for the original MISDP. 
Hence, we increment the value of $\hat{\gamma}$ to the best known lower bound 
plus an epsilon value and continue to solve the BSDP. This procedure of 
bisection is repeated until the BSDP gets infeasible, which implies that we 
have an optimal solution to the MISDP. The pseudo code of this procedure is 
outlined in Algorithm \ref{Algo:minimize_degree}.
Finding an eigenvalue cut that removes the infeasible solution at each iteration is clearly shown in this
algorithm.

\begin{algorithm}[h!]
\caption{\textbf{: $EA_3$ (BSDP approach)}}
Let $\mathfrak{F}$ denote a set of cuts which must be
satisfied by any feasible solution
\label{Algo:minimize_degree}
\begin{algorithmic}[1]
		\STATE Input: Graph $G=(V,E,w_e)$, $e\in E$, a root vertex, $r$, and a finite number of Fiedler vectors, $v_i, i=1 \ldots M$
        \STATE Choose a maximum spanning tree as an initial feasible solution, $x^*$
		\STATE $\hat{\gamma} \gets \lambda_2(L(x^*))$
		\LOOP
		\STATE $\mathfrak{F}$$ \gets \emptyset$
		\STATE Solve:
		  {\begin{equation}
				  \begin{array}{ll}
						\min & \sum_{e \in \delta(r)} x_e, \\
						\text{s.t.} & \sum_{e \in E} x_e ({v_i} \cdot L_e {v_i}) \geq \hat{\gamma} \quad \forall i=1,..,M,  \\
					  & \sum_{e \in E} x_e \leq q, \\
					 &\sum_{e \in \delta(S)} x_e \geq 1, \ \ \forall \ S \subset V, \\
					  & x_e \in \{0, 1\}^{|E|}, \\
        			&x_e \ \textrm{satisfies the constraints in }\mathfrak{F}.
				  \end{array}
		  \end{equation}}

		\IF{ the above ILP is infeasible}
				\STATE \textbf{break loop} \COMMENT{$x^*$ is the optimal solution with maximum algebraic connectivity}
		\ELSE   \STATE Let $x^*$ be an optimal solution to the above ILP. Let $\gamma^*$ and $v^*$ be the algebraic connectivity 
		and the Fiedler vector corresponding to $x^*$ respectively.
		\IF{$\sum_{e \in E} x_e^* L_e \nsucceq \gamma^* (I_n - e_0 \otimes e_0)$}
						\STATE Augment $\mathfrak{F}$ with a constraint $\sum_{e \in E} x_e ({v^*} \cdot L_e {v^*}) \geq {\gamma^*} $.
						\STATE Go to step 6.
				\ENDIF
		\ENDIF
		\STATE $\hat{\gamma} \gets \hat{\gamma} + \epsilon$ \COMMENT{let $\epsilon$ be a small number}
		\ENDLOOP
\end{algorithmic}
\end{algorithm}

\subsection{Performance of $EA_3$}
\label{Sec:results2}
All the computations in this section were performed with the same computer specifics as mentioned in 
section \ref{Sec:results1}.
In $EA_3$ (Algorithm \ref{Algo:minimize_degree}), for a given random complete graph,
we chose maximum spanning tree as an initial feasible solution 
since it was computationally inexpensive to evaluate using the standard greedy algorithms.
For the bisection step, we assumed $\epsilon = 0.01$. 

In Table \ref{Table:EA3_n8}, we compare the computational performance
of solving a sequence of BSDPs (in formulation \eqref{eq:F2_degree}) with bisection technique directly using the MATLAB's
SDP solver with the proposed $EA_3$ based on cutting plane method implemented 
in CPLEX. Clearly, the computation time for the $EA_3$ is much faster (46.45 times)
than solving BSDPs directly in MATLAB. For the same set of random instances of eight nodes,
it is also worthy to note that, $EA_3$ performs computationally better
than solving MILPs using basic $EA_1$ and $EA_2$ in CPLEX as indicated
in Table \ref{Table:Yalmip_Cplex_n_8_ch2}.

For the problems with nine nodes, $EA_3$ significantly reduced 
the average computational time from eight hours to around five to six hours.

\begin{table}[h!]
  \centering
  \caption{Comparison of CPU time to directly solve the BSDPs 
  in bisection procedure using Matlab's SDP solver ($T_1$) with the proposed
  $EA_3$ using CPLEX solver ($T_2$) for networks with eight nodes.}
    \label{Table:EA3_n8}
    {
  \begin{tabular}{cccc} \toprule
    Instance No. & $\lambda_2^*$  & \textbf{$T_1$} & \textbf{$T_2$} \\
	 &  & (seconds)& (seconds)  \\
	\cmidrule(r){1-4}
    1  & 22.8042 &  15729.51  & 254.45  \\
    2  & 24.3207 &  3652.41  & 314.62  \\
    3  & 26.4111 & 3075.72  &  378.42   \\
    4  & 28.6912  & 23794.01 &  420.96  \\
    5  & 22.5051 &  10032.71  &  263.06 \\
    6  & 25.2167 & 20340.92 &  382.28  \\
    7  & 22.8752 & 16717.06 & 484.46   \\
    8  & 28.4397 & 16837.90 & 512.47   \\
    9  & 26.7965 & 44008.42 & 306.56  \\
   10  & 27.4913 & 7366.67 & 204.57 \\ 
	\cmidrule(r){1-4}
	Avg. & &15955.68 & 351.72 \\
   \bottomrule
  \end{tabular}
  }
\end{table}

\subsection{Performance of $EA_1$ with an improved relaxation of the semi-definite constraint}
\label{Sec:results3}
In section \ref{Sec:results1}, the performance of $EA_1$ was based on the initial relaxation 
of the semi-definite constraint using the Fiedler vectors of random feasible solutions. 
With such a relaxation, the initial upper bound
can be weak and hence can incur larger time to compute the optimal algebraic connectivity.
However, in section \ref{subsec:UB_1}, we discussed a method to provide tight 
upper bounds by relaxing the semi-definite constraint using the Fiedler vectors 
of feasible solutions with higher values of algebraic connectivity. Therefore, in this section, 
we discuss the performance of $EA_1$
with an improved initial relaxation of the semi-definite constraint 
as discussed in section \ref{subsec:UB_1}.

By choosing thousand Fiedler vectors of good feasible solutions 
to initially relax the semi-definite constraint, the
computation time to obtain optimal solutions are shown 
in Tables \ref{Table:improved_EA1_n8} and \ref{Table:improved_EA1_n9}.
$T_1$ corresponds to the time required to enumerate fifteen thousand
spanning trees from the maximum spanning tree and 
$T_2$ corresponds to $EA_1$'s time to compute optimal solutions with 
an improved initial relaxation.

Based on the results in Table \ref{Table:improved_EA1_n8} for the 
eight nodes problem, the average total computation time of $EA_1$ 
with an improved relaxation is \textit{eight} times faster than
the computation time of $EA_1$ without an improved relaxation. 
Also, the average total computation
time in Table \ref{Table:improved_EA1_n8} is \textit{two} times 
faster than the BSDP approach in $EA_3$.

Based on the results in Table \ref{Table:improved_EA1_n9} for the 
nine nodes problem, the average total computation time 
is around 2.9 hours with the best case instance being 26 minutes.
This is a great improvement compared to the eight hours of 
computation time for $EA_1$ without an
improved relaxation and compared to five hours of computation
time for $EA_3$. The optimal networks with maximum algebraic connectivity 
for networks with nine nodes are shown in Figure \ref{Fig:Optimal_graphs_n9}.

\begin{table}[h!]
  \centering
  \caption{Performance of $EA_1$ with an improved relaxation of the 
positive semi-definite constraint for networks with eight nodes.}
    \label{Table:improved_EA1_n8}
    {
  \begin{tabular}{ccccc} \toprule
    Instance No. & $\lambda_2^*$  & \textbf{$T_1$} & \textbf{$T_2$} & \textbf{$T_1+T_2$}\\
	 &  & (seconds)& (seconds) & (seconds)  \\
	\cmidrule(r){1-5}
    1  & 22.8042 & 70  & 12 & 82  \\
    2  & 24.3207 & 70  & 142& 212  \\
    3  & 26.4111 & 70  & 9  & 79 \\
    4  & 28.6912 & 70  & 6  & 76 \\
    5  & 22.5051 & 70  & 7  & 77 \\
    6  & 25.2167 & 70  & 86 & 156 \\
    7  & 22.8752 & 70  & 20 & 90  \\
    8  & 28.4397 & 70  & 10 & 80  \\
    9  & 26.7965 & 70  & 39 & 109 \\
   10  & 27.4913 & 70  & 5  & 75  \\ 
	\cmidrule(r){1-5}
	Avg. & & & & 103.6 \\
   \bottomrule
  \end{tabular}
  }
\end{table}

\begin{table}[h!]
  \centering
  \caption{Performance of $EA_1$ with an improved relaxation of the 
positive semi-definite constraint for networks with nine nodes.}
    \label{Table:improved_EA1_n9}
    {
  \begin{tabular}{ccccc} \toprule
    Instance No. & $\lambda_2^*$  & \textbf{$T_1$} & \textbf{$T_2$} & \textbf{$T_1+T_2$}\\
	 &  & (seconds)& (seconds) & (seconds)  \\
	\cmidrule(r){1-5}
    1  & 28.2168 & 170  & 4295 & 4465  \\
    2  & 26.3675 & 170  & 8093 & 8263  \\
    3  & 29.8184 & 170  & 5377   & 5547 \\
    4  & 25.8427 & 170  & 32788   & 32958 \\
    5  & 24.2756 & 170  & 8880   & 9050 \\
    6  & 30.0202 & 170  & 3981  & 4151 \\
    7  & 25.6410 & 170  & 20458  & 20628  \\
    8  & 26.9705 & 170  & 13796  & 13966  \\
    9  & 33.5068 & 170  & 2908  & 3078 \\
   10  & 31.7445 & 170  & 1417   & 1587  \\ 
	\cmidrule(r){1-5}
	Avg. & & & & 10369.3 \\
   \bottomrule
  \end{tabular}
  }
\end{table}
\begin{figure}[htp]
	\centering
	\subfigure[$\gamma^*$ = 28.2168]{
	\includegraphics[scale=0.30]{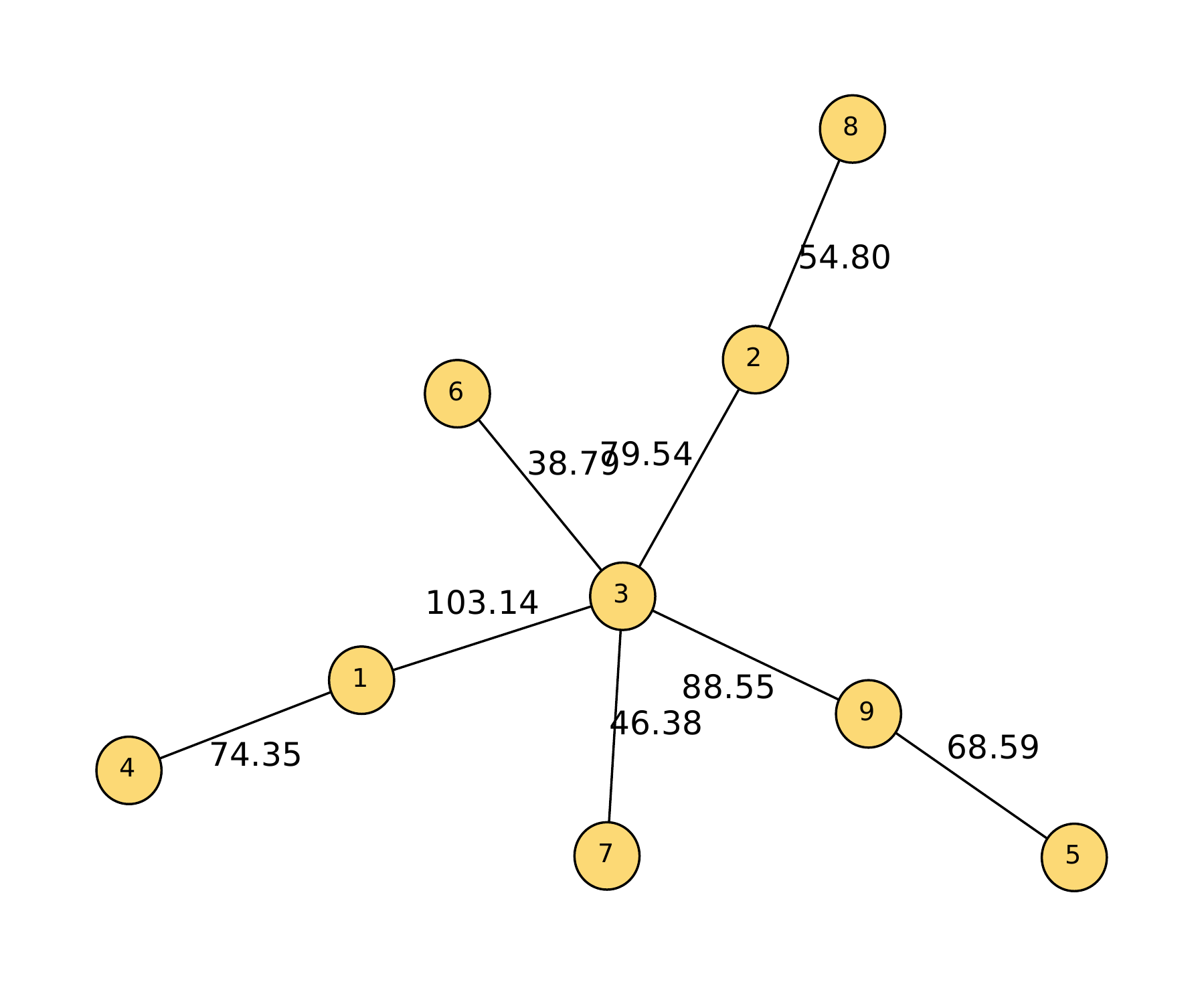}}
	\subfigure[$\gamma^*$ = 26.3675]{
	\includegraphics[scale=0.30]{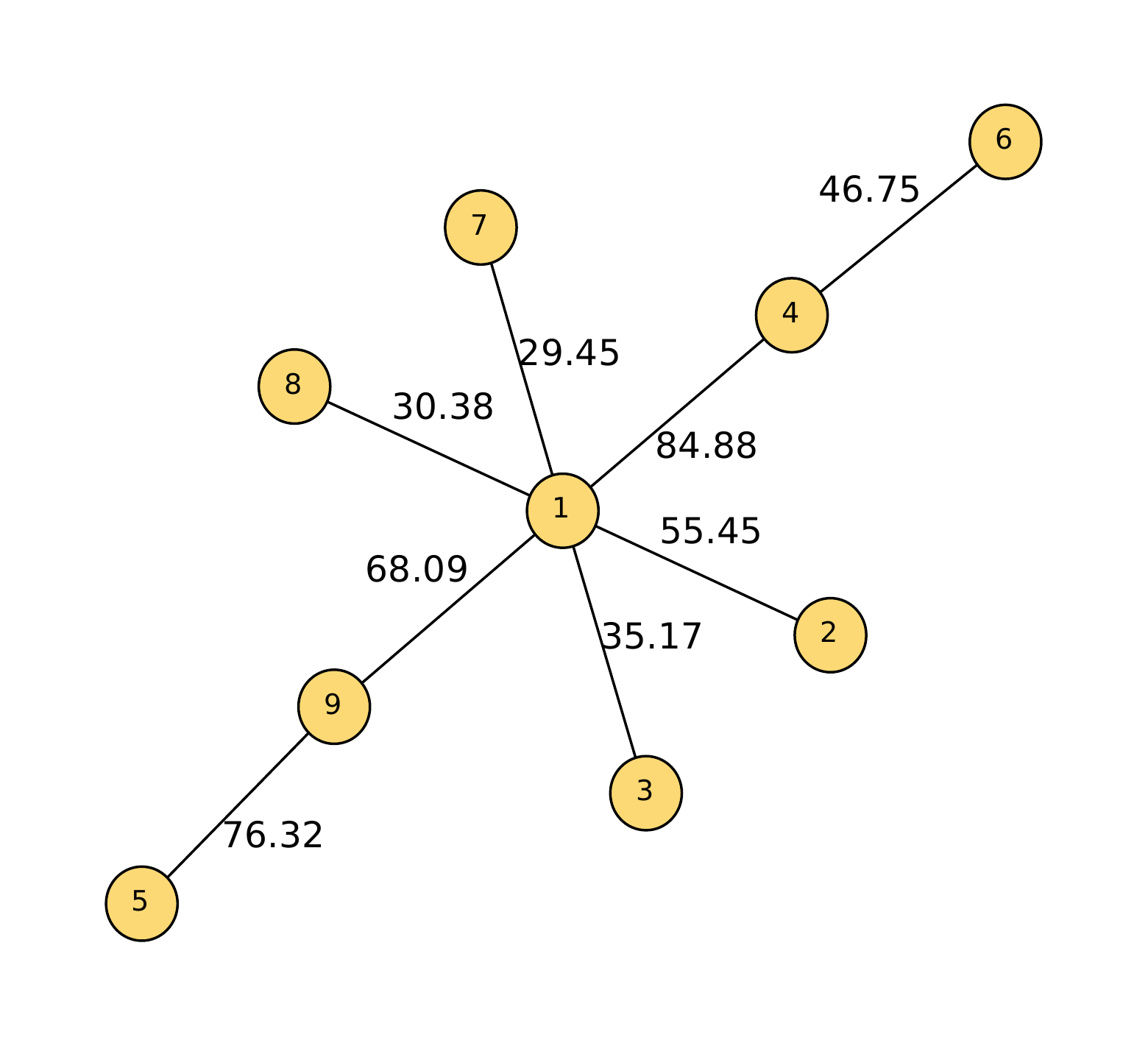}}
	\subfigure[$\gamma^*$ = 29.8184]{
	\includegraphics[scale=0.30]{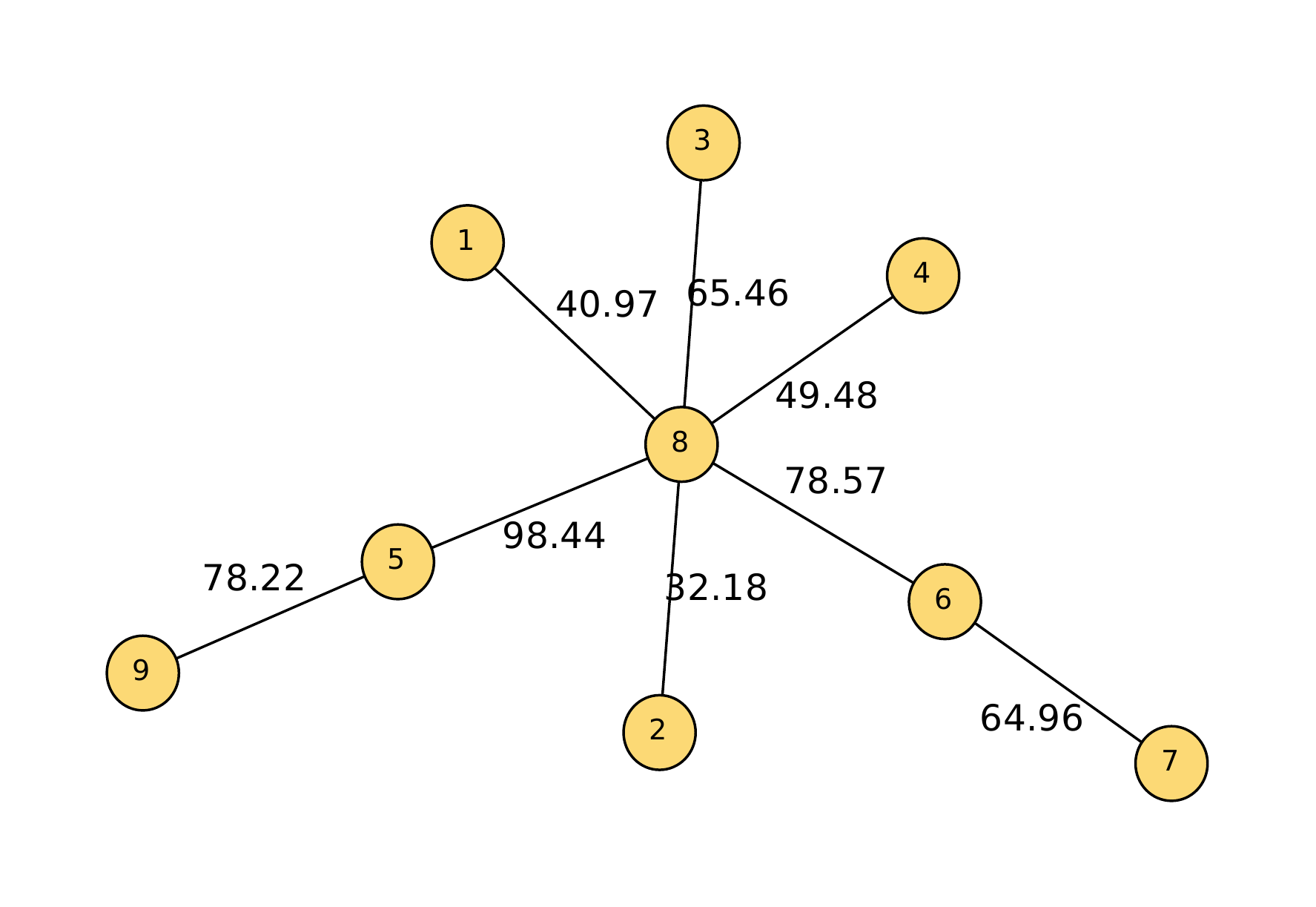}}
	\subfigure[$\gamma^*$ = 25.8427]{
	\includegraphics[scale=0.30]{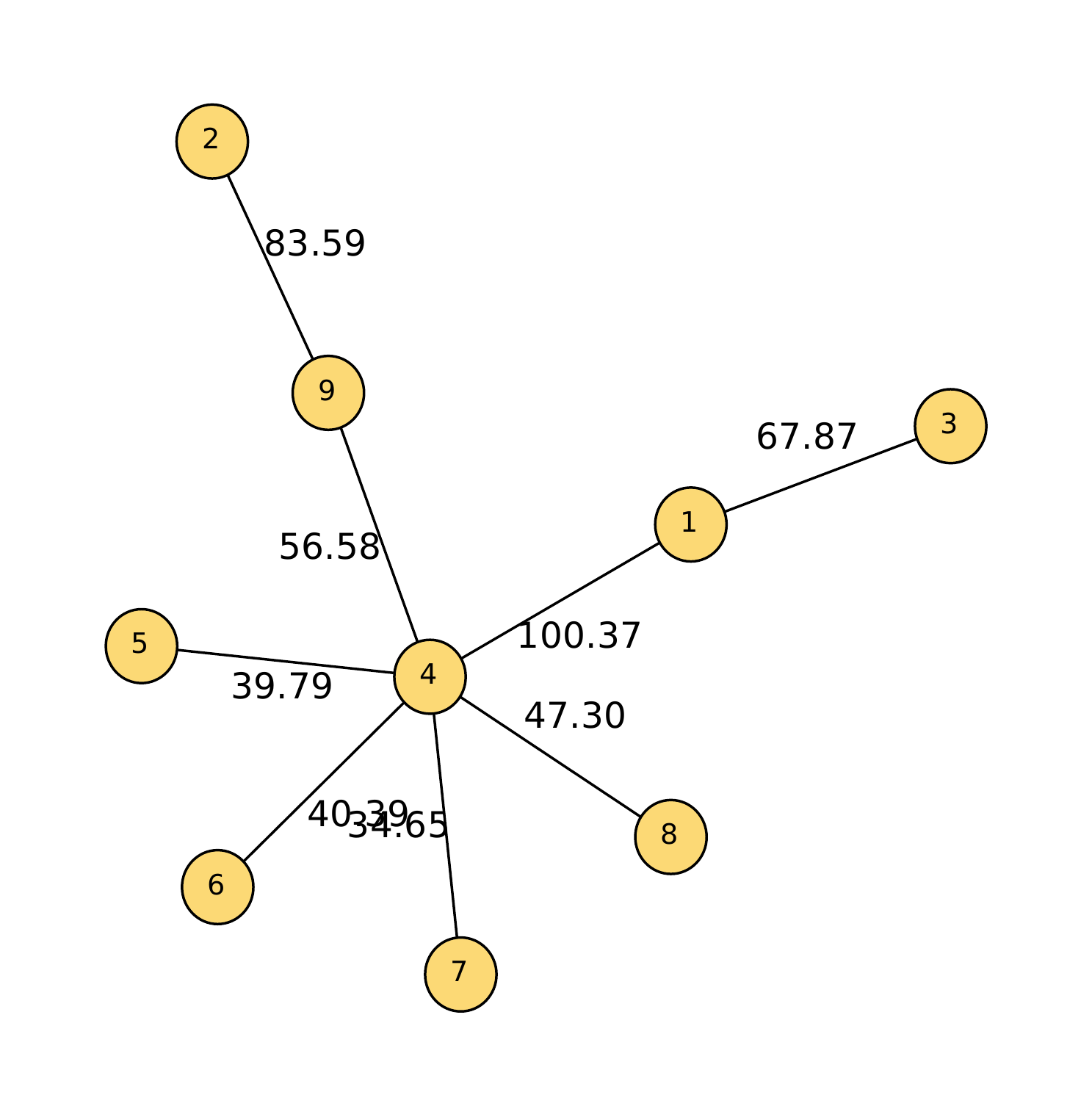}}
	\subfigure[$\gamma^*$ = 24.2756]{
	\includegraphics[scale=0.30]{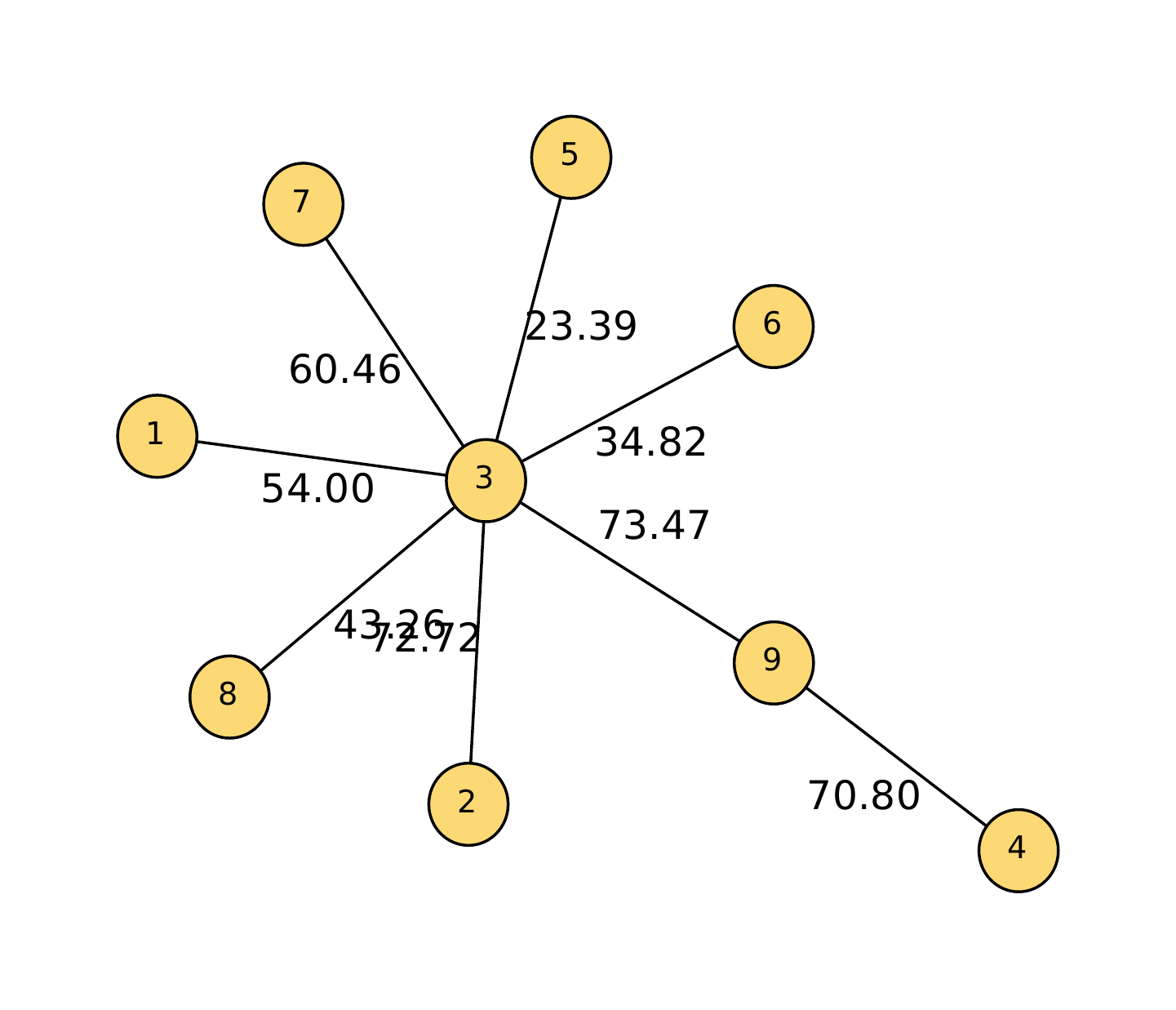}}
	\subfigure[$\gamma^*$ = 30.0202]{
	\includegraphics[scale=0.30]{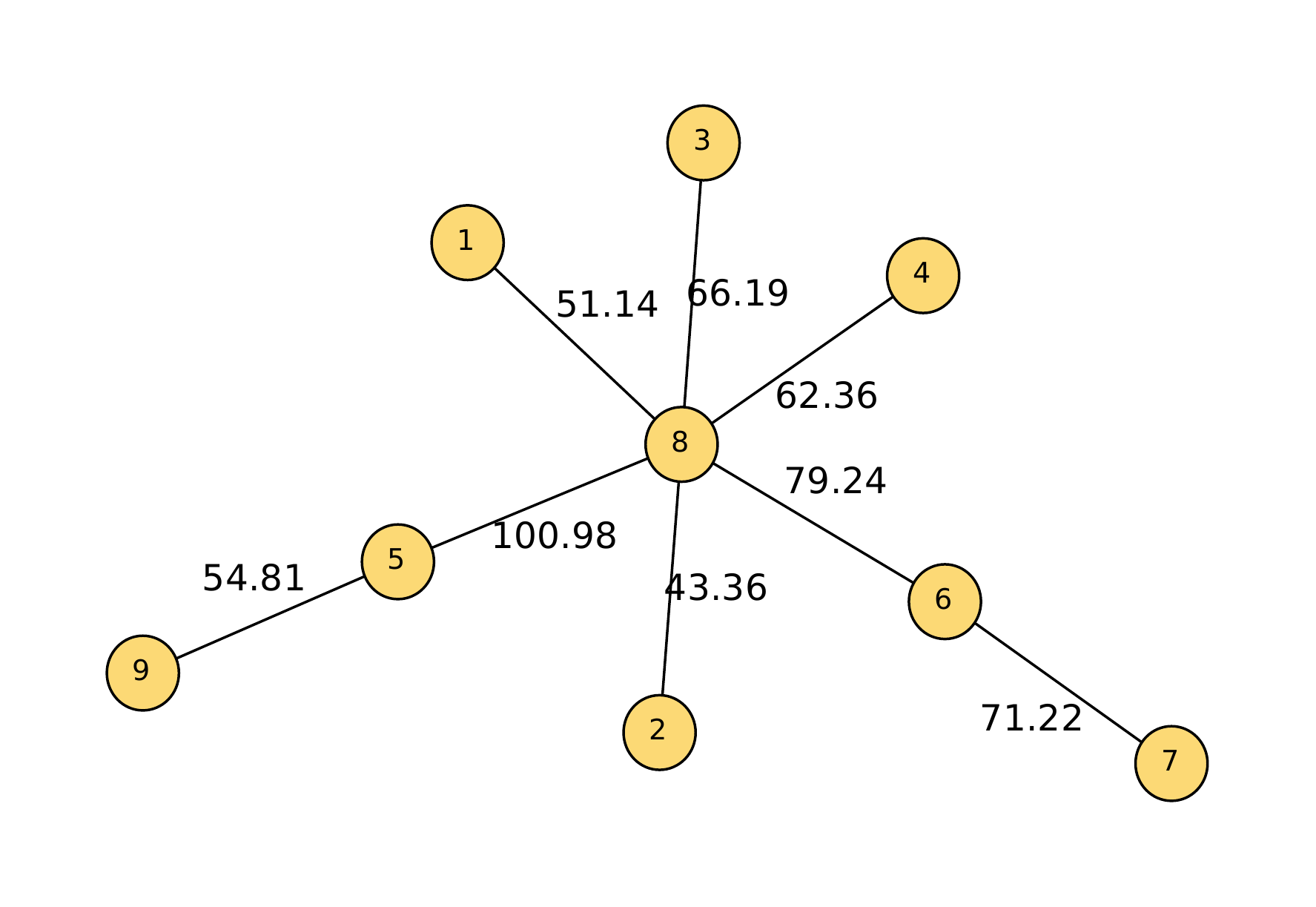}}
	\subfigure[$\gamma^*$ = 25.6410]{
	\includegraphics[scale=0.30]{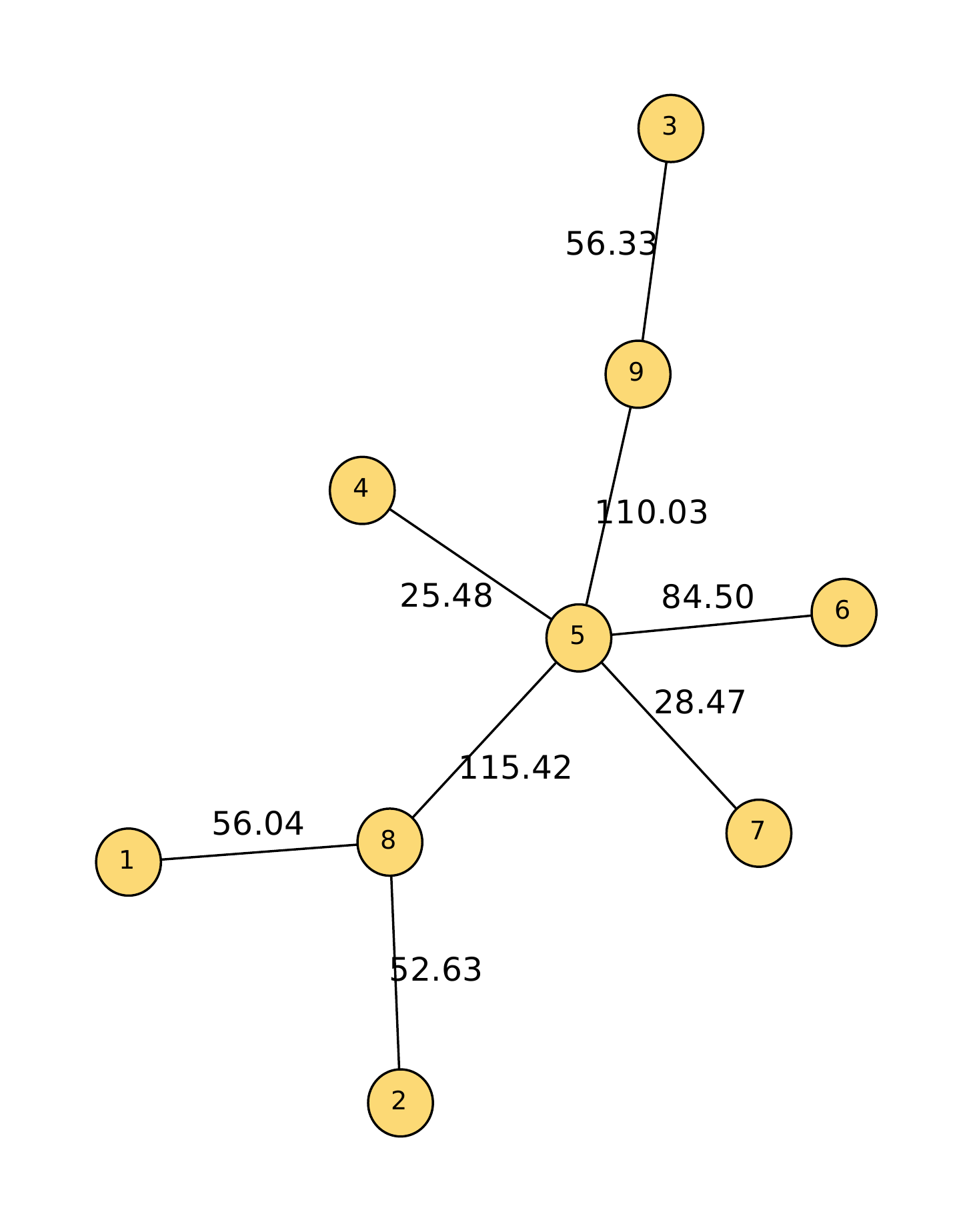}}
	\subfigure[$\gamma^*$ = 26.9705]{
	\includegraphics[scale=0.30]{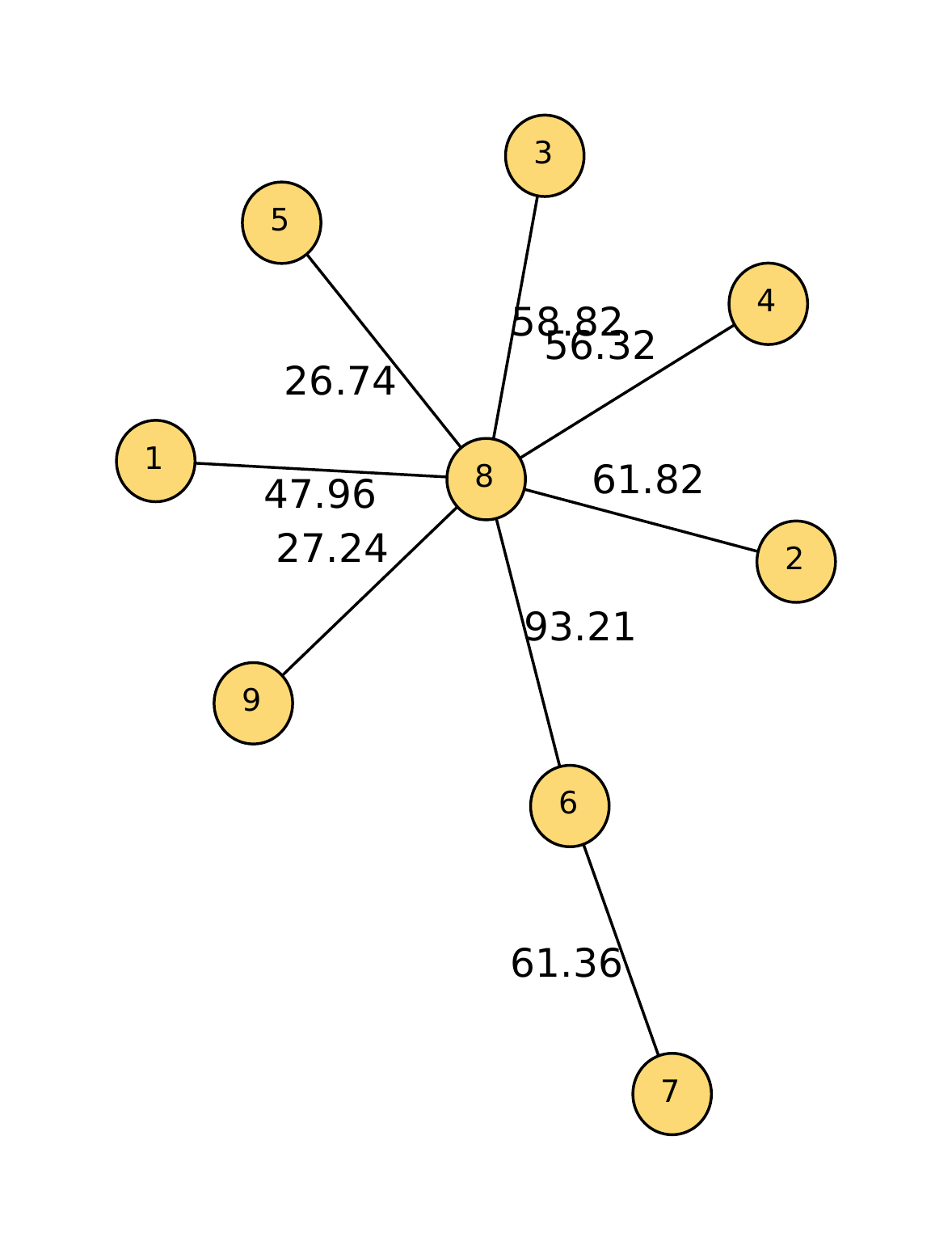}}
	\subfigure[$\gamma^*$ = 33.5068]{
	\includegraphics[scale=0.30]{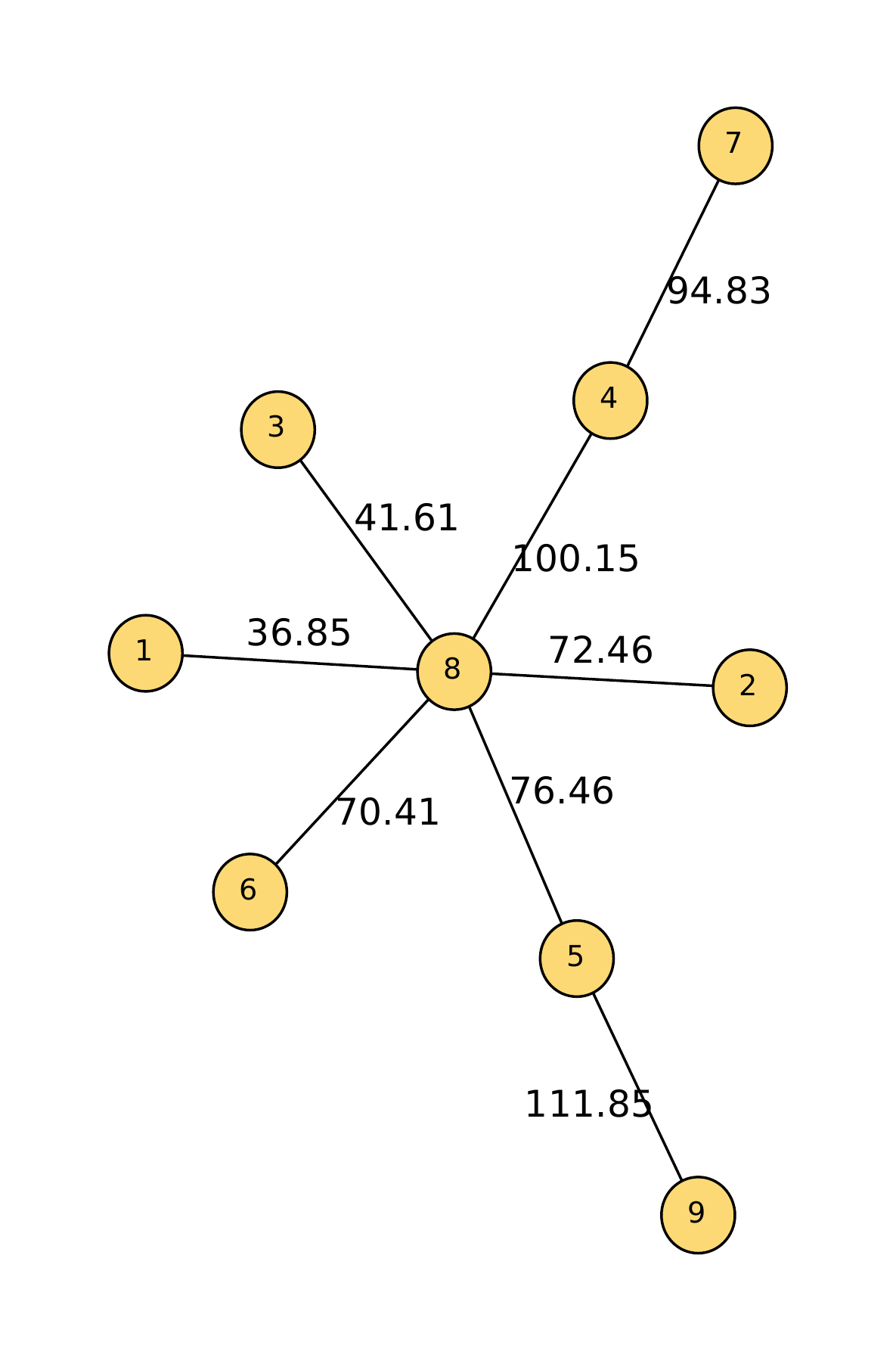}}
	\subfigure[$\gamma^*$ = 31.7445]{
	\includegraphics[scale=0.30]{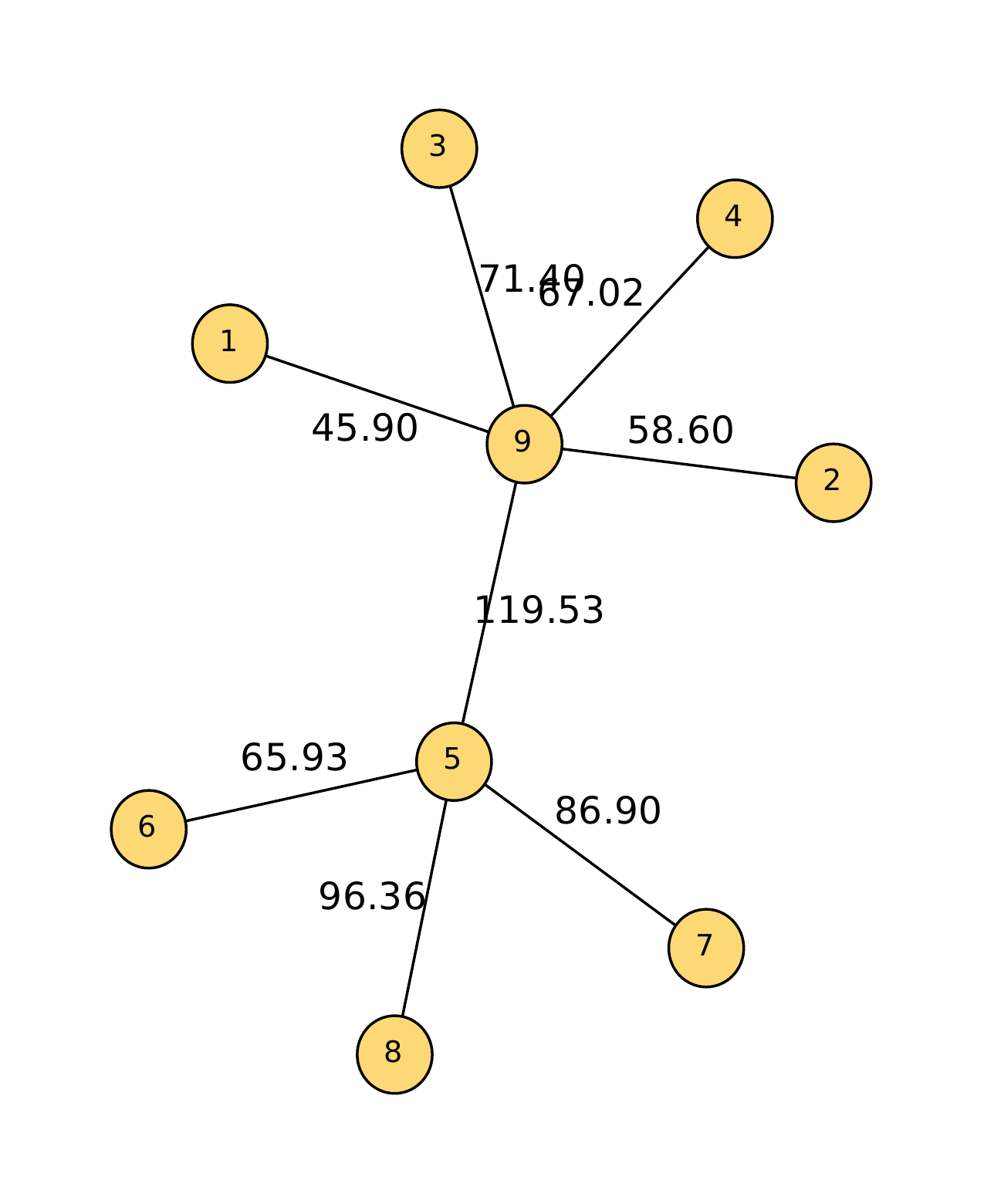}}
	\caption{Optimal networks of nine nodes with maximum algebraic connectivity 
	for the random instances shown in Table \ref{Table:improved_EA1_n9} and the corresponding 
	adjacency matrices in Appendix \ref{ch:appendix}.}  
	\label{Fig:Optimal_graphs_n9}
\end{figure}

\section{Neighborhood search heuristics}
\label{Sec:heuristics}
A general approach to developing heuristics for NP-hard problems primarily involves the 
following two phases: a) Design of algorithms, also known as construction heuristics, 
that can provide an initial feasible solution for the problem,
and b) Design of a systematic procedure, also known as improvement heuristics, 
to iteratively modify this initial feasible solution to improve
its quality. 
Since the feasible solutions discussed in this dissertation mainly concern the 
construction of spanning trees, development 
of a construction heuristic for the proposed problem is quite trivial. 
However, it is non-trivial to improve a feasible solution to obtain another feasible solution
with better algebraic connectivity. 

In section \ref{Sec:BSDP}, we discussed an exact algorithm based on the BSDP
approach wherein every iteration of the bisection, we are guaranteed to obtain
a feasible solution with algebraic connectivity greater than or equal to a pre-specified value. 
However, the main drawback of this approach was the solving of MILPs of increasing 
complexity in every iteration without
any guarantee on a finite computation time.

However, there are several improvement heuristics available in the
literature for sequencing problem and traveling salesman type problems.
Some of them include neighborhood search methods, tabu search
\cite{gendreau1994tabu} and even genetic algorithms \cite{potvin1996genetic}. 
In this section, we focus on developing quick improvement heuristics for the 
problem of maximizing algebraic connectivity based on neighborhood search methods. 
The remainder of this section is organized as follows: we initially develop
$k$-opt heuristic and later an improved $k$-opt wherein the size of the search space in
the neighborhood of a feasible solution is reduced significantly. We conclude the section 
with the computational performance of the proposed heuristics for networks up to sixty nodes.

\subsection{$k$-opt heuristic}
\label{Subsec:kopt}

We consider a neighborhood search heuristic called 2-opt exchange heuristic,
a special case of a more general $k$-opt heuristic which
has been successfully used for solving traveling salesman problems \cite{Croes1958}.
We extend the idea of this heuristic to solve the problem of maximizing
algebraic connectivity. This  heuristic can also be easily extended to the 
problem of maximizing algebraic connectivity with resource constraints 
which will be discussed in the later sections. This section is primarily based on 
\cite{nagarajan2014heuristics}.

Any new feasible solution $\mathcal{T}_2$ for this
problem is said to be in the $k$-exchange neighborhood of a feasible solution
$\mathcal{T}_1$ if $\mathcal{T}_2$ is obtained by replacing
$k$ edges in $\mathcal{T}_1$. In case of 2-opt,
we start with a feasible solution,
which is a spanning tree satisfying the resource constraints, and iteratively perform
2-opt exchanges for every pair of edges in the initial spanning tree
until no such exchanges can be made while improving the solution. 
A 2-opt exchange on one such pair
of edges of a random spanning tree is shown in
Figure \ref{Fig:2_opt_trees}. An important issue to be addressed is to make
sure that the solutions resulting after 2-opt exchanges are also feasible.
Ensuring feasibility in the case of spanning trees is relatively easy as 
after removing 2 edges, we are guaranteed to have 3 connected
components ($C_1,C_2,C_3$); therefore, by suitably adding any 2 edges
connecting all the 3 components, one is guaranteed to 
obtain a spanning tree. 

\begin{figure}[h!]
	\centering
	 \subfigure[Sample graph]{
	\includegraphics[scale=0.96]{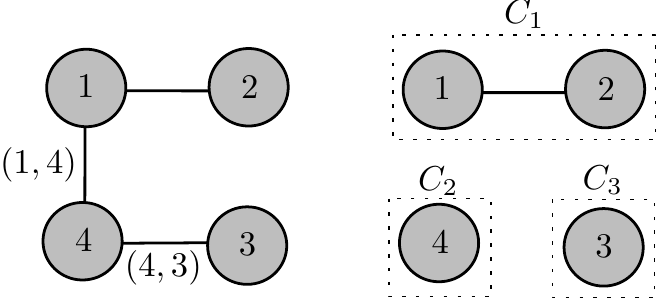}}
	\subfigure[2-opt feasible solutions]{
	\includegraphics[scale=0.96]{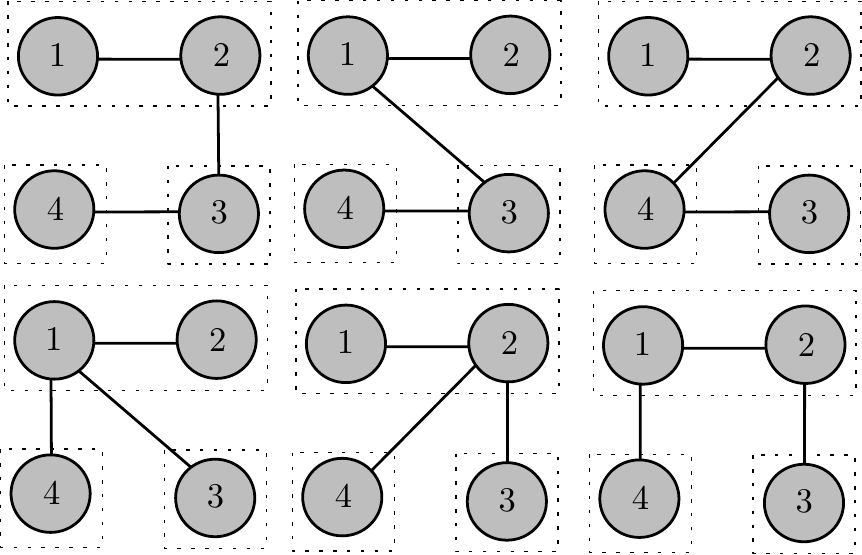}}
	\caption{This figure illustrates the 2-opt heuristic on an initial feasible solution, $\mathcal{T}_0$.
	After removing a selected pair of edges $\{(1,4)(4,3)\}$ from $\mathcal{T}_0$, the three connected components
   are shown in (a). Part (b) shows the 2-opt exchange on the connected components to obtain new feasible solutions. Source \cite{nagarajan2014heuristics}}
	\label{Fig:2_opt_trees}
\end{figure}

The new spanning tree (say, $\mathcal{T}_{2}$)
is considered for replacing the initial spanning tree ($\mathcal{T}_{1}$)
if it has a better algebraic connectivity than $\mathcal{T}_{1}$.
A pseudo-code of 2-opt heuristic is given in Algorithm \ref{algo:2opt_1}.
Note that this heuristic performs a 2-opt exchange on a given initial spanning
tree to obtain a new tree with better algebraic connectivity satisfying the
resource constraint. This procedure is 
repeated on the current feasible solution iteratively 
until no improvement is possible.

\begin{algorithm}[!h]
    \begin{algorithmic}[1]
    \STATE $\mathcal{T}_0 \leftarrow$ Initial feasible solution
	 \STATE $\lambda_0 \leftarrow$ $\lambda_2(L(\mathcal{T}_0)))$
        \FOR{each pair of edges $\{(u_1,v_1),(u_2,v_2)\} \in \mathcal{T}_0$}
            \STATE Let $\mathcal{T}_{opt}$ be the best spanning tree in the 2-exchange neighborhood of $\mathcal{T}_{0}$ obtained by replacing edges $\{(u_1,v_1),(u_2,v_2)\}$ in $\mathcal{T}_{opt}$ with a different pair of edges.
				\IF {$\lambda_2(L(\mathcal{T}_{opt})) > \lambda_0$ and $\mathcal{T}_{opt}$ satisfies the resource constraint}
                \STATE $\mathcal{T}_0 \leftarrow \mathcal{T}_{opt} $
                \STATE $\lambda_0 \leftarrow \lambda_2(L(\mathcal{T}_{opt}))$
				\ENDIF
        \ENDFOR
    \STATE $\mathcal{T}_0$ is the best spanning tree in the solution space
	 with respect to the initial feasible solution
    \end{algorithmic}
\caption{: 2-opt exchange heuristic}
\label{algo:2opt_1}
\end{algorithm}

\subsection{Improved $k$-opt heuristic}
\label{Subsec:kopt2}

\begin{figure}[!h]
	\centering
	\includegraphics[scale=1.3]{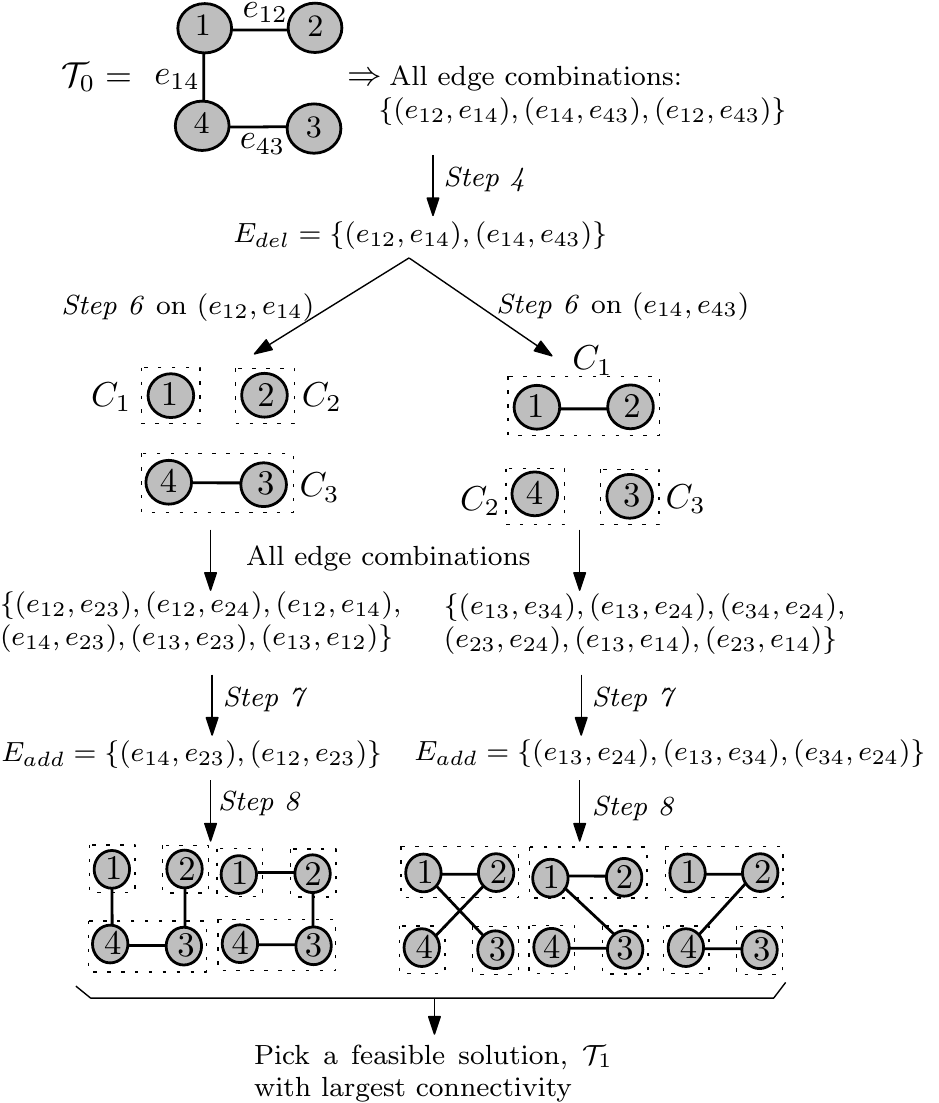}
	\caption{An example illustrating an improved 2-opt exchange heuristic for a network of 4 nodes. Source \cite{nagarajan2014heuristics}}
	\label{Fig:2_opt_trees}
\end{figure}

\begin{algorithm}[h!]
	\begin{algorithmic}[1]
		\STATE $\mathcal{T}_0 \leftarrow$ Initial feasible solution	
		\STATE $\lambda_0 \leftarrow$ $\lambda_2(L(\mathcal{T}_0)))$
		\STATE $E_{del} \leftarrow$ Subset of $k$-edge combinations considered for possible deletion as obtained by the edge ranking procedure
		\FOR{each edge combination in $E_{del}$}
		\STATE Delete the $k$ edges present in the edge combination to obtain connected components $C_1,C_2,C_3,\ldots, C_{k+1}$
		\STATE $E_{add} \leftarrow$ Subset of $k$-edge combinations considered for possible addition as obtained by the edge ranking procedure
		\STATE Let $\mathcal{T}_{1}$ be the spanning tree which is feasible and has the largest connectivity obtained from adding the edges in an edge combination from $E_{add}$
				\IF {$\lambda_2(L(\mathcal{T}_{opt})) > \lambda_0$ and $\mathcal{T}_{opt}$ satisfies the resource constraint}
		\STATE $\mathcal{T}_0 \leftarrow \mathcal{T}_{1} $
		\STATE $\lambda_0 \leftarrow \lambda_2(L(\mathcal{T}_{1}))$
		\ENDIF
		\ENDFOR
		\STATE Output $\mathcal{T}_0$ as the (new) current solution
	\end{algorithmic}
	\caption{: $k$-opt exchange}
	\label{algo:2opt_2}
\end{algorithm}

As discussed in section \ref{Subsec:kopt}, 
there are two main steps in the $k$-opt exchange heuristic: 
Choosing a collection of $k$ edges to remove from the current solution and then reconnecting 
the resulting, disjoint components with a new collection of $k$ edges. Clearly, there are 
several combinations of $k$ edges that can be removed from ($or$ added to) a given solution, 
especially when the number of nodes in the graph is large. For example, while performing
3-opt heuristic on a network of 40 nodes with four connected components after deleting 
any three edges, there would be at least 16,000 combinations of edges which can be connected to 
form a spanning tree. Therefore, choosing an efficient procedure for the deletion and addition of edges is 
critical for developing a relatively fast algorithm.  In the following subsections,
we provide procedures for implementing these steps.

\vspace{0.5cm}
\noindent 
\textbf{Selecting a collection of $k$-edge combinations to delete:}
The basic idea here is to list all the possible combinations of $k$ edges that 
can be deleted from the current feasible solution, assign a value for each 
combination, rank the combinations based on these values, and then choose a 
subset of these combinations for further processing. We assign a value to a 
combination of edges by first asking the following basic question: Which are 
the $k$ edges that needs to be deleted from the current solution 
$\mathcal{T}_0$ so that $\mathcal{T}_0$ (possibly) incurs the smallest 
reduction in the algebraic connectivity? To answer this question, let 
$\mathcal{T}_0 \setminus \{e\}$ denote the graph obtained
by deleting an edge $e$ from the graph $\mathcal{T}_0$ and by abuse of notation,
let the Laplacian of the graph, $\mathcal{T}_0 \setminus \{e\}$ be denoted
by $L(\mathcal{T}_0 - e)$. By variational characterization, we have the
following inequality:

\begin{subequations}
	\label{eq:edge_del}
	\begin{align}
		\lambda_2(L(\mathcal{T}_0-e)) &\leq v' L(\mathcal{T}_0-e) v \ ~\forall v\in \mathcal{V} \\
		&= v' L(\mathcal{T}_0) v -  v' L_e v \ ~\forall v\in \mathcal{V} \\
		&= v' L(\mathcal{T}_0) v -  w_{e}(v_i-v_j)^2 \ ~\forall v\in \mathcal{V}
	\end{align}
\end{subequations}
where, $\mathcal{V} := \{v:  \sum_i v_i = 0, \|v\|=1 \}$,
$w_{e}$ is the weight of edge $e=(i,j)$ and
$v_i$ represents the $i^{th}$ component of the vector $v$. One may
observe from the above inequality that by choosing an edge with
a minimum value of $w_{e}(v_i-v_j)^2$, the upper bound on the algebraic connectivity
of the graph, $\mathcal{T}_0 \setminus \{e\}$, is kept as high as possible.
Also, we numerically observed that, $w_{e}(v_i-v_j)^2$ was kept to a minimum
by choosing $v$ as the eigenvector corresponding to the maximum eigenvalue 
of $L(\mathcal{T}_0)$. Hence, for any combination of $k$ edges denoted by 
$S$, we assign a value given by $\sum_{e=(i,j)\in S} w_{e}(v_i-v_j)^2$. 
Then, we rank all the combinations based on the increasing values and 
choose a subset of these combinations that corresponds to the lowest values. 
In this work, the fraction of combinations that is considered for deletion 
is specified through a parameter called the edge deletion factor. 
The edge deletion factor is defined as the ratio of the number of 
$k$-edge combinations considered for deletion to the maximum number 
of possible $k$-edge combinations ($i.e.$, $\binom{n-1}{k}$). We 
will discuss more about this factor later in section \ref{Subsec:ch4_results}.

\vspace{0.5cm}
\noindent 
\textbf{Selecting a collection of $k$-edge combinations to add:}
In the case of spanning trees, after removing $k$ edges, we are 
guaranteed to have a graph $\tilde{\mathcal{T}}_0$ with exactly $k+1$ 
connected components $\{C_1,C_2,C_3,\ldots, C_{k+1}\}$; therefore,
by suitably adding a collection of $k$ edges connecting all the $k+1$ 
components in $\tilde{\mathcal{T}}_0$,
one is guaranteed to obtain a spanning tree, $\mathcal{T}_1$. Also, we 
add these edges while ensuring that the resulting tree satisfies the 
diameter constraints. The new feasible solution $\mathcal{T}_{1}$
is considered for replacing $\mathcal{T}_{0}$ if it has a larger algebraic 
connectivity than $\mathcal{T}_{0}$.

As in the edge-deletion procedure, checking for every addition of $k$ 
edges may be computationally intensive for large instances. Therefore, 
we develop another edge ranking procedure for adding edges as follows: 
Let $\tilde{\mathcal{T}}_0 \cup \{e\}$ denote the graph obtained
by adding an edge $e=(i,j)$ to the graph $\tilde{\mathcal{T}}_0$ and
let the Laplacian of the graph,  $\tilde{\mathcal{T}}_0 \cup \{e\}$, be denoted
by $L(\tilde{\mathcal{T}}_0  + e)$. By variational characterization, we have the
following inequality:

\begin{subequations}
	\label{eq:edge_add}
	\begin{align}
		\lambda_2(L(\tilde{\mathcal{T}}_0+e)) &\leq v' L(\tilde{\mathcal{T}}_0+e) v \ ~ \forall v\in \mathcal{V} \\
		&= v' L(\tilde{\mathcal{T}}_0) v +  v' L_e v \ ~ \forall v\in \mathcal{V} \\
		&= v' L(\tilde{\mathcal{T}}_0) v +  w_{e}(v_i-v_j)^2 \ ~ \forall v\in \mathcal{V} .
	\end{align}
\end{subequations}

One may observe from the above inequality that by choosing an edge with
a maximum value of $w_{e}(v_i-v_j)^2$, the upper bound on the algebraic connectivity
of the graph, $\tilde{\mathcal{T}}_0 \cup \{e\}$ is kept as high as possible.
Just like the edge deletion step, let $v$ be the
eigenvector corresponding to the maximum eigenvalue of $L(\tilde{\mathcal{T}}_0)$.
Hence, for any combination of $k$ edges denoted by $S$, we assign a value 
given by $\sum_{e=(i,j)\in S} w_{e}(v_i-v_j)^2$. Then, we rank all the 
combinations based on decreasing values and choose a subset of these 
combinations that corresponds to the highest values. The number of 
combinations that are considered for addition is another parameter 
and can be specified based on the problem instances. 

A pseudo-code of the $k$-opt exchange is outlined in Algorithm \ref{algo:2opt_2}.
An illustration of such a procedure on one such pair ($k=2$) of edges 
for a spanning tree with 4 nodes is shown in Figure \ref{Fig:2_opt_trees}. 
This exchange is iteratively applied on the current solution until no improvements can be made.

\subsection{Performance of $k$-opt and improved $k$-opt heuristic}
\label{Subsec:ch4_results}

All the computations in this section were performed with the same 
computer specifics as mentioned in section \ref{Sec:results1}.

In this section, we performed all the simulations for the case of $k$ equal to two and 
three, which we shall refer as 2-opt and 3-opt heuristics. The 2-opt 
(Algorithm \ref{algo:2opt_1}) and improved 2-opt 
heuristics (Algorithm \ref{algo:2opt_2}) were implemented in Matlab 
since the heuristics terminated in a reasonable
amount of time. But, we implemented the improved 3-opt heuristic 
(Algorithm \ref{algo:2opt_2}) in \texttt{C++} programming language 
which could handle up to sixty nodes in a reasonable amount of time. 


\vspace{0.5cm}
\noindent \textbf{Construction of initial feasible solution:}
Since the primary idea of the $k$-opt heuristic is to search in the 
neighborhood space of an initial feasible solution, it would be important 
to construct a good initial feasible solution. As we saw in the previous section
on exact algorithms (Figure \ref{Fig:Optimal_graphs_n8}), the networks with maximum algebraic
connectivity tend to be clustered and are low in diameter. With this intuition, the
procedure to construct an initial feasible solution is as follows: 
For a given complete weighted graph with $n$ nodes, 
sort the $n$ possible star graphs (two diameter graphs) in the decreasing 
order of the sum of the weights of the edges incident on the internal node (weighted degree) 
of the graph. After sorting, we observed that, performing $k$-opt exchange on
five of these ranked star graphs provided a great improvement in the
algebraic connectivity.

\vspace{0.5cm}
\noindent \textbf{Selecting edge deletion and edge addition factor:}
For the improved $k$-opt heuristic, we set the edge deletion factor
to be equal to 0.15 in all the simulations.
We chose this value based on the simulation results shown in Figure \ref{Fig:hindawi_plot}.
This figure shows the average algebraic connectivity of the final solution
(and the computation time) obtained using the improved $3-$opt heuristic as a function
of the edge deletion factor. We observed that there was not much improvement
in the quality of the feasible solutions beyond a value of 0.15
(for the edge deletion factor) even for large instances ($n=30,40$).
Hence, we chose 0.15 as the edge deletion factor. Also, we set the number
of combinations of edges to be added 
to be at most equal to $5^k$. For improved 3-opt, this parameter was set to 125. We chose this 
value based on the simulation results shown in Figure \ref{Fig:hindawi_plot1}. 
For improved 2-opt, this parameter was set to 25 after performing similar simulations. 
\begin{figure}[h!]
	\centering
	\subfigure[]{
	\includegraphics[scale=0.415]{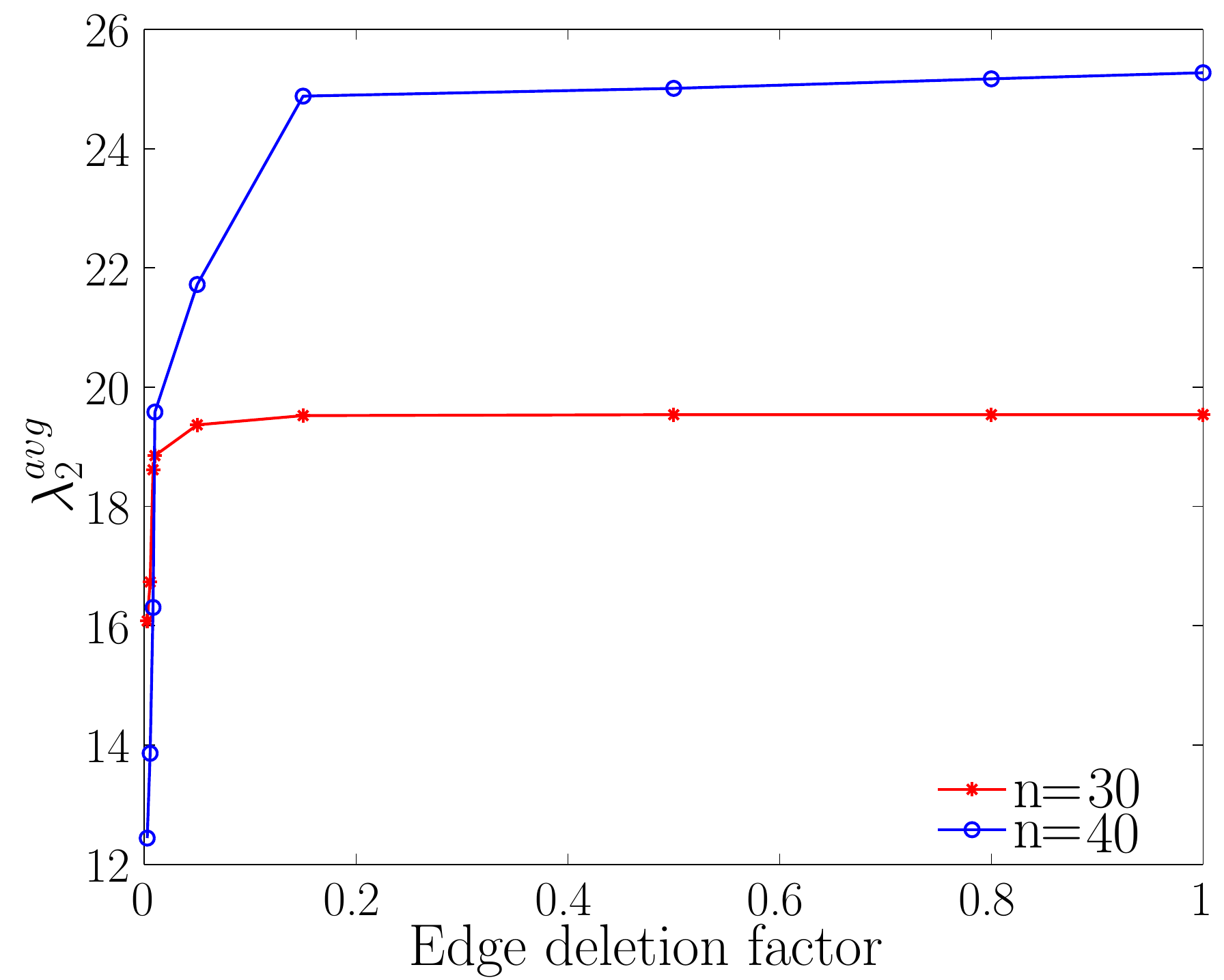}}
	\subfigure[]{
	\includegraphics[scale=0.42]{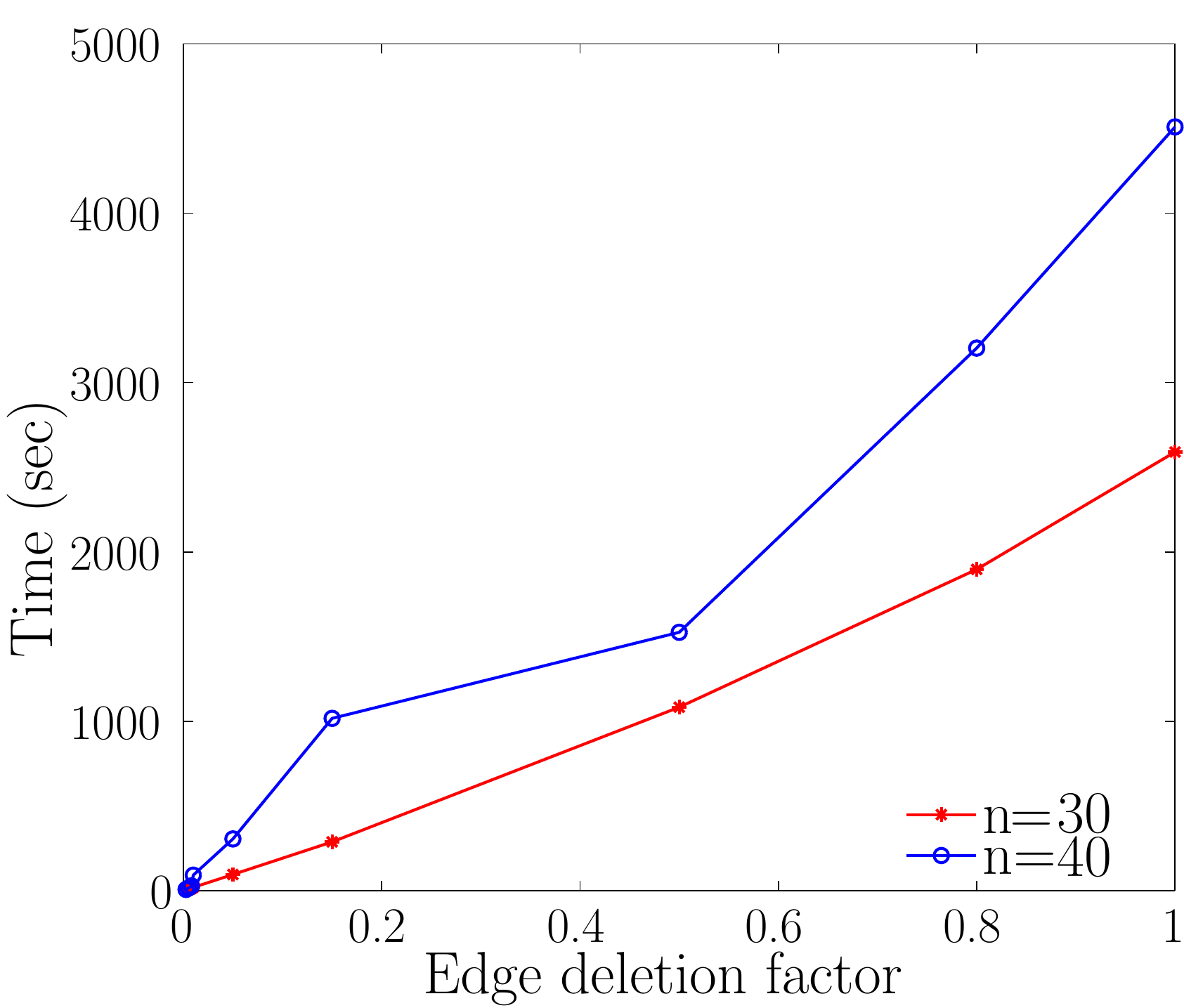}}
	\caption{Average values of the algebraic connectivity (a) and computation 
	times (b) obtained as a function of the edge deletion factor using the 
	improved 3-opt heuristic over ten instances. In these computations, 
	the maximum number of edge combinations considered for addition between 
	any two components was set to 125.}
	\label{Fig:hindawi_plot}
\end{figure}

\begin{figure}[h!]
	\centering
	\subfigure[]{
	\includegraphics[scale=0.42]{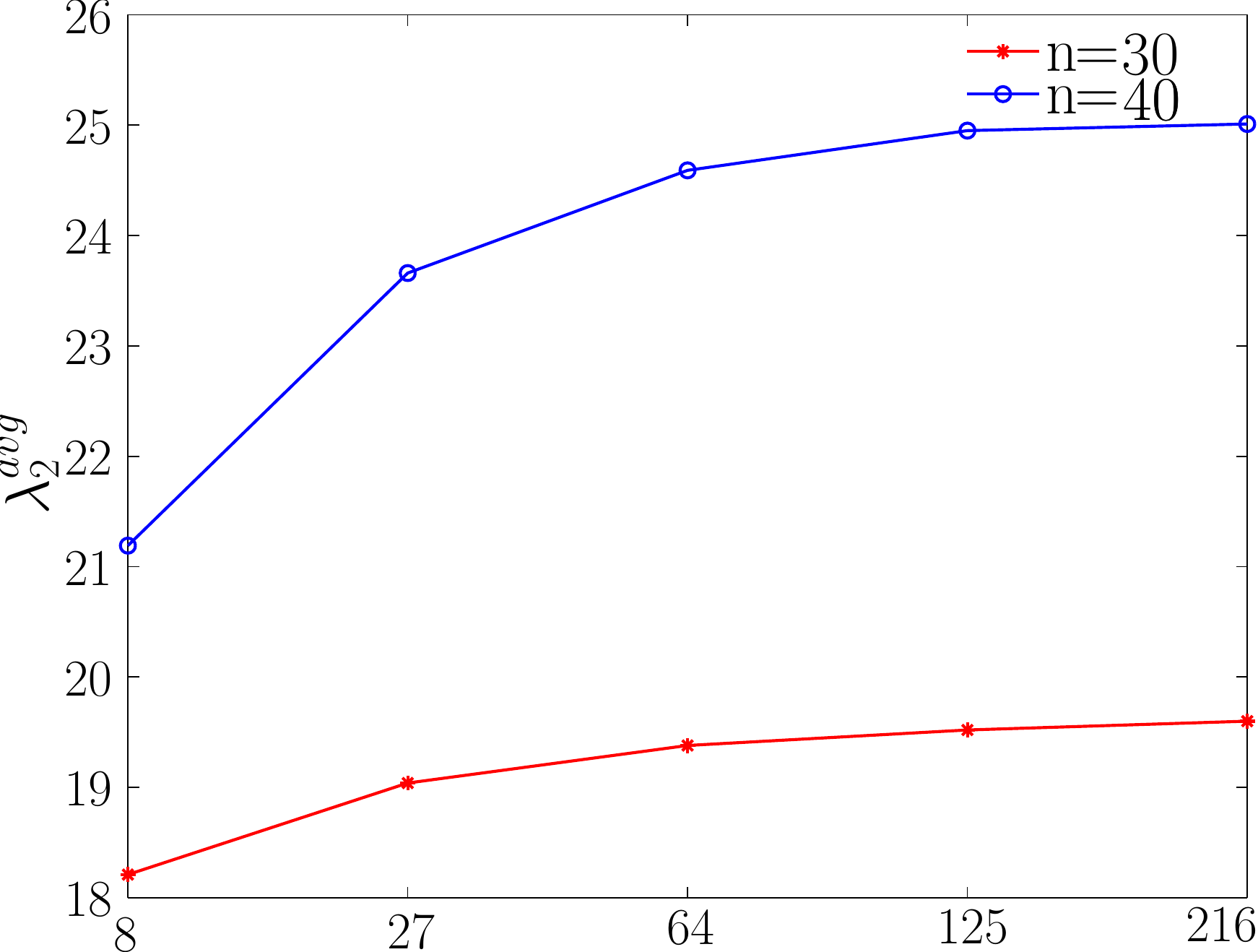}}
	\subfigure[]{
	\includegraphics[scale=0.42]{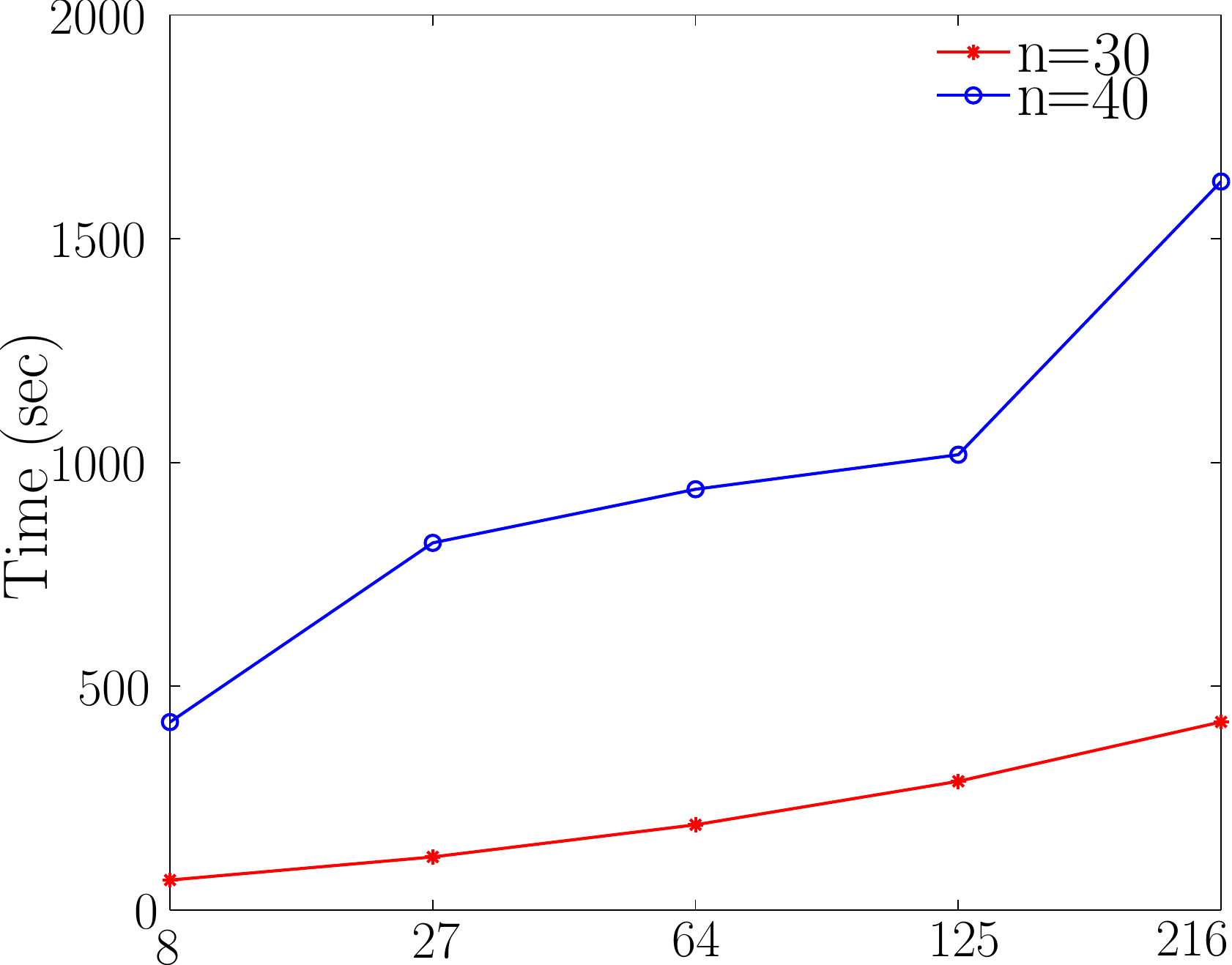}}
	\caption{The average algebraic connectivity (a) and computation times (b) obtained 
	as a function of the maximum number of edge combinations considered for addition 
	between any two components in the improved 3-opt heuristic over ten instances. 
	In these computations, the edge deletion factor was set to 0.15.}
	\label{Fig:hindawi_plot1}
\end{figure}

\vspace{0.5cm}
\noindent \textbf{Performance of $k$-opt and improved $k$-opt with respect to optimal
solutions:}
For the problem with 8 nodes, we define the solution quality of the proposed heuristic as 
$$ \mathrm{Solution \ quality} = \frac{\lambda_2^* - \lambda_2^{kopt}}{\lambda_2^*}\times 100$$ 
where $\lambda_2^{kopt}$ denotes the algebraic connectivity of the solution 
found by the $k$-opt heuristic and $\lambda_2^*$ represents the optimum. The 
results shown in Table \ref{opti_2opt_tabu} are for $10$ random instances generated for
networks with 8 nodes. Based on the results in Table \ref{opti_2opt_tabu},
it can be seen that the quality of solutions found by the $k$-opt ($k=2,3$) 
and the improved $k$-opt ($k=2,3$) heuristic were very good and gave optimal solutions for all 
the 10 random instances. Also, on an average, the computation time for the heuristics 
were less than 1.5 seconds to obtain the best feasible solution for the problem with 
8 nodes. An improvement in the computational time for improved $k$-opt heuristic can
be observed for larger instances as discussed in the later parts of this section. 
Instance $1$ of Table \ref{opti_2opt_tabu} is 
pictorially shown in Figure \ref{Fig:optimal_tree_n8_1}.

\begin{figure}[!h]
	\centering
	\subfigure[Complete network]{
	\includegraphics[scale=0.54]{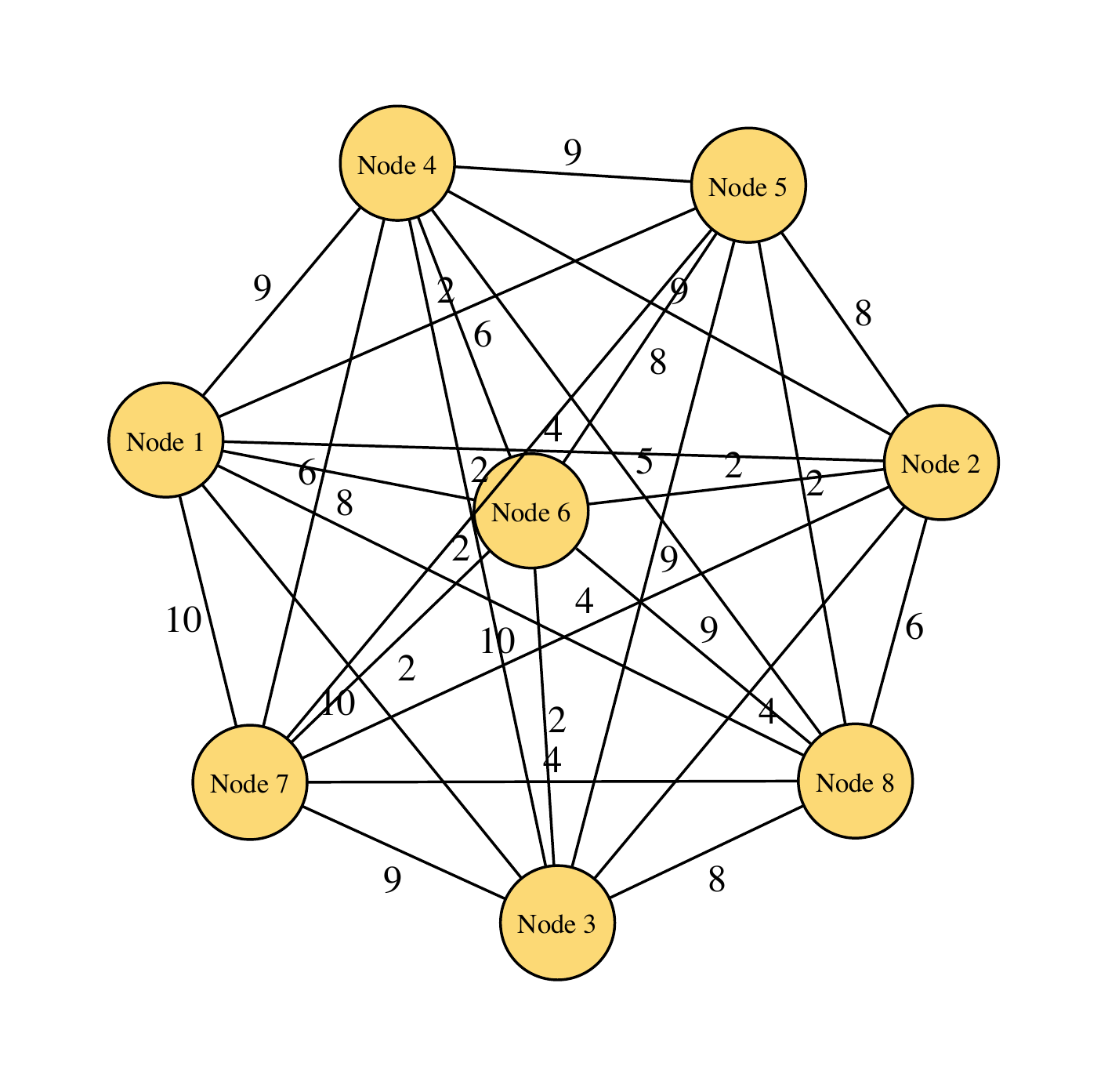}}
	\subfigure[$\lambda_2^{initial}$ = 2.2045]{
	\includegraphics[scale=0.40]{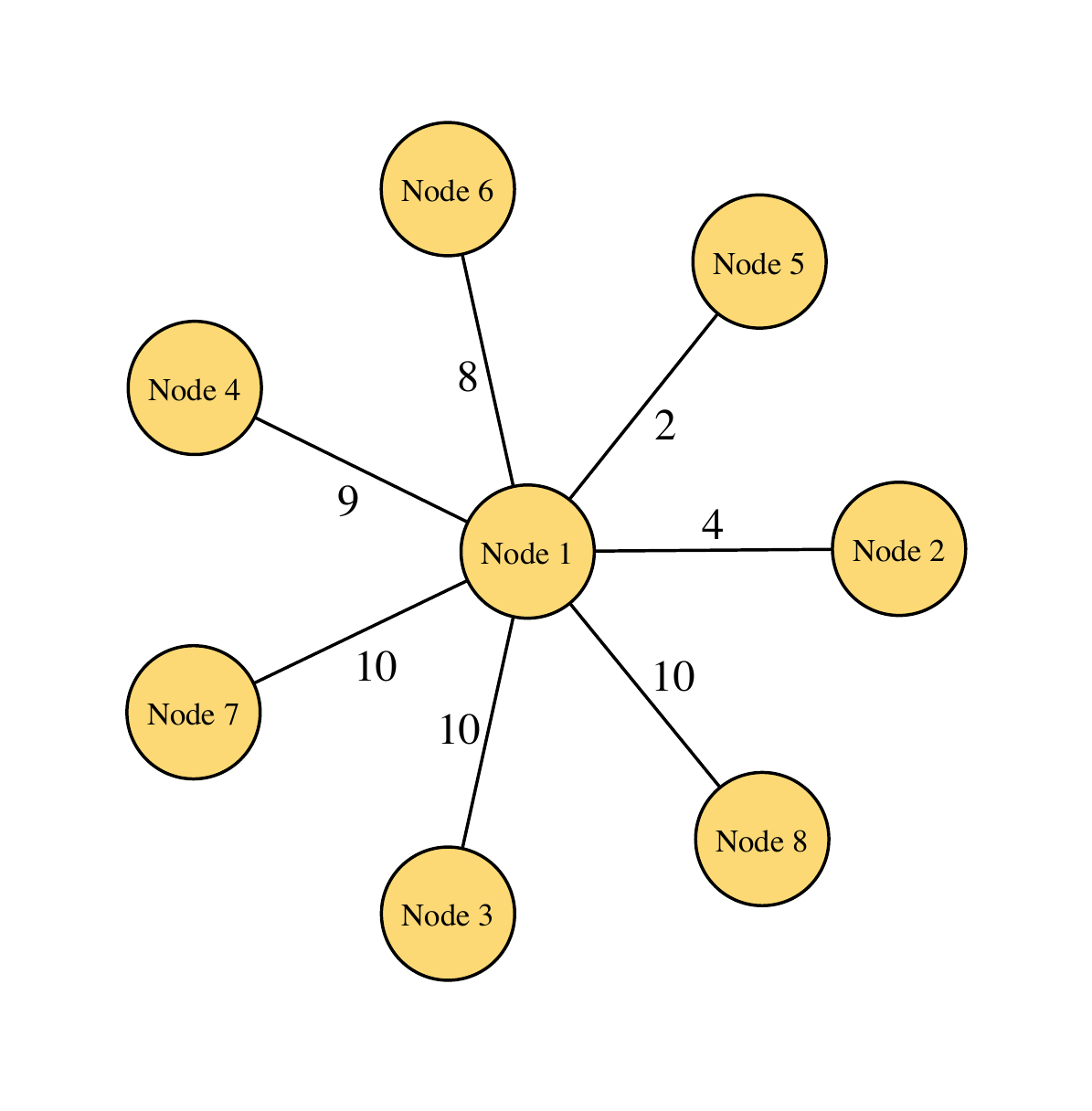}}
	\subfigure[$\lambda_2^{2opt}$,$\lambda_2^{3opt}$ = 3.9712]{
	\includegraphics[scale=0.40]{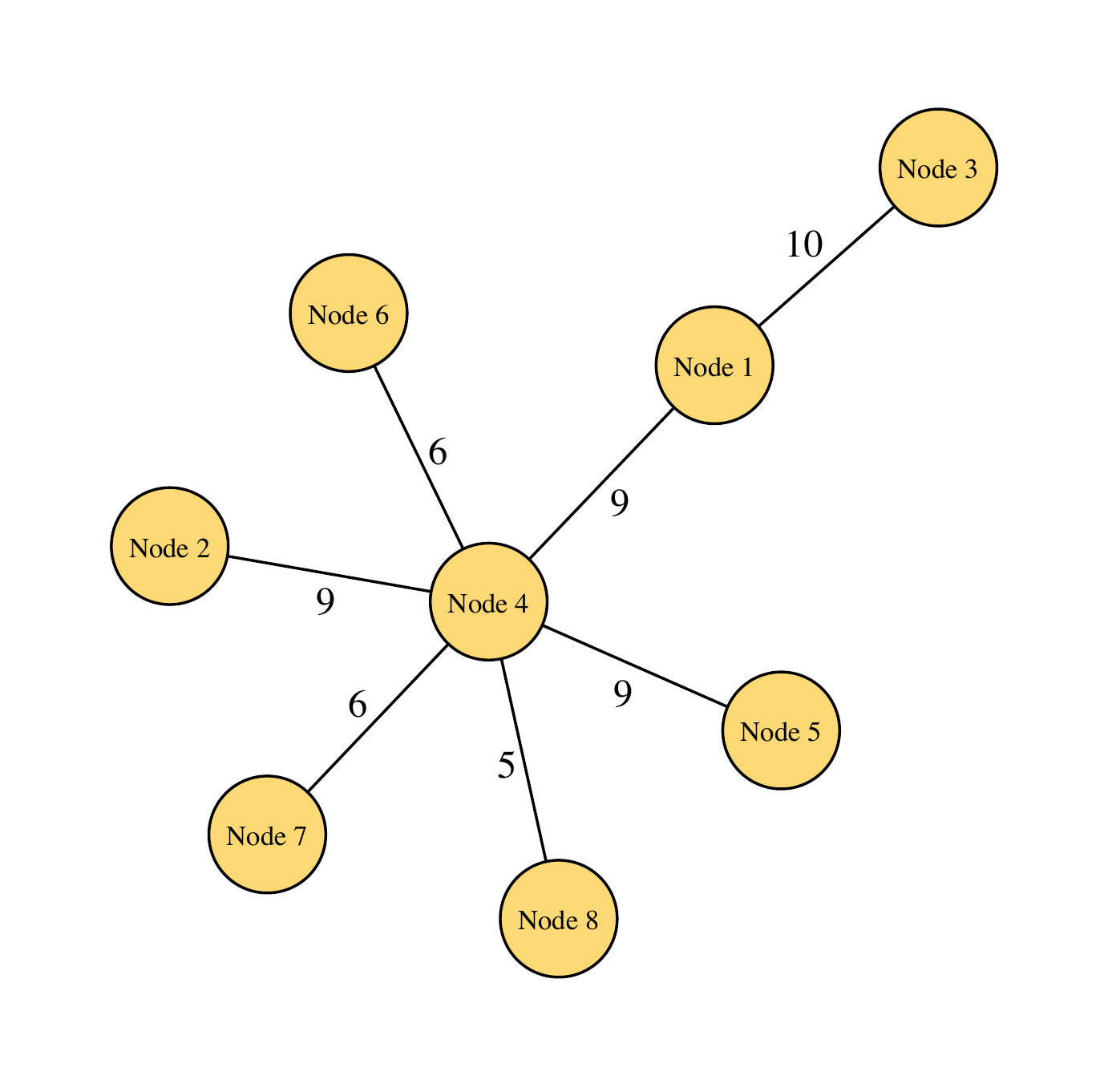}}
	\caption{A network with all possible edges connecting 8 nodes including edge weights 
	are shown in (a). (b) represents the initial feasible solution which
	is a star graph. (c) represents an optimal network which also happens to be the
	solution found by the 2-opt and 3-opt heuristics.}
	\label{Fig:optimal_tree_n8_1}
\end{figure}

\begin{table}[h!]
		  \caption{Comparison of the quality of the solutions found by the $k$-opt heuristic (Algorithm \ref{algo:2opt_1}) 
for networks with 8 nodes. $\lambda_2^*$ is the optimal algebraic connectivity.}
\begin{center}
\begin{tabular}{cccccc}
\toprule
\multicolumn{1}{l}{} & \multicolumn{ 2}{c}{Optimal solution} & \multicolumn{ 3}{c}{$k$-opt, Improved $k$-opt ($k$=2,3)}  \\
\cmidrule(l{0.25em}r{0.25em}){2-3} \cmidrule(l{0.25em}r{0.25em}){4-6} 
No. & \multicolumn{1}{l}{$\lambda_2^*$} & Time & \multicolumn{1}{l}{$\lambda_2^{kopt}$} & Solution  & Time \\
&&(sec)&&quality& (sec) \\
\cmidrule(r){1-6}
1 & 3.9712 & 180.57 & 3.9712 & 0.00 & 1.1  \\
2 & 4.3101 & 408.10 & 4.3101 & 0.00 & 2.1  \\
3 & 3.9297 & 621.85 & 3.9297 & 0.00 & 1.3  \\
4 & 3.5275 & 216.79 & 3.5275 & 0.00 & 2.3  \\
5 & 3.8753 & 470.63 & 3.8753 & 0.00 & 0.8  \\
6 & 3.7972 & 342.14 & 3.7972 & 0.00 & 1.2  \\
7 & 3.7125 & 377.47 & 3.7125 & 0.00 & 1.7  \\
8 & 3.9205 & 313.12 & 3.9205 & 0.00 & 1.6  \\
9 & 3.7940 & 341.84 & 3.7940 & 0.00 & 2.3  \\
10 & 3.8923 & 316.86 & 3.8923 & 0.00 & 2.1 \\
\cmidrule(r){1-6}
Avg. &&358.43 &&0.00 & 1.5 \\
 \bottomrule
\end{tabular}
\end{center}
\label{opti_2opt_tabu}
\end{table}

\vspace{0.5cm}
\noindent \textbf{Performance of $k$-opt and improved $k$-opt for large instances:}
For problems with larger instances ($n \geq 10$), in Table \ref{tab:summary1}, we analyze the 
improvement in the computation time of improved 2-opt heuristic with respect to the standard 2-opt heuristic 
and also compare their solution qualities. 

For a given initial feasible solution, the neighborhood 
search space for the improved $2$-opt is a
subset of the neighborhood search space for the standard $2$-opt. 
Hence, we define the reduction in the value of the algebraic connectivity 
of improved 2-opt from the standard 2-opt as 
$$ \mathrm{percent \ reduction} = \frac{\lambda_2^{2opt} - \lambda_2^{2opt_{imp}}}{\lambda_2^{2opt}} \times 100$$ 
where $\lambda_2^{2opt}$ is the algebraic connectivity of a solution obtained 
from the 2-opt heuristic and $\lambda_2^{2opt_{imp}}$ is the algebraic connectivity of a solution obtained 
from the improved 2-opt heuristic. From the results in Table \ref{tab:summary1}, it can be seen 
that the improved 2-opt heuristic performed almost as good as the 2-opt heuristic without much 
reduction in the quality of the solution but with 
a very remarkable improvement in the computational time to obtain the feasible solution. Therefore,
it can be observed that the greedy procedure of deletion and addition of the edges 
based on the variational characterization of the eigenvalues has reduced the 
neighborhood search space very effectively. 

\begin{table}[htbp]
\centering
\caption{Comparison of 2-opt with improved 2-opt heuristic solutions for various
	problem sizes. Here, the solution quality was averaged over ten random instances for each $n$. }
\begin{tabular}{cccc}
	\toprule
$n$ & \multicolumn{1}{c}{2-opt} & \multicolumn{2}{c}{Improved 2-opt}  \\
\cmidrule(l{0.25em}r{0.25em}){2-2} \cmidrule(l{0.25em}r{0.25em}){3-4} 
 \multicolumn{ 1}{c}{}  & Time & Percent & Time  \\
 \multicolumn{ 1}{c}{}  & (sec) & reduction & (sec)  \\
\cmidrule(r){1-4}
10 &  0.88 & 0.00 & 0.10 \\
15 &  8.45 & 0.00 & 0.52 \\
20 &  60.47 & 0.00 & 1.65 \\
25 &  240.57 & 0.60 & 3.59  \\
30 &  1533.98 & 0.81 & 12.13 \\
35 &  3468.75 & 0.79 & 37.85 \\
40 &  5899.62 & 1.20 & 57.28 \\
45 &  8897.69 & 1.16 &  116.09 \\
50 &  10089.31 & 1.27 & 139.99  \\
55 &  12980.78 & 1.80 & 350.83 \\
60 &  16001.02 & 2.01 & 505.36 \\
	\bottomrule
\end{tabular}
\label{tab:summary1}
\end{table}

In Table \ref{tab:summary_table}, we further study the improvement 
in the solution quality of the improved 3-opt heuristic with respect to the 
improved 2-opt heuristic. For this purpose, 
we define the percent improvement as follows: 
$$ \mathrm{percent \ improvement} = \frac{\lambda_2^{3opt_{imp}} - \lambda_2^{2opt_{imp}}}{\lambda_2^{3opt_{imp}}} \times 100$$ 
where $\lambda_2^{3opt_{imp}}$ is the algebraic connectivity of a solution obtained 
from the improved 3-opt heuristic and $\lambda_2^{2opt_{imp}}$ is the algebraic connectivity of a solution obtained 
from the improved 2-opt heuristic. From the results in Table \ref{tab:summary_table}, it can be seen 
that the improved 3-opt heuristic performed consistently better than the 
improved 2-opt, though the quality of solution was quite comparable in an average sense.
It can also be observed that there were instances where the improvement in the solution
quality of the improved 3-opt from the improved 2-opt heuristic was up to around 18 percent for large instances.

In summary, computational results suggested that the improved 3-opt heuristic performed 
the best while the improved 2-opt heuristic provided a good trade-off between 
finding good feasible solutions and the required computation time.
Figure \ref{Fig:heuristics_n40} illustrates the solutions obtained 
from the improved 2-opt and $3$-opt search heuristic for a network with 40 nodes.

\begin{table}[h!]
\centering
\caption{Comparison of improved 3-opt and improved 2-opt heuristic solutions for various
	problem sizes. The percent improvement values were 
	averaged over ten random instances for each $n$. }
\begin{tabular}{ccccc}
	\toprule
$n$ & \multicolumn{1}{c}{Improved 2-opt} & \multicolumn{3}{c}{Improved 3-opt}   \\
\cmidrule(l{0.25em}r{0.25em}){2-2} \cmidrule(l{0.25em}r{0.25em}){3-5} 
 \multicolumn{ 1}{c}{}  & Time & Time & Average & Maximum \\
 \multicolumn{ 1}{c}{}  & (sec) & (sec) & percent improvement & percent improvement  \\
\cmidrule(r){1-5}
10 &  0.10  & 0.29    & 0.00  & 0.00  \\
15 &  0.52  & 3.06    & 0.01  & 0.08  \\
20 &  1.65  & 16.32   & 0.27  & 2.73   \\
25 &  3.59  & 60.38   & 0.60  & 4.92   \\
30 &  12.13 & 274.93  & 2.07  & 7.56    \\
35 &  37.85 & 480.15  & 2.02  & 12.59   \\
40 &  57.28 & 1016.99 & 5.62  & 17.89    \\
45 &  116.09& 2309.60 & 7.10  & 15.41   \\
50 &  139.99& 4219.17 & 1.38  & 5.10    \\
55 &  350.83& 6798.34 & 6.98  & 17.56 \\
60 &  505.36& 8974.46 & 7.97  & 16.92\\
	\bottomrule
\end{tabular}
\label{tab:summary_table}
\end{table}

\begin{figure}[!h]
	\centering
	\subfigure[$\lambda_2^{2opt_{imp}} = 21.4088$]{
	\includegraphics[scale=0.37]{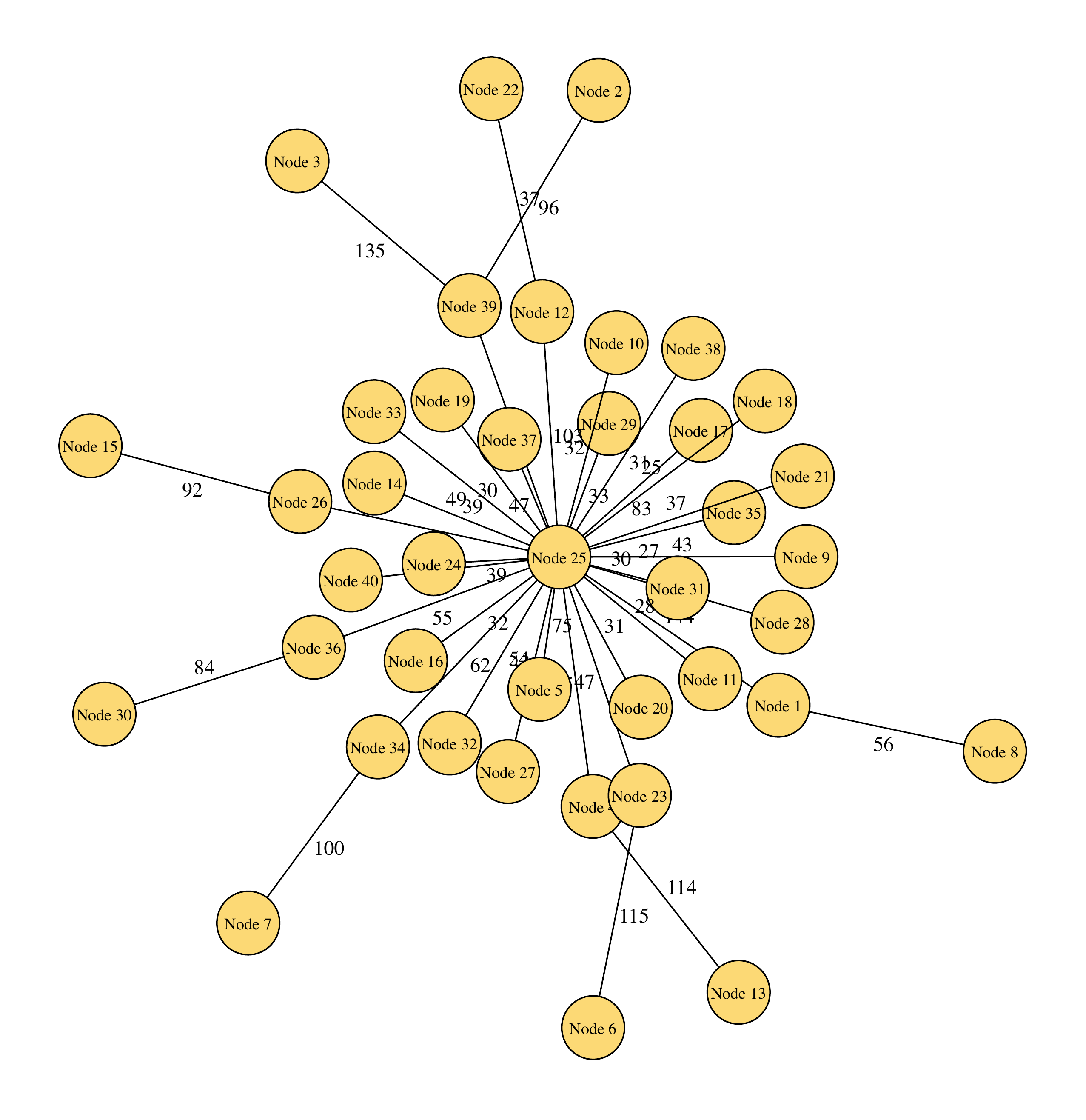}}
	\subfigure[$\lambda_2^{3opt_{imp}} = 26.0731$]{
	\includegraphics[scale=0.37]{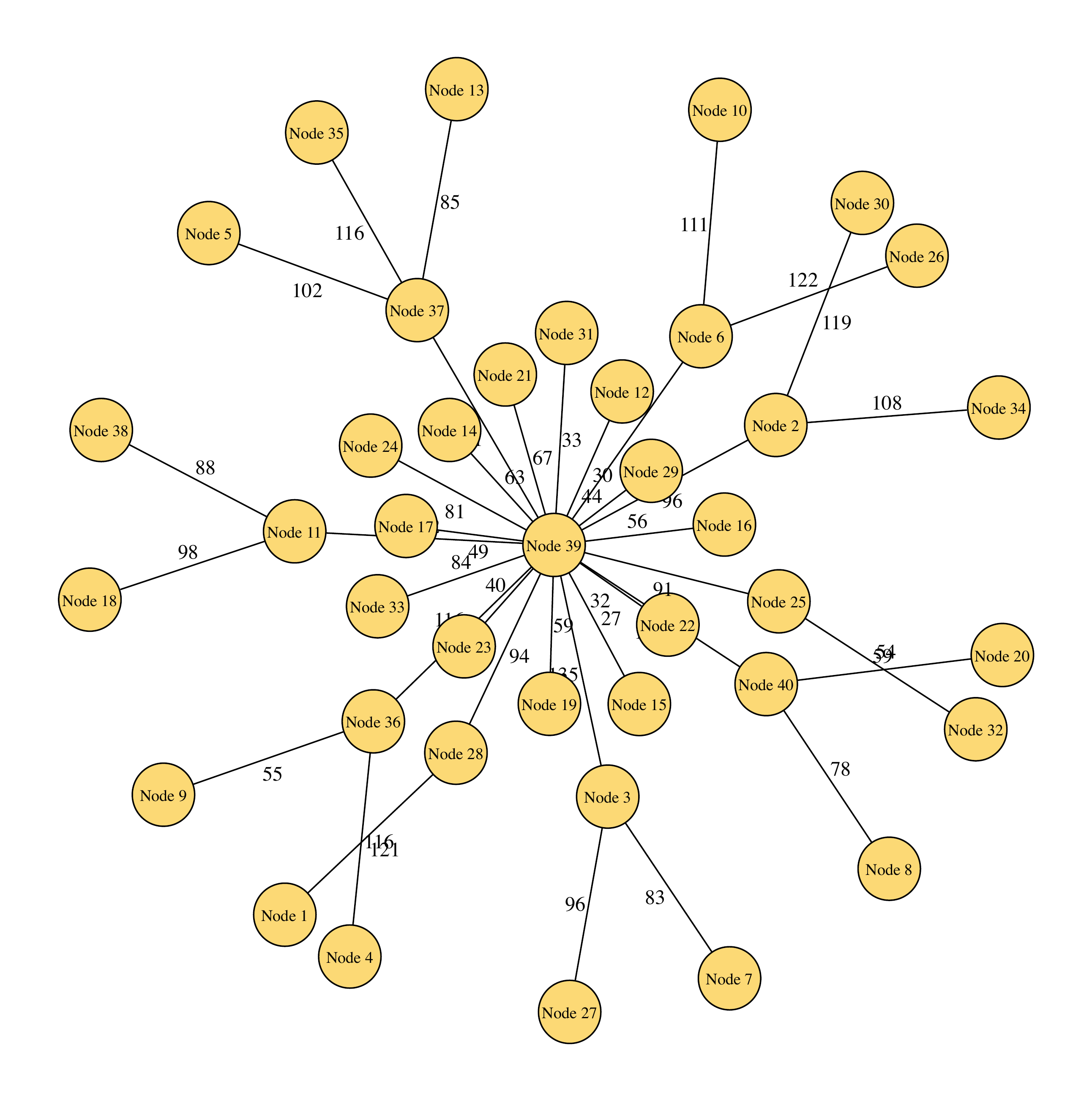}}
	\caption{Improved 2-opt and 3-opt exchange heuristic solutions for a network with 40 nodes. Source: \cite{nagarajan2012synthesizing,nagarajan2014heuristics}}
	\label{Fig:heuristics_n40}
\end{figure}

%% file: Chapters/section3.tex
%
%
%

\chapter{\uppercase {Maximization of algebraic connectivity under resource constraints}}

\label{ch3}


We discussed earlier 
the variant of \textbf{BP} involving the construction of an 
adhoc infrastructure network with UAVs that can be central to 
civilian and military applications. We also alluded briefly to the 
desirable attributes of the UAV adhoc network such as: a) Lower diameter to 
minimize latency in communicating data/information across the network, 
b) A limit/budget on the power consumed by the UAVs due to their limited 
battery capacities and c) High isoperimetric number so that the 
bottlenecking in a network can only occur at higher data rates while at the same time
be robust to node and link failures. 

In addition to the bound on the number of communication links, treating the 
requirements on diameter and power consumption as the constraints
on the resources, a variant of the \textbf{BP} that arises in the UAV application
may be posed as follows: Given a collection of UAVs which can serve as backbone nodes, 
how should they be arranged and connected so that 
\begin{itemize}
\item[(i)] the convex hull of the projections of their locations
on the ground spans a minimum area of coverage,
\item[(ii)] the resources such as the diameter of the network, total UAV power consumption for maintaining connectivity  and the total number of communication links employed are within their respective prescribed bounds, and
\item[(iii)] algebraic connectivity of the network is maximum among all possible networks satisfying the constraints (i) and (ii).
\end{itemize}

Since each of these resource constraints makes the problem much harder,
we separately formulate the diameter and power consumption constraint in the 
forthcoming sections. Hence, in the remainder of this section, we shall discuss the respective 
mathematical formulations and extend the 
algorithms based on cutting plane techniques as discussed in the 
in section \ref{sec:exact_algo} to synthesize optimal networks. The derivations and algorithms in this chapter
are primarily extracts from \cite{nagarajan2015synthesizing,nagarajan2012dscc}.

\section{Maximization of algebraic connectivity with diameter constraint}
\label{Sec:dia_cons}

\subsection{Problem formulation}
\label{Subsec:form_dia}
Based on the notation defined in section \ref{Sec:ch1_notation}, 
the problem of choosing at most $q$ edges from $E$ so that the
algebraic connectivity of the augmented network is maximized and the
diameter of the network is within a given constant ($D$) can be posed as follows:

\begin{equation}
		\label{eq:formulation1}
		\begin{array}{ll}
	  		\gamma^* = &\max  \lambda_2(L(x)), \\
			\text{s.t.} & \sum_{e \in E} x_e \leq q, \\
			& \delta_{uv}(x) \leq D \quad \forall u,v \in V, \\
			& x_e \in \{0, 1\}^{|E|}.
		\end{array}
\end{equation}
where $\delta_{uv}(x) $ represents the number of edges on the shortest path
joining the two nodes $u$ and $v$ in the network with an incident vector $x$. 
In the above formulation, there are two challenges that need to be overcome before one can
pose the above problem as a MISDP. First,
the objective is a non-linear function of the edges in the network; secondly,
the diameter constraint as stated in formulation \eqref{eq:formulation1} requires one to
implicitly compute the number of edges in the shortest path joining any
two vertices. 
We have already discussed in section \ref{Sec:ch1_notation} of section \ref{ch2}, how to 
address the above non-linear problem as a MISDP which is as follows:  

\begin{equation}
		\label{eq:formulation2}
		\begin{array}{ll}
	  		\gamma^* = &\max  \gamma, \\
			\text{s.t.} & \sum_{e \in E} x_e L_e \succeq \gamma (I_n - e_0 \otimes e_0),\\
			& \sum_{e \in E} x_e \leq q, \\
			& \delta_{uv}(x) \leq D \quad \forall u,v \in V, \\
			& x_e \in \{0, 1\}^{|E|}.
		\end{array}
\end{equation}

\begin{figure}[!h]
	\centering
	\subfigure[Original graph augmented with a source node]{
	\includegraphics[scale=0.6]{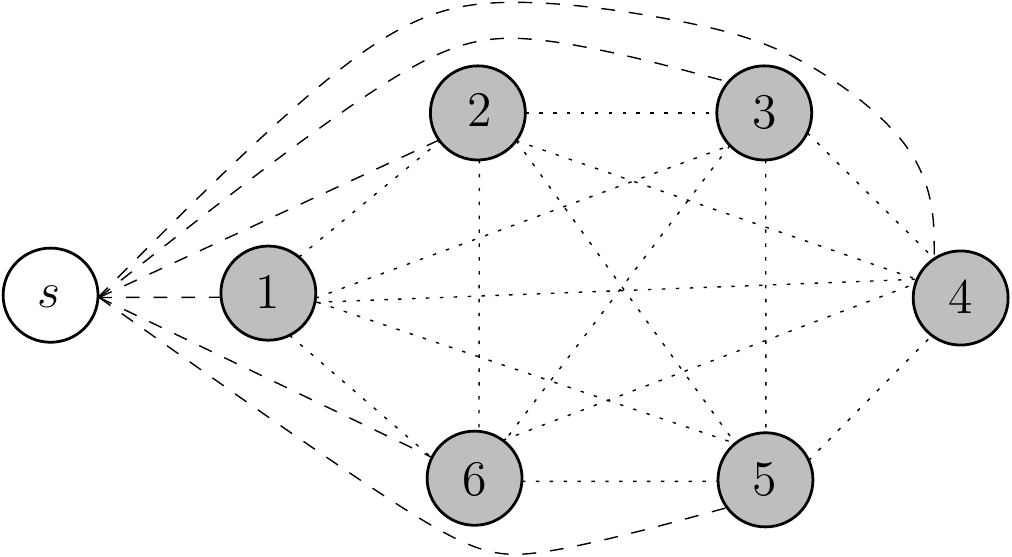}}
	\subfigure[Feasible solution]{
	\includegraphics[scale=0.6]{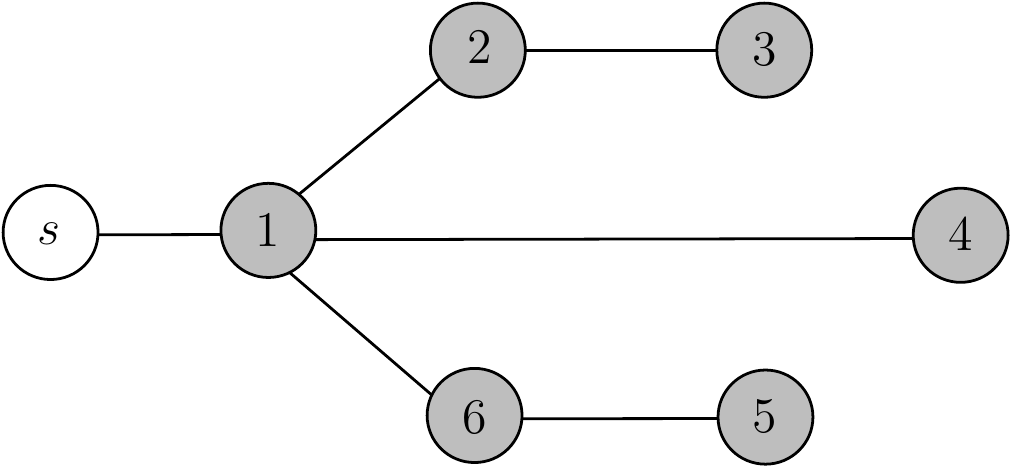}}
	\caption{Illustration of an addition of the source node ($s$)
	to the original (complete) graph represented by shaded nodes. If one were given that
	the diameter of the original graph must be at most $D=4$, then restricting
	the length of each of the paths from the source node to $(D/2)+1 = 3$, and
   allowing only one incident edge on $s$ will suffice as shown in (b). Source: \cite{nagarajan2012diameter}}
	\label{Fig:cons_path_model}
\end{figure}

The next difficulty one needs to address stems from the diameter constraints
formulated in \eqref{eq:formulation1}. To simplify the presentation,
let us limit our search of an optimal network to the set of all the spanning
trees. Also, let the parameter $D$ which limits the diameter of the network
be an even number. Then, it is well known \cite{gouveia2003network} that a
spanning tree has a diameter no more than an even integer ($D$) if and only
if there exists a central node $p$ such that the path from $p$ to any other
node in the graph consists of at most $D/2$ edges. If the central
node $p$ is given, then one can use the multicommodity flow formulation
\cite{magnanti1995optimal} to keep track of the number of edges present
in any path originating from node $p$. However, since $p$ is not known
a priori, a common way to address this issue is to augment the network
with a source node ($s$) and connect this source node to each of the
remaining vertices in the network with an edge
(refer to figure \ref{Fig:cons_path_model}). If one were to find a
spanning tree in this augmented network such that there is only one edge
incident on the source node and the path from the source node to any
other node in the graph consists of at most $\frac{D}{2}+1$ edges, the
diameter constraint	 for the original network will be naturally satisfied.

In order to impose the diameter constraints formulated in \eqref{eq:formulation1},
we add a source node ($s$) to the graph $(V, E)$ and add an edge
joining $s$ to each vertex in $V$, $i.e.$, $\tilde{V}=V \cup \{s\}$ and
$\tilde{E} = E \cup (s,j) ~~ \forall j \in V $. We then construct a tree spanning all the nodes
in $\tilde{V}$ while restricting the length of the path from $s$ to any other node in $\tilde{V}$. The
additional edges emanating from the source node are used only to formulate
the diameter constraints, and they do not play any role in determining the
algebraic connectivity of the original graph.

Constraints representing a spanning tree are commonly formulated in the
literature using the multicommodity flow formulation. In this formulation,
a spanning tree is viewed as a network which facilitates the flow of a
unit of commodity from the source node to each of the remaining
vertices in $\tilde{V}$. A commodity can flow directly between two
nodes if there is an edge connecting the two nodes in the network.
Similarly, a commodity can flow from the source node to a vertex $v$
if there is a path joining the source node to vertex $v$ in the network.
One can guarantee that all the vertices in $V$ are connected to the
source node by constructing a network that allows for a distinct unit
of commodity to be shipped from the source node to each vertex in $V$.
To formalize this further, let a distinct unit of commodity (also
referred to as the $k^{th}$ commodity) corresponding to the $k^{th}$ 
vertex be shipped from the source node. Let $f_{ij}^k$ be the $k^{th}$
commodity flowing from node $i$ to node $j$. Then, the constraints
which express the flow of the commodities from the source node to
the vertices can be formulated as follows:
\begin{subequations}
\label{eq:MCF_cons}
	\begin{align}
	\label{eq:F1_flow_2}
	\sum_{j \in \tilde{V} \setminus \{s\}} (f^k_{ij}-f^k_{ji}) = 1 ~~ \forall k \in V ~ \text{and} ~ i = s, \\
       	\label{eq:F1_flow_3}
        \sum_{j \in \tilde{V}} (f^k_{ij}-f^k_{ji}) = 0 ~~ \forall i,k \in V ~ \text{and} ~ i \neq  k, \\
        	\label{eq:F1_flow_4}
        \sum_{j \in \tilde{V}} (f^k_{ij}-f^k_{ji}) = -1 ~~ \forall i,k \in V ~ \text{and} ~ i =  k, \\
        	\label{eq:F1_flow_5}
         f^k_{ij} + f^k_{ji} \leq x_{e} ~~ \forall  ~ e:=(i,j) \in \tilde{E}, \forall k \in V, \\
        \label{eq:F1_flow_6}
        \sum_{e \in \tilde{E}} x_{e} = |\tilde{V}| -1, \\
         0 \leq ~  f^k_{ij} ~ \leq 1 ~~ \forall i,j \in \tilde{V},\forall k \in V, \\
            x_{e} \in \{0,1\} ~~ \forall e \in \tilde{E}. 
	\end{align}
\end{subequations}

Constraints (\ref{eq:F1_flow_2}) through (\ref{eq:F1_flow_4}) state that
each commodity must originate at the source node and terminate at its
corresponding vertex. Equation (\ref{eq:F1_flow_5}) states that
the flow of commodities between two vertices is possible only if there is an edge joining the two vertices.
Constraint (\ref{eq:F1_flow_6}) ensures that the number of edges in the chosen network corresponds to that of a spanning tree.
An advantage of using this formulation is that one now has access
directly to the number of edges on the path joining the source node
to any vertex in the graph. That is, $\sum_{(i,j)\in \tilde{E}} f^k_{ij}$
denotes the length of the path from $s$ to $k$. Therefore, the diameter
constraints now can be expressed as:

\begin{subequations}
\label{eq:path_cons}
	\begin{align}
         \sum_{(i,j)\in \tilde{E}} f^k_{ij} \leq (D/2+1) ~~ \forall  k \in V,  \\
         \sum_{j\in V} x_{sj} = 1.
	\end{align}
\end{subequations}

To summarize, the MISDP for the network synthesis problem with diameter
constraints is:

\begin{equation}
		\label{eq:formulation3}
		\begin{array}{ll}
	  		\gamma^* = &\max  \gamma, \\
			\text{s.t.} & \sum_{e \in E} x_e L_e \succeq \gamma (I_n - e_0 \otimes e_0),\\
			& \sum_{e \in E} x_e \leq q, \\
			& \mathrm{Constraints \ in \ \eqref{eq:MCF_cons} \ and \ \eqref{eq:path_cons}},\\
			& x_e \in \{0, 1\}^{|E|}.
		\end{array}
\end{equation}

Note that the formulation in \eqref{eq:formulation3} is for the case when the desired network is a
spanning tree and the bound on the diameter of the spanning tree is an even
number. Using the results in \cite{gouveia2003network}, similar formulations
can also be stated for more general networks with no restrictions on the
parity of the bound. In this section, we will concentrate on the formulation
presented in \eqref{eq:formulation3}.

\subsection{Algorithm for determining maximum algebraic connectivity with diameter constraint}
\label{Subsec:dia_cons_algo}
In order to pose the problem as a BSDP, let the specified 
level of connectivity be $\hat{\gamma}$. The decision problem can be 
mathematically formulated as follows: Is there an incident vector $x$ such that

\begin{equation}
		\begin{array}{ll}
				  & \sum_{e \in E} x_e L_e \succeq \hat{\gamma} (I_n - e_0 \otimes e_0),\\
			& \sum_{e \in E} x_e \leq q, \\
			& \mathrm{Constraints \ in \ \eqref{eq:MCF_cons} \ and \ \eqref{eq:path_cons}},\\
			& x_e \in \{0, 1\}^{|E|} \ ?
		\end{array}
\end{equation}

The above problem can be posed as a BSDP by marking any vertex in $V$ as a root vertex $r$
and then choosing to find a feasible tree
that minimizes the degree of this root vertex \footnote{There are several ways to
formulate the decision problem as a BSDP. For example, one can also aim to minimize
the total weight of the augmented graph defined as $\sum_e w_ex_e$. We chose to
minimize the degree of a node as it gave reasonably good computational results.}.
In this formulation, the only decision variables would be the binary variables
denoted by $x_e$ and the flow variables denoted by $f^k_{ij}$. Therefore, the BSDP we have is the following:

\begin{equation}
	\label{eq:F2_degree}
		\begin{array}{ll}
	  		& \min \sum_{e \in \delta(r)} x_e, \\
			\text{s.t.} & \sum_{e \in E} x_e L_e \succeq \hat{\gamma} (I_n - e_0 \otimes e_0),\\
			& \sum_{e \in E} x_e \leq q, \\
			& \mathrm{Constraints \ in \ \eqref{eq:MCF_cons} \ and \ \eqref{eq:path_cons}},\\
			& x_e \in \{0, 1\}^{|E|}.
		\end{array}
\end{equation}

As expected, the cutting plane algorithm for the above BSDP in conjunction with bisection techniques to solve
the original MISDP \eqref{eq:formulation1} to optimality is in very similar lines as discussed in 
Algorithm \ref{Algo:minimize_degree}. Hence, we present just the pseudo code of the procedure in Algorithm  
\ref{Algo:minimize_degree_with_dia} without discussing the details.

\begin{algorithm}[h!]
\caption{\textbf{:Algorithm for determining maximum algebraic connectivity with diameter constraint}}
Let $\mathfrak{F}$ denote a set of cuts which must be
satisfied by any feasible solution
\label{Algo:minimize_degree_with_dia}
\begin{algorithmic}[1]
		\STATE Input: Graph $G=(V,E,w_e)$, $e\in E$, a root vertex $r$, diameter $D$ and a finite number of Fiedler vectors, $v_i, i=1 \ldots M$
        \STATE Choose a maximum spanning tree as an initial feasible solution, $x^*$
		\STATE $\hat{\gamma} \gets \lambda_2(L(x^*))$
		\LOOP
		\STATE $\mathfrak{F}$$ \gets \emptyset$
		\STATE Solve:
		  {\begin{equation}
				  \begin{array}{ll}
						\min & \sum_{e \in \delta(r)} x_e, \\
						\text{s.t.} & \sum_{e \in E} x_e ({v_i} \cdot L_e {v_i}) \geq \hat{\gamma} \quad \forall i=1,..,M,  \\
					  & \sum_{e \in E} x_e \leq q, \\
					  & x_e \in \{0, 1\}^{|E|}, \\
					& \mathrm{Constraints \ in \ \eqref{eq:MCF_cons} \ and \ \eqref{eq:path_cons}},\\
        			&x_e \ \textrm{satisfies the constraints in }\mathfrak{F}.
				  \end{array}
		  \end{equation}}

			\IF{ the above ILP is infeasible}
					\STATE \textbf{break loop} \COMMENT{$x^*$ is the optimal solution with maximum algebraic connectivity}
			\ELSE   \STATE Let $x^*$ be an optimal solution to the above ILP. Let $\gamma^*$ and $v^*$ be the algebraic connectivity 
			and the Fiedler vector corresponding to $x^*$ respectively.
			\IF{$\sum_{e \in E} x_e^* L_e \nsucceq \gamma^* (I_n - e_0 \otimes e_0)$}
							\STATE Augment $\mathfrak{F}$ with a constraint $\sum_{e \in E} x_e ({v^*} \cdot L_e {v^*}) \geq {\gamma^*} $.
							\STATE Go to step 6.
					\ENDIF
			\ENDIF
			\STATE $\hat{\gamma} \gets \hat{\gamma} + \epsilon$ \COMMENT{let $\epsilon$ be a small number}
			\ENDLOOP
\end{algorithmic}
\end{algorithm}

\begin{table}[htbp]
\caption{Comparison of computational time (CPU time) of the proposed algorithm for different
limits on the diameter of the graph and $\gamma^*$ is the optimal algebraic connectivity.
The algorithm was implemented in CPLEX for instances involving  6 nodes.}
\begin{center}
\begin{tabular}{ccccc}
\toprule
\multicolumn{ 1}{c}{\textbf{Instance}} & \multicolumn{ 2}{c}{\textbf{diameter $\leq$ 4}} & \multicolumn{ 2}{c}{\textbf{no diameter constraint}} \\
\cmidrule(l{0.25em}r{0.25em}){2-3}  \cmidrule(l{0.25em}r{0.25em}){4-5}
\multicolumn{ 1}{c}{} & \multicolumn{1}{c}{$\gamma^*$} & $T_1$ & \multicolumn{1}{c}{$\gamma^*$} & $T_2$ \\
\multicolumn{ 1}{c}{} & \multicolumn{1}{c}{} & (sec) & \multicolumn{1}{c}{} & (sec) \\
\cmidrule(r){1-5}
    1   & 39.352 & 7 & 559.539 & 8 \\
    2   & 39.920 & 4 & 546.915 & 8 \\
    3   & 67.270 & 6 & 765.744 & 6 \\
    4   & 50.262 & 10 & 713.925 & 5 \\
    5   & 31.218 & 8 & 569.959 & 4 \\
    6   & 52.344 & 8 & 662.326 & 7 \\
    7   & 35.513 & 7 & 637.331 & 6 \\
    8   & 38.677 & 7 & 704.89 & 6 \\
    9   & 46.427 & 11 & 574.132 & 5 \\
   10   & 40.945 & 7 & 597.241 & 5 \\
   11   & 36.770 & 10 & 586.950 & 9 \\
   12   & 42.885 & 6 & 587.027 & 5 \\
    13   & 30.880 & 8 & 569.482 & 10 \\
    14   & 47.583 & 3 & 543.145 & 6 \\
    15   & 37.277 & 4 & 517.401 & 9 \\
    16   & 37.439 & 11 & 704.228 & 8 \\
    17   & 51.434 & 10 & 639.456 & 3 \\
    18   & 42.476 & 3 & 620.974 & 10 \\
    19   & 29.934 & 3 & 576.275 & 4 \\
    20   & 46.980 & 6 & 536.366 & 6 \\
    21   & 25.955 & 6 & 630.748 & 9 \\
    22   & 49.220 & 6 & 601.309 & 4 \\
   23   & 53.282 & 6 & 607.615 & 6 \\
    24   & 45.909 & 5 & 524.214 & 6 \\
    25   & 48.120 & 3 & 549.210 & 3 \\
    \bottomrule
\end{tabular}
\end{center}
\label{Table:Yalmip_Cplex_n_6}
\end{table}

\begin{table}[htbp]
\caption{Comparison of computational time (CPU time) of the proposed algorithm for different
limits on the diameter of the graph and $\gamma^*$ is the optimal algebraic connectivity.
The algorithm was implemented in CPLEX for instances involving  8 nodes.}
\begin{center}
\begin{tabular}{ccccccc}
\toprule
\multicolumn{ 1}{c}{\textbf{Instance}} & \multicolumn{ 2}{c}{\textbf{diameter $\leq$ 4}} & \multicolumn{ 2}{c}{\textbf{diameter $\leq$ 6}} & \multicolumn{ 2}{c}{\textbf{no diameter constraint}} \\
\cmidrule(l{0.25em}r{0.25em}){2-3} \cmidrule(l{0.25em}r{0.25em}){4-5} \cmidrule(l{0.25em}r{0.25em}){6-7}
\multicolumn{ 1}{c}{} & \multicolumn{1}{c}{$\gamma^*$} & $T_1$ & \multicolumn{1}{c}{$\gamma^*$} & $T_2$ & \multicolumn{1}{c}{$\gamma^*$} & $T_3$ \\
\multicolumn{ 1}{c}{} & \multicolumn{1}{c}{} & (sec) & \multicolumn{1}{c}{} & (sec) & \multicolumn{1}{c}{} & (sec) \\
\cmidrule(r){1-7}
    1   &  66.1636  &  298.10  &  93.0846  &  184.26  &  631.739  & 495.23 \\
    2   &  39.2994  &  477.34  &  54.3061 &  416.43  &  631.883  & 980.98 \\
    3   & 44.8588   &  803.45  &  45.9793  &  634.54  &  604.213 & 4253.01 \\
    4   &  66.5337   &   394.02  &  78.7357 &  221.21  &  757.490  & 815.01 \\
    5   &  33.8383  &    519.28  &  53.8226  &  480.23  &  755.205 & 706.25 \\
    6   &  46.6083  &  1033.09  &  75.6113  &  349.12  &  513.994  & 586.34 \\
    7   &  51.1379  &  781.07   &  63.3915  &  385.17  &  550.717  & 949.30 \\
    8   & 42.8026   &  931.50   &  77.4458  &  319.51  &  807.108  & 333.93 \\
    9   &  58.1182  &   489.43   &  84.7166  &  348.82  &  769.641  & 482.55 \\
   10   &  50.5110  &   492.11   &  54.3155  &  323.33  &  646.711 & 1789.64 \\
   11   &  43.6888  &   791.01   &  107.1820  &  212.34  &  729.171  & 472.71 \\
   12   &  47.5213  &  693.13  &  82.2919  &  219.20  &  655.867  & 1061.16 \\
    13   &  42.4918  &  468.44  &  53.2514  &  698.21  &  698.129  & 1421.38 \\
    14   &  41.1752  &  445.26  &  48.9485  &  261.18  &  523.118  & 977.67 \\
    15   &  44.8202   &  518.13  &  63.8735  &  509.77  &  639.540  & 504.42 \\
    16   &  40.1853   &  540.19  &  72.1540  &  396.25  &  690.719  & 661.91 \\
    17   &  66.6196  &   480.70  &  108.0970  &  254.47  &  735.361  & 476.87 \\
    18   &  62.9801  &   499.78  &  69.1063  &  233.33  &  622.840  & 1372.58 \\
    19   &  40.7602  &  542.69  &  54.9466  &  343.04  & 650.096  & 236.65 \\
    20   &  60.1121  &  607.19   &  81.2138  &  209.15  &  607.008  & 590.38 \\
    21   &  66.3578  &  588.31   &  80.3600  &  408.78  &  609.370  & 730.82 \\
    22   &  42.8765  &   776.38  &  75.5561  &  458.80  &  666.251  & 734.43 \\
   23   &  42.7949  &   400.03   &  62.8144  &  638.11  &  444.903  & 942.26 \\
    24   &  63.1568  &   590.91   &  73.7841  &  333.03  &  680.411  & 804.27 \\
    25   &  31.3830  &   232.18   &  44.6972  &  231.16  &  630.107  & 818.93 \\
    \bottomrule
\end{tabular}
\end{center}
\label{Table:Yalmip_Cplex_n_8}
\end{table}

\subsection{Performance of proposed algorithm}
\label{Sec:ch3_results1}
All the computations in this section were performed with the same computer specifics as mentioned in 
section \ref{Sec:results1}.

As discussed in earlier sections, the semi-definite programming toolboxes in Matlab
could not be used to solve the proposed formulation with the semi-definite and diameter constraints
even for instances with 6 nodes primarily due to the inefficient memory management.
However, due to the combinatorial explosion resulting from
the increased size of the problem, the proposed algorithm with CPLEX solver could
provide optimal solutions in a reasonable amount of run time for instances upto 8 nodes.

We shall now compare the computational times of the proposed algorithm  to obtain optimal solutions
for different values of the bound on the diameter. The results shown in Tables
(\ref{Table:Yalmip_Cplex_n_6}) and (\ref{Table:Yalmip_Cplex_n_8})  are for $25$ random instances generated for
networks with 6 and 8 nodes, respectively. Based on the results in Table (\ref{Table:Yalmip_Cplex_n_6}),
we observed that the average run time for obtaining optimal solution for the 6 nodes problem with diameter constraint
was (average $T_1$) $6.6s$ and without diameter constraint was (average $T_2$) $6.3s$.
Based on the results in Table (\ref{Table:Yalmip_Cplex_n_8}),
we observed that the average run time for the problem without diameter constraints
(average $T_3= 927.95s$) was 1.61 times greater than the average run time for the problem
with diameter $\leq 4$ (average $T_1= 575.75s$) and 2.56 times greater than the average run time for the problem
with diameter $\leq 6$ (average $T_2= 362.77s$). Optimal networks with various diameters
corresponding to instances $1$ and $2$ of Table (\ref{Table:Yalmip_Cplex_n_8}) with 8 nodes may be found in
Figure \ref{Fig:optimal_tree_n8_1}.

\begin{figure}[htp]
	\centering
	\subfigure[Instance 1, diameter $\leq$ 4]{
	\includegraphics[scale=0.25]{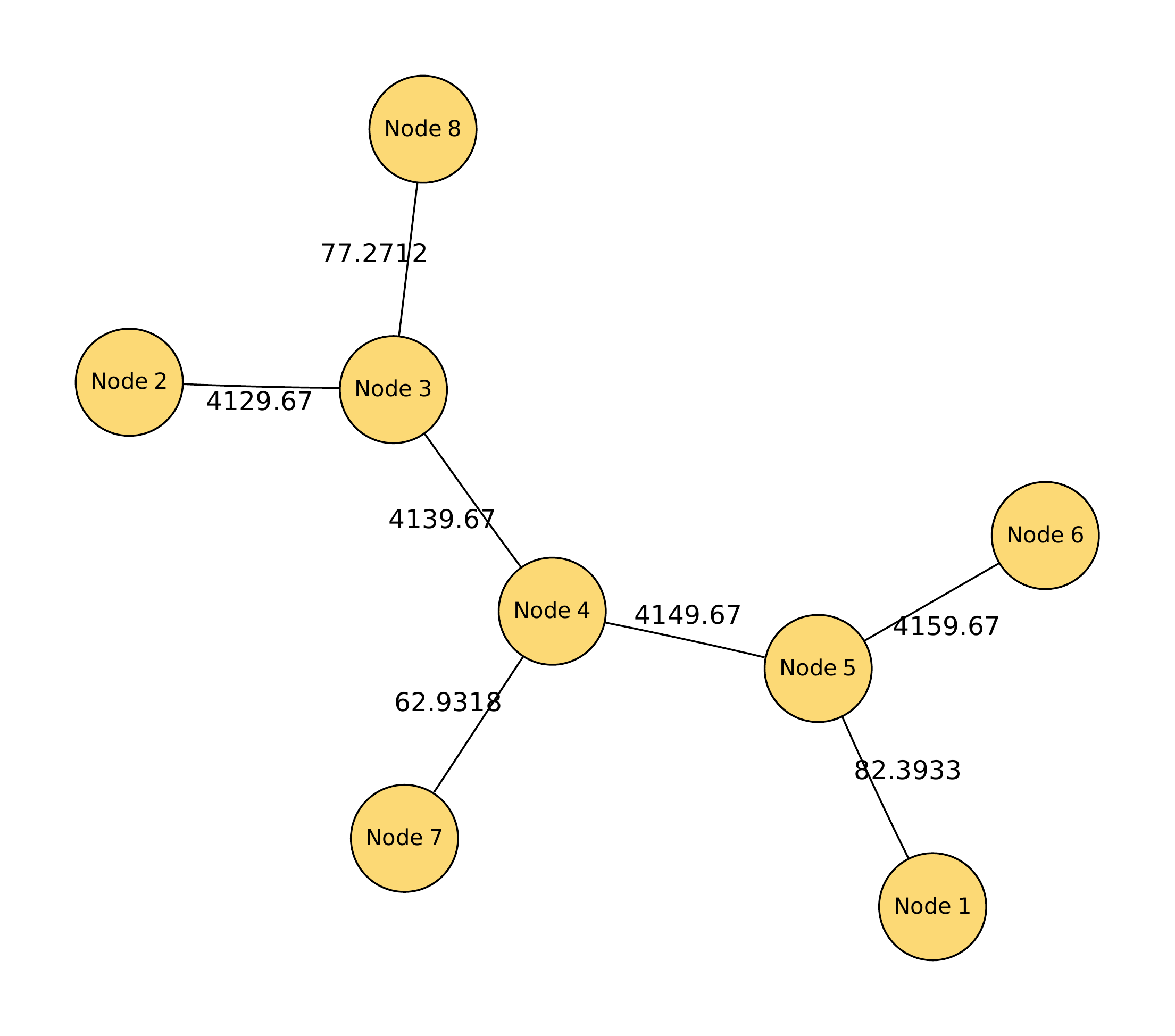}}
	\subfigure[Instance 2, diameter $\leq$ 4]{
	\includegraphics[scale=0.25]{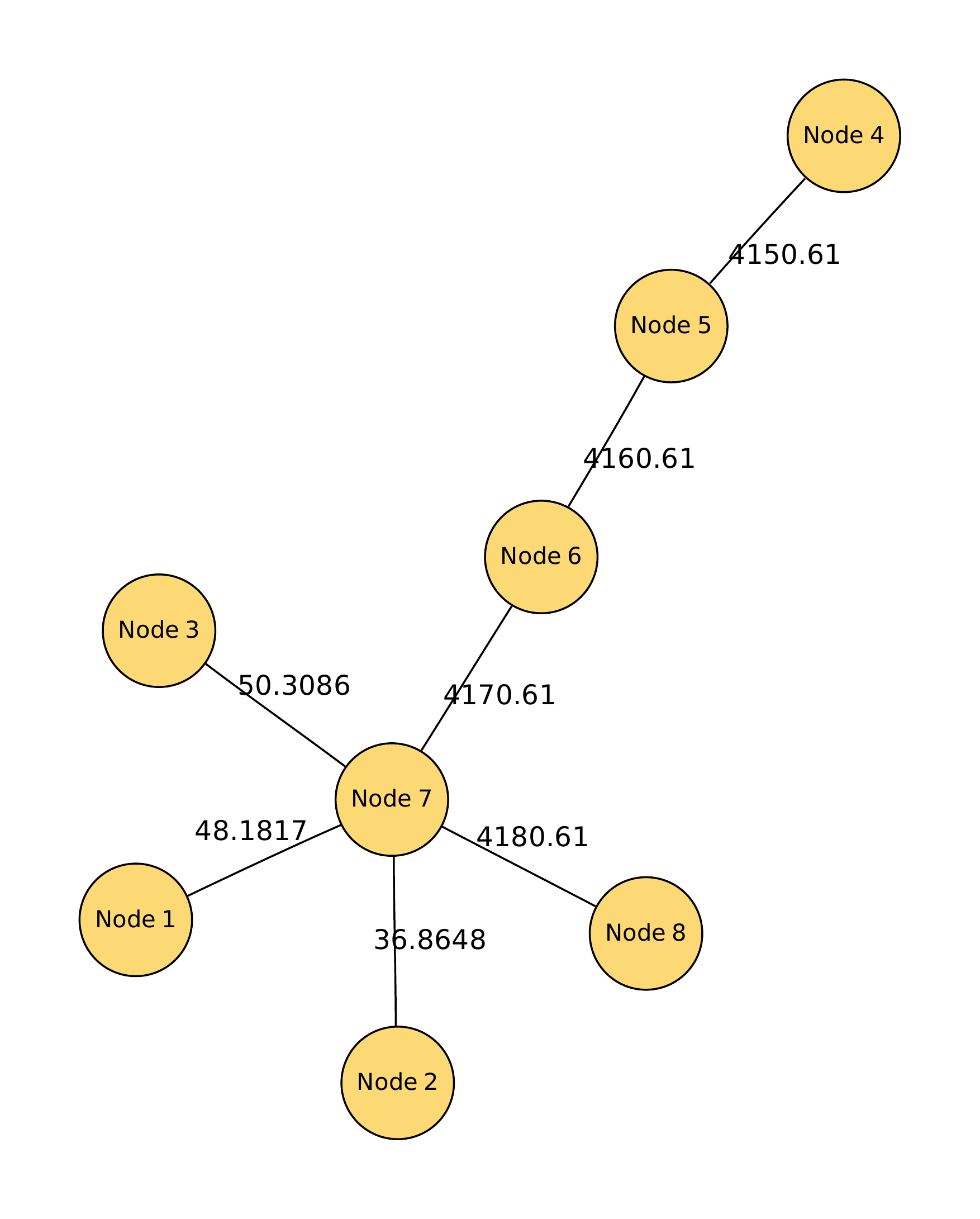}}
	\subfigure[Instance 1, diameter $\leq$ 6]{
	\includegraphics[scale=0.25]{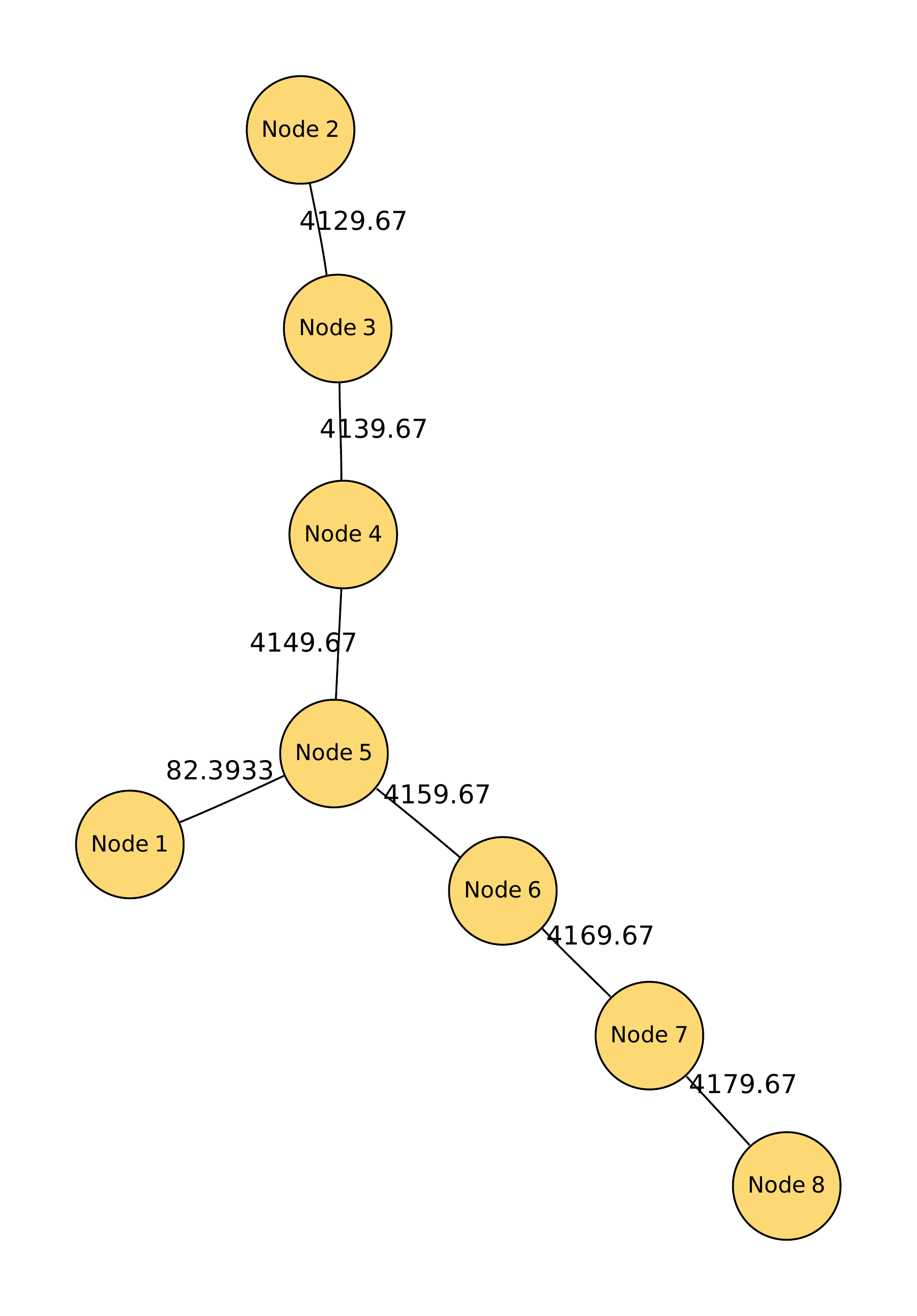}}
	\subfigure[Instance 2, diameter $\leq$ 6]{
	\includegraphics[scale=0.31]{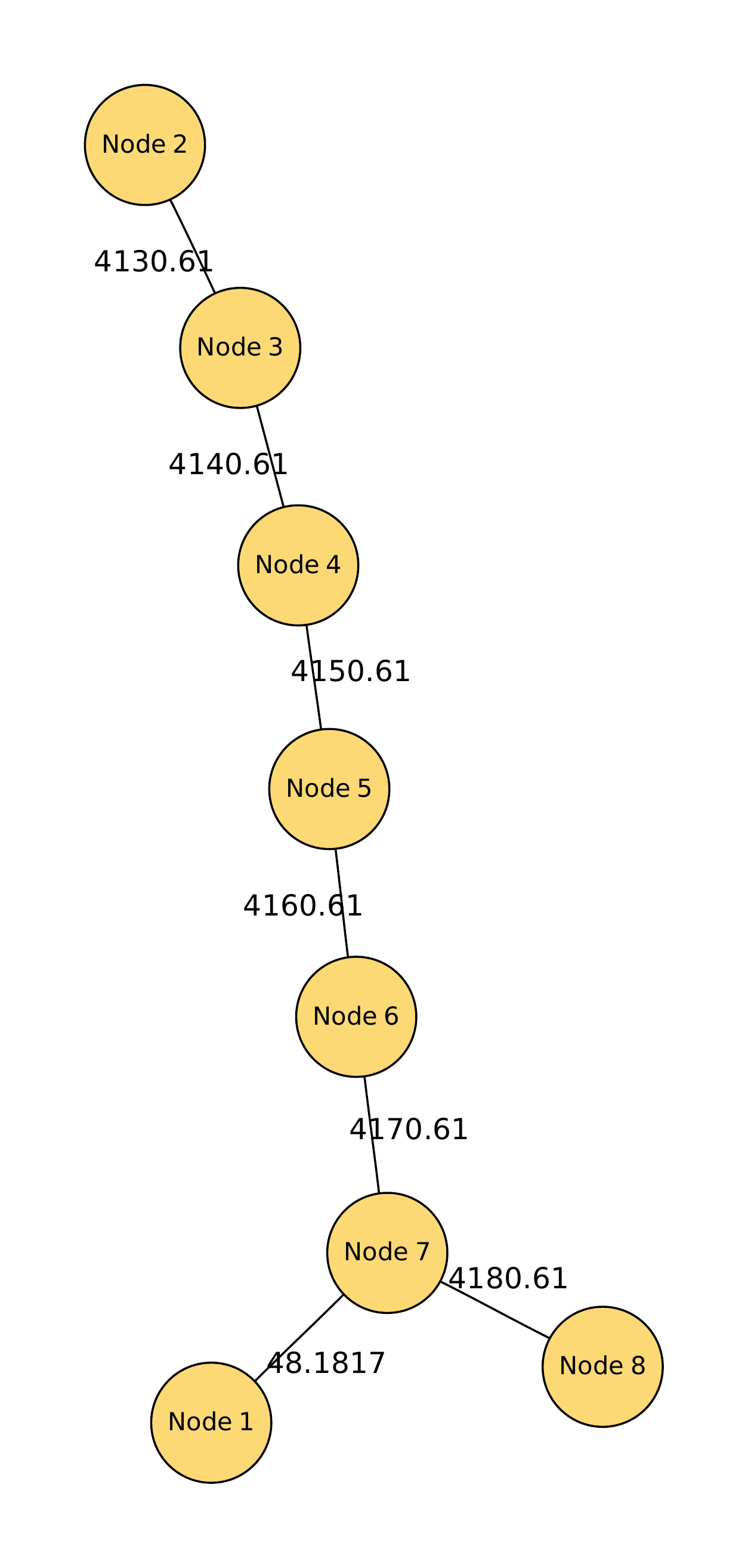}}
	\subfigure[Instance 1, no diameter constraint]{
	\includegraphics[scale=0.28]{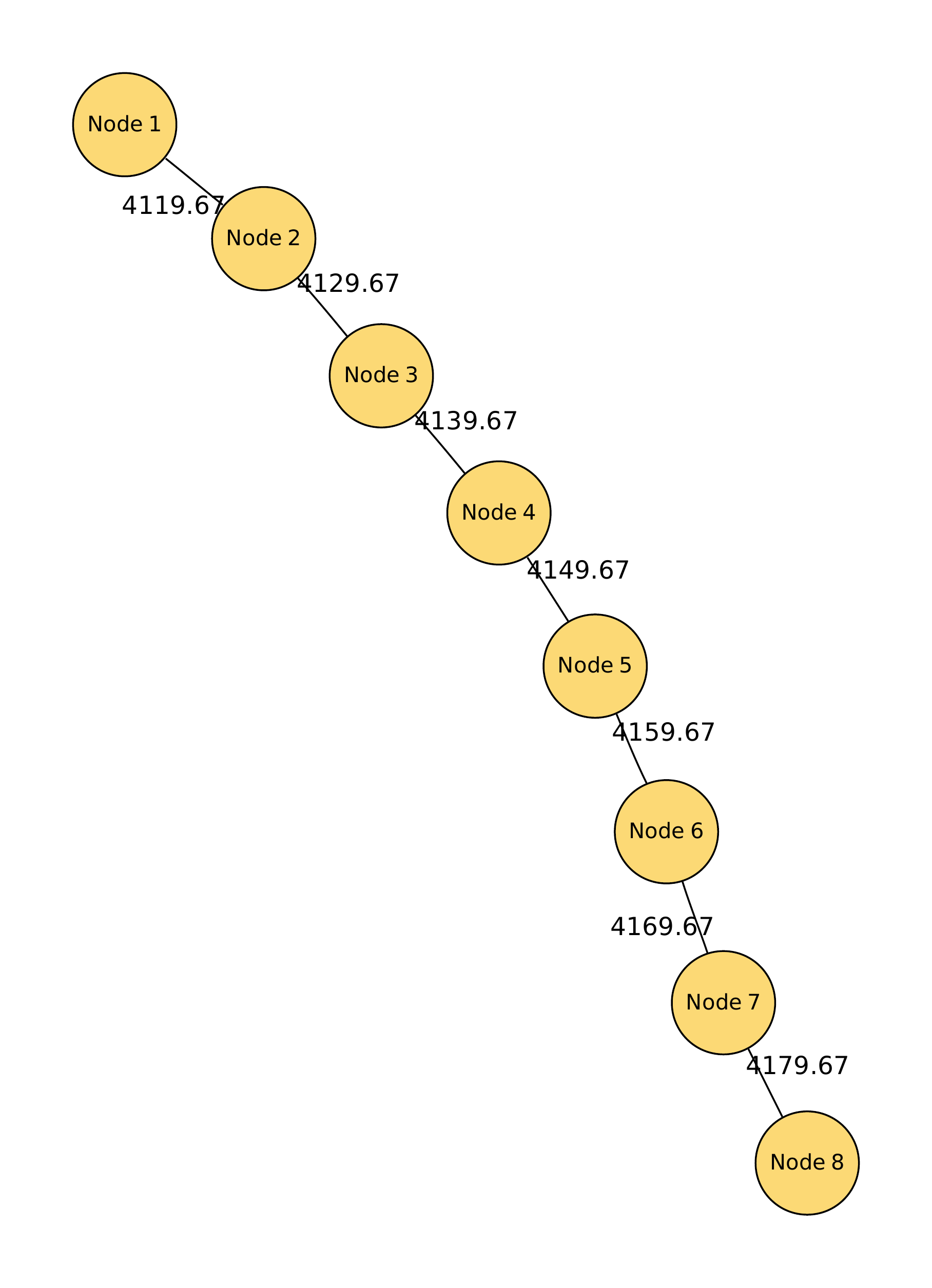}}
	\subfigure[Instance 2, no diameter constraint]{
	\includegraphics[scale=0.28]{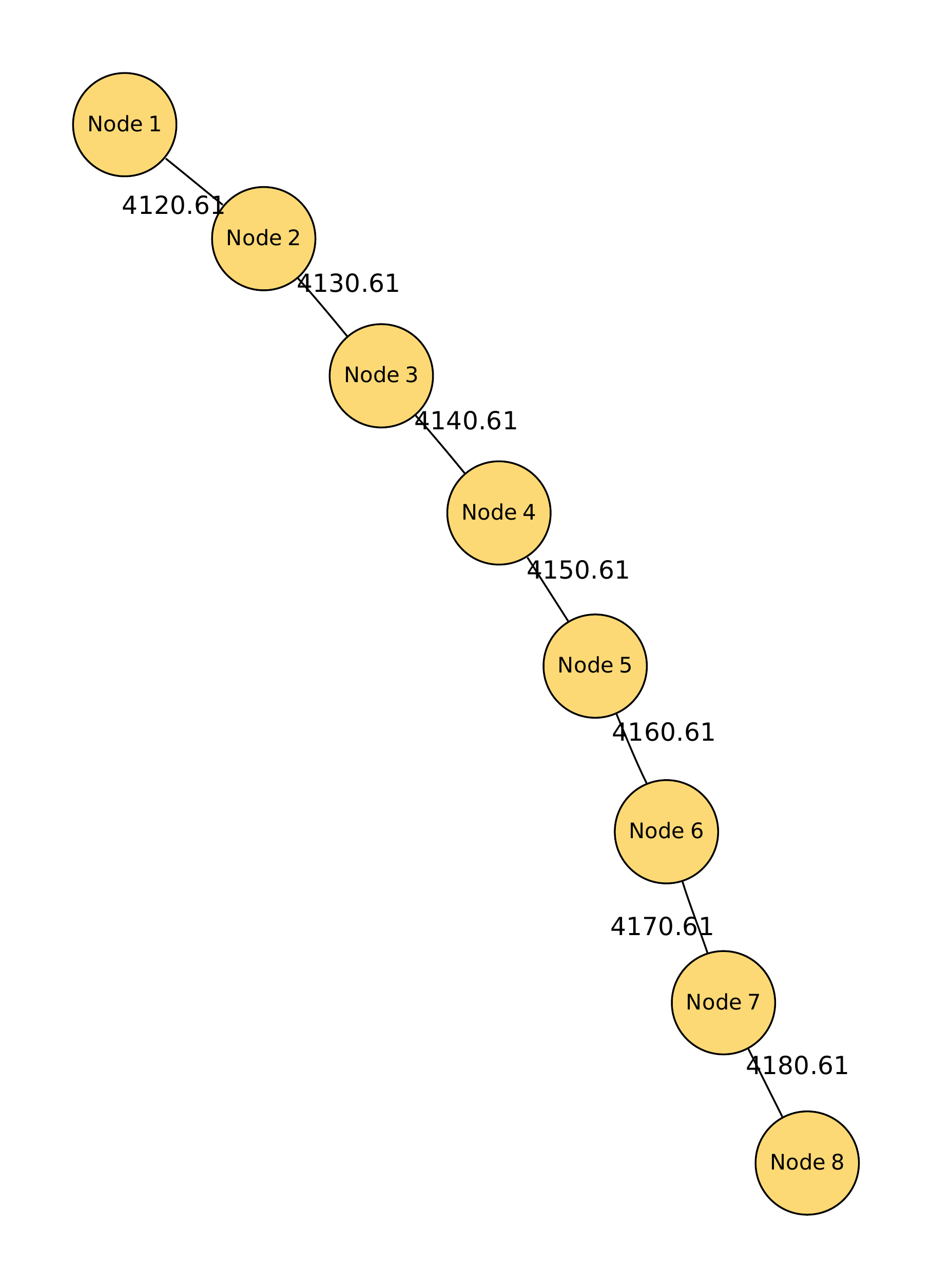}}
	\caption{ (a),(c),(e) correspond to optimal networks with maximum
	algebraic connectivity subject to various diameter constraints for
	instance 1 (from Table (\ref{Table:Yalmip_Cplex_n_8})).
	Similar plots for instance $2$ are also shown in (b),(d),(e).}
	\label{Fig:optimal_tree_n8_1}
\end{figure}

\section{Maximization of algebraic connectivity with power consumption constraint}
\label{Sec:power_cons}
In this section, we mathematically formulate the total power consumed by the UAV 
network as a function of the eigenvalues of the Laplacian and later pose the problem of 
synthesizing a robust backbone UAV network subject to power consumption
constraint as a MISDP problem. Essentially, the problem is to determine the backbone UAV network with maximum algebraic 
connectivity subject to the constraints on the total power consumed by the network 
and the number of communication links.

\subsection{Related literature}
The idea of using UAVs to communicate data has been proposed in the literature; for example,
the use of UAVs as relays in disaster areas has been proposed in
\cite{connectivityUAV07}, \cite{FrewJINT08Comm}, \cite{MANET2009}, \cite{relay09}, \cite{relay10} to
facilitate a mobile communication network connecting the emergency responders,
control towers and different agencies, thereby enabling a timely exchange of
information between the relevant entities. A similar architecture was 
envisioned in \cite{RadioNavigation2011} for GPS denied navigation of UAVs. Employing some UAVs
primarily as data transmitters has several advantages. For instance, each
vehicle may not have the high-power transmitter and antennas to
communicate to the ground station, and even if it does, such direct
links are not suitable for environments with obstructions or
non-line-of-sight communications\cite{FrewJINT08Comm}. In addition,
the UAVs may be operating in dynamic environments where regular cellular
towers are either damaged or non-existent. For these reasons,
researchers have proposed meshing architectures
\cite{FrewJINT08Comm},\cite{heirarchical2000}. In this architecture, UAVs with
a higher communication capability act as mobile
base stations (also referred to as backbone nodes) and its primary job
is to connect the individual vehicles or the regular nodes with limited
communication capability to the control stations. A typical example
of such a network is shown in Figure \ref{Fig:uav_backbone}.
Various objectives have been considered in the recent work on optimization of networks
with backbone nodes; for example,  in  \cite{srinivas2009construction}, the objective is
to minimize the number of mobile backbone nodes so that
all the regular nodes are connected; in \cite{craparo2011throughput}, the objective is to
optimally chose the location of mobile backbone nodes so as to maximize the number of regular
nodes achieving a minimum throughput. However, in \cite{srinivas2009construction,craparo2011throughput},
either the backbone nodes are not allowed to communicate with each other or the connectivity among the backbone
nodes is enforced by requiring a minimum number of communication links that connect them.
\begin{figure}[!h]
	\centering
	\includegraphics[scale=1]{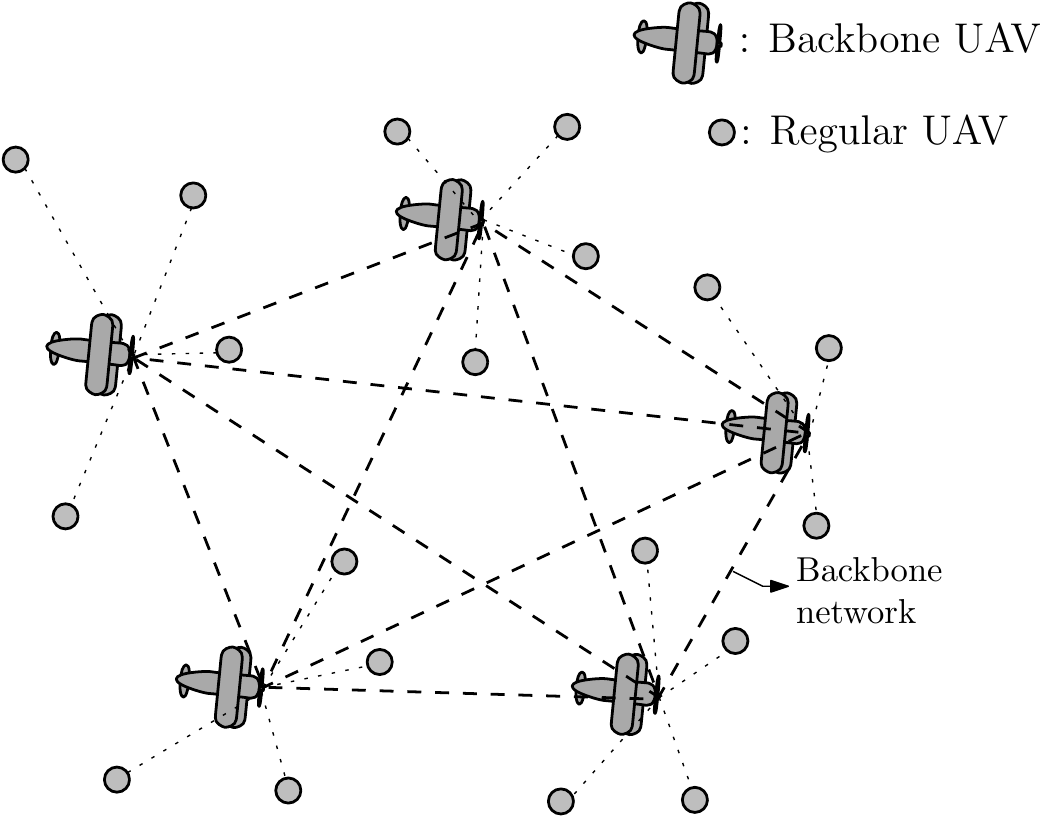}
	\caption{A typical representation of the UAV backbone network where
	backbone UAVs/nodes provide communication support to the regular nodes
	and each regular node is assigned to one backbone node as shown.}
	\label{Fig:uav_backbone}
\end{figure}

In this work, we consider another objective to optimize the network of mobile nodes; this objective
better reflects the robustness of connectivity among the backbone nodes due to random or unexpected
failure of communication networks.

The problem dealing with maximization of algebraic connectivity
subject to wiring cost constraint considered by Varshney
\cite{varshney2010distributed} is closely related to the proposed problem. Varshney proposes
the use of reverse convex program to find a relaxed solution and proposes the use
of rounding to get a feasible solution from the relaxed solution. Their work
lack the development of a systematic procedure to obtain optimal network. Also, the numerical
results presented in \cite{varshney2010distributed} are limited to instances of problems
with at most 7 nodes.  

Therefore, in this section, we propose an algorithm based on cutting plane technique 
to determine an optimal network for the problem of maximizing
algebraic connectivity subject to power consumption constraint. 
Lastly, we apply a 2-opt heuristic (developed in section \ref{Subsec:kopt}) to
find feasible solutions, and use a simple bound on the optimal algebraic connectivity
stemming from the resource constraint to estimate the quality of the feasible
solutions. 



\subsection{Mathematical formulation of the power consumption constraint}
\label{Subsec:power_cons}
In this subsection, we first formulate the power consumption constraint, 
and in the subsequent subsections, we pose the algebraic connectivity problem with all the 
resource constraints as a Mixed Integer Semi-Definite Program.

One may sometimes have the choice of positioning UAVs that serve as backbone nodes.
In this case, there is a natural problem of determining the positions of the UAVs such
that the convex hull of the projection of the locations of these UAVs on to two dimensions
is at least equal to the minimum area of coverage, $A_0$.  Since the total power consumed
by these UAVs is an important consideration, as defined in the introduction, if there 
is a communication link between the $i^{th}$ and $j^{th}$ UAV, the power consumed is 
given by $P_{ij} :=\alpha_{ij} d_{ij}^2$. The total power consumed by the collection 
of backbone UAVs is $\sum_{(i,j) \in E} P_{ij}x_{ij}$. To formulate the power 
consumption constraint, we now pose the following subproblem: Given the network 
topology of the backbone UAVs, what would be an optimal placement of the nodes 
such that
\begin{itemize}
\item[a.] the total power consumed is minimum, and
\item[b.] the projected area of the convex hull of the backbone UAVs in the ground plane in is at least $A_0$?
\end{itemize}
Clearly, if one were to solve this subproblem, one can solve the original problem of synthesizing the network of backbone UAVs by considering only those topologies that result in the total power consumption within the specified budget, $P_{max}$, and then picking one network topology among them with the maximum algebraic connectivity. Hence, we shall discuss the formulation of this subproblem in the
remainder of this section.

Suppose the location of the $i^{th}$ UAV is given by $(a_i, b_i, c_i)$, so that
for a given topology of communication (provided by the set $E_g$ of edges), the
total power consumed may be written as:

\begin{subequations}
	\label{eq:}
	\begin{align}
	\sum_{(i,j) \in E_g} P_{ij} & = \sum_{(i,j) \in E_g} \alpha_{ij} ((a_i - a_j)^2 + (b_i - b_j)^2 +(c_i-c_j)^2), \\
 &= {\mathbf a} \cdot \sum_{e=(i,j) \in E_g} \alpha_{ij} L_e \mathbf{a} + \mathbf{b} \cdot \sum_{e=(i,j) \in E_g} \alpha_{ij} L_e
\mathbf{b} + \mathbf{c} \cdot \sum_{e=(i,j) \in E_g} \alpha_{ij} L_e \mathbf{c},
	\end{align}
\end{subequations}
where $L_e$ is the local Laplacian matrix corresponding to the edge $e$ 
and $\mathbf{a}$, $\mathbf{b}$ and ${\mathbf c}$ are the vectors whose
$i^{th}$ components provide respectively  the $a$, $b$ and $c$ coordinates of the $i^{th}$ UAV.

In order to improve spatial spread, we will require that the area of the convex hull
of the projections of UAVs' locations on the horizontal plane be at least a
specified amount, say $A_0$. Without loss of generality, assuming that the origin is at the
centroid of the convex hull as shown in figure \ref{Fig:convex_hull} (for a simple case of five UAVs),
we have the following constraints,
$\mathbf{1}\cdot\mathbf{a}=0$ and $\mathbf{1}\cdot\mathbf{b}=0$ where $\mathbf{1} \in \mathbb{R}^n$ is
a vector of ones. Since we are dealing with the convex hull of the projected locations of
the UAVs, one may number the projected locations and order them appropriately,
so that the area may be triangulated with the centroid being one of the
vertices of every triangle in the triangulation and  that  the area of the
convex hull may be expressed as a bilinear function: $A(\mathbf a, \mathbf b)$ has the the property that
$A(\mathbf{a}, \mathbf{b}) = -A( \mathbf{b}, \mathbf{a})$.
Hence, for some skew-symmetric matrix, $\Omega$, one may express the area as:
$$A(\mathbf a, {\mathbf b}) = \mathbf {a} \cdot \Omega \mathbf b. $$
For the case of five nodes shown in Figure \ref{Fig:convex_hull}, the skew symmetric 
matrix is given by: 
\begin{align*}
		  	 {\Omega} &= \frac{1}{2} \begin{pmatrix}
		0 & 1 & 0 & 0 & -1 \\
		-1 & 0 & 1 & 0 & 0 \\
		0 & -1 & 0 & 1 & 0 \\
		0 & 0 & -1 & 0 & 1 \\
		1 & 0 & 0 & -1 & 0 \\
		\end{pmatrix}
\end{align*}
Since $\Omega$ is skew-symmetric, $\mathbf{b} \cdot \Omega \mathbf{b} = 0$,
the component of the vector $\mathbf{a}$ along $\mathbf{b}$ will
not contribute to the projected area. Therefore, we may require
that $\mathbf{a} \cdot \mathbf{b} = 0$ so that $\mathbf{a}$
contributes fully to the projected area.

\begin{figure}[!h]
	\centering
	\includegraphics[scale=0.9]{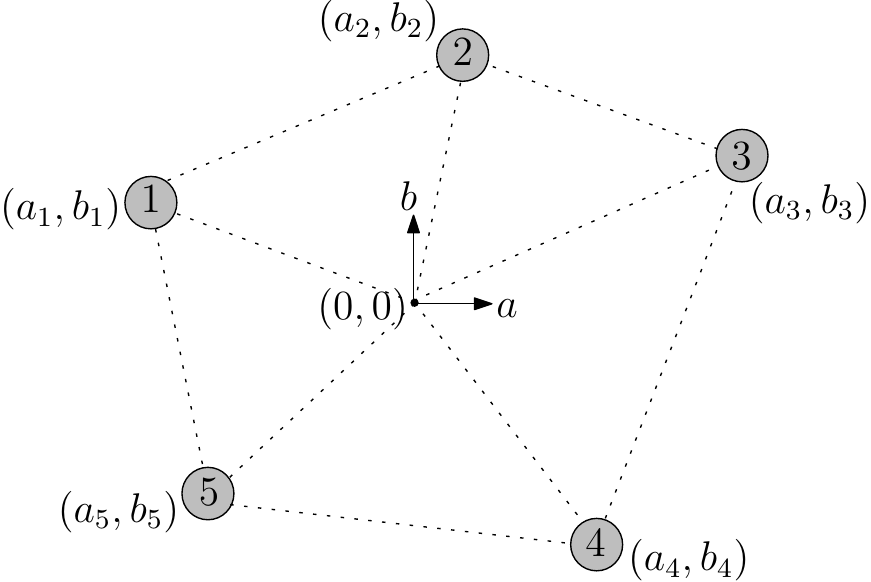}
	\caption{Convex hull of the projections of five UAVs' locations on the horizontal plane
	 with the centroid of the area at the origin.}
	\label{Fig:convex_hull}
\end{figure}

Imposing the non-linear constraint $\mathbf{a} \cdot \Omega \mathbf{b} \geq A_0$ is
hard; for this reason, we will alternatively specify the constraint on the spread of 
UAVs indirectly through requiring the variance in their coordinates to be at least $R^2$ units.
Hence, we recast the problem for locating the UAVs so that the total power consumption is a minimum as:

\begin{align*}
		  \mathrm{Minimize} \ \ &P = \mathbf{a} \cdot L \mathbf{a} + \mathbf{b} \cdot L \mathbf{b} + \mathbf{c} \cdot L \mathbf{c} \\
		  \mathrm{subject \ to} \ \
		  &\mathbf{1}\cdot\mathbf{a}=0, \ \mathbf{1}\cdot\mathbf{b}=0 \\
		  &\mathbf{a} \cdot \mathbf{b} = 0, \\
		  &\mathbf{a} \cdot \mathbf{a} \geq R^2,  \ \mathbf{b} \cdot \mathbf{b} \geq R^2.
\end{align*}

Therefore, by the variational characterization of eigenvalues, the minimum total power consumed for a given network
topology of the UAVs is given by: \[R^2(\lambda_2(L) + \lambda_3(L)),\]
where the optimal ${\mathbf c}$ is along the vector ${\mathbf 1}$
(corresponding to the zero eigenvalue of $L$), $\mathbf{a}$
and $\mathbf{b}$ along the eigenvectors corresponding to the second and third smallest
eigenvalues of $L$. In other words, the optimal location of UAVs is such that they must lie in the
same plane (i.e., with their $c$ coordinates being the same) and their $a$ and $b$ coordinates must lie
along the eigenvectors corresponding to the second and third smallest eigenvalues of the Laplacian
for the specified network topology so that the total power consumption of the communication network is a minimum.

\subsection{Problem of maximizing algebraic connectivity with power consumption constraint}
\label{Subsec:}
As we discussed in section \ref{Subsec:power_cons}, we showed that, given
a network topology for the UAVs, an optimal placement of the nodes would be
along the second and third eigenvector directions and that the minimum power
consumed would be $R^2(\lambda_2(L) + \lambda_3(L))$ where $L$ is the
Laplacian of that particular topology.
Naturally, one would also be interested in synthesizing a network topology ($x$) which
connects all the UAVs and such that a) the topology is robust/well-connected against
random failure of links and b) the total power consumed is bounded by a prescribed
upper bound ($\tilde{P}_{max}$), i.e,
\begin{align*}
	\lambda_2(L(x)) + \lambda_3(L(x)) &\leq P_{max}
\end{align*}
where $P_{max} = \tilde{P}_{max}/R^2$.

\begin{figure}[!h]
	\centering
	\subfigure[8 nodes]{
			  \includegraphics[scale=0.53]{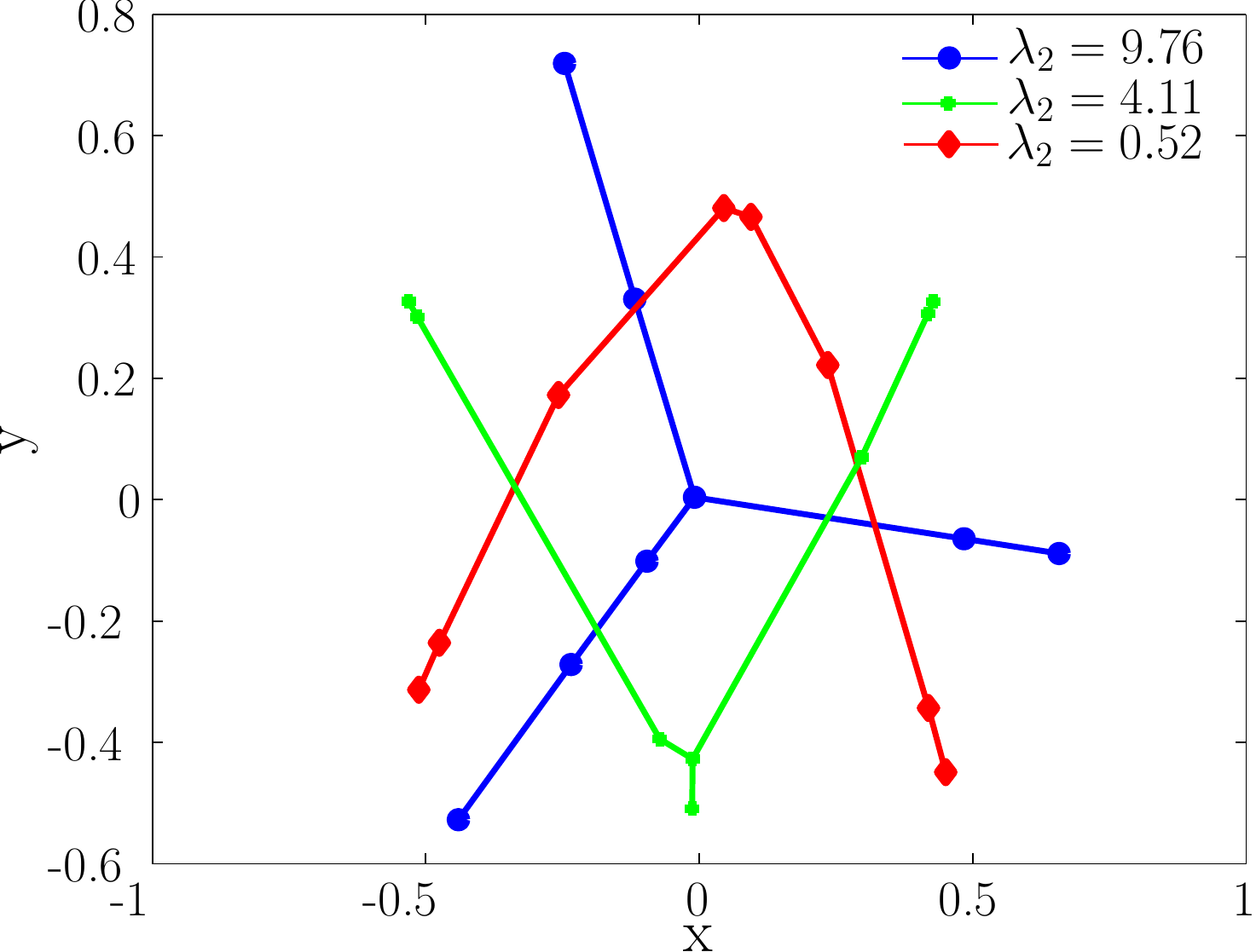}}
	\subfigure[20 nodes]{
			  \includegraphics[scale=0.53]{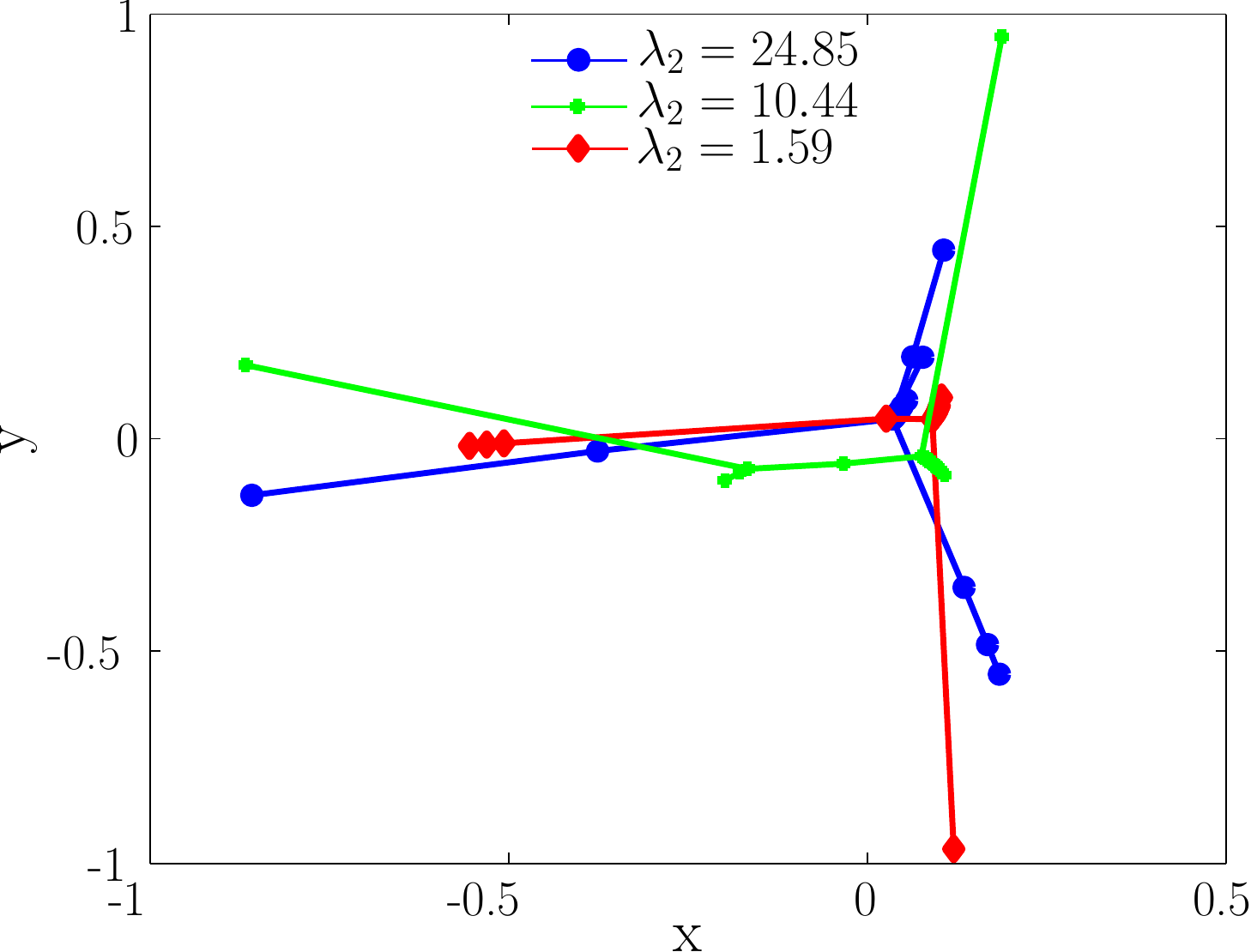} }
	\caption{This figure represents the positioning of UAVs for various objective values subject to
			  power consumption constraint. Maximizing
	$\lambda_2(L)$ indicates that the UAV locations are more uniformly distributed with well connected
	topologies.}
	\label{Fig:uav_positions}
\end{figure}

\begin{figure}[htp]
	\centering
	\subfigure[$\lambda_2 = 9.76$]{
			  \includegraphics[scale=0.15]{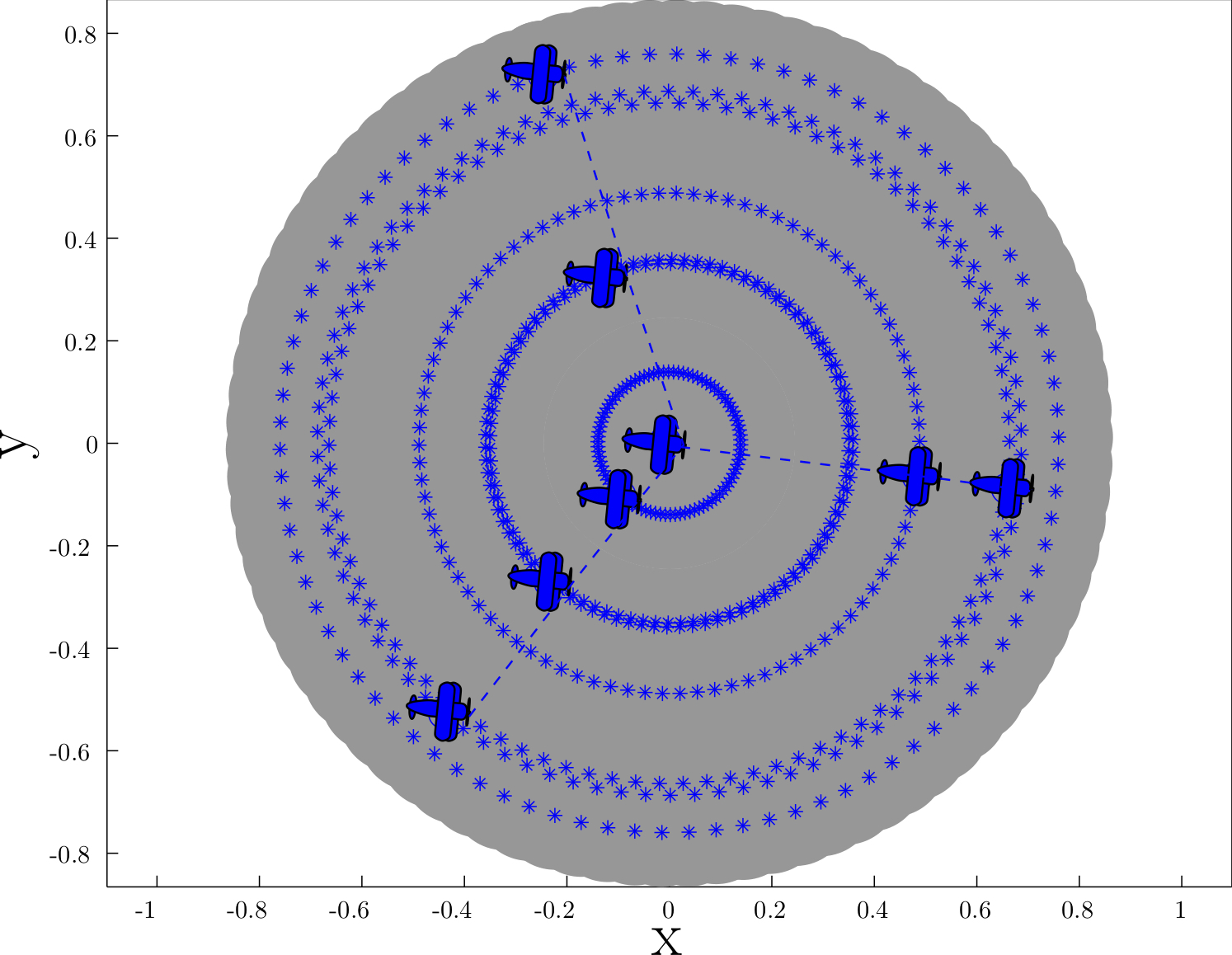}}
	\subfigure[$\lambda_2 = 4.11$]{
			  \includegraphics[scale=0.15]{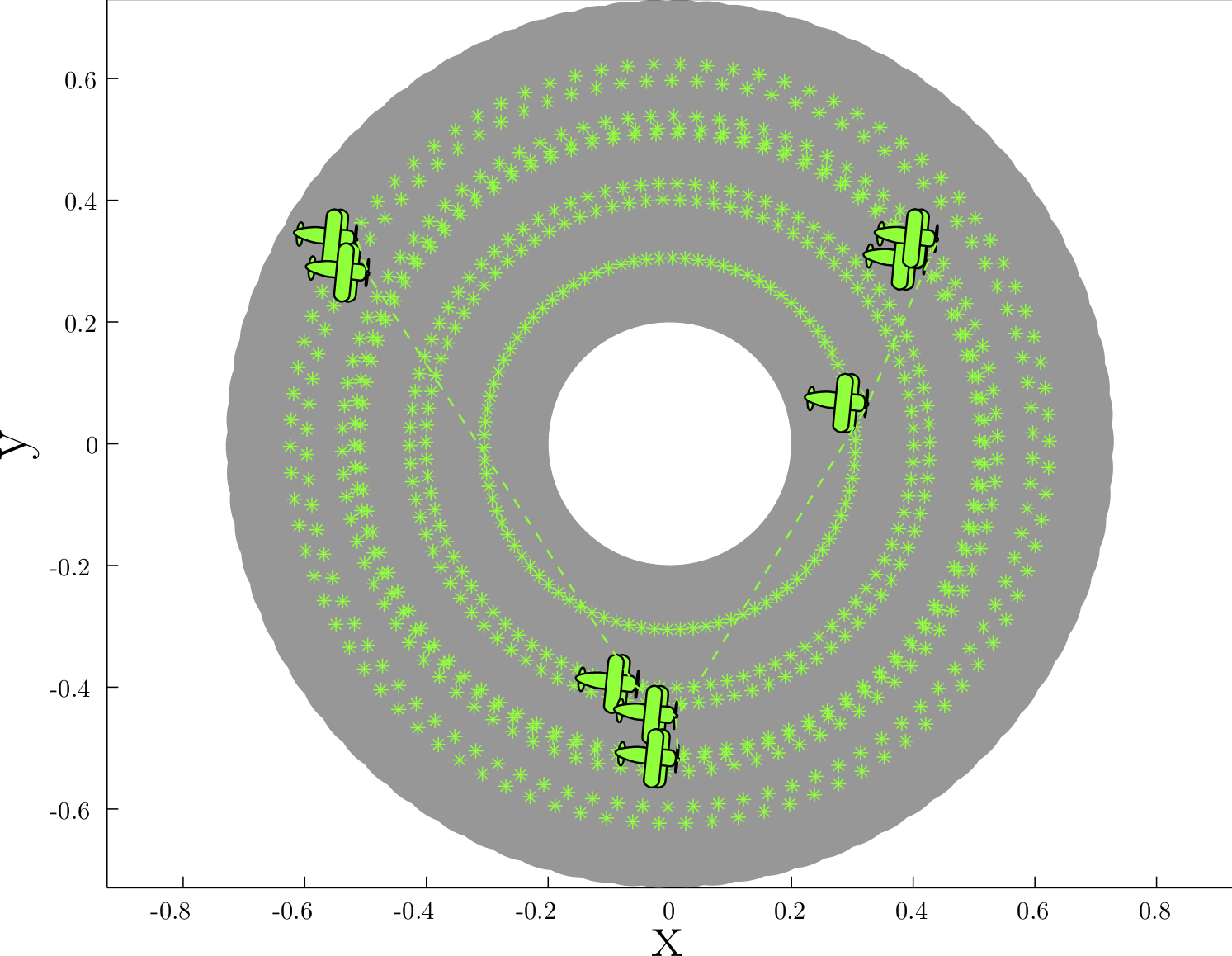} }
	\subfigure[$\lambda_2 = 0.52$]{
			  \includegraphics[scale=0.15]{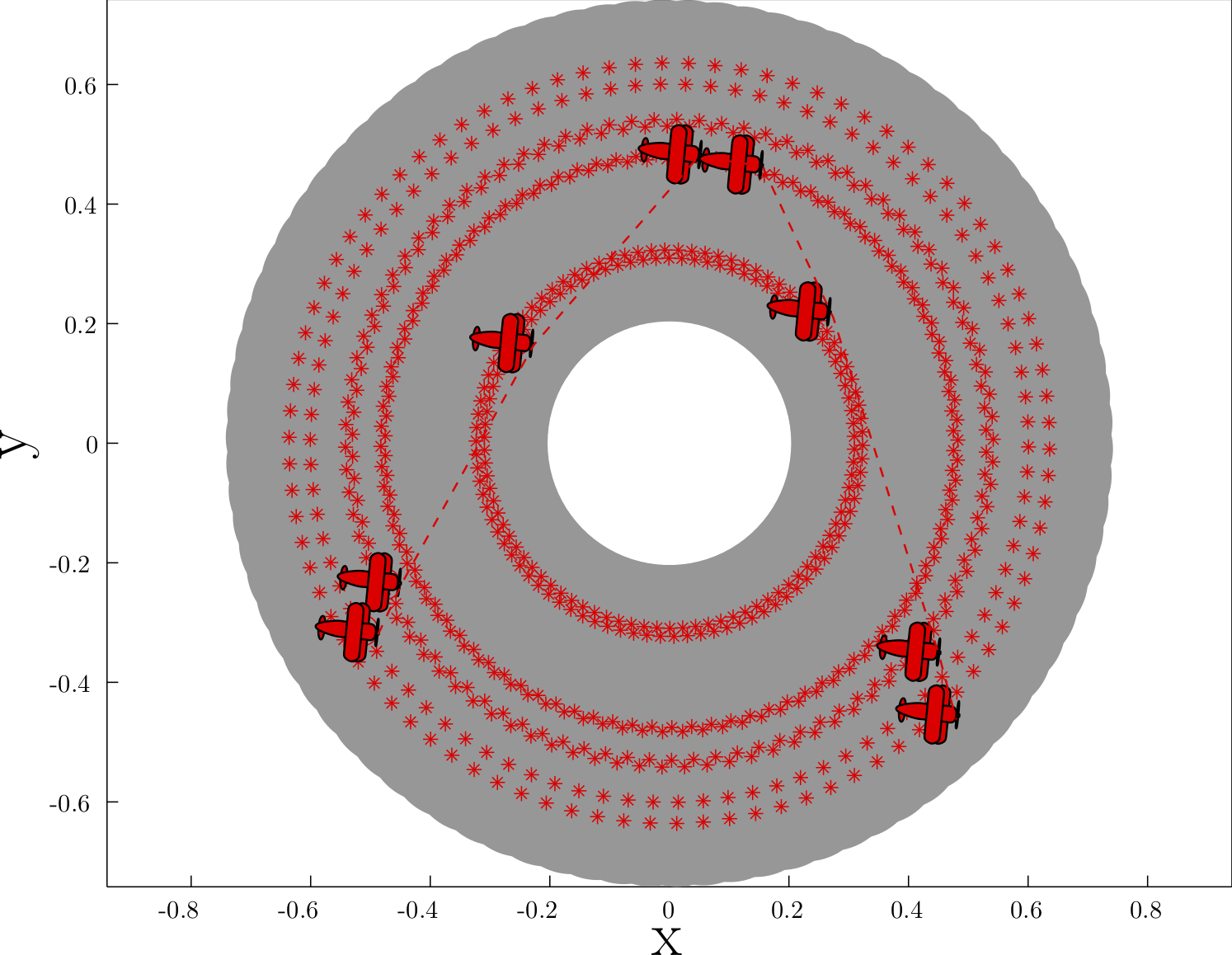} }
	\caption{This figure represents the trajectories of the UAVs when the backbone UAV network (8 nodes) is subject to
			  a rigid body rotation by 360 degrees about their respective centroids. Radius of communication
			  of 0.1 was chosen for all the UAVs. Note that, the network corresponding to largest $\lambda_2$ value
           has the maximum coverage unlike the networks with with lower $\lambda_2$. Source: \cite{nagarajan2012dscc,nagarajan2015synthesizing} }
	\label{Fig:uav_rotation}
\end{figure}

In order to synthesize a well-connected topology,
we chose to maximize the $\lambda_2$ or the algebraic connectivity
of the weighted Laplacian of the network where weights
of the edges correspond to the proportionality constant, $\alpha_{ij}$. As discussed
in the introduction, algebraic connectivity has been extensively used in the literature as a measure for
robustness of networks.  This objective of maximizing algebraic connectivity is reasonable
since we observed that the topology of networks with higher $\lambda_2$ tend to spread the UAV 
locations better. Sample feasible networks for 8 and 20 nodes with
varying values of $\lambda_2$ are shown in
figure \ref{Fig:uav_positions}. In the case of 8 nodes, networks with lower $\lambda_2$ are weakly connected
since a random removal of any node/edge can disconnect the entire network. Networks with lower $\lambda_2$
seem to have a higher diameter which can incur more delays in the communication of data among UAVs.
The situation is similar in the case of 20 nodes in figure \ref{Fig:uav_positions}.

Based on the model of the power constraint and the notation defined earlier,
the problem of choosing at most $q$ (positive integer) edges from $E$ so that
the algebraic connectivity of the augmented network is maximized
can be posed as follows:


\begin{equation}
		\label{eq:ch3_F1}
		\begin{array}{ll}
	  		&\max  \lambda_2(L(x)), \\
		\text{s.t.} &  \lambda_2(L(x)) + \lambda_3(L(x)) \leq P_{max},\\
			& \sum_{e \in E} x_e \leq q, \\
			& x \in \{0, 1\}^{|E|}.
		\end{array}
\end{equation}

Since the objective of this problem is non-linear, it can be converted to a more tractable but an
equivalent MISDP with the non-linear power consumption constraint
as follows:

\begin{equation}
		\label{eq:ch3_F2}
		\begin{array}{ll}
	  		&\max  \gamma, \\
		\text{s.t.} &  L(x) \succeq \gamma (I_n - e_0 \otimes e_0), \\
			&  \lambda_2(L(x)) + \lambda_3(L(x)) \leq P_{max},\\
			& \sum_{e \in E} x_e \leq q, \\
			& x \in \{0, 1\}^{|E|}.
		\end{array}
\end{equation}

The correctness of this formulation exactly follows the proof given in
section \ref{Subsec:F1}. 

\subsection{Algorithm for determining maximum algebraic connectivity with power consumption constraint}
\label{Sec:cutting}
Currently, efficient tools for solving MISDPs are not available.
In this section, we extend the idea of the algorithm proposed for solving
the \textbf{BP}, which was based on the cutting plane method. 
The basic idea of this method is to find an outer approximation (relaxation) of the 
feasible set of the MISDP problem and solve the optimization problem over the  outer 
approximation (which we refer to as a relaxed problem). If the optimal solution for 
the relaxed problem is feasible for the MISDP problem, it is also clearly optimal for 
the MISDP problem; otherwise, one must refine the outer approximation, e.g., via the 
introduction of additional linear inequalities (referred to as cuts). One may then 
iteratively refine the outer approximation until the optimal solution of the outer 
approximation is also feasible for the MISDP.

We initially tried solving the MISDP problem by relaxing the non-convex power consumption 
constraint and adding linear inequalities to cut off optimal solutions of the relaxed semi-definite programs. 
We used Matlab and state-of-the-art semi-definite solvers such as the Sedumi for this implementation 
and found that it could not handle problems of size greater than 5 nodes. Hence, we opted 
to pose this problem as a MILP problem so that the available
high performance solvers such as CPLEX can be used.

It is important to understand the sources of difficulty when implementing a 
cutting plane algorithm for the MISDP problem under consideration. For every vector
$v \in \Re^n$ such that $v \cdot \mathbf{1} = 0$ and $\|v\|_2 = 1$, the semi-definiteness requirement is equivalent to 
$$\sum_{e \in E} x_e v \cdot L_e v \geq \gamma. $$ 
In essence, for every such vector $v$, there is a linear constraint in $x_e$ and $\gamma$ and 
since the number of vectors $v$ satisfying $v \cdot \mathbf{1} = 0$ and $\|v\|_2 = 1$ is 
uncountably infinite, the semi-definite requirement is equivalent to an uncountable number of linear constraints in 
the discrete variables $x_e, \; e \in E$ and $\gamma$. This constraint makes the use of 
standard ILP tools difficult. One can pick any finite subset of 
these linear constraints to construct a polyhedral outer approximation. 

The non-convex nature of the power consumption constraint in \eqref{eq:ch3_F2} makes 
it difficult to be taken care of directly by the standard ILP tools. For this reason, 
we relax this constraint and provide a method for the construction of ``cuts'' that cut 
off any feasible solution of the relaxed problem that is not feasible for the MISDP.  

The schema for solving  the MISDP  is as follows:
\begin{itemize}
\item {\bf Step 0:} {\bf Initialization:} Pick a finite set of unit vectors, 
		  say $v_i, \; i=1, 2, \ldots, M$ that are perpendicular to $\mathbf{1}$, 
		  and the polyhedral outer approximation, $\mathcal{P}_0$ is the feasible set of the inequalities:
$$\sum_{e \in E } x_e v_i \cdot L_e v_i \geq \gamma, \; \; i=1, 2, \ldots, M, $$
and $x_e \in \{0,1\}, \; e \in E$. 
\item {\bf Step 1:} {\bf Refinement of Polyhedral Outer Approximation:}  This step involves the developing of 
``cuts''. Since the initial polyhedral approximation relaxes the semi-definite constraint and the power consumption constraint \eqref{eq:ch3_F2}, we outline a method to find the linear inequalities
that cut off solutions that are not feasible for either of these constraints.  
\begin{itemize}
\item[a.] {\bf Cut for the semi-definite constraint violation:}  A violation of semi-definite constraint can result in a graph being disconnected. However, we augment the constraints from a multicommodity flow formulation 
\cite{magnanti1995optimal} in order to ensure that the optimal solution of the relaxed problem is not disconnected. 

If the semi-definiteness requirement \eqref{eq:ch3_F2} is violated by the optimal 
solution given by $(x_e^*,\gamma^*), e \in E$, (which we assume is connected now) one 
may readily use the eigenvalue cut, i.e., if
$$\sum_{e \in E} x_e^* L_e - \gamma^* (I_n - e_0 \otimes e_0) \nsucceq 0, $$
then there is at least one eigenvalue of the matrix on the left hand side of the 
above inequality that is negative. Hence, if one were to consider the 
corresponding normalized eigenvector, say $v$, then 
$$  \sum_{e \in E} x_e^* v \cdot L_e v  < \gamma^*, $$
and hence, one may refine the polyhedral outer approximation by augmenting an 
additional constraint that must be satisfied by any feasible solution to MISDP:
\begin{equation}
		  \label{eq:valid_ineq_1}
	\sum_{e \in E} x_e v \cdot L_e v \geq \gamma.
\end{equation}
This additional constraint ensures that the solution $x_e^*, \; e \in E$ 
that was optimal for the relaxed problem will not be feasible now for the 
augmented set of inequalities and the feasible set of the augmented set of 
inequalities is a refined outer approximation. Also, it can be easily proved 
that the inequality \eqref{eq:valid_ineq_1} is a valid inequality for the original
problem based on the variational characterization of the eigenvalues. 

\item[b.] {\bf Cut for the power consumption constraint violation:}
If the constraint on power consumption in \eqref{eq:ch3_F2} is violated by 
$x_e^*$ where $e \in E^* \subset E$, one may 
introduce a constraint requiring that not all the edges of the optimal 
solution may be used, and can introduce a branch according to 
$$\sum_{e \in E^*} x_e \leq \sum_{e \in E^*} x_e^* -1, $$ or 
$$\sum_{e \in E^*} x_e \geq \sum_{e \in E^*} x_e^*+1. $$
Since we seek spanning trees in the numerical examples, we only require 
the former constraint, namely
\begin{equation}
		  \label{eq:valid_ineq_2}
\sum_{e \in E^*} x_e \leq \sum_{e \in E^*} x_e^*-1,
\end{equation}
to be enforced in the algorithm. It can again be easily proved that 
inequality \eqref{eq:valid_ineq_2} is a valid inequality as follows:
Since \eqref{eq:valid_ineq_2} is an inequality on the number of edges in the spanning tree, 
let $\tau_i$ be the $i^{th}$ spanning tree among $n^{n-2}$ possible 
spanning trees and $E_{\tau_i}$ be the edges in $\tau_i^{th}$ spanning tree. Then we know that 
$$ 0 \leq |E_{\tau_i} \cap E_{\tau_j}| \leq n-2  \ \ \ \forall \ i,j=1,\ldots ,n^{n-2}, i \neq j.$$
From this, it is clear that 
$$ \sum_{e \in E_{\tau_i}} x_e \leq n-2$$ uniquely eliminates $\tau_i$ retaining 
all other spanning trees valid. 

\end{itemize} 
\item{\bf Step 2:} Solve the relaxed problem, i.e., solve the optimization problem over the feasible set of 
the refined approximation using ILP solvers to get an updated solution $x_e^*, \; e \in E$ and go to Step 1. 
\end{itemize}
The pseudo code of this procedure is outlined in
Algorithm \ref{Algo:wiring_cost_algo}. The Algorithm \ref{Algo:wiring_cost_algo}
is guaranteed to terminate in finite number of iterations since the number of feasible solutions
for this problem is finite ($n^{n-2}$ for a problem with $n$ nodes). 
The cut for eliminating solutions that do not satisfy the semi-definite constraint is shown in steps 7 through 11 of Algorithm \ref{Algo:wiring_cost_algo}. Step 12 of Algorithm \ref{Algo:wiring_cost_algo}
 corresponds to the cut for eliminating solutions that violate the power consumption constraint.

\begin{algorithm}[h]
\caption{\textbf{:Algorithm for determining maximum algebraic connectivity with power constraint}}
\label{Algo:wiring_cost_algo}
\textit{Notation:} Let $\hat{I} =(I_n - e_0 \otimes e_0)$.

Let $\mathfrak{F}$ denote a set of cuts which must be satisfied by any feasible solution
\vspace{0.1cm}
\begin{algorithmic}[1]
\STATE Input: A graph $G=(V,E)$, a weight ($w_e$) for each edge $e\in E$, $P_{max}$, and a finite number of Fiedler vectors, $v_i, i=1 \ldots M$
\STATE Choose any spanning tree, $x_0$ such that $\lambda_2(L(x_0))+\lambda_3(L(x_0)) > P_{max}$
\STATE $x^* \gets x_0$
\STATE $\mathfrak{F}$$ \gets \emptyset$
\WHILE{$\lambda_2(L(x^*))+\lambda_3(L(x^*)) > P_{max}$}
\STATE Solve:

		  {\begin{equation*}
				  \begin{array}{ll}
						\max & \gamma, \\
						\text{s.t.} & \sum_{e \in E} x_e ({v_i} \cdot L_e {v_i}) \geq {\gamma} \quad \forall i=1,..,M,  \\
					  & \sum_{e \in E} x_e \leq q, \\
					 &\sum_{e \in \delta(S)} x_e \geq 1, \ \ \forall \ S \subset V, \\
					  & x_e \in \{0, 1\}^{|E|}, \\
        			&x_e \ \textrm{satisfies the constraints in }\mathfrak{F}.
				  \end{array}
		  \end{equation*}}

Let $(x^*,\gamma^*)$ be an optimal solution to the above problem.
\IF{$\sum_{e \in E} x_e^* L_e \nsucceq \gamma^* \hat{I}$}
\STATE Find the Fiedler vector $v^*$ corresponding to  $x^*$.
\STATE Augment $\mathfrak{F}$ with a constraint ${v^*}\cdot L(x^*){v^*} \geq \gamma^*$.
\STATE Go to step 6.
\ENDIF
\STATE If $\lambda_2(L(x^*))+\lambda_3(L(x^*)) \nleq P_{max}$, augment $\mathfrak{F}$ with a cut
${\mathbf 1} \cdot x \leq {\mathbf 1} \cdot x^* -1$.
\STATE Go to step 6.

\ENDWHILE
\end{algorithmic}	
\end{algorithm}

\subsection{Performance of proposed algorithm}
\label{Sec:ch3_results2}
In this section, we discuss the computational performance of the proposed
algorithm (Algorithm \ref{Algo:wiring_cost_algo}) to solve 
the problem of maximizing algebraic connectivity 
subject to the consumption constraint. All the computations in this 
section were performed with the same computer specifics as mentioned in 
section \ref{Sec:results1}.

In order to solve the MISDP in step 6 of the algorithm (\ref{Algo:wiring_cost_algo}), 
we used the Sedumi solver in Matlab and found that they could
not handle problems with more than five nodes. However, solving 
the same MISDP using the proposed cutting plane algorithm performed
comparatively better using the CPLEX solver and could handle up to eight nodes though the
computation time was in the order of hours.

Table \ref{Table:exact_algo_n7} shows the optimal solutions for ten random instances 
with seven nodes. For the case of 7 nodes, we assumed $P_{max}$ equal to fifteen 
to ensure that the problem had a feasible solution. From this table, it can be observed that the average
computational time to obtain an optimal solution based on the proposed 
algorithm was around seven hours. 
Even though the algorithm \eqref{Algo:wiring_cost_algo} provides successive tighter polyhedral approximations
with the augmentation of valid inequalities, the convergence to an optimal 
solution is very slow. The reason for the slow convergence can be attributed to the 
strength of the cuts added due to the violation of the power consumption constraint. These
cuts are merely solution elimination constraints which eliminates only the current 
infeasible integral solution. Generating more valid and stronger cuts at every iteration of the 
algorithm can possibly reduce the computation time to obtain optimal solutions. 
A sample network with maximum algebraic connectivity satisfying the
power consumption constraint for instance $\#1$ of Table \ref{Table:exact_algo_n7} is
shown in Figure \ref{Fig:wiring_cost_n7_1}.

\begin{table}[h!]
  \caption{Computational performance of algorithm \eqref{Algo:wiring_cost_algo} to solve
  the problem of maximizing algebraic connectivity subject to the power consumption constraint. 
  $T_1$ corresponds to the CPU time taken by CPLEX solver to 
  solve instances with 7 nodes. Note that $\lambda_2^*+\lambda_3^*$ represents the power incurred
  by each network with optimal connectivity as indicated under $\lambda_2^*$. $P_{max}$ is chosen to 
  be equal to fifteen for all the instances.} 
    \label{Table:exact_algo_n7}
\begin{center}
\begin{tabular}{cccc}
\toprule
\multicolumn{ 1}{c}{\textbf{Instance}} & $\lambda_2^*$ &  $\lambda_2^*+\lambda_3^*$  & $T_1$\\
\multicolumn{ 1}{c}{} & & & (sec)\\
\cmidrule(r){1-4}
1 & 7.1278 & 14.7192 & 12291.28 \\
2 & 7.1457 & 14.9988 & 10350.02 \\
3 & 6.7300 & 14.8166 & 58481.36 \\		
4 & 6.9879 & 14.9829 & 58845.25 \\
5 & 7.2684 & 14.8568 & 28891.07  \\
6 & 6.4437 & 14.9999 & 26543.17 \\
7 & 7.0472 & 14.9261 & 17218.10  \\
8 & 7.0047 & 14.9225 & 27589.19  \\
9 & 7.0526 & 14.6940 & 16413.76  \\
10 & 7.1569 & 14.5328 & 12353.42\\
\cmidrule(r){1-4}
Avg. & & & 26894.41\\
\bottomrule
\end{tabular}
\end{center}
\end{table}

\begin{figure}[htp]
	\centering
	\subfigure[Complete graph for $n=7$]{
	\includegraphics[scale=0.6]{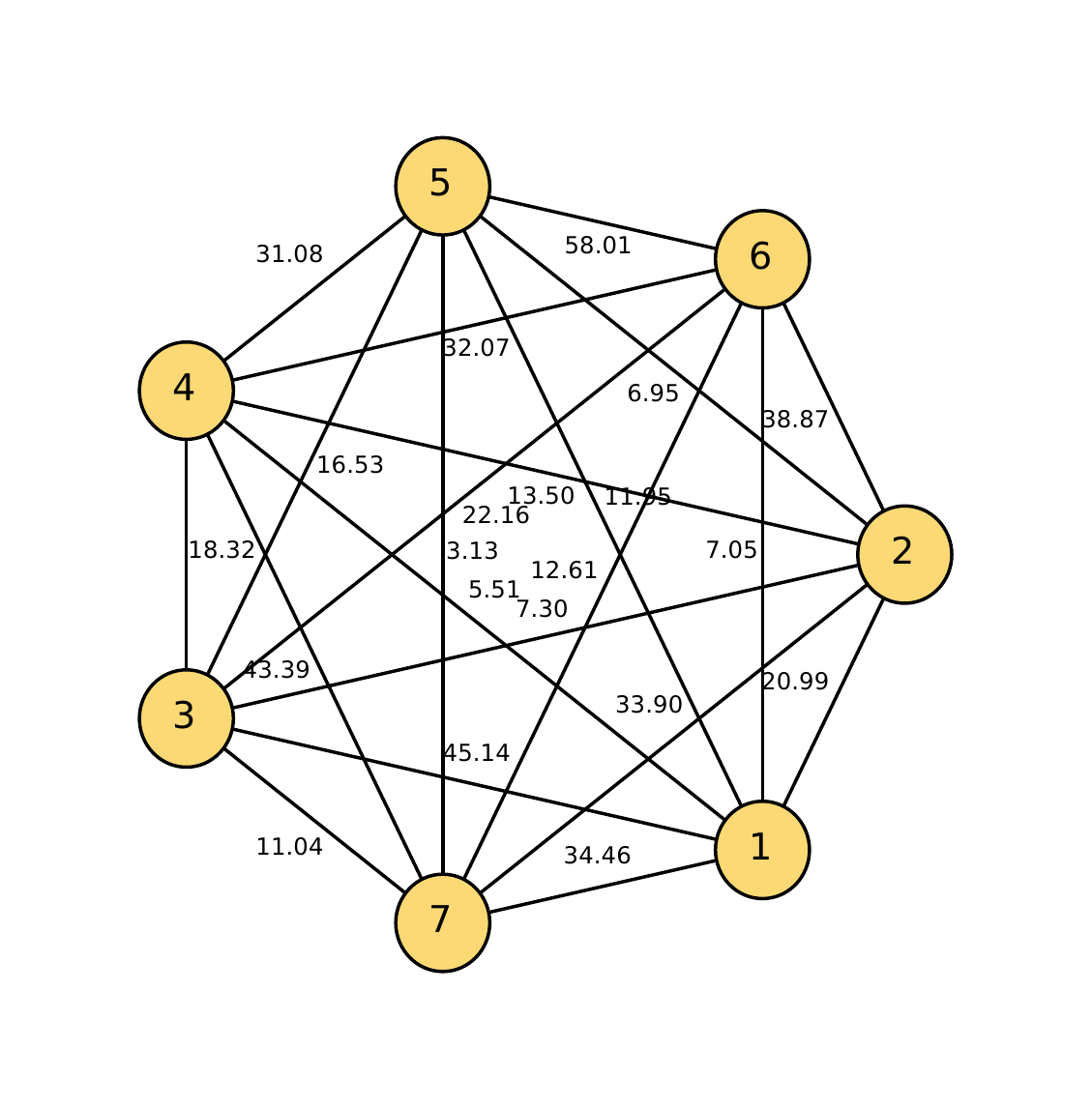}}
	\subfigure[Optimal network for $n=7$]{
			  \includegraphics[scale=0.6]{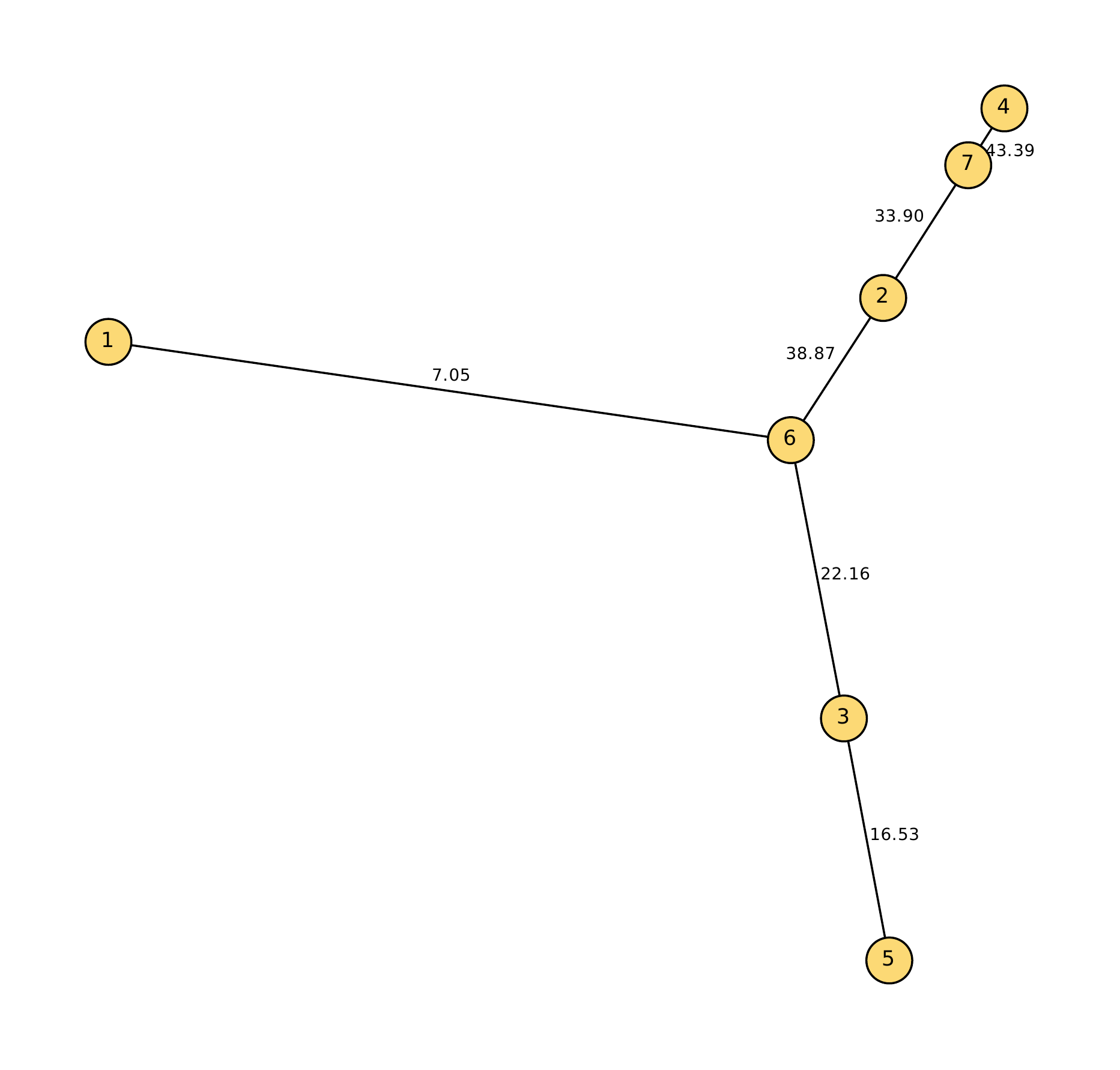}}
	\caption{In this figure, part (a) represents a complete graph of 7 nodes with random
	edge weights. Part (b) represents an optimal network with maximum algebraic 
	connectivity ($\lambda_2^* = 7.1278$) synthesized from the complete 
	graph which satisfies the power consumption constraint 
	($\lambda_2^*+\lambda_3^* \leq 15$). Note that
	the locations of the nodes in (b) are aligned in the second and 
	third eigenvector directions.}
   \label{Fig:wiring_cost_n7_1}
\end{figure}

\subsection{Lower bounds based on the BSDP approach}
\label{Sec:ch3_results3}
In section \ref{Sec:ch3_results2}, we observed that the convergence of the 
algorithm \eqref{Algo:wiring_cost_algo} to optimality was very slow 
and was in the order of many hours. Hence, in this section, we modify the proposed 
algorithm based on the techniques developed in the earlier sections on BSDP approach 
to obtain quick lower bounds and corresponding feasible solutions.

The MISDP in Step 6 of the algorithm (\ref{Algo:wiring_cost_algo}) can also 
be solved to optimality based on the BSDP approach as discussed in section \ref{Sec:BSDP}.
The basic idea of this approach is to solve the feasibility problem where we 
are interested to obtain a network with a specified level of connectivity which 
satisfies the power consumption constraint. 
For completeness, we summarize the algorithm to obtain lower bounds for the problem of 
maximizing algebraic connectivity with power consumption constraint in \eqref{Algo:power_LBnd}.
Clearly, at every iteration of this algorithm, $\hat{\gamma}$ serves as the lower bound whose value 
monotonically increases until the optimality is reached. Correspondingly, $x^*$ at every
iteration serves as the feasible network satisfying the power consumption constraint.  

\begin{algorithm}[h!]
\caption{\textbf{: Lower bounding algorithm (BSDP approach)}}
Let $\mathfrak{F}$ denote a set of cuts which must be
satisfied by any feasible solution
\label{Algo:power_LBnd}
\begin{algorithmic}[1]
		\STATE Input: Graph $G=(V,E,w_e)$, $e\in E$, a root vertex, $r$, $P_{max}$ and a finite number of Fiedler vectors, $v_i, i=1 \ldots M$
		\STATE Choose any spanning tree, $x^*$ such that $\lambda_2(L(x^*))+\lambda_3(L(x^*)) \leq P_{max}$
		\STATE $\hat{\gamma} \gets \lambda_2(L(x^*))$
		\LOOP
		\STATE $\mathfrak{F}$$ \gets \emptyset$
		\STATE Solve:
		  {\begin{equation}
				  \begin{array}{ll}
						\min & \sum_{e \in \delta(r)} x_e, \\
						\text{s.t.} & \sum_{e \in E} x_e ({v_i} \cdot L_e {v_i}) \geq \hat{\gamma} \quad \forall i=1,..,M,  \\
					  & \sum_{e \in E} x_e \leq q, \\
					 &\sum_{e \in \delta(S)} x_e \geq 1, \ \ \forall \ S \subset V, \\
					  & x_e \in \{0, 1\}^{|E|}, \\
        			&x_e \ \textrm{satisfies the constraints in }\mathfrak{F}.
				  \end{array}
		  \end{equation}}

		\IF{ the above ILP is infeasible}
				\STATE \textbf{break loop} \COMMENT{$x^*$ is an optimal solution with maximum algebraic connectivity}
		\ELSE   \STATE Let $x^*$ be an optimal solution to the above ILP. Let $\gamma^*$ and $v^*$ be the algebraic connectivity 
		and the Fiedler vector corresponding to $x^*$ respectively.
		\IF{$\sum_{e \in E} x_e^* L_e \nsucceq \gamma^* (I_n - e_0 \otimes e_0)$}
						\STATE Augment $\mathfrak{F}$ with a constraint $\sum_{e \in E} x_e ({v^*} \cdot L_e {v^*}) \geq {\gamma^*} $.
						\STATE Go to step 6.
				\ENDIF
		\ENDIF
		\IF{$\lambda_2(L(x^*))+\lambda_3(L(x^*)) \nleq P_{max}$} 
		\STATE augment $\mathfrak{F}$ with a cut ${\mathbf 1} \cdot x \leq {\mathbf 1} \cdot x^* -1$. 
		\ELSE
		\STATE $\hat{\gamma} \gets \hat{\gamma} + \epsilon$ \COMMENT{let $\epsilon$ be a small number}
		\STATE Go to step 6.
		\ENDIF
		\ENDLOOP
\end{algorithmic}
\end{algorithm}

Since for this particular problem, we know that $\lambda_2(L(x)) + \lambda_3(L(x)) \leq P_{max}$ and 
$\lambda_2(L(x)) \leq \lambda_3(L(x))$, a trivial upper bound on the algebraic connectivity would be
$$\lambda_2(L(x)) \leq \frac{P_{max}}{2}.$$ Hence, we use this upper bound to corroborate 
the quality of the lower bounds obtained for larger instances. 

\noindent
\textbf{Construction of an initial feasible solution:}
As we discussed in the lower bounding procedure in \eqref{Algo:power_LBnd}, construction of an
initial feasible solution is the first step. For the problem we have considered in this section, 
an initial feasible solution is any spanning tree whose power incurred will be less than a 
given value of upper bound. Though the eigenvalues of the
Laplacian are non-linear functions of the edge weights of the graph, we observed that
the values of $\lambda_2 + \lambda_3$ are reasonably low for spanning trees with relatively low
edge weights as seen in figure \ref{Fig:enumeration_n6}. Since we found that
enumerating a fixed number of spanning trees (10000 trees) starting from a minimum spanning tree
using the algorithm discussed in \cite{gabow1986efficient} was computationally easier, the best initial
feasible solution which satisfies the power consumption constraint was chosen from these enumerated trees.
\begin{figure}[!h]
	\centering
	\includegraphics[scale=0.5]{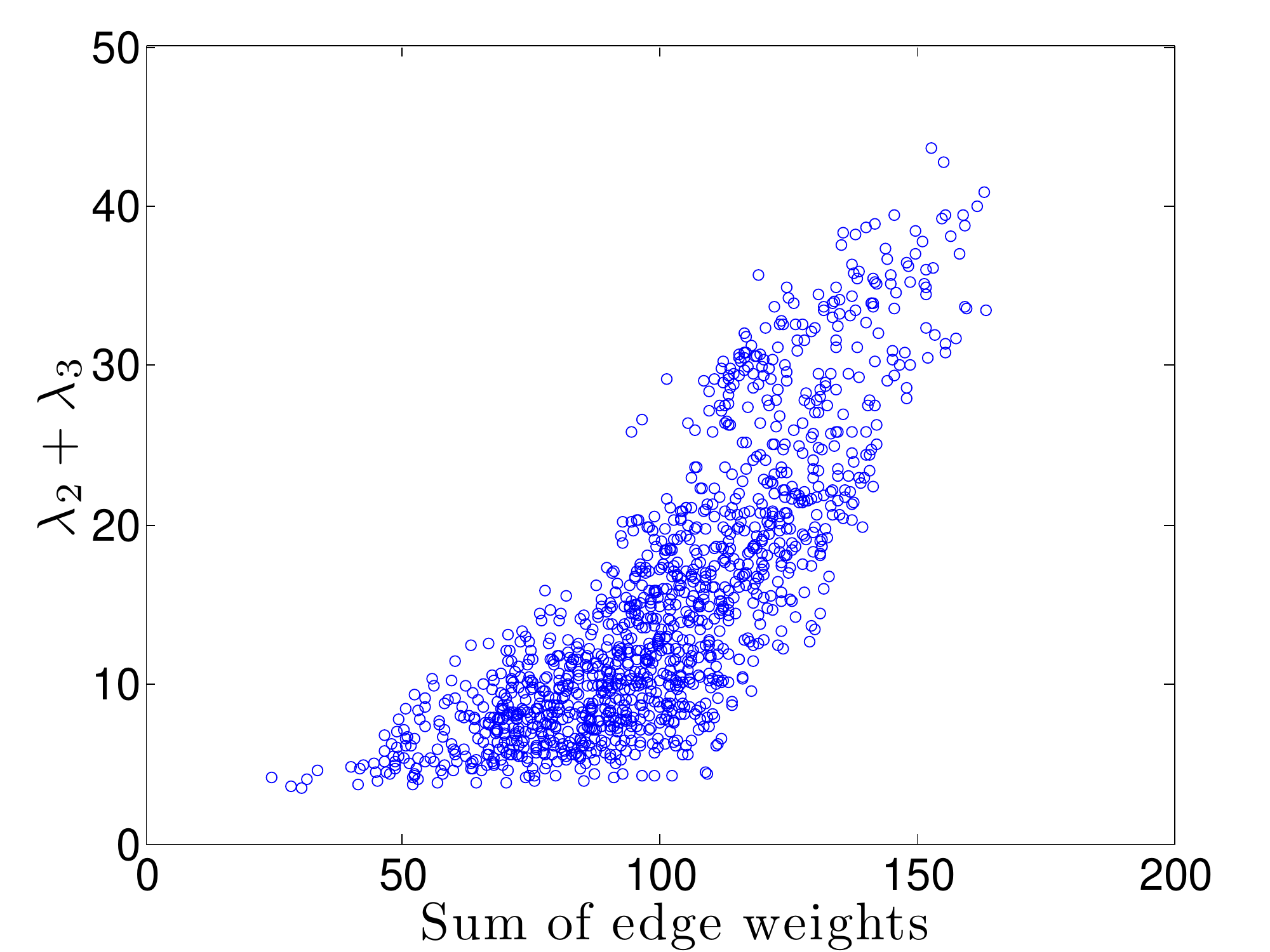}
	\caption{Enumeration of all spanning trees for a random instance with six nodes. It can be observed that
	spanning trees with lesser sum of edge weights incur lesser power consumption.}
	\label{Fig:enumeration_n6}
\end{figure}

\vspace{0.5cm}
\noindent
\textbf{Quality of lower bounds:}
Computationally, we observed that the proposed lower bounding procedure based on the BSDP
approach provided very good quality lower bounds. 
For the seven nodes problem, we limited the computation time of the lower bounding 
algorithm to three minutes. As shown in Table \ref{Table:LBnd_n7}, on an average, the lower 
bound obtained was within 3.5 \% from the optimal solution. This was indeed a tremendous 
improvement in terms of the computation time compared to the algorithm discussed in the 
previous section. Certainly, the computation time of the lower bounding procedure depends 
on the value chosen for $P_{max}$. As expected, higher the value of $P_{max}$, 
larger would be the computation time for the BSDP 
approach since the number of bisection steps would increase. 

We also tested the computational performance of the lower bounding procedure for 
the case of ten nodes with $P_{max}$ equal to thirty. The computation time was limited to 
three minutes. As shown in Table \ref{Table:LBnd_n10}, 
on an average, the lower bound obtained was within 15.2 \% from the upper bound as described earlier. 
Since, this percent gap was with respect to the upper bound, we expect this gap to further reduce when evaluated 
with respect to the optimal solutions.

\begin{table}[h!]
  \caption{	Quality of lower bounds obtained based on the BSDP approach 
for the problem with \emph{seven} nodes. $\lambda_2^{LB}$ represents the 
lower bound obtained by terminating the algorithm \eqref{Algo:power_LBnd} in three minutes. The value of $P_{max}$ is equal to fifteen.} 
    \label{Table:LBnd_n7}
\begin{center}
\begin{tabular}{ccccc}
\toprule
\multicolumn{ 1}{c}{\textbf{Instance}} & \multicolumn{ 2}{c}{\textbf{Optimal solution}} & \multicolumn{ 2}{c}{\textbf{Lower bound}} \\
\cmidrule(l{0.25em}r{0.25em}){2-3}  \cmidrule(l{0.25em}r{0.25em}){4-5}
\multicolumn{ 1}{c}{} & $\lambda_2^*$ &  $\lambda_2^*+\lambda_3^*$  &  ${\lambda_2}^{LB}$ & $\frac{\lambda_2^* - {\lambda_2}^{LB}}{\lambda_2^*}$ \\
\multicolumn{ 1}{c}{} & & & & (\% gap)\\
\cmidrule(r){1-5}
1 & 7.1278 & 14.7192  & 7.1067& 0.3  \\
2 & 7.1457 & 14.9988  & 6.8177& 4.6 \\
3 & 6.7300 & 14.8166  & 6.4897& 3.6 \\		
4 & 6.9879 & 14.9829  & 6.6520& 4.8 \\
5 & 7.2684 & 14.8568  & 7.2608& 0.1 \\
6 & 6.4437 & 14.9999  & 6.3397& 1.6 \\
7 & 7.0472 & 14.9261  & 6.8744& 2.5 \\
8 & 7.0047 & 14.9225  & 6.7295& 3.9 \\
9 & 7.0526 & 14.6940  & 6.2335& 11.6 \\
10 & 7.1569 & 14.5328 & 6.9836& 2.4 \\
\cmidrule(r){1-5}
Avg. & & & & 3.5 \\
\bottomrule
\end{tabular}
\end{center}
\end{table}

\begin{table}[h!]
	\centering
	\caption{Quality of lower bounds obtained based on the BSDP approach 
	for the problem with \emph{ten} nodes. $\lambda_2^{LB}$ represents the 
	lower bound obtained by terminating the algorithm \eqref{Algo:power_LBnd} in three minutes.
 	Note that $\frac{P_{max}}{2}$ is an upper bound on the optimal solution and 
	value of $P_{max}$ is equal to thirty.}
	\label{Table:LBnd_n10}
	\begin{tabular}{cccc} \toprule
			  \textbf{Instances} & ${\lambda_2}^{LB}$  & $\frac{\frac{P_{max}}{2}-{\lambda_{2}}^{LB}}{\frac{P_{max}}{2}}$ & ${\lambda_2}^{LB}$+ ${\lambda_3}^{LB}$  \\
	&  &(\% gap) &  \\
\cmidrule(r){1-4}
	1  & 12.9106 & 14.0 & 29.176   \\
	2  & 12.4978 & 16.7 & 29.204  \\
	3  & 12.3475 & 17.7 & 29.623   \\
	4  & 12.7198 & 15.2 & 29.329   \\
	5  & 12.7786 & 14.8 & 29.393   \\
	6 & 12.2478 & 18.3 & 29.381   \\
	7  & 11.3301 & 24.5 & 28.804   \\
	8  & 14.0376 & 6.4 & 28.632  \\
    9  & 12.7420 & 15.1 & 28.062  \\
	10  & 13.5049 & 10.0 & 29.187   \\ 
\cmidrule(r){1-4}
Avg. && 15.2& \\
	\bottomrule
	\end{tabular}
\end{table}

\subsection{Performance of $2$-opt heuristic}
\label{Sec:ch3_results4}
In this section, we discuss the performance of 2-opt heuristic to obtain feasible 
solutions for the problem of maximizing algebraic connectivity under the power consumption 
constraint. Section \ref{Subsec:kopt} has dealt in detail with the $k$-opt heuristic for 
synthesizing feasible solutions for the \textbf{BP}. Since the extension of this heuristic 
for incorporating an additional constraint on the power consumption is quite straight forward, we do not 
delve into the details. Instead, for completeness, we summarize the 2-opt heuristic
including the additional resource constraint in \eqref{algo:alg_conn_2opt}.

\begin{algorithm}[htp]
    \begin{algorithmic}[1]
    \STATE $\mathcal{T}_0 \leftarrow$ Initial feasible solution satisfying resource constraints
	 \STATE $\lambda_0 \leftarrow$ $\lambda_2(L(\mathcal{T}_0)))$
	 \STATE Input: $P_{max}$
        \FOR{each pair of edges $\{(u_1,v_1),(u_2,v_2)\} \in \mathcal{T}_0$}
            \STATE Let $\mathcal{T}_{opt}$ be the best spanning tree in the 2-exchange neighborhood of $\mathcal{T}_{0}$ obtained by replacing edges $\{(u_1,v_1),(u_2,v_2)\}$ in $\mathcal{T}_{opt}$ with a different pair of edges.
				\IF {$\lambda_2(L(\mathcal{T}_{opt})) > \lambda_0$ and $\lambda_2(L(\mathcal{T}_{opt})) + \lambda_3(L(\mathcal{T}_{opt}))\leq P_{max}$}
                \STATE $\mathcal{T}_0 \leftarrow \mathcal{T}_{opt} $
                \STATE $\lambda_0 \leftarrow \lambda_2(L(\mathcal{T}_{opt}))$
				\ENDIF
        \ENDFOR
    \STATE $\mathcal{T}_0$ is the best spanning tree in the solution space
	 with respect to the initial feasible solution
    \end{algorithmic}
\caption{: 2-opt exchange heuristic}
\label{algo:alg_conn_2opt}
\end{algorithm}

\vspace{0.5cm}
\noindent
\textbf{Solution quality of 2-opt heuristic:}
In Table \ref{Table:wiring_cost_n_7}, we present the solution quality of 
the 2-opt heuristic solutions with respect to the optimal solutions for the
problem instances of seven nodes. We define the solution quality or percent gap as 
$$\frac{\lambda_2^* - \lambda_2^{2opt}}{\lambda_2^*}* 100$$
where $\lambda_2^*$ is the algebraic connectivity of the optimal solution and $\lambda_2^{2opt}$ is the
algebraic connectivity of the 2-opt heuristic solution.
It is clear from Table \ref{Table:wiring_cost_n_7} that the average percent gap of the 2-opt 
solution from the optimal solution was around one percent and many of the 2-opt solutions were indeed optimal.
Also, we empirically observed that the 2-opt heuristic solutions were optimal for 60 \% of
the random (100) instances. Therefore, from this short numerical study, we 
observed that the performance of 2-opt heuristic was phenomenal since it could generate 
feasible solutions within one percent gap from the optimal solutions within a few seconds of
the CPU time.

In Table \ref{Table:2opt_large_n}, we present the scalability of 2-opt 
heuristic solutions for instances up to 25 nodes. Though we could not obtain
optimal solutions for larger instances, as discussed earlier in section \ref{Sec:ch3_results3}, we 
used the trivial upper bound on the optimal $\lambda_2$ which stems from the power consumption constraint.
Hence, the percent gap in Table \ref{Table:2opt_large_n} is given by
$$\frac{\frac{P_{max}}{2} - \lambda_2^{2opt}}{ \frac{P_{max}}{2} }* 100.$$ The average
percent gap was an average value evaluated over ten random instances for each size of the
problem. $P_{max}$ values in Table \ref{Table:2opt_large_n} were chosen randomly such that 
there is existed a feasible solution.
Again we observed that the 2-opt heuristic performed very well for larger instances
and also the average percent gap reduced with the increase in the size of the problem.
Certainly, the percent gap depends on the values of $P_{max}$ chosen, that is, the larger the 
value of $P_{max}$, larger would be the percent gap. 
A sample network satisfying the power consumption constraint found by the 
2-opt heuristic can be seen in figure \ref{Fig:2opt_n25}.

\begin{figure}[htp]
	\centering
	\includegraphics[scale=0.5]{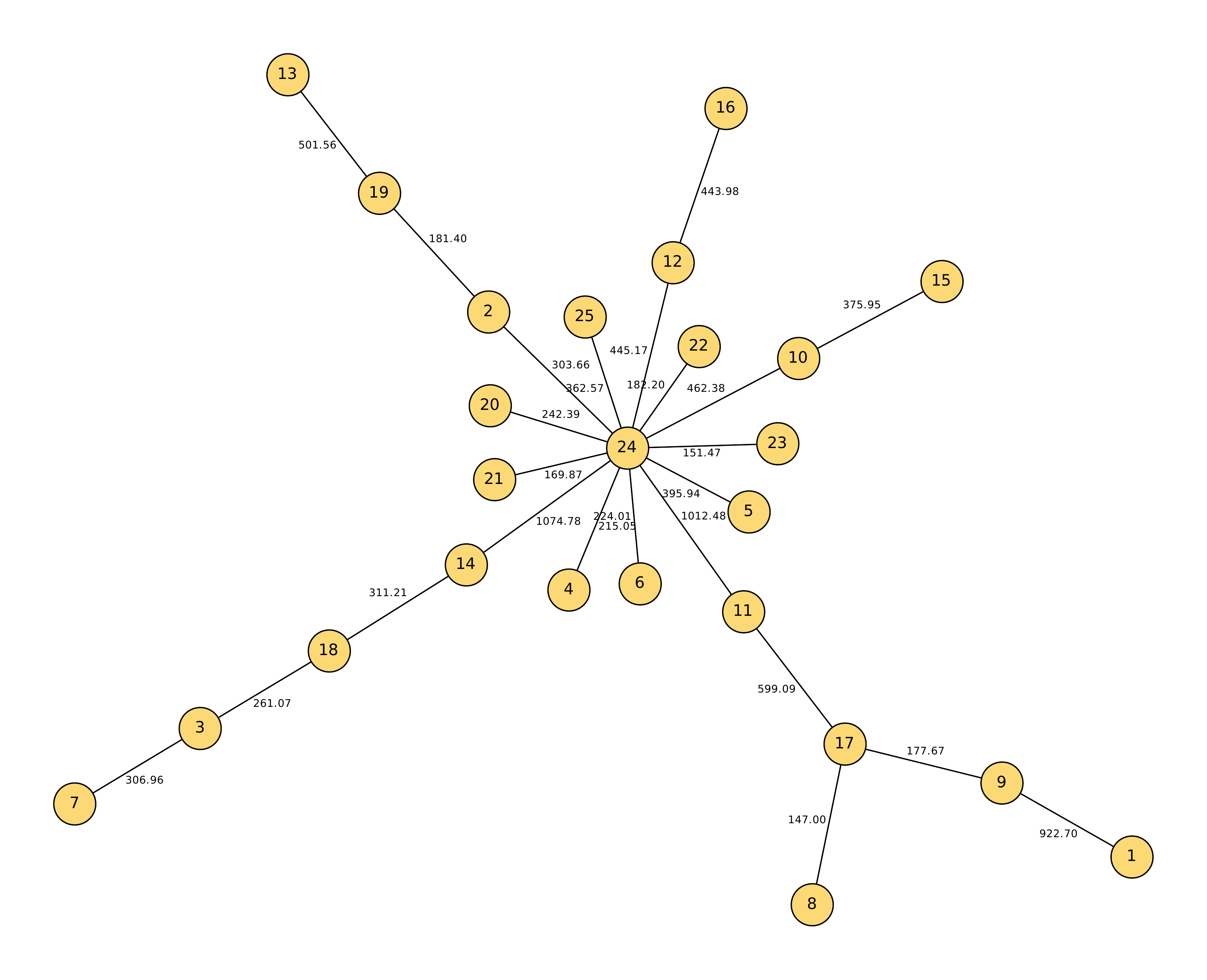}
	\caption{2-opt heuristic solution for a problem with 25 nodes and random edge weights. 
	For the shown network, $\lambda_2^{2opt} = 49.9379$ (percent gap = 0.12) 
	and $P_{max} = 100$. This figure is just a representation of the connectivity of the network and 
	does not necessarily represent the location of nodes.} 
	\label{Fig:2opt_n25}
\end{figure}

\begin{table}[htbp]
  \caption{2-opt heuristic solutions for
  the problem of maximizing algebraic connectivity with power consumption constraint.
  The results in this table are for instances with seven nodes.
	Note that $\lambda_2^*+\lambda_3^*$ represents the total power incurred
  by each network with optimal connectivity as indicated under $\lambda_2^*$}
    \label{Table:wiring_cost_n_7}
\begin{center}
\begin{tabular}{cccccc}
\toprule
\multicolumn{1}{c}{\textbf{Instances}} & \multicolumn{2}{c}{\textbf{Optimal solution}} & \multicolumn{3}{c}{\textbf{2-opt solution}} \\
 \cmidrule(l{0.25em}r{0.25em}){2-3}\cmidrule(l{0.25em}){4-6}
 \multicolumn{1}{c}{} & $\lambda_2^*$ &  $\lambda_2^*+\lambda_3^*$ & $\lambda_2^{2opt}$ & $\lambda_2^{2opt}+\lambda_3^{2opt}$ &\% gap \\
\cmidrule(r){1-6}
1 & 7.1278 & 14.7192  & 7.1278 & 14.7192 & 0.00 \\
2 & 7.1457 & 14.9988  & 7.0880 & 14.9557 & 0.81 \\
3 & 6.7300 & 14.8166  & 6.7300 & 14.8166 & 0.00 \\
4 & 6.9879 & 14.9829  & 6.7837 & 14.9440 & 2.92 \\
5 & 7.2684 & 14.8568  & 7.2684 & 14.8568 & 0.00 \\
6 & 6.4437 & 14.9999  & 6.3817 & 14.9933 & 0.96 \\
7 & 7.0472 & 14.9261  & 7.0472 & 14.9261 & 0.00 \\
8 & 7.0047 & 14.9225  & 7.0047 & 14.9225 & 0.00 \\
9 & 7.0526 & 14.6940  & 6.6520 & 14.0760 & 5.68 \\
10 & 7.1569 & 14.5328  & 7.1569 & 14.5328 & 0.00 \\
\cmidrule(r){1-6}
Avg. & & & & & 1.04 \\
\bottomrule
\end{tabular}
\end{center}
\label{tab:heuristics_9nodes}
\end{table}

\begin{table}[htp]
  \centering
  \caption{2-opt heuristic solutions for the problem of maximizing algebraic
	connectivity with power consumption constraint. Corresponding to every $n$, the
	value of $\lambda_2^{2opt}$ and the percent gap is averaged over ten random instances.}
    \label{Table:2opt_large_n}
    {
  \begin{tabular}{cccc} \toprule
	$n$ & \multicolumn{1}{c}{$P_{max}$} & \multicolumn{1}{c}{Average} & Average\\
	& & \multicolumn{1}{c}{$\lambda_2^{2opt}$} & $\%$ gap \\
	\cmidrule(r){1-4}
	8 & 20 & 9.2964 & 7.04 \\
	9 & 20 & 9.6658 & 3.34 \\
	10 & 20 & 9.8182 & 1.82 \\
	12 & 25 & 12.3150 & 1.48 \\
	15 & 30 & 14.9268 & 0.49 \\
	20 & 50 & 24.8734  & 0.51\\
	25 & 100 & 49.9549 & 0.10\\
	 \bottomrule
  \end{tabular}
  }
\end{table}

%% file: Chapters/section6.tex
%
%
%

\chapter{\uppercase{Conclusions}}
\label{ch4}

In this dissertation\footnote{Other papers of H. Nagarajan whose topics were excluded from this dissertation include 
\cite{nagarajan2011enforcing,nakshatrala2016numerical,nagarajan2017optimal}.}, 
we aimed at understanding the relevance of a simplified 
version of an open problem in system realization theory which has 
several important applications in disparate fields of engineering. The basic
problem in the context of mechanical systems we considered was as follows: 
Given a collection of masses and a set of linear springs with a specified cost and 
stiffness, a resource constraint in terms of a budget on the total cost, the problem 
was to determine an optimal connection of masses and springs so that the resulting 
structure was as stiff as possible. Under certain assumptions, we showed that the 
the structure is stiff when the second non-zero natural frequency of the interconnection 
is maximized. 

We also aimed at understanding the relevance of the variants of this problem 
in deploying UAVs for civilian and military applications. In particular, 
we were interested in synthesizing a communication network among the UAVs 
subject to resource and performance constraints. Some of the important 
resource constraints considered were: limit on the maximum number of
communication links,  power consumed and maximum latency in routing the information
between any pair of UAVs in the network. As a performance objective, we
considered algebraic connectivity (second non-zero eigenvalue of the network's Laplacian) 
as the measure since it determines the convergence rate of consensus 
protocols and error attenuation in UAV formations. 

The mechanical/UAV network synthesis problem, formulated as a 
Mixed Integer Semi-Definite Problem (MISDP), had scarce literature on 
the development of systematic procedures to solve this problem. To address 
this void in the literature, we developed \emph{novel} algorithms to obtain 
optimal solutions and upper bounds 
for moderate sized problems and fast heuristic algorithms 
to obtain good sub-optimal solutions for larger problems. 

We posed the problem of maximizing algebraic connectivity as three 
equivalent formulations: MISDP formulation, MISDP formulation with 
connectivity constraints and Fiedler vector formulation as MILP.
We observed that the binary relaxation of the MISDP formulation 
did not necessarily satisfy the cutset constraints and hence 
invoked the multicommodity flow formulation to ensure that the 
connectivity requirements were satisfied. We also posed this problem 
equivalently as a MILP using the Fiedler vectors of the feasible solutions
since there are not many efficient MISDP solvers available. 
We observed that the binary relaxations of these formulations
were very weak (up to 128 percent deviation from the optimal solution 
for eight nodes problem). However, owing to the various useful features of the 
three equivalent formulations, we developed effective 
methods for obtaining upper bounds and optimal solutions. 

Relaxing the feasible set by outer approximating the semi-definite constraint 
in the MISDP formulation with a finite number of Fiedler vectors would naturally 
lead to an upper bound on the maximum algebraic connectivity. Based on this idea,
we proposed a procedure to effectively enumerate the Fiedler vectors of the 
feasible solutions to obtain tight upper bounds. We observed that the 
upper bound was tighter when the semi-definite constraint was relaxed with 
Fiedler vectors of solutions with higher values of
algebraic connectivity. With thousand such Fiedler 
vectors used for relaxation, the average percent deviation of the 
upper bound from the optimal solution 
was within 4.13\% (best deviation = 0\%) for the eight nodes problem 
and was within 42.9\% (best deviation = 25.6\%) 
for nine nodes problem. For the problem with ten and twelve nodes, 
the average percent deviation of the upper bound from the best known feasible solution 
was within 65.7\% (best deviation = 37.7\%) and 116.1\% (best deviation = 92.1\%) 
respectively. However, the main drawback of this procedure is the 
enumeration of good feasible solutions. If the construction of 
feasible solutions is non-trivial, then this procedure would not be 
effective to obtain tight upper bounds.

We also proposed three cutting plane algorithms to 
solve the proposed MISDP to optimality. 
Firstly, algorithm $EA_1$ 
was based on the construction of successively tighter polyhedral approximations 
of the positive semi-definite set. 
Secondly, iterative primal-dual algorithm $EA_2$ 
considered the Lagrangian relaxation of the semi-definite constraint where 
the primal feasible solution was updated iteratively with a better solution 
obtained by solving the related dual problem. 
Thirdly, algorithm $EA_3$ was based on the Binary Semi-Definite Program (BSDP) approach 
in conjunction with cutting plane and bisection techniques. Computationally, 
we observed that the proposed algorithms implemented in CPLEX performed much better 
than the available MISDP solvers in Matlab. In particular, though the performance of
$EA_1$ and $EA_2$ were comparable, an improved relaxation of the semi-definite constraint
in $EA_1$ with good Fiedler vectors tremendously reduced the computation time to obtain 
optimal solutions. Computationally, $EA_1$ with an improved relaxation of the feasible set performed
at least eight times better than the standard $EA_1$ and $EA_2$ and at least two times 
better than $EA_3$. However, without an a priori knowledge of good Fiedler vectors to relax the feasible set, 
$EA_3$ performed computationally better than $EA_1$ and $EA_2$ for problems up to nine nodes. 
Another useful feature of $EA_3$ was the continually improving lower bound with a
corresponding feasible solution at every bisection step. This was very useful
to obtain quick feasible solutions with a good lower bound for the problem 
of maximizing algebraic connectivity under power consumption constraint.
 
We also developed quick improvement heuristics for the 
problem of maximizing algebraic connectivity based on neighborhood search methods.
In particular, we extended the idea of the well known $k$-opt search 
which has been successfully implemented for traveling salesman problems. $k$-opt 
search aims to iteratively search for better solutions by performing an 
exchange of edges in each iteration. The standard 2-opt (two edges exchanged in every iteration) 
search performed very well and provided optimal solutions for
problems with up to nine nodes. However, for larger problems ($n \geq 15$), owing to the
exponential rise in the number of edge deletion and addition combinations,
standard 2-opt was very slow. Hence, we proposed an improved $k$-opt search 
where the search space was significantly but effectively reduced 
based on the variational characterization of eigenvalues. Computational results suggested 
that the improved 3-opt search performed the best 
while the improved 2-opt search provided a good trade-off between finding good solutions and 
the required computation time.

Finally, we proposed algorithms to address the variants of \textbf{BP} subject to resource 
constraints such as, the diameter constraint and the power consumption constraint.
We posed the problem of maximizing algebraic connectivity of a network as a MISDP
and the diameter of the graph was formulated using a multicommodity
flow formulation. We provided computational results for
problems involving seven and eight nodes under varying limits on the diameter of the graph.
Even though the proposed algorithm was an improvement over state-of-the-art
MISDP solvers, there is definitely a need for faster algorithms
that can handle more number of vertices. 

We posed the problem of maximizing algebraic connectivity of a network as a MISDP
and mathematically formulated the power consumption constraint by relating it to the
second and third eigenvalues of the networks's Laplacian. We proposed 
an algorithm to obtain optimal solutions based on cutting plane method.
Though this algorithm was an improvement over the existing MISDP
solvers, it could handle only smaller instances (up to seven nodes) without much guarantee on the run
time. We employed the BSDP approach in conjunction with the bisection technique 
to obtain quick lower bounds from the associated feasible solutions. 
Terminating the algorithm in three minutes, the average percent 
deviation of the lower bound from the optimal solution for seven 
nodes problem was 3.5\%. For the problem with ten nodes, the lower bound
obtained was within 15.2\% from the upper bound (a simple bound on the optimal algebraic connectivity
stemming from the power consumption constraint). 
Lastly, we applied 2-opt heuristic to find good feasible solutions. For the
seven nodes problem, the average percent deviation of the 2-opt solution 
from the optimal solution was within 1.04\%. For larger problem sizes (up to 25 nodes),
2-opt heuristic performed very well with respect to the values of $P_{max}$
chosen.

%% file: Chapters/appendices.tex
%
%
%

\begin{appendices}

\input{Chapters/appendix1}

\end{appendices}

%% file: Chapters/appendix1.tex
%
%
%



\chapter{APPENDIX}
\label{ch:appendix}
All the computational results in section \ref{sec:exact_algo} on algorithms for computing 
optimal solutions are based on the weighted adjacency matrices shown below.

\noindent
\textbf{Random weighted adjacency matrices for \textit{eight} nodes problem}
\begin{align*}	
		  \scriptsize
  A_{1} =
  \left(\begin{array}{cccccccc}
0 & 4.561 & 19.020 & 37.537 & 82.393 & 18.295 & 50.073 & 5.511 \\ 
4.561 & 0 & 50.358 & 2.819 & 5.916 & 34.933 & 43.855 & 44.377 \\ 
19.020 & 50.358 & 0 & 16.268 & 11.806 & 2.159 & 45.568 & 77.271 \\ 
37.537 & 2.819 & 16.268 & 0 & 28.642 & 45.083 & 62.932 & 24.352 \\ 
82.393 & 5.916 & 11.806 & 28.642 & 0 & 2.590 & 23.840 & 13.704 \\ 
18.295 & 34.933 & 2.159 & 45.083 & 2.590 & 0 & 4.041 & 35.791 \\ 
50.073 & 43.855 & 45.568 & 62.932 & 23.840 & 4.041 & 0 & 55.830 \\ 
5.511 & 44.377 & 77.271 & 24.352 & 13.704 & 35.791 & 55.830 & 0 
    \end{array}\right)
\end{align*}

\begin{align*}	
		  \scriptsize
  A_{2} =
  \left(\begin{array}{cccccccc}
 0 & 7.991 & 19.023 & 40.147 & 46.093 & 9.834 & 48.182 & 39.823 \\ 
7.991 & 0 & 82.412 & 17.293 & 26.714 & 31.590 & 36.865 & 22.808 \\ 
19.023 & 82.412 & 0 & 34.046 & 22.715 & 18.902 & 50.309 & 14.671 \\ 
40.147 & 17.293 & 34.046 & 0 & 25.462 & 10.701 & 51.117 & 34.138 \\ 
46.093 & 26.714 & 22.715 & 25.462 & 0 & 38.596 & 53.231 & 16.664 \\ 
9.834 & 31.590 & 18.902 & 10.701 & 38.596 & 0 & 13.779 & 58.921 \\ 
48.182 & 36.865 & 50.309 & 51.117 & 53.231 & 13.779 & 0 & 53.351 \\ 
39.823 & 22.808 & 14.671 & 34.138 & 16.664 & 58.921 & 53.351 & 0 \\ 
    \end{array}\right)
\end{align*}

\begin{align*}	
		  \scriptsize
  A_{3} =
  \left(\begin{array}{cccccccc}
 0 & 5.449 & 13.087 & 39.460 & 14.189 & 26.056 & 30.279 & 41.788 \\ 
5.449 & 0 & 23.490 & 18.772 & 24.992 & 43.876 & 14.074 & 66.580 \\ 
13.087 & 23.490 & 0 & 13.379 & 44.093 & 11.845 & 45.530 & 65.366 \\ 
39.460 & 18.772 & 13.379 & 0 & 28.403 & 54.327 & 68.801 & 30.908 \\ 
14.189 & 24.992 & 44.093 & 28.403 & 0 & 31.147 & 62.558 & 8.237 \\ 
26.056 & 43.876 & 11.845 & 54.327 & 31.147 & 0 & 21.427 & 78.777 \\ 
30.279 & 14.074 & 45.530 & 68.801 & 62.558 & 21.427 & 0 & 61.276 \\ 
41.788 & 66.580 & 65.366 & 30.908 & 8.237 & 78.777 & 61.276 & 0 
    \end{array}\right)
\end{align*}

\begin{align*}	
		  \scriptsize
  A_{4} =
  \left(\begin{array}{cccccccc}
0 & 3.166 & 10.819 & 69.610 & 7.771 & 35.867 & 47.759 & 11.385 \\ 
3.166 & 0 & 23.452 & 26.608 & 13.743 & 63.817 & 56.875 & 12.734 \\ 
10.819 & 23.452 & 0 & 16.165 & 30.174 & 46.717 & 41.704 & 66.899 \\ 
69.610 & 26.608 & 16.165 & 0 & 5.841 & 57.495 & 67.210 & 14.102 \\ 
7.771 & 13.743 & 30.174 & 5.841 & 0 & 63.502 & 61.732 & 23.618 \\ 
35.867 & 63.817 & 46.717 & 57.495 & 63.502 & 0 & 11.427 & 38.997 \\ 
47.759 & 56.875 & 41.704 & 67.210 & 61.732 & 11.427 & 0 & 98.913 \\ 
11.385 & 12.734 & 66.899 & 14.102 & 23.618 & 38.997 & 98.913 & 0 
    \end{array}\right)
\end{align*}

\begin{align*}	
		  \scriptsize
  A_{5} =
  \left(\begin{array}{cccccccc}
0 & 2.544 & 18.566 & 23.983 & 44.333 & 11.513 & 47.634 & 8.196 \\ 
2.544 & 0 & 17.548 & 20.902 & 29.848 & 56.828 & 16.094 & 45.784 \\ 
18.566 & 17.548 & 0 & 20.030 & 21.883 & 21.306 & 19.583 & 13.961 \\ 
23.983 & 20.902 & 20.030 & 0 & 33.448 & 50.940 & 7.763 & 22.462 \\ 
44.333 & 29.848 & 21.883 & 33.448 & 0 & 60.604 & 57.279 & 7.599 \\ 
11.513 & 56.828 & 21.306 & 50.940 & 60.604 & 0 & 19.492 & 7.163 \\ 
47.634 & 16.094 & 19.583 & 7.763 & 57.279 & 19.492 & 0 & 98.613 \\ 
8.196 & 45.784 & 13.961 & 22.462 & 7.599 & 7.163 & 98.613 & 0 
    \end{array}\right)
\end{align*}

\begin{align*}	
		  \scriptsize
  A_{6} =
  \left(\begin{array}{cccccccc}
0 & 3.368 & 5.354 & 64.684 & 66.925 & 28.203 & 41.094 & 53.284 \\ 
3.368 & 0 & 34.119 & 8.390 & 27.285 & 35.904 & 11.076 & 51.050 \\ 
5.354 & 34.119 & 0 & 33.155 & 33.273 & 28.636 & 34.563 & 59.182 \\ 
64.684 & 8.390 & 33.155 & 0 & 28.884 & 20.305 & 43.513 & 15.110 \\ 
66.925 & 27.285 & 33.273 & 28.884 & 0 & 62.458 & 34.925 & 3.265 \\ 
28.203 & 35.904 & 28.636 & 20.305 & 62.458 & 0 & 4.674 & 27.095 \\ 
41.094 & 11.076 & 34.563 & 43.513 & 34.925 & 4.674 & 0 & 45.437 \\ 
53.284 & 51.050 & 59.182 & 15.110 & 3.265 & 27.095 & 45.437 & 0   
    \end{array}\right)
\end{align*}

\begin{align*}	
		  \scriptsize
  A_{7} =
  \left(\begin{array}{cccccccc}
 0 & 5.721 & 8.828 & 22.020 & 55.966 & 5.384 & 34.178 & 43.546 \\ 
5.721 & 0 & 17.823 & 18.462 & 31.074 & 26.090 & 18.068 & 28.879 \\ 
8.828 & 17.823 & 0 & 23.527 & 25.014 & 48.801 & 40.533 & 53.078 \\ 
22.020 & 18.462 & 23.527 & 0 & 37.835 & 38.275 & 4.024 & 19.766 \\ 
55.966 & 31.074 & 25.014 & 37.835 & 0 & 50.395 & 50.884 & 11.786 \\ 
5.384 & 26.090 & 48.801 & 38.275 & 50.395 & 0 & 12.491 & 35.477 \\ 
34.178 & 18.068 & 40.533 & 4.024 & 50.884 & 12.491 & 0 & 71.750 \\ 
43.546 & 28.879 & 53.078 & 19.766 & 11.786 & 35.477 & 71.750 & 0 
    \end{array}\right)
\end{align*}

\begin{align*}	
		  \scriptsize
  A_{8} =
  \left(\begin{array}{cccccccc}
0 & 1.537 & 12.505 & 45.077 & 68.271 & 6.608 & 20.672 & 37.893 \\ 
1.537 & 0 & 76.166 & 11.996 & 10.903 & 25.450 & 57.973 & 36.482 \\ 
12.505 & 76.166 & 0 & 37.794 & 22.848 & 20.843 & 15.406 & 39.688 \\ 
45.077 & 11.996 & 37.794 & 0 & 37.311 & 29.056 & 36.097 & 27.623 \\ 
68.271 & 10.903 & 22.848 & 37.311 & 0 & 63.989 & 59.293 & 4.220 \\ 
6.608 & 25.450 & 20.843 & 29.056 & 63.989 & 0 & 12.757 & 33.223 \\ 
20.672 & 57.973 & 15.406 & 36.097 & 59.293 & 12.757 & 0 & 105.431 \\ 
37.893 & 36.482 & 39.688 & 27.623 & 4.220 & 33.223 & 105.431 & 0 
    \end{array}\right)
\end{align*}

\begin{align*}	
		  \scriptsize
  A_{9} =
  \left(\begin{array}{cccccccc}
 0 & 7.473 & 13.871 & 74.945 & 59.785 & 28.499 & 36.559 & 41.392 \\ 
7.473 & 0 & 63.104 & 1.118 & 18.255 & 56.460 & 30.670 & 28.415 \\ 
13.871 & 63.104 & 0 & 21.090 & 12.332 & 26.304 & 31.328 & 38.784 \\ 
74.945 & 1.118 & 21.090 & 0 & 34.870 & 35.743 & 13.807 & 6.835 \\ 
59.785 & 18.255 & 12.332 & 34.870 & 0 & 74.240 & 78.291 & 8.182 \\ 
28.499 & 56.460 & 26.304 & 35.743 & 74.240 & 0 & 13.607 & 60.731 \\ 
36.559 & 30.670 & 31.328 & 13.807 & 78.291 & 13.607 & 0 & 100.509 \\ 
41.392 & 28.415 & 38.784 & 6.835 & 8.182 & 60.731 & 100.509 & 0
    \end{array}\right)
\end{align*}

\begin{align*}	
		  \scriptsize
  A_{10} =
  \left(\begin{array}{cccccccc}
0 & 4.673 & 11.233 & 47.921 & 20.123 & 5.275 & 11.570 & 41.965 \\ 
4.673 & 0 & 59.460 & 26.490 & 24.895 & 48.453 & 49.937 & 45.337 \\ 
11.233 & 59.460 & 0 & 20.843 & 21.083 & 33.312 & 3.120 & 56.785 \\ 
47.921 & 26.490 & 20.843 & 0 & 23.790 & 14.368 & 57.961 & 26.491 \\ 
20.123 & 24.895 & 21.083 & 23.790 & 0 & 63.058 & 84.360 & 10.774 \\ 
5.275 & 48.453 & 33.312 & 14.368 & 63.058 & 0 & 6.137 & 37.142 \\ 
11.570 & 49.937 & 3.120 & 57.961 & 84.360 & 6.137 & 0 & 82.681 \\ 
41.965 & 45.337 & 56.785 & 26.491 & 10.774 & 37.142 & 82.681 & 0  
    \end{array}\right)
\end{align*}

\noindent
\textbf{Random weighted adjacency matrices for \textit{nine} nodes problem}
\begin{align*}	
		  \scriptsize
  A_{1} =
  \left(\begin{array}{ccccccccc}
0 & 51.109 & 103.141 & 74.350 & 3.664 & 13.229 & 15.797 & 18.230 & 30.797\tabularnewline
51.109 & 0 & 79.543 & 7.805 & 19.555 & 18.661 & 25.386 & 54.808 & 67.820\tabularnewline
103.141 & 79.543 & 0 & 25.047 & 4.786 & 38.796 & 46.383 & 6.685 & 88.554\tabularnewline
74.350 & 7.805 & 25.047 & 0 & 28.353 & 23.511 & 55.800 & 46.123 & 91.246\tabularnewline
3.664 & 19.555 & 4.786 & 28.353 & 0 & 39.345 & 74.242 & 116.722 & 68.593\tabularnewline
13.229 & 18.661 & 38.796 & 23.511 & 39.345 & 0 & 61.739 & 65.714 & 3.377\tabularnewline
15.797 & 25.386 & 46.383 & 55.800 & 74.242 & 61.739 & 0 & 6.930 & 25.114\tabularnewline
18.230 & 54.808 & 6.685 & 46.123 & 116.722 & 65.714 & 6.930 & 0 & 30.790\tabularnewline
30.797 & 67.820 & 88.554 & 91.246 & 68.593 & 3.377 & 25.114 & 30.790 & 0\tabularnewline
    \end{array}\right)
\end{align*}

\begin{align*}	
		  \scriptsize
  A_{2} =
  \left(\begin{array}{ccccccccc}
0 & 55.451 & 35.171 & 84.885 & 5.505 & 20.855 & 29.453 & 30.388 & 68.093\tabularnewline
55.451 & 0 & 66.881 & 7.059 & 15.901 & 21.996 & 19.397 & 65.193 & 64.995\tabularnewline
35.171 & 66.881 & 0 & 11.186 & 20.066 & 14.621 & 61.816 & 69.104 & 45.769\tabularnewline
84.885 & 7.059 & 11.186 & 0 & 52.853 & 46.755 & 65.175 & 47.878 & 72.586\tabularnewline
5.505 & 15.901 & 20.066 & 52.853 & 0 & 14.783 & 39.858 & 15.650 & 76.328\tabularnewline
20.855 & 21.996 & 14.621 & 46.755 & 14.783 & 0 & 63.148 & 55.653 & 6.730\tabularnewline
29.453 & 19.397 & 61.816 & 65.175 & 39.858 & 63.148 & 0 & 3.325 & 11.846\tabularnewline
30.388 & 65.193 & 69.104 & 47.878 & 15.650 & 55.653 & 3.325 & 0 & 30.335\tabularnewline
68.093 & 64.995 & 45.769 & 72.586 & 76.328 & 6.730 & 11.846 & 30.335 & 0\tabularnewline
    \end{array}\right)
\end{align*}

\begin{align*}	
		  \scriptsize
  A_{3} =
  \left(\begin{array}{ccccccccc}
0 & 34.962 & 106.416 & 83.430 & 3.962 & 19.667 & 16.809 & 40.972 & 23.189\tabularnewline
34.962 & 0 & 62.136 & 6.460 & 21.299 & 20.235 & 55.212 & 32.185 & 65.989\tabularnewline
106.416 & 62.136 & 0 & 3.817 & 19.672 & 35.388 & 38.132 & 65.466 & 31.264\tabularnewline
83.430 & 6.460 & 3.817 & 0 & 23.553 & 67.925 & 56.553 & 49.485 & 72.971\tabularnewline
3.962 & 21.299 & 19.672 & 23.553 & 0 & 64.328 & 41.192 & 98.448 & 78.225\tabularnewline
19.667 & 20.235 & 35.388 & 67.925 & 64.328 & 0 & 64.970 & 78.577 & 6.603\tabularnewline
16.809 & 55.212 & 38.132 & 56.553 & 41.192 & 64.970 & 0 & 5.339 & 19.096\tabularnewline
40.972 & 32.185 & 65.466 & 49.485 & 98.448 & 78.577 & 5.339 & 0 & 16.332\tabularnewline
23.189 & 65.989 & 31.264 & 72.971 & 78.225 & 6.603 & 19.096 & 16.332 & 0\tabularnewline
    \end{array}\right)
\end{align*}

\begin{align*}	
		  \scriptsize
  A_{4} =
  \left(\begin{array}{ccccccccc}
0 & 37.880 & 67.875 & 100.379 & 6.067 & 8.170 & 32.228 & 30.653 & 34.148\tabularnewline
37.880 & 0 & 80.509 & 5.478 & 21.916 & 35.606 & 22.303 & 51.762 & 83.597\tabularnewline
67.875 & 80.509 & 0 & 9.380 & 22.829 & 41.847 & 43.423 & 38.497 & 80.200\tabularnewline
100.379 & 5.478 & 9.380 & 0 & 39.792 & 40.394 & 34.656 & 47.301 & 56.581\tabularnewline
6.067 & 21.916 & 22.829 & 39.792 & 0 & 82.544 & 76.565 & 82.045 & 15.671\tabularnewline
8.170 & 35.606 & 41.847 & 40.394 & 82.544 & 0 & 60.790 & 137.604 & 5.053\tabularnewline
32.228 & 22.303 & 43.423 & 34.656 & 76.565 & 60.790 & 0 & 4.473 & 16.443\tabularnewline
30.653 & 51.762 & 38.497 & 47.301 & 82.045 & 137.604 & 4.473 & 0 & 23.297\tabularnewline
34.148 & 83.597 & 80.200 & 56.581 & 15.671 & 5.053 & 16.443 & 23.297 & 0\tabularnewline
    \end{array}\right)
\end{align*}

\begin{align*}	
		  \scriptsize
  A_{5} =
  \left(\begin{array}{ccccccccc}
0 & 74.434 & 54.004 & 49.180 & 3.425 & 23.574 & 14.026 & 46.679 & 27.705\tabularnewline
74.434 & 0 & 72.725 & 12.127 & 5.963 & 32.650 & 47.091 & 6.417 & 5.720\tabularnewline
54.004 & 72.725 & 0 & 6.283 & 23.399 & 34.824 & 60.464 & 43.262 & 73.479\tabularnewline
49.180 & 12.127 & 6.283 & 0 & 34.042 & 36.230 & 30.105 & 61.880 & 70.808\tabularnewline
3.425 & 5.963 & 23.399 & 34.042 & 0 & 41.705 & 37.664 & 71.445 & 28.397\tabularnewline
23.574 & 32.650 & 34.824 & 36.230 & 41.705 & 0 & 49.857 & 64.825 & 4.676\tabularnewline
14.026 & 47.091 & 60.464 & 30.105 & 37.664 & 49.857 & 0 & 6.500 & 13.624\tabularnewline
46.679 & 6.417 & 43.262 & 61.880 & 71.445 & 64.825 & 6.500 & 0 & 17.935\tabularnewline
27.705 & 5.720 & 73.479 & 70.808 & 28.397 & 4.676 & 13.624 & 17.935 & 0\tabularnewline
    \end{array}\right)
\end{align*}

\begin{align*}	
		  \scriptsize
  A_{6} =
  \left(\begin{array}{ccccccccc}
0 & 64.068 & 10.484 & 82.702 & 5.059 & 17.211 & 41.722 & 51.143 & 34.027\tabularnewline
64.068 & 0 & 38.358 & 13.136 & 19.432 & 8.179 & 36.737 & 43.368 & 44.477\tabularnewline
10.484 & 38.358 & 0 & 4.736 & 19.992 & 35.610 & 68.747 & 66.199 & 100.487\tabularnewline
82.702 & 13.136 & 4.736 & 0 & 26.705 & 69.996 & 24.366 & 62.367 & 68.319\tabularnewline
5.059 & 19.432 & 19.992 & 26.705 & 0 & 6.220 & 42.855 & 100.982 & 54.818\tabularnewline
17.211 & 8.179 & 35.610 & 69.996 & 6.220 & 0 & 71.220 & 79.242 & 6.379\tabularnewline
41.722 & 36.737 & 68.747 & 24.366 & 42.855 & 71.220 & 0 & 4.100 & 13.469\tabularnewline
51.143 & 43.368 & 66.199 & 62.367 & 100.982 & 79.242 & 4.100 & 0 & 20.063\tabularnewline
34.027 & 44.477 & 100.487 & 68.319 & 54.818 & 6.379 & 13.469 & 20.063 & 0\tabularnewline
    \end{array}\right)
\end{align*}

\begin{align*}	
		  \scriptsize
  A_{7} =
  \left(\begin{array}{ccccccccc}
0 & 84.178 & 40 & 94.496 & 3.252 & 19.661 & 19.108 & 56.048 & 40.033\tabularnewline
84.178 & 0 & 110.743 & 7.442 & 15.846 & 35.148 & 24.472 & 52.636 & 21.187\tabularnewline
40 & 110.743 & 0 & 19.213 & 19.871 & 9.568 & 67.812 & 51.830 & 56.333\tabularnewline
94.496 & 7.442 & 19.213 & 0 & 25.484 & 26.072 & 51.210 & 33.379 & 75.134\tabularnewline
3.252 & 15.846 & 19.871 & 25.484 & 0 & 84.510 & 28.471 & 115.426 & 110.035\tabularnewline
19.661 & 35.148 & 9.568 & 26.072 & 84.510 & 0 & 34.251 & 47 & 6.161\tabularnewline
19.108 & 24.472 & 67.812 & 51.210 & 28.471 & 34.251 & 0 & 3.757 & 9.776\tabularnewline
56.048 & 52.636 & 51.830 & 33.379 & 115.426 & 47 & 3.757 & 0 & 31.904\tabularnewline
40.033 & 21.187 & 56.333 & 75.134 & 110.035 & 6.161 & 9.776 & 31.904 & 0\tabularnewline
    \end{array}\right)
\end{align*}

\begin{align*}	
		  \scriptsize
  A_{8} =
  \left(\begin{array}{ccccccccc}
0 & 46.103 & 105.745 & 61.239 & 5.627 & 18.101 & 24.459 & 47.970 & 58.582\tabularnewline
46.103 & 0 & 26.179 & 6.670 & 23.523 & 30.729 & 48.579 & 61.829 & 49.850\tabularnewline
105.745 & 26.179 & 0 & 7.958 & 26.819 & 26.925 & 39.978 & 58.829 & 59.187\tabularnewline
61.239 & 6.670 & 7.958 & 0 & 16.321 & 65.826 & 27.566 & 56.328 & 93.999\tabularnewline
5.627 & 23.523 & 26.819 & 16.321 & 0 & 19.572 & 19.077 & 26.750 & 87.654\tabularnewline
18.101 & 30.729 & 26.925 & 65.826 & 19.572 & 0 & 61.369 & 93.219 & 4.440\tabularnewline
24.459 & 48.579 & 39.978 & 27.566 & 19.077 & 61.369 & 0 & 5.917 & 14.002\tabularnewline
47.970 & 61.829 & 58.829 & 56.328 & 26.750 & 93.219 & 5.917 & 0 & 27.248\tabularnewline
58.582 & 49.850 & 59.187 & 93.999 & 87.654 & 4.440 & 14.002 & 27.248 & 0\tabularnewline
    \end{array}\right)
\end{align*}

\begin{align*}	
		  \scriptsize
  A_{9} =
  \left(\begin{array}{ccccccccc}
0 & 49.524 & 47.001 & 97.199 & 5.978 & 14.685 & 32.794 & 36.858 & 34.073\tabularnewline
49.524 & 0 & 92.926 & 2.890 & 25.966 & 23.370 & 19.856 & 72.462 & 61.856\tabularnewline
47.001 & 92.926 & 0 & 20.361 & 10.725 & 26.954 & 27.427 & 41.613 & 39.825\tabularnewline
97.199 & 2.890 & 20.361 & 0 & 21.730 & 43.284 & 94.839 & 100.156 & 61.767\tabularnewline
5.978 & 25.966 & 10.725 & 21.730 & 0 & 52.430 & 78.240 & 76.468 & 111.859\tabularnewline
14.685 & 23.370 & 26.954 & 43.284 & 52.430 & 0 & 129.134 & 70.418 & 1.984\tabularnewline
32.794 & 19.856 & 27.427 & 94.839 & 78.240 & 129.134 & 0 & 2.300 & 16.639\tabularnewline
36.858 & 72.462 & 41.613 & 100.156 & 76.468 & 70.418 & 2.300 & 0 & 31.418\tabularnewline
34.073 & 61.856 & 39.825 & 61.767 & 111.859 & 1.984 & 16.639 & 31.418 & 0\tabularnewline
    \end{array}\right)
\end{align*}

\begin{align*}	
		  \scriptsize
  A_{10} =
  \left(\begin{array}{ccccccccc}
0 & 73.479 & 78.550 & 57.077 & 2.770 & 16.830 & 27.284 & 13.703 & 45.902\tabularnewline
73.479 & 0 & 96.045 & 7.971 & 14.967 & 18.793 & 24.575 & 22.947 & 58.603\tabularnewline
78.550 & 96.045 & 0 & 9.560 & 14.197 & 14.033 & 69.942 & 64.482 & 71.409\tabularnewline
57.077 & 7.971 & 9.560 & 0 & 44.664 & 40.189 & 51.466 & 33.023 & 67.021\tabularnewline
2.770 & 14.967 & 14.197 & 44.664 & 0 & 65.935 & 86.901 & 96.362 & 119.533\tabularnewline
16.830 & 18.793 & 14.033 & 40.189 & 65.935 & 0 & 90.886 & 50.091 & 3.862\tabularnewline
27.284 & 24.575 & 69.942 & 51.466 & 86.901 & 90.886 & 0 & 3.593 & 20.033\tabularnewline
13.703 & 22.947 & 64.482 & 33.023 & 96.362 & 50.091 & 3.593 & 0 & 14.719\tabularnewline
45.902 & 58.603 & 71.409 & 67.021 & 119.533 & 3.862 & 20.033 & 14.719 & 0\tabularnewline
    \end{array}\right)
\end{align*}